\documentclass{amsart}

\calclayout

\usepackage{stmaryrd}
\usepackage{enumerate}
\usepackage[pdftex]{graphicx}
\usepackage{amsmath}
\usepackage{amsthm}
\usepackage{xcolor}
\usepackage[normalem]{ulem}

\newtheorem{thm}{Theorem}[section]
\newtheorem*{thm*}{Theorem}
\newtheorem{prop}[thm]{Proposition}
\newtheorem*{prop*}{Proposition}
\newtheorem{lem}[thm]{Lemma}
\newtheorem*{lem*}{Lemma}
\newtheorem{cor}[thm]{Corollary}

\theoremstyle{definition}
\newtheorem{definition}[thm]{Definition}

\theoremstyle{remark}
\newtheorem{notation}[thm]{Notation}
\newtheorem{remark}[thm]{Remark}

\numberwithin{equation}{section}

\newcommand{\Span}{\mathrm{Span}} 
 
\newcommand{\Hit}{\mathrm{Hit}} 
\newcommand{\Hom}{\mathrm{Hom}} 

\newcommand{\id}{\mathrm{id}}

\newcommand{\Amc}{\mathcal{A}} 
\newcommand{\Bmc}{\mathcal{B}} 
 
\newcommand{\Dmc}{\mathcal{D}} 
\newcommand{\Emc}{\mathcal{E}} 
\newcommand{\Fmc}{\mathcal{F}} 
 
\newcommand{\Hmc}{\mathcal{H}} 
 
\newcommand{\Jmc}{\mathcal{J}}

\newcommand{\Pmc}{\mathcal{P}} 
\newcommand{\Qmc}{\mathcal{Q}} 
 
\newcommand{\Smc}{\mathcal{S}} 
\newcommand{\Tmc}{\mathcal{T}} 
\newcommand{\Vmc}{\mathcal{V}} 
\newcommand{\Xmc}{\mathcal{X}} 
\newcommand{\Ymc}{\mathcal{Y}}
\newcommand{\Zmc}{\mathcal{Z}} 
 
\newcommand{\Hbbb}{\mathbb{H}} 
\newcommand{\Pbbb}{\mathbb{P}} 
\newcommand{\Rbbb}{\mathbb{R}}

\newcommand{\PGL}{\mathrm{PGL}} 
\newcommand{\PSL}{\mathrm{PSL}} 
\newcommand{\SL}{\mathrm{SL}}
\newcommand{\Gr}{\mathrm{Gr}}

\begin{document}


\title{Flows on the $\PGL(V)$-Hitchin component}

\author{Zhe Sun}
\address{Department of mathematics, University of Luxembourg, Maison du nombre, 6 avenue de la Fonte, L-4364 Esch-sur-Alzette, Luxembourg}
\email{zhe.sun@uni.lu}

\author{Anna Wienhard}
\address{Ruprecht-Karls Universit\"at Heidelberg, Mathematisches Institut, Im Neuenheimer Feld~288, 69120 Heidelberg, Germany
\newline HITS gGmbH, Heidelberg Institute for Theoretical Studies, Schloss-Wolfs\-brunnen\-weg 35, 69118 Heidelberg, Germany }
\email{wienhard@uni-heidelberg.de}

\author{Tengren Zhang}
\address{Mathematics Department, National University of Singapore, 10 Lower Kent Ridge Road, Singapore 119076
}
\email{matzt@nus.edu.sg}

\thanks{ZS was partially supported by the Luxembourg National Research Fund(FNR) AFR bilateral grant COALAS 11802479-2. AW was partially supported by the National Science Foundation under agreements DMS-1536017 and 1566585, by the Sloan Foundation, by the Deutsche Forschungsgemeinschaft, by the European Research Council under ERC-Consolidator grant 614733, and by the Klaus Tschira Foundation. TZ was partially supported by the National Science Foundation under agreements DMS-1536017, and by the NUS-MOE grant R-146-000-270-133. The authors acknowledge support from U.S. National Science Foundation grants DMS 1107452, 1107263, 1107367 "RNMS: GEometric structures And Representation varieties" (the GEAR Network)}

%
%
%
\maketitle

\begin{abstract}
In this article we define new flows on the Hitchin components for $\PGL(V)$. Special examples of these flows are associated to simple closed curves on the surface and give generalized twist flows. Other examples, so called eruption flows, are associated to pair of pants in $S$ and capture new phenomena which are not present in the case when $n=2$. We determine a global coordinate system on the Hitchin component. Using the computation of the Goldman symplectic form on the Hitchin component, that is developed by two of the authors in a companion paper to this article \cite{SunZhang}, this gives a global Darboux coordinate system on the Hitchin component. 
\end{abstract}
\tableofcontents

\section{Introduction} 

Let $S$ be a closed oriented surface of genus at least $2$, and let $\mathcal{H}(S)$ denote the space of marked hyperbolic structures on $S$. By the uniformization theorem $\mathcal{H}(S)$ can be identified with the Teichm\"uller space of $S$, which is a smooth cover of the moduli space of Riemann surfaces. There is a natural symplectic structure on $\mathcal{H}(S)$, given by the Weil-Petersson symplectic form. 
Given a simple closed curve $c$ on $S$ the length function with respect to $c$ is the function on $\mathcal{H}(S)$ which associates to a marked hyperbolic structure the hyperbolic length of the unique geodesic in the free homotopy class of $c$. The Fenchel-Nielsen twist flow associated to the simple closed curve $c$ is one of the simplest flows on $\mathcal{H}(S)$. Geometrically it can be described by cutting $S$ along the curve $c$ and regluing after a twist. Wolpert \cite{Wolpert1,Wolpert2} showed that the Fenchel-Nielsen twist flow associated to a simple closed curve $c$ is precisely the Hamiltonian flow of the length function with respect to $c$. 
Furthermore, he proved that the twist flows associated to non-intersecting simple closed curves have Poisson commuting Hamiltonian functions. In particular, given a maximal family of pairwise non-intersecting simple closed curves, the twist flows generate a Lagrangian submanifold of $\mathcal{H}(S)$. 

A maximal family of pairwise non-intersecting simple closed curves gives a decomposition of $S$ into pairs of pants. Fixing a pair of pants decomposition and a transversal to each pants curve, $\mathcal{H}(S)$ can be parametrized by Fenchel-Nielsen coordinates, which consist of the length functions $l_i$ of the $3g-3$ pants curves, and $3g-3$ twist functions $\theta_i$, which measure the twisting across each pants curve. 
Wolpert \cite{Wolpert1,Wolpert2} gave a beautiful explicit description of the Weil-Petersson symplectic structure in terms the Fenchel-Nielsen coordinates. He showed that the length and twist functions give global Darboux coordinates for Teichm\"uller space, i.e. the Weil-Petersson symplectic form can be written as
\begin{equation}\label{eqn:magic formula}\omega = \sum_{i=1}^{3g-3} d \ell_i \wedge d \theta_i.\end{equation}

Goldman \cite{Goldmansymplectic} showed that for any semisimple Lie group $G$, there is a natural symplectic structure on the set of smooth points on $\Hom(\pi_1(S),G)/G$. We will refer to this symplectic structure as the \emph{Goldman symplectic structure}. He also defined generalized twist flows and showed that they are precisely the Hamiltonian flows associated to generalized length functions on $\Hom(\pi_1(S),G)/G$ \cite{Goldman_twist}. Associating to a marked hyperbolic structure its holonomy provides an embedding of $\mathcal{H}(S)$ into $\Hom(\pi_1(S),\PGL_2(\Rbbb))/\PGL_2(\Rbbb)$ as the connected component consisting entirely of discrete and faithful representations, which is smooth. Goldman then showed that via this embedding, the restriction of the Goldman symplectic form to $\mathcal{H}(S)$ is (up to scaling) the Weil-Petersson symplectic form.

In recent years there has been a lot of interest in studying more general representations varieties $\Hom(\pi_1(S),G)/G$, where $G$ is a semisimple Lie group. For some Lie groups $G$ (of higher rank) there exist connected components in $\Hom(\pi_1(S),G)/G$ that consist entirely of discrete and faithful representations. 
The study of these connected components is the central theme of the burgeoning field of higher Teichm\"uller theory. Hitchin \cite{Hitchin} introduced the first such component, which is now called the $G$-Hitchin component, and proved that it is homeomorphic to a cell of dimension $(2g-2)\dim G$.

In this article, and its companion paper \cite{SunZhang}, we extend the results of Wolpert to the setting of the $\PGL(V)$-Hitchin component $\Hit_V(S)$. 
The main goal of this paper is to define new flows on the Hitchin component. These flows are defined with respect to an ideal triangulation of $S$ and a set of transversals. We give a clean description of these flows in terms of a mild but non-trivial reparametrization of the Bonahon-Dreyer parametrization of $\Hit_V(S)$ \cite{BD1,BD2}. 
The flows we define pairwise commute and provide a trivialization of the tangent bundle of $\Hit_V(S)$. 

When the ideal triangulation and the set of transverals are subordinate to a pair of pants decomposition of $S$, we use a special family of these new flows to construct a new coordinate system on $\Hit_V(S)$ that generalizes Fenchel-Nielsen coordinates on $\mathcal{H}(S)$ in the setting of the $\PGL(V)$-Hitchin component. One of the key new features is that the flows and coordinates are not only associated to simple closed curves, but also to pairs of pants in $S$ given by the pants decomposition.

In the companion paper \cite{SunZhang} it is proved that the flows defined here are in fact Hamiltonian flows, and that the coordinate system we construct here is a global Darboux coordinate system for $\Hit_V(S)$.

\subsection{Frenet curves} 
A key tool we use in our approach is a theorem due to \cite{Labourie2006} and Guichard \cite{Guichard}, which says that there is a canonical embedding of the Hitchin component $\Hit_V(S)$ into the space $\mathrm{Fre}(V)$ of projective classes of Frenet curves. Frenet curves are maps from $S^1$ into the space of (complete) flags $\Fmc(V)$ of $V$, which satisfy strong continuity and transversality properties (see Definition~\ref{def:Frenet}). More precisely, they show that using the identification of the Gromov boundary $\partial \pi_1(S)$ of $\pi_1(S)$ with $S^1$, the $\PGL(V)$-Hitchin component $\Hit_V(S)$ can be identified with projective classes of Frenet curves from $S^1 = \partial \pi_1(S)$ into $\Fmc(V)$ that are $\rho$-equivariant for some representation $\rho:\pi_1(S) \to \PGL(V)$. In the case when $\dim(V)=2$, this is just the classical fact that a representation $\rho:\pi_1(S)\to\PGL_2(\Rbbb)$ is the holonomy of a hyperbolic surface if and only if there is a $\rho$-equivariant, continuous, injective map $\xi:\partial\pi_1(S)\to\Rbbb\Pbbb^1$.

We first introduce two types of flows, the \emph{elementary shearing flows} and the \emph{elementary eruption flows}, on the space of projective classes of Frenet curves $\mathrm{Fre(V)}$, that do not preserve the subset $\Hit_V(S)$. 

The elementary shearing flows generalize the shear along a geodesic in the Poincar\'e disk. They are associated to a pair of distinct points $\mathbf z:=(z_1,z_2)$ in $S^1$ and a pair of positive integers $\mathbf k:=(k_1,k_2)$ that sum to $n$. Given a Frenet curve $\xi:S^1\to\Fmc(V)$, the pair $\mathbf k$ determines a one-parameter family of projective transformations that fix the two flags $\xi(z_1)$ and $\xi(z_2)$. We then apply this one-parameter family of projective transformations to one of the two connected components of $\xi(S^1\setminus\{z_1,z_2\})$, and its inverse to the other connected component (see Section \ref{sec:elementaryshearing} for the precise definition). The elementary shearing flow keeps each component of $\xi(S^1\setminus\{z_1,z_2\})$ projectively invariant, but changes the way the two components are glued together at $\xi(\mathbf z)$. 

The elementary eruption flows are a new feature that only arises when $\dim(V)>2$. They are associated to a triple $\mathbf x:=(x_1,x_2,x_3)$ of pairwise distinct points in $S^1$ and a triple of positive integers $\mathbf i:=(i_1,i_2,i_3)$ that sum to $n$. Given a Frenet curve $\xi:S^1\to\Fmc(V)$, we use $\mathbf i$ to specify three one-parameter families of projective transformations, and deform the three connected components of $\xi(S^1\setminus\{x_1,x_2,x_3\})$ using these three one-parameter families (see Section \ref{sec:elementaryeruption} for the precise definition). The elementary eruption flows change the projective class of the triple of flags $\xi(\mathbf x)$. In the case when $\dim(V) = 3$, these flows have a particularly nice geometric description as changing the gluing parameter of a triple of flags, see \cite{Wienhard_Zhang}.

In order to define flows that preserve the Hitchin component $\Hit_V(S)$ inside of $\mathrm{Fre(V)}$, we fix an ideal triangulation $\Tmc$ of $S$ and a set of transversals $\Jmc$, that we call a compatible bridge system. We lift this to an ideal triangulation $\widetilde{\Tmc}$ of $\widetilde{S}$. Note that any edge of $\widetilde{\Tmc}$ corresponds to a pair of distinct points in $S^1$, and any triangle of $\widetilde{\Tmc}$ correponds to a triple of pairwise distinct points in $S^1$. 
To define a flow on the Hitchin component $\Hit_V(S)$ we thus want to perform an infinite $\pi_1(S)$-invariant family of elementary shearing flows along the edges of $\widetilde{\Tmc}$, and of elementary eruption flows along the ideal triangles of $\widetilde{\Tmc}$. This is a simple idea, but it turns out that formally defining these flows is rather delicate; one has to compose this infinite family of elementary flows in a certain order to ensure convergence. These technical difficulties are addressed in Section~\ref{sec:Hitchin} and Section~\ref{sec:main theorem}. In fact, we construct, for any equivariant Frenet curve $\xi$ and any tangent vector $\mu$ in $T_{[\xi]}\Hit_V(S)$, such a flow whose tangent vector at $[\xi]$ is $\mu$. The flow that we construct this way is called the \emph{$(\Tmc,\Jmc)$-parallel flow associated to $\mu$} (see Definition \ref{def:parallelflow}).

A key ingredient to prove that the $(\Tmc,\Jmc)$-parallel flows are well-defined is that the choice of an ideal triangulation and a set of transversal gives a real analytic parametrization of $\Hit_V(S)$. Such a parametrization was first given by Bonahon and Dreyer \cite{BD1, BD2}, based on work of Fock and Goncharov in the case of surface with punctures \cite{FockGoncharov}. See Section \ref{sec:Bonahon-Dreyer} for more details.

The first main result of the paper can be condensed into the following theorem.

\begin{thm}\label{thm:main_intro}[Theorem \ref{thm:reparametrization},Theorem \ref{thm:main theorem}]
Let $\Tmc$ be an ideal triangulation on $S$ and $\Jmc$ a compatible bridge system. Denote by $\Theta$ the set of ideal triangles of $\Tmc$. Then there is a $(n^2-1)(2g-2)$-dimensional subspace 
\[W_\Tmc\subset\Rbbb^{|\Tmc|(n-1)+|\Theta|\frac{(n-1)(n-2)}{2}},\] 
an open convex polytope $C_\Tmc\subset W_\Tmc$, and a real analytic diffeomorphism
\[\Omega=\Omega_{\Tmc,\Jmc}:\Hit_V(S)\to C_\Tmc\]
suchsuch that that the $(\Tmc,\Jmc)$-parallel flow associated to $\mu$, $\phi_t^\mu:\Hit_V(S)\to\Hit_V(S)$, is given by
\[\phi_t^\mu[\xi]=\Omega^{-1}(\Omega[\xi]+t\mu)\]
where the vector $\mu$ in $T_{[\xi]}\Hit_V(S)$ is viewed as a vector in $W_\Tmc$ via $\Omega$. In particular, any two $(\Tmc,\Jmc)$-parallel flows commute. 
\end{thm}

The parametrization $\Omega$ in the above theorem is a slight reparametrization of the Bonahon-Dreyer parametrization of $\Hit_V(S)$ in \cite{BD1}, where we replace their edge invariants associated to the closed curves in $\Tmc$ with a new invariant that we call the \emph{symplectic closed edge invariants}. 

Theorem \ref{thm:main_intro} implies that $\phi^\mu_t[\xi]$ is a well-defined projective class of Frenet curves in $\Hit_V(S)$, as long as $\Omega[\xi]+t\mu$ lies in $C_\Tmc$. 

More conceptually, Theorem \ref{thm:main_intro} states that any pair $(\Tmc,\Jmc)$ of an ideal triangulation and a compatible bridge system determines a trivialization 
of the tangent bundle $T\Hit_V(S)$ of $\Hit_V(S)$, and the $(\Tmc,\Jmc)$-parallel flows are exactly the flows generated by the vector fields parallel to this trivialization.

As a by-product of the proof of Theorem \ref{thm:main_intro}, we can explicitly specify the projective transformations that govern the $(\Tmc,\Jmc)$-parallel flows (see Section \ref{sec:closedform}). 

\subsection{Coordinates on $\Hit_V(S)$.}
Each choice of an ideal triangulation $\Tmc$ and a compatible bridge system $\Jmc$ determines a trivialization of the tangent bundle $T\Hit_V(S)$ of $\Hit_V(S)$. Via the parametrization $\Omega_{\Tmc,\Jmc}$, any basis of the vector space $W_\Tmc$ determines a family of $(\Tmc,\Jmc)$-parallel flows on $\Hit_V(S)$, with the property that their tangent fields $\{\Xmc_1,\dots,\Xmc_{(n^2-1)(2g-2)}\}$ form a basis of the tangent space to $\Hit_V(S)$ at every point in $\Hit_V(S)$. If one can show that these vector fields are Hamiltonian vector fields and that they give a Darboux basis of the tangent space at every point in $\Hit_V(S)$, then the potential functions of these vector fields will define a Darboux coordinate system on $\Hit_V(S)$. With this in mind, we construct an explicit global coordinate system on $\Hit_V(S)$ as follows.

Choose a pants decomposition on $S$. We then fix an appropriate ideal triangulation that is subordinate to a pair of pants decomposition of $S$, and an appropriate compatible bridge system $\Jmc$, see Section \ref{sec:special1} for more details. With these choices we define four types of \emph{special} $(\Tmc, \Jmc)$-parallel flows, namely the \emph{twist flows}, \emph{length flows}, \emph{eruption flows}, and \emph{hexagon flows}. There are $n-1$ twist flows and $n-1$ length flows associated to each simple closed curve in the pants decomposition, and $\frac{(n-1)(n-2)}{2}$ eruption flows, and $\frac{(n-1)(n-2)}{2}$ hexagon flows associated to each pair of pants in the pants decomposition. 

The twist flows associated to a simple closed curve $c$ do not change the Hitchin representation restricted to $\pi_1(S\setminus c)$ (up to conjugation). The length flows associated to the simple closed curve $c$ change the $n-1$ generalized length functions associated to $c$, but they do not change the Hitchin representation restricted to the fundamental groups of the connected components of $S\setminus (P_1\cup P_2)$, where $P_1$ and $P_2$ are the two pairs of pants that share $c$ as a common boundary component (it is possible that $P_1=P_2$). When $\dim(V)=2$, the twist flow agrees (up to scaling) with the Fenchel-Nielsen twist flow, and the length flow increases the length of the curve it is associated to while keeping the lengths of all the other curves in the pants decomposition (and the way the pairs of pants are glued together) unchanged.

The eruption and hexagon flows associated to a pair of pants $P$ change the Hitchin representation restricted to $\pi_1(P)$, while keeping the Hitchin representation restricted to the fundamental group of each connected component of $S\setminus P$ unchanged. When $\dim V = 2$, the holonomies of the boundary curves uniquely determine the hyperbolic structure on a pair of pants, therefore the eruption and hexagon flows are not present in Teichm\"uller space. They are a new feature arising when $\dim V>2$. In the case when $\dim(V)=3$, the eruption, twist, and hexagon flows were previously described in a more geometric way in \cite{Wienhard_Zhang}.

To any pair of positive integers $\mathbf k:=(k_1,k_2)$ that sum to $n$ and any closed curve in the pants decomposition we can associate a twist flow and a length flow. We denote their tangent fields by $\Smc^\mathbf k_\mathbf c$ and $\Ymc^\mathbf k_\mathbf c$ respectively.
To any triple of positive integers $\mathbf i:=(i_1,i_2,i_3)$ that sum to $n$ and any pair of pants of the pants decomposition, we can associate an eruption flow and a hexagon flow. We denote their tangent fields by $\Emc^\mathbf i_\mathbf x$ and $\Hmc^\mathbf i_\mathbf x$ respectively. We call these tangent vector fields the \emph{special $(\Tmc,\Jmc)$-parallel vector fields}, and set 
\[\Xmc^*:=\left\{\begin{array}{ll}
\Smc^\mathbf k_\mathbf c&\text{if }\Xmc=\Ymc^\mathbf k_\mathbf c;\\
-\Ymc^\mathbf k_\mathbf c&\text{if }\Xmc=\Smc^\mathbf k_\mathbf c;\\
\Emc^\mathbf i_\mathbf x&\text{if }\Xmc=\Hmc^\mathbf i_\mathbf x;\\
-\Hmc^\mathbf i_\mathbf x&\text{if }\Xmc=\Emc^\mathbf i_\mathbf x.
\end{array}\right.\]

\begin{thm}[Corollary \ref{cor:coordinates}, Theorem \ref{thm:Darboux}, Theorem \ref{thm:Darboux2}]\label{thm:coordinates_intro}
For any special $(\Tmc,\Jmc)$-parallel vector field $\Xmc$, there is a real analytic function 
\[H(\Xmc):\Hit_V(S)\to\Rbbb\] 
whose derivative in the direction of $\Xmc^*$ is $1$, and whose derivative in the direction of any other special $(\Tmc,\Jmc)$-parallel vector field is $0$. In particular, the collection of functions
\[\{H(\Xmc):\Xmc\text{ is a special }(\Tmc,\Jmc)\text{-parallel vector field}\}\]
defines a global coordinate system on $\Hit_V(S)$. Furthermore, $H(\Xmc)$ can be described explicitly in terms of the coordinate functions of $\Omega: \Hit_V(S)\to C_\Tmc\subset W_\Tmc$. 
\end{thm}

This coordinate system on $\Hit_V(S)$ generalizes the Fenchel-Nielsen coordinates. Compared to other such coordinate systems, which have been constructed by Goldman \cite{Goldman_convex} when $\dim(V)=3$, and by one of the authors \cite{Zhang_internal, Zhang_thesis} for the general case, this new coordinate system has the additional advantage that it is compatible with the symplectic structure on $\Hit_V(S)$.
In the companion paper \cite{SunZhang}, two of the authors show that any $(\Tmc,\Jmc)$-parallel flow is a Hamiltonian flow, and the vector fields $\Smc^\mathbf k_\mathbf c$ and $\Ymc^\mathbf k_\mathbf c$, and $\Emc^\mathbf i_\mathbf x$ and $\Hmc^\mathbf i_\mathbf x$ are dual to each other with respect to the Goldman symplectic structure. This implies that the coordinate system given in Theorem \ref{thm:coordinates_intro} is a global Darboux coordinate system, and $H(\Xmc)$ is the Hamiltonian function of the integral flow of the vector field $\Xmc$.

\begin{remark}
Theorem \ref{thm:main_intro} suggests that we use the coordinate system $\Omega$ to define the $(\Tmc,\Jmc)$-parallel flows. However, we in fact define them more geometrically and then show that they can be described very naturally in the coordinate system $\Omega$. 
The geometric definition plays an important role 
in order to compute the symplectic pairing between a pair of $(\Tmc,\Jmc)$-parallel vector fields, and to show that $(\Tmc,\Jmc)$-parallel flows are Hamiltonian.
\end{remark}

{\bf Structure of the paper:} 
In Section~\ref{sec:projective_top}, we review definitions of cross ratios, triple ratios, and Frenet curves. The elementary shearing and eruption flows for Frenet curves are introduced in Section~\ref{sec:elementary}. In Section \ref{sec:comb}, we describe a combinatorial description of any pair of vertices in of $\Tmc$, which is a key tool that we use repeatedly in our proofs. In Section~\ref{sec:Hitchin_param}, we describe the Bonahon-Dreyer parametrization of the Hitchin component as well as the (re)parametrization $\Omega$. In Section~\ref{sec:Hitchin}, we define the $(\Tmc, \Jmc)$-parallel flows and state our main results that imply Theorem \ref{thm:main_intro}, before we prove the technical convergence results in Section \ref{sec:main theorem}. Finally, the twist flows, length flows, eruption flows, hexagon flows are defined in Section~\ref{sec:Goldman}, where we also prove Theorem \ref{thm:coordinates_intro}. The appendix contains proofs of a few technical statements.

\section{The Hitchin component and Frenet curves} \label{sec:projective_top}
In this section we recall the definition of the $\PGL(V)$-Hitchin component and of Frenet curves, review some of their properties, and state an important theorem due to Labourie and Guichard relating these two objects. We introduce two important projective invariants, the \emph{cross ratio} and \emph{triple ratio}, and recall the definition of a positive curve, which was used by Fock and Goncharov to give a characterization of the $\PGL(V)$-Hitchin component.

\subsection{The $\PGL(V)$-Hitchin component} 
Let $\Gamma:=\pi_1(S)$ be the fundamental group of a closed, connected, oriented surface $S$ of genus at least $2$. We consider the space $\Hom(\Gamma,\PGL(V))$ of homomorphisms of $\Gamma$ to $\PGL(V)$. Since $\PGL(V)$ is a real algebraic group, $\Hom(\Gamma,\PGL(V))$ is naturally a real algebraic variety. Thus, it has a natural real analytic structure away from its singular points, which is compatible with the compact-open topology on $\Hom(\Gamma,\PGL(V))$.

\begin{definition}\label{def:Hitchin component}
A representation $\rho:\Gamma\to\PGL(V)$ is a $\PGL(V)$-\emph{Hitchin representation} if it is a continuous deformation in $\Hom(\Gamma,\PGL(V))$ (with respect to the compact-open topology) of a faithful representation whose image is discrete and lies in the image of an irreducible representation from $\PSL(2,\Rbbb)$ to $\PGL(V)$.
\end{definition}

\begin{remark}
Since $\PSL(2,\Rbbb)$ is connected, the image of any $\PGL(V)$-Hitchin representation must in fact lie in the identity component of $\PGL(V)$, which is $\PSL(V)$.
\end{remark}

The set $\widetilde{\Hit}_V(S)$ of Hitchin representations form two connected components in $\Hom(\Gamma,\PGL(V))$. 
Under the quotient map 
\[\Hom(\Gamma,\PGL(V))\to\Hom(\Gamma,\PGL(V))/\PGL(V),\] 
$\widetilde{\Hit}_V(S)$ is mapped to a single connected component of $\Hom(\Gamma,\PGL(V))/\PGL(V)$, called the $\PGL(V)$-\emph{Hitchin component}. We denote the $\PGL(V)$-Hitchin component by $\Hit_V(S)$. 

The Hitchin component was introduced by Hitchin \cite{Hitchin}, who, using techniques from the theory of Higgs bundle, proved that $\Hit_V(S)$ is homeomorphic to $\Rbbb^{(2g-2)(n^2-1)}$. 
In the case when $\dim(V)=2$, $\Hit_V(S)$ is exactly the space of conjugacy classes of faithful representations from $\Gamma$ to $\PGL(V)$ with discrete image. These representations are precisely holonomy representations of hyperbolic structures on $S$. Thus, in this case, $\Hit_V(S)$ can be identified with the space $\mathcal{H}(S)$ of marked hyperbolic structures on $S$. 

\begin{remark}
We work with $\PGL(V)$ instead of $\PGL(n,\mathbb{R})$ in order to avoid fixing a basis of $V$. We consider the Hitchin component as a connected component of the representation variety $\mathrm{Hom}(\Gamma, \PGL(V))/\PGL(V)$. If prefered, the reader can also consider the representation variety $\mathrm{Hom}(\Gamma, \PSL(V))/\PSL(V)$. Then, when $n$ is even, there are two connected components homeomorphic to $\Hit_V(S)$, and our results apply to each component separately. 
\end{remark}

\subsection{Frenet curves} 
Let $V$ be an $n$-dimensional real vector space. A (complete) \emph{flag} $F$ in $V$ is a sequence of properly nested subspaces $F^{(1)}\subset$ $\dots$ $\subset F^{(n-1)}$, where $\dim F^{(i)}=i$. We denote the set of (complete) flags in $V$ by $\Fmc(V)$. 
The space $\Fmc(V)$ can be equipped with a real-analytic structure such that the inclusion $\Fmc(V)\to\prod_{i=1}^{n-1}\Gr_i(V)$ given by $F\mapsto (F^{(1)},\dots,F^{(n-1)})$ is a real-analytic embedding. 

\begin{definition}\
A $k$-tuple of flags $(F_1,\dots,F_k)$ is said to be \emph{generic} if for all positive integers $i_1$, $\dots$, $i_k$ that sum to $n$, one has $F_1^{(i_1)}+\dots+F_k^{(i_k)}=V$. 
The set of generic $k$-tuples of flags is denoted by $\Fmc(V)^{[k]}$. When $k=2$, a pair of generic flags is also called a \emph{transverse} pair of flags.
\end{definition}

\begin{remark}
Note that requiring a triple of flags to be generic is strictly stronger than requiring the triple to be pairwise transverse. Most of our later constructions involving triples of flags do not work for a pairwise transverse triple of flags; we do indeed need that the triple of flags is generic.
\end{remark}

\begin{definition}\label{def:Frenet}
A continuous curve $\xi:S^1\to\Fmc(V)$ is said to be \emph{Frenet} if the following conditions hold for all positive integers $n_1$, $\dots$, $n_k$ that sum to $d\leq n$.
\begin{enumerate}
\item For all pairwise distinct points $x_1$, $\dots$, $x_k$ in $S^1$ , the subspace 
\[\sum_{j=1}^k\xi^{(n_j)}(x_j)\subset V\] 
is of dimension $d$.
\item Let $x$ be a point in $S^1$. For all sequences $\{(x_{i,1},\dots,x_{i,k})\}_{i=1}^\infty$ of pairwise distinct $k$-tuples in $S^1$ such that $\lim_{i\to\infty}x_{i,j}=x$ for all $j=1$, $\dots$, $k$, we have 
\[\lim_{i\to\infty}\sum_{j=1}^k\xi^{(n_j)}(x_{i,j})=\xi^{(d)}(x).\]
\end{enumerate}
The space of Frenet curves from $S^1$ to $\Fmc(V)$ is denoted by $\widetilde{\mathrm{Fre}}(V)$.
\end{definition}

Note that every Frenet curve is injective, and moreover sends every k-tuple of pairwise distinct points in $S^1$ to a generic k-tuple of flags. 

The space $\widetilde{\mathrm{Fre}}(V)$ can be equipped with the topology of uniform convergence. With this topology, the continuous action of $\PGL(V)$ on $\Fmc(V)$ induces a continuous action of $\PGL(V)$ on $\widetilde{\mathrm{Fre}}(V)$. We denote the quotient space $\widetilde{\mathrm{Fre}}(V)/\PGL(V)$ by $\mathrm{Fre}(V)$, and equip $\mathrm{Fre}(V)$ with the quotient topology.

\begin{remark}\label{rem:transitive}
We denote the projective classes of non-zero vectors in $V$ by $\Pbbb(V)$, and identify this with the space of lines through the origin in $V$. Similarly, we identify the space of hyperplanes through the origin in $V$ with $\Pbbb(V^*)$, the space of projective classes of non-zero covectors in $V^*$. 

It is an elementary linear algebra argument that $\PGL(V)$ acts transitively and freely on the set
\[\{(F,G,P)\in\Fmc(V)^{[2]}\times\Pbbb(V):F^{(i)}+G^{(n-i-1)}+P=V\text{ for all }i=0,\,\dots,\,n-1\}.\] 
In particular, if we choose a generic triple of flags $F,G,H$ in $\Fmc(V)$ and a triple of distinct points $x,y,z$ in $S^1$, then any projective class of Frenet curves $[\xi]$ in $\mathrm{Fre}(V)$ has a unique representative $\xi$ such that $\xi(x)=F$, $\xi(y)=G$ and $\xi^{(1)}(z)=H^{(1)}$. We make use of this normalization throughout the paper. 
\end{remark}

Labourie \cite{Labourie2006} proved that for every $\PGL(V)$-Hitchin representation $\rho$, there exists a unique $\rho$-equivariant Frenet curve $\xi:\partial\Gamma\to\Fmc(V)$. (Here $\partial \Gamma$ is the Gromov boundary of $\Gamma$. It ss topologically a circle.) Guichard \cite{Guichard} later proved the converse to this statement. Their combined work gives the following result that is crucial for this paper. 

\begin{thm} \cite[Theorem 1.4]{Labourie2006},\cite[Theorem 1]{Guichard}\label{thm:Labourie,Guichard}
Let $\rho:\Gamma\to\PGL(V)$ be a representation. Then the conjugacy class $[\rho]$ lies in $\Hit_V(S)$ if and only if there is a $\rho$-equivariant Frenet curve $\xi_\rho:\partial\Gamma\to\Fmc(V)$. Moreover, if such a Frenet curve exists, then it is necessarily unique. Furthermore, if $\xi_\rho=\xi_{\rho'}$, then $\rho=\rho'$. In particular, the Hitchin component $\Hit_V(S)$ is naturally emdedded into $\mathrm{Fre}(V)$ as the locus of Frenet curves that are $\rho$-equivariant for some representation $\rho:\Gamma\to\PGL(V)$.
\end{thm}

\subsection{Projective invariants}\label{sec:projective}
We recall the definition of the cross ratio and triple ratio, as well as some of their basic properties.

\subsubsection{Cross ratio} We begin with the definition of the cross ratio.
\begin{definition}
Let $P_1,P_2\in\Pbbb(V)$ be two lines in $V$ and let $K_1,K_2\in\Pbbb(V^*)$ be two hyperplanes in $V$ such that $P_i$ is not contained in $K_j$ for all $i$, $j$. For $j=1$, $2$ choose a non-zero vector $v_j$ in $P_j$ and a non-zero covector $\alpha_j$ such that $\ker(\alpha_j) = K_j$. Then the \emph{cross ratio} of $(K_1,P_1,P_2,K_2)$ is
\[C(K_1,P_1,P_2,K_2):=\frac{\alpha_1(v_2)\alpha_2(v_1)}{\alpha_1(v_1)\alpha_2(v_2)}.\]
\end{definition}
Note that the cross ratio $C(K_1,P_1,P_2,K_2)$ does not depend on the choices of $v_j$ and $\alpha_i$ for $i$, $j =1$, $2$.

The cross ratio is a projective invariant. For this recall that a projective transformation $g$ in $\PGL(V)$ acts on $\Pbbb(V)$ by $g\cdot [v] = [g\cdot v]$, and on $\Pbbb(V^*)$ by $g\cdot[\alpha] = [ \alpha \circ g^{-1}]$. Thus the definition implies immediately that
\[C(K_1,P_1,P_2,K_2)=C(g\cdot K_1,g\cdot P_1,g\cdot P_2,g\cdot K_2).\]
The following symmetries of the cross ratio are also direct consequences of the definition:
\[C(K_1,P_1,P_2,K_2)=C(K_2,P_2,P_1,K_1),\] 
\[C(K_1,P_1,P_2,K_2)\cdot C(K_1,P_2,P_3,K_2)=C(K_1,P_1,P_3,K_2),\]
and 
\[C(K_1,P_1,P_2,K_2)\cdot C(K_1,P_2,P_1,K_2)=1,\]
whenever the expressions are defined. 

We apply the cross ratio most commonly in the following setting. Let $\mathbf H:=(F_1,G_1,F_2,G_2)$ be a quadruple of flags in $\Fmc(V)$ such that $(F_1,G_1,F_2)$ and $(F_1,F_2,G_2)$ are both generic triples of flags, and let $\mathbf k:=(k_1,k_2)$ be a pair of positive integers that sum to $n$. Then define
\[C^{\mathbf k}(\mathbf H):=C\big(F_1^{(k_1)}+F_{2}^{(k_2-1)},G_1^{(1)},G_2^{(1)},F_1^{(k_1-1)}+F_2^{(k_2)}\big).\]

Observe that if $\mathbf k_m:=(k_m,k_{m+1})$ and $\mathbf H_m:=(F_m,G_m,F_{m+1},G_{m+1})$ for $m=1$, $2$, where the arithmetic in the subscripts are done modulo $2$, then 
\[C^{\mathbf k_1}(\mathbf H_1)=C^{\mathbf k_{2}}(\mathbf H_2).\]

\subsubsection{Triple ratio} 
We now introduce the triple ratio. Suppose first that $\dim(V) =3$. Let $(F_1, F_2, F_3)$ be a generic triple of flags in $\Fmc(V)^{[3]}$. Write $F_i = (P_i, K_i)$, where $P_i \in \Pbbb(V)$ is a line in $V$ and $K_i \in \Pbbb(V^*)$ a hyperplane in $V$ containing $P_i$. For $i=1$, $2$, $3$, choose non-zero vectors $v_i$ in $P_i$ and non-zero covectors $\alpha_i$ such that $\ker(\alpha_i) = K_i$. Then $\alpha_i(v_i) = 0$, but by genericity $\alpha_i(v_j) \neq 0 $ if $i\neq j$. 
The triple ratio of $(F_1, F_2, F_3)$ is then defined by 
\[T(F_1, F_2, F_3)=\frac{\alpha_1(v_2)\alpha_2(v_3)\alpha_3(v_1)}{\alpha_1(v_3)\alpha_3(v_2)\alpha_2(v_1)},\] 
which is independent of the choices of vectors and covectors. In this case, the triple ratio determines generic triples of flags up to projective equivalence.

The definition of the triple ratio generalizes to the case when $\dim(V)\geq 3$ in the following way. 

\begin{definition}
Let $K_1$, $K_2$, $K_3\in\Pbbb(V^*)$ be hyperplanes in $V$ and let $P_1$, $P_2$, $P_3\in\Pbbb(V)$ be lines such that for all $i=1$, $2$, $3$, $P_i$ is contained in $K_i$, but not in $K_{i-1}\cup K_{i+1}$ (here, arithmetic in the subscripts are done modulo $3$). Choose non-zero vectors $v_i$ in $P_i$ and non-zero covectors $\alpha_i$ such that $\ker(\alpha_i) = K_i$. Then the \emph{triple ratio} of $(K_1,P_1,K_2,P_2,K_3,P_3)$ is 
\[T(K_1,P_1,K_2,P_2,K_3,P_3)=\frac{\alpha_1(v_2) \alpha_2(v_3)\alpha_3(v_1)}{\alpha_1(v_3) \alpha_3(v_2) \alpha_2(v_1)}.\]
\end{definition}

The triple ratio is a projective invariant: 
\[T(g\cdot K_1,g\cdot P_1,g\cdot K_2, g\cdot P_2, g\cdot K_3,g \cdot P_3) = T(K_1,P_1,K_2,P_2,K_3,P_3)\] for all projective transformations $g$ in $\PGL(V)$. It also satisfies the symmetry 
\[T(K_1,P_1,K_2,P_2,K_3,P_3)=T(K_2,P_2,K_3,P_3,K_1,P_1)=\frac{1}{T(K_3,P_3,K_2,P_2,K_1,P_1)}.\]

We apply the triple ratio to give projective invariants of generic triples of flags as follows. Let $\mathbf F:=(F_1, F_2, F_3)$ be a generic triple of flags in $\Fmc(V)^{[3]}$, and let $\mathbf i:=(i_1,i_2,i_3)$ be a triple of positive integers that sum to $n$. For all $m=1$, $2$, $3$, let 
\[K_m:=F_m^{(i_m+1)}+F_{m+1}^{(i_{m+1}-1)}+F_{m-1}^{(i_{m-1}-1)},\] 
and let $P_m$ be a line such that $P_m+F_m^{(i_m-1)}=F_m^{(i_m)}$. Clearly, $P_m$ lies in $K_m$, and the genericity of $\mathbf F$ also implies that $P_m$ does not lie in $K_{m-1}$ and $K_{m+1}$. Then define
\[T^{\mathbf i}(\mathbf{F}) :=T(K_1,P_1,K_2,P_2,K_3,P_3).\]
Observe that $T^{\mathbf i}(\mathbf F)$ does not depend on the choice of $P_m$. Also, if we set $\mathbf i_m:=(i_m,i_{m+1},i_{m-1})$ and $\mathbf F_m:=(F_m,F_{m+1},F_{m-1})$ for $m=1$, $2$, $3$, then 
\[T^{\mathbf i_1}(\mathbf F_1)=T^{\mathbf i_2}(\mathbf F_2)=T^{\mathbf i_3}(\mathbf F_3).\]

Using ideas from Fock-Goncharov \cite{FockGoncharov}, one can prove that the triple ratios in fact give a parametrization of the space $\Fmc(V)^{[3]}/\PGL(V)$ of generic triples of flags in $V$ considered up to projective transformations (see \cite[Lemma 2.3.7]{Zhang_thesis}). 

\begin{prop}\label{prop:parameterize triple}
Let $\Bmc$ be the set of triples of positive integers that sum to $n$. 
The map 
\[\Fmc(V)^{[3]}/\PGL(V)\to(\Rbbb\setminus\{0\})^{\frac{(n-1)(n-2)}{2}}\] 
given by $[\mathbf F]\mapsto\big(T^{\mathbf i}(\mathbf F)\big)_{\mathbf i\in\Bmc}$ is a real-analytic diffeomorphism.
\end{prop}

\subsection{Positive maps}
In this section we briefly discuss the notion of positive maps, introduced by Fock and Goncharov \cite{FockGoncharov} and discuss its relation to Frenet curves. 
The notion of positive maps relies on the notion of totally positive unipotent matrices. 

\begin{definition}
Let $B:=\{f_1,\dots,f_n\}$ be a basis of $V$. A projective transformation $u$ in $\PGL(V)$ is a unipotent projective transformation is \emph{totally positive with respect to $B$} if in the basis $B$, it is represented by an upper-triangular matrix with ones on the diagonal, such that all the minors are positive except those that are forced to be zero by $u$ being upper-triangular.
\end{definition}

\begin{remark}Unipotent here means that $u_m$ has a representative that is conjugate to an upper triangular matrix with only $1$'s along the diagonal.\end{remark}

\begin{definition}
A $k$-tuple of flags $(F_1,\dots F_k)$ in $\Fmc(V)$ ($k\geq 3$) is said to be \emph{positive} if $F_1$ and $F_2$ are transverse, and for some fixed basis $B=\{f_1,\dots,f_n\}$ of $V$ such that $f_i$ lies in $F_1^{(i)}\cap F_2^{(n-i+1)}$ for all $i=1$, $\dots$, $n$, 
\begin{itemize}
\item there are projective transformations $u_1$, $\dots$, $u_{k-2}$ in $\PGL(V)$ that are totally positive with respect to $B$, and
\item there is some projective transformation $g$ in $\PGL(V)$ that fixes both $F_1$ and $F_2$,
\end{itemize}
such that 
\[g\cdot (F_1,F_2,F_3,\dots,F_k)=(F_1,F_2,F_3',\dots,F_k'),\]
where $F_i':=\left(u_1u_2\dots u_{i-2}\right)\cdot F_2$ for all $i=3$, $\dots$, $k$. 
We denote the space of positive $k$-tuples of flags in $\Fmc(V)$ by $\Fmc(V)_+^k$.
\end{definition}

The notion of a positive $k$-tuple of flags does not depend on the choice of basis $B$. Fock and Goncharov \cite[Theorem 1.2, Section 5.5--5.11]{FockGoncharov} proved that if $(F_1,\dots,F_k)$ is a positive $k$-tuple of flags, then so is $(F_2,\dots,F_k,F_1)$.

\begin{remark}\label{rem:remove}
If $(F_1,F_2,\dots,F_k)$ is a positive $k$-tuple of flags, then for any $1\leq i_1<i_2<\dots<i_l\leq k$ with $l\geq 3$, $(F_{i_1},\dots,F_{i_l})$ is also a positive $l$-tuple of flags.
\end{remark}

To define the notion of positive maps, we endow $S^1$ with a clockwise orientation.

\begin{definition}\label{def:cyclically oriented}
A $k$-tuple of pairwise distinct points $(x_1,\dots,x_k)$ in $S^1$ is \emph{cyclically ordered} if $x_1<\dots<x_k<x_1$, where $<$ denotes the cyclic ordering on $S^1$.
\end{definition}

\begin{definition}\cite{FockGoncharov} \label{def:positive}
A map $\phi:S^1\to\Fmc(V)$ is \emph{positive} if for any $k\geq 3$ and any cyclically ordered $k$-tuple of points $(x_1,\dots,x_k)$ in $S^1$, we have that $(\phi(x_1),\dots,\phi(x_k))$ is a positive $k$-tuple of flags. 
\end{definition}

Fock and Goncharov proved the following characterization of a positive triple of flags (this is a consequence of \cite[Theorem 9.1(a)]{FockGoncharov}).

\begin{thm}\cite{FockGoncharov}
Let $\mathbf F:=(F_1,F_2,F_3)$ be a triple of flags in $\Fmc(V)$. This triple is positive if and only if for any triple of positive integers $\mathbf i:=(i_1,i_2,i_3)$ that sum to $n$, we have $T^{\mathbf i}(\mathbf F)>0$. 
\end{thm}

More generally, Fock and Goncharov \cite[Section 9]{FockGoncharov} showed that the the space of projective classes of positive $k$-tuples of flags can be explicitly parametrized using cross ratios and triple ratios. To describe the parametrization of $\Fmc(V)_+^k/\PGL(V)$ consider an oriented planar polygon $M$ with $k$ vertices, and choose a triangulation $\Tmc$ of $M$ such that the vertices of the triangles in the triangulation are exactly the vertices of $M$ (see Figure \ref{fig:polygon}). We associate to a vertex $v$ of $M$ the flag $F(v)$.

\begin{figure}[ht]
\centering
\includegraphics[scale=0.7]{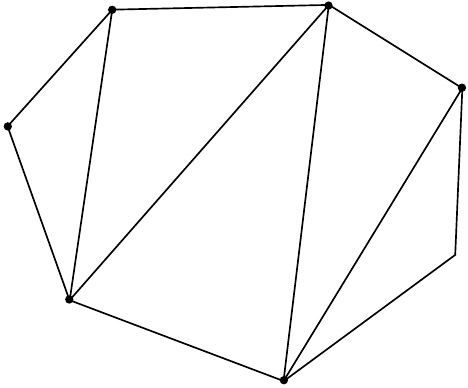}
\small
\put (-130, 131){$r_{e,1}$}
\put (-143, 24){$r_{e,2}$}
\put (-55, 132){$s_{e,1}=x_{T,3}$}
\put (-170, 92){$s_{e,2}$}
\put (0, 100){$x_{T,1}$}
\put (-70, -3){$x_{T,2}$}
\put (-40, 70){$T$}
\put (-132, 80){$e$}
\caption{The oriented planar polygon $M$, where the orientation induces a clockwise orientation on $\partial M$.}\label{fig:polygon}
\end{figure}

We fix the following notation to simplify the statements. For each interior edge $e$ of $\Tmc$ let $r_{e,1}$ and $r_{e,2}$ be the endpoints of $e$ respectively, and let $s_{e,1}$ and $s_{e,2}$ be the two vertices of $M$ such that the triangles with vertices $r_{e,1}$, $r_{e,2}$, $s_{e,1}$ and $r_{e,1}$, $r_{e,2}$, $s_{e,2}$ are both triangles of $\Tmc$, and $r_{e,1}<s_{e,1}<r_{e,2}<s_{e,2}<r_{e,1}$ according to the clockwise orientation on $\partial M$ (induced by the orientation on $M$, see Figure \ref{fig:polygon}). Then denote the quadruple of flags $\big(F(r_{e,1}),F(s_{e,1}),F(r_{e,2}),F(s_{e,2})\big)$ by $\mathbf F(e)$. 

 For each triangle $T$ of $\Tmc$, let $x_{T,1}$, $x_{T,2}$ and $x_{T,3}$ be the three vertices of $T$ such that $x_{T,1}<x_{T,2}<x_{T,3}<x_{T,1}$ according to the clockwise orientation on $\partial M$ (see Figure \ref{fig:polygon}). Then denote the triple of flags $\big(F(x_{T,1}),F(x_{T,2}),F(x_{T,3})\big)$ by $\mathbf F(T)$.

\begin{prop}\cite[Theorem 9.1(a)]{FockGoncharov}\label{prop:Fock-Goncharov parametrization}\label{thm:ktuple}
Let $M$ be an oriented planar $k$-gon ($k\geq 3$) and $\Tmc$ a triangulation of $M$ as described above. Let $T_1$, $\dots$, $T_{k-2}$ be the triangles given by $\Tmc$ and $e_1,\dots,e_{k-3}$ be the interior edges of $\Tmc$. Also, let $\Bmc$ denote the set of triples of positive integers that sum to $n$, and let $\Amc$ denote the set of pairs of positive integers that sum to $n$. The map 
\[\Fmc(V)_+^k/\PGL(V)\to(\Rbbb^+)^{\frac{(n-1)(n-2)(k-2)}{2}}\times(\Rbbb^+)^{(n-1)(k-3)}\] 
given by
\[[F(v_1),\dots, F(v_k)]\mapsto\left(\left(T^{\mathbf i}(\mathbf F(T_j))\right)_{j\in[1,k-2];\mathbf i\in\Bmc},\left(-C^{\mathbf k}(\mathbf F(e_j))\right)_{j\in[1,k-3];\mathbf k\in\Amc}\right)
\]
is a real analytic diffeomorphism.
\end{prop}

The next theorem relates Frenet curves to positive maps. Namely, the cross ratios and triple ratios of points along a Frenet curve are well-behaved.

\begin{thm}\label{prop:positive}
Let $\xi:S^1\to\Fmc(V)$ be a Frenet curve. 
\begin{enumerate}
\item For all pairwise distinct triples of points $\mathbf x:=(x_1,x_2,x_3)$ in $S^1$ and for all triples of positive integers $\mathbf i:=(i_1,i_2,i_3)$ that sum to $n$, we have 
\[T^{\mathbf i}\big(\xi(\mathbf x)\big)>0.\]
\item For all cyclically ordered quadruples $\mathbf a:=(r_1,s_1,r_2,s_2)$ of points in $S^1$, and for all pairs of positive integers $\mathbf k:=(k_1,k_2)$ that sum to $n$, we have
\[C^{\mathbf k}\big(\xi(\mathbf a)\big)<0.\]
\end{enumerate}
In particular every Frenet curve is a positive map. 
\end{thm}

This theorem can be proved using similar arguments to the ones Fock and Goncharov use in \cite[Theorem 9.1]{FockGoncharov}, see for example \cite[Proposition 2.5.7]{Zhang_thesis} or \cite[Appendix B]{LabourieMcShane}.

Conversely, the positivity of a $k$-tuple ensures a strong transversality condition, which we prove in Appendix \ref{app:Fock-Goncharov}. A slightly weaker statement was proven in \cite[Proposition~9.4.]{FockGoncharov}. 

\begin{prop}\cite{FockGoncharov}\label{prop:Fock-Goncharov}
Let $(F_1,\dots,F_k)$ be a positive $k$-tuple of flags in $\Fmc(V)_+^k$. Then for any positive integers $n_1$, $\dots$, $n_k$ that sum to $d\leq n$, we have that 
\[\dim\left(\sum_{j=1}^kF_j^{(n_j)}\right)=d.\]
\end{prop}

This proposition might make the the reader think that there is no difference between Frenet curves and positive maps. But positive maps $\phi: S^1 \to \Fmc(V)$ do not even have to be continuous in general, and thus are in general not Frenet. 
However, if we take a representation $\rho: \Gamma \to \PGL(V)$ into play, then the two notions agree: any $\rho$-equivariant positive map $\phi: \partial\Gamma \to \Fmc(V)$ is Frenet \cite[Theorem 1.15, Section 7.6-7.9]{FockGoncharov}.

\section{Elementary flows on Frenet curves}\label{sec:elementary}

In this section we define two types of flows, the \emph{elementary eruption flows} and \emph{elementary shearing flows} on the space of Frenet curves $\mathrm{Fre}(V)$. Later, we use infinite combinations of these elementary flows to define flows on the Hitchin component $\Hit_V(S)$. We equip $S^1$ with a clockwise orientation.

\subsection{The elementary eruption flow} \label{sec:elementaryeruption}
Let $\mathbf F:=(F_1,F_2,F_3)$ be a generic triple of flags in $\Fmc(V)^{[3]}$, and let $\mathbf i:=(i_1,i_2,i_3)$ be a triple of positive integers that sum to $n$. For any real number $t$ we define $b^{\mathbf i}_{\mathbf F}(t)$ to be the projective transformation in $\PGL(V)$ that is the projectivization of the linear automorphism of $V$ that acts as 
\begin{itemize}
\item the identity on $F_1^{(i_1)}$,
\item scaling by $e^{\frac{t}{3}}$ on $F_{2}^{(i_{2})}$, and
\item scaling by $e^{-\frac{t}{3}}$ on $F_{3}^{(i_{3})}$.
\end{itemize}
With respect to the decomposition $V = F_1^{(i_1)} + F_{2}^{(i_{2})}+F_{3}^{(i_{3})}$, the transformation $b^{\mathbf i}_{\mathbf F}(t)$ is represented by the matrix 
\[
\left(\begin{matrix}
\mathrm{Id}_{i_1} & 0 & 0 \\
0 & e^{\frac{t}{3}}\mathrm{Id}_{i_2} & 0 \\
0& 0& e^{-\frac{t}{3}}\mathrm{Id}_{i_3} \\
\end{matrix}\right). 
\]

\begin{lem}\label{lem:fixflag1}
Let $\mathbf i:=(i_1,i_2,i_3)$ be a triple of positive integers that sum to $n$, and let $\mathbf F:=(F_1,F_2,F_3)$ be a generic triple of flags. For $m=1$, $2$, $3$, set $\mathbf i_m:=(i_m,i_{m+1},i_{m-1})$ and $\mathbf F_m:=(F_m,F_{m+1},F_{m-1})$, where the arithmetic in the subscripts is done modulo $3$. Then for all $t\in \Rbbb$, 
\[b^{\mathbf i_{m-1}}_{\mathbf F_{m-1}}(t)\cdot F_m=b^{\mathbf i_{m+1}}_{\mathbf F_{m+1}}(t)\cdot F_m.\]
\end{lem}

\begin{proof}
Observe first that by the definition of $b^{\mathbf i}_{\mathbf F}(t)$, we have that for $l\leq i_m$, 
\[b^{\mathbf i_{m-1}}_{\mathbf F_{m-1}}(t)\cdot F_m^{(l)}=b^{\mathbf i_{m+1}}_{\mathbf F_{m+1}}(t)\cdot F_m^{(l)}.\]
Thus, we only need to show that the same is true for $l>i_m$. For $k=1$, $2$, $3$ and $j=1$, $\dots$, $n-1$, let $f_{k,j}$ be a vector in $F_k^{(j)}\setminus F_k^{(j-1)}$ (we use the convention that $F_k^{(0)}=\{0\}$). Since $\mathbf F$ is a generic triple,
\[\{f_{1,1},\dots,f_{1,i_1},f_{2,1},\dots,f_{2,i_2},f_{3,1},\dots,f_{3,i_3}\}\]
is a basis for $\Rbbb^n$. Thus, for any non-zero vector $f$ in $F_m^{(l)}$, we can write 
\[f=\sum_{j=1}^{i_1}\alpha_{1,j}f_{1,j}+\sum_{j=1}^{i_2}\alpha_{2,j}f_{2,j}+\sum_{j=1}^{i_3}\alpha_{3,j}f_{3,j},\]
for some real numbers $\alpha_{k,j}$. 
A direct computation gives that 
\[\begin{aligned}
b^{\mathbf i_{m-1}}_{\mathbf F_{m-1}}(t)\cdot [f]&=\left[\sum_{j=1}^{i_{m-1}}\alpha_{m-1,j}f_{m-1,j}+e^\frac{t}{3}\sum_{j=1}^{i_m}\alpha_{m,j}f_{m,j}+e^{-\frac{t}{3}}\sum_{j=1}^{i_{m+1}}\alpha_{m+1,j}f_{m+1,j}\right],
\end{aligned}\]
\[\begin{aligned}
b^{\mathbf i_{m+1}}_{\mathbf F_{m+1}}(t)\cdot [f]&=\left[e^\frac{t}{3}\sum_{j=1}^{i_{m-1}}\alpha_{m-1,j}f_{m-1,j}+e^{-\frac{t}{3}}\sum_{j=1}^{i_m}\alpha_{m,j}f_{m,j}+\sum_{j=1}^{i_{m+1}}\alpha_{m+1,j}f_{m+1,j}\right].
\end{aligned}\]
Hence, $b^{\mathbf i_{m-1}}_{\mathbf F_{m-1}}(t)\cdot [f]$ and $b^{\mathbf i_{m+1}}_{\mathbf F_{m+1}}(t)\cdot [f]$ are both lines that lie in the subspace of $V$ spanned by the $i_m+1$ vectors
\[f_{m,1}\,\,,\,\,\dots\,\,,\,\,f_{m,i_m}\,\,,\,\,e^\frac{t}{3}\sum_{j=1}^{i_{m-1}}\alpha_{m-1,j}f_{m-1,j}+\sum_{j=1}^{i_{m+1}}\alpha_{m+1,j}f_{m+1,j}.\]
In other words, 
\[b^{\mathbf i_{m-1}}_{\mathbf F_{m-1}}(t)\cdot \left(F_m^{(i_m)}+[f]\right)=b^{\mathbf i_{m+1}}_{\mathbf F_{m+1}}(t)\cdot \left(F_m^{(i_m)}+[f]\right)\]
for any non-zero vector $f$ in $F_m^{(l)}$. This concludes the proof.
\end{proof}

With Lemma \ref{lem:fixflag1}, we can now define the elementary eruption flow. 

\begin{definition}\label{def:elementary eruption flow}
Let $\mathbf i:=(i_1,i_2,i_3)$ be a triple of positive integers that sum to $n$, and let $\mathbf x:=(x_1,x_2,x_3)$ be a cyclically ordered triple of points in $S^1$. For $m=1$, $2$, $3$, let $\mathbf i_m:=(i_m,i_{m+1},i_{m-1})$ and $\mathbf x_m:=(x_m,x_{m+1},x_{m-1})$, where arithmetic in the subscripts are done modulo $3$. The $\mathbf i$-\emph{elementary eruption flow with respect to $\mathbf{x}$} is the continuous flow
\[\left(\epsilon^{\mathbf i}_{\mathbf x}\right)_t:\widetilde{\mathrm{Fre}}(V)\to\widetilde{\mathrm{Fre}}(V),\]
defined by $\xi \mapsto \xi_t = \left(\epsilon^{\mathbf i}_{\mathbf x}\right)_t(\xi)$, 
where 
\[\xi_t(p):=\left\{
\begin{array}{ll}
b^{\mathbf i_1}_{\xi(\mathbf x_1)}(t)\cdot \xi(p)&\text{if }x_{2}\leq p\leq x_{3};\\
b^{\mathbf i_{2}}_{\xi(\mathbf x_{2})}(t)\cdot \xi(p)&\text{if }x_{3}\leq p\leq x_1;\\
b^{\mathbf i_{3}}_{\xi(\mathbf x_{3})}(t)\cdot \xi(p)&\text{if }x_1\leq p\leq x_{2}.\\
\end{array}\right.\]
Here, the inequality $\leq$ is with respect to the clockwise orientation on $S^1$.
\end{definition}

\begin{figure}[ht]
\centering
\includegraphics[scale=0.4]{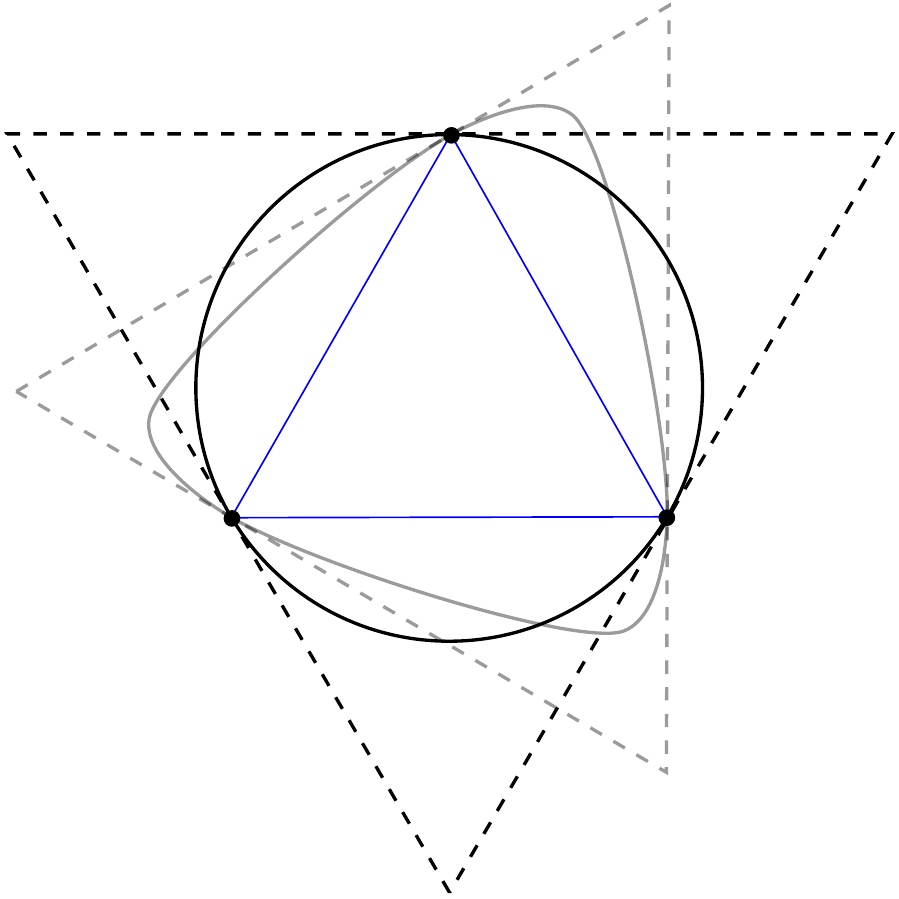}
\small
\put (-150, 67){$\xi(x_1)$}
\put (-103, 152){$\xi(x_2)$}
\put (-43, 67){$\xi(x_3)$}
\caption{When $n=3$, the dark conic is $\xi^{(1)}(S^1)\subset\Rbbb\Pbbb^2$ and the dark dotted lines are $\xi^{(2)}(x_1)$, $\xi^{(2)}(x_2)$ and $\xi^{(2)}(x_3)$. After applying an elementary eruption at $(x_1,x_2,x_3)$, $\xi_t^{(1)}(S^1)$ is the light curve and $\xi_t^{(2)}(x_1)$, $\xi_t^{(2)}(x_2)$ and $\xi_t^{(2)}(x_3)$ are the light dotted lines.}\label{fig:elementaryeruption}
\end{figure}

\begin{remark}
In the case when $n=3$ and $i_1=i_2=i_3=1$, the elementary eruption flow is (up to projective transformations) the eruption flow defined in \cite[Section~4.2.]{Wienhard_Zhang}. In this case, the eruption flow admits a nice geometric interpretation as changing the gluing parameter of a pair of nested triangles determined by $(\xi(x_1), \xi(x_2), \xi(x_3))$ (see Figure \ref{fig:elementaryeruption}). This geometric interpretation also suggested the name eruption flow. 
\end{remark}

It is not obvious that the elementary eruption flows are well defined. Lemma \ref{lem:fixflag1} implies that for all triples $\mathbf{i}:=(i_1,i_2,i_3)$ of positive integers that sum to $n$, and all triples $\mathbf x:=(x_1,x_2,x_3)$ of cyclically ordered points in $S^1$, the curve $\xi_t=\left(\epsilon^{\mathbf i}_{\mathbf x}\right)_t(\xi):S^1\to\Fmc(V)$ is a well-defined continuous curve. But to show that $\left(\epsilon^{\mathbf i}_{\mathbf x}\right)_t$ is well-defined as a flow on $\widetilde{\mathrm{Fre}}(V)$, we have show that $\xi_t$ is in fact a Frenet curve. We do so now.

For $i=1$, $2$, $3$, let $\mathbf i_m:=(i_m,i_{m+1},i_{m-1})$ and $\mathbf x_m:=(x_m,x_{m+1},x_{m-1})$, where arithmetic in the subscripts are done modulo $3$. For convenience, we consider the curve $b^{\mathbf i_m}_{\xi(\mathbf x_ m)}(-t)\cdot\xi_t$ which is projectively equivalent to $\xi_t$. Setting
\begin{equation}\label{eqn:a computation}
a^{\mathbf i_m}_{\mathbf F_m}(t):=b^{\mathbf i_{m+1}}_{\mathbf F_{m+1}}(-t)b^{\mathbf i_{m-1}}_{\mathbf F_{m-1}}(t),
\end{equation}
the curve $b^{\mathbf i_m}_{\xi(\mathbf x_m)}(-t)\cdot\xi_t:S^1\to\widetilde{\mathrm{Fre}}(V)$ is given by
\[b^{\mathbf i_m}_{\xi(\mathbf x_m)}(-t)\cdot\xi_t(p)=\left\{
\begin{array}{ll}
\xi(p)&\text{if }x_{m+1}\leq p\leq x_{m-1};\\
a^{\mathbf i_{m-1}}_{\xi(\mathbf x_{m-1})}(t)\cdot \xi(p)&\text{if }x_{m-1}\leq p\leq x_m;\\
a^{\mathbf i_{m+1}}_{\xi(\mathbf x_{m+1})}(t)^{-1}\cdot \xi(p)&\text{if }x_m\leq p\leq x_{m+1}.
\end{array}\right.\]
This has the advantage that for all $x_{m+1}\leq p\leq x_{m-1}$ we have $b^{\mathbf i_m}_{\xi(\mathbf x_m)}(-t)\cdot\xi_t(p)=\xi(p)$. Furthermore, we have the following lemma, which is useful in later computations. 

\begin{lem}\label{lem:other eruption}
Let $\mathbf F:=(F_1,F_2,F_3)$ be a generic triple of flags, and let $\mathbf i:=(i_1,i_2,i_3)$ be a triple of positive integers that sum to $n$. For $m=1$, $2$, $3$, let $\mathbf F_m:=(F_m,F_{m+1},F_{m-1})$ and $\mathbf i_m:=(i_m,i_{m+1},i_{m-1})$. Choose a basis $\{f_{m,1},\dots,f_{m,n}\}$ of $V$ such that $F_m^{(l)}=\Span_{\Rbbb}\{f_{m,1},\dots,f_{m,l}\}$ for all $l=1$, $\dots$, $n-1$. Then $a^{\mathbf i_m}_{\mathbf F_m}(t)$ is represented in this basis by an upper triangular matrix where the first $i_m$ entries down the diagonal are $e^{\frac{2t}{3}}$ and the last $n-i_m$ entries down the diagonal are $e^{-\frac{t}{3}}$.
\end{lem}

\begin{proof}
By Lemma \ref{lem:fixflag1}, $a^{\mathbf i_m}_{\mathbf F_m}(t)$ fixes the flag $F_m$. Therefore with respect to the basis $\{f_1,\dots,f_n\}$, $a^{\mathbf i_m}_{\mathbf F_m}(t)$ is represented by an upper triangular matrix. 
Furthermore, $a^{\mathbf i_m}_{\mathbf F_m}(t)$ is the projectivization of the linear map which acts as
\begin{itemize}
\item scaling by $e^{-\frac{t}{3}}$ on $F_{m-1}^{(i_{m-1})}+F_{m+1}^{(i_{m+1})}$ and
\item scaling by $e^{\frac{2t}{3}}$ on $F_{m}^{(i_{m})}$.
\end{itemize}
The lemma now follows. 
\end{proof}

As a preliminary step to prove that $\xi_t:=\left(\epsilon^{\mathbf i}_{\mathbf x}\right)_t(\xi)$ is a Frenet curve, we prove the following proposition, which describes how certain cross ratios and triple ratios along the image of $\xi_t$ change with $t$. 

\begin{prop}\label{prop:eruption projective invariants}
Let $\mathbf i:=(i_1,i_2,i_3)$ be a triple of positive integers that sum to $n$, and let $\mathbf x:=(x_1,x_2,x_3)$ be a cyclically ordered triple of points in $S^1$. Let $\left(\epsilon^{\mathbf i}_{\mathbf x}\right)_t$ be the $\mathbf{i}$-elementary eruption flow with respect to $\mathbf{x}$, and let $\xi_t:=\left(\epsilon^{\mathbf i}_{\mathbf x}\right)_t(\xi)$. Then we have the following:
\begin{enumerate}
\item For any triple $\mathbf j:=(j_1,j_2,j_3)$ of positive integers that sum to $n$, we have 
\[T^{\mathbf j}(\xi_t(\mathbf x))=\left\{
\begin{array}{ll}
T^{\mathbf j}(\mathbf x)&\text{if }\mathbf j\neq\mathbf i;\\
e^t\cdot T^{\mathbf j}(\mathbf x)&\text{if }\mathbf j=\mathbf i.
\end{array}
\right.\]
\item For any triple $\mathbf y:=(y_1,y_2,y_3)$ of pairwise distinct points in $S^1$ such that $y_1$, $y_2$, $y_3$ lie in the closure of a connected component of $S^1\setminus\{x_1,x_2,x_3\}$, we have $T^{\mathbf j}(\xi_t(\mathbf y))=T^{\mathbf j}(\xi(\mathbf y))$ for any triple $\mathbf j:=(j_1,j_2,j_3)$ of positive integers that sum to $n$. 
\item For any cyclically ordered quadruple $\mathbf{a}:=(r_1,s_1,r_2,s_2)$ of points in $S^1$ such that $r_1$, $s_1$, $r_2$ lie in the closure of a connected component of $S^1\setminus\{x_1,x_2,x_3\}$, and $s_2$ either lies in that same closure or is equal to the $x_i$ that is not a boundary point of that closure, we have $C^{\mathbf k}(\xi_t(\mathbf{a}))=C^{\mathbf k}(\xi(\mathbf {a}))$ for any pair $\mathbf k:=(k_1,k_2)$ of positive integers that sum to $n$. 
\end{enumerate}
\end{prop}

For the proof of this proposition we express the triple ratios using determinants. 
To describe this, let $\mathbf F:=(F_1,F_2,F_3)$ be generic triple of flags in $\Fmc(V)^{[3]}$ and let $\mathbf i:=(i_1,i_2,i_3)$ be a triple of positive integers that sum to $n$. For $m=1$, $2$, $3$, we choose an ordered basis $\{f_{m,1},\dots,f_{m,n}\}$ of $V$ such that 
$F_m^{(l)}=\Span_\Rbbb(f_{m,1},\dots, f_{m,l})$ for all $l=1$, $\dots$, $n-1$. By genericity of the triple, we have that 
\[F_1^{(i_1)}\wedge F_2^{(i_2)}\wedge F_3^{(i_3)}:=\det(f_{1,1},\dots,f_{1,i_1},f_{2,1},\dots,f_{2,i_2},f_{3,1},\dots,f_{3,i_3})\]
is non-zero. Of course, $F_1^{(i_1)}\wedge F_2^{(i_2)}\wedge F_3^{(i_3)}$ depends on the choice of the three bases, but one can verify that the ratio
\[\frac{F_1^{(i_1+1)}\wedge F_2^{(i_2)}\wedge F_3^{(i_3-1)}\cdot F_1^{(i_1-1)}\wedge F_2^{(i_2+1)}\wedge F_3^{(i_3)}\cdot F_1^{(i_1)}\wedge F_2^{(i_2-1)}\wedge F_3^{(i_3+1)}}{F_1^{(i_1+1)}\wedge F_2^{(i_2-1)}\wedge F_3^{(i_3)}\cdot F_1^{(i_1)}\wedge F_2^{(i_2+1)}\wedge F_3^{(i_3-1)}\cdot F_1^{(i_1-1)}\wedge F_2^{(i_2)}\wedge F_3^{(i_3+1)}}\]
does not, and in fact evaluates to the triple ratio $T^{\mathbf i}(\mathbf F)$. 

\begin{proof}[Proof of Proposition~\ref{prop:eruption projective invariants}]
Proof of (1). For $m=1$, $2$, $3$, let $\mathbf i_m:=(i_m,i_{m+1},i_{m-1})$, $\mathbf j_m:=(j_m,j_{m+1},j_{m-1})$, and $\mathbf x_m:=(x_m,x_{m+1},x_{m-1})$, where the arithmetic in the subscripts are done modulo $3$. From the definition of $a^{\mathbf i_m}_{\xi(\mathbf x_m)}(t)$ and that of $\xi_t$, it is clear that
\[b^{\mathbf i_{m+1}}_{\xi(\mathbf x_{m+1})}(-t)\cdot\xi_t(\mathbf x_m)=\big(\xi(x_m),a^{\mathbf i_m}_{\xi(\mathbf x_m)}(t)\cdot\xi(x_{m+1}),\xi(x_{m-1})\big).\]

Suppose first that $\mathbf j\neq\mathbf i$. Then $i_{m+1}\geq j_{m+1}+1$ for some $m=1$, $2$, $3$. By Lemma \ref{lem:other eruption}, $a^{\mathbf i_m}_{\xi(\mathbf x_m)}(t)$ acts on $\xi^{(j_{m+1}+1)}(x_{m+1})\subset \xi^{(i_{m+1})}(x_{m+1})$ by scaling by $\lambda:=e^{-\frac{t}{3}}$. 
Since the triple ratio is a projective invariant, we have
\begin{eqnarray*}
&&T^{\mathbf j}(\xi_t({\mathbf x}))= T^{\mathbf j_m}\left(b^{\mathbf i_{m+1}}_{\xi({\mathbf x}_{m+1})}(-t)\cdot \xi_t(\mathbf x_m)\right) \\
&=& T^{\mathbf j_m}\big(\xi(x_m),a^{\mathbf i_m}_{\xi(\mathbf x_m)}(t)\cdot\xi(x_{m+1}),\xi(x_{m-1})\big)\\
&=&\frac{\xi^{(j_{m-1}+1)}(x_{m-1})\wedge \xi^{(j_m)}(x_m)\wedge \xi^{(j_{m+1}-1)}(x_{m+1})\cdot\lambda^{j_{m+1}-1}}{\xi^{(j_{m-1}+1)}(x_{m-1})\wedge \xi^{(j_m-1)}(x_m)\wedge \xi^{(j_{m+1})}(x_{m+1})\cdot\lambda^{j_{m+1}}}\cdot\\
&&\frac{\xi^{(j_{m-1}-1)}(x_{m-1})\wedge \xi^{(j_m+1)}(x_m)\wedge \xi^{(j_{m+1})}(x_{m+1})\cdot\lambda^{j_{m+1}}}{\xi^{(j_{m-1})}(x_{m-1})\wedge \xi^{(j_m+1)}(x_m)\wedge \xi^{(j_{m+1}-1)}(x_{m+1})\cdot\lambda^{j_{m+1}-1}}\cdot\\
&&\frac{\xi^{(j_{m-1})}(x_{m-1})\wedge \xi^{(j_m-1)}(x_m)\wedge \xi^{(j_{m+1}+1)}(x_{m+1})\cdot\lambda^{j_{m+1}+1}}{\xi^{(j_{m-1}-1)}(x_{m-1})\wedge \xi^{(j_m)}(x_m)\wedge \xi^{(j_{m+1}+1)}(x_{m+1})\cdot\lambda^{j_{m+1}+1}}\\
&=&T^{\mathbf j_m}(\xi(\mathbf x_m))=T^{\mathbf j}(\xi(\mathbf x)).
\end{eqnarray*}

On the other hand, $a^{\mathbf i}_{\xi(\mathbf x)}(t)$ acts on 
\begin{itemize}
\item $\xi^{(i_2)}(x_2)$ as scaling by $e^{-\frac{t}{3}}$, 
\item $\left(\xi^{(i_2+1)}(x_2)+\xi^{(i_3)}(x_3)\right)\cap\xi^{(i_1)}(x_1)$ as scaling by $e^{\frac{2t}{3}}$, 
\item $\left(\xi^{(i_1)}(x_1)+\xi^{(i_2+1)}(x_2)\right)\cap\xi^{(i_3)}(x_3)$ as scaling by $e^{-\frac{t}{3}}$.
\end{itemize} 
Thus,
\begin{eqnarray*}
T^{\mathbf i}(\xi_t(\mathbf x))&=&T^{\mathbf i}\big(\xi(x_1),a^{\mathbf i}_{\xi(\mathbf x)}(t)\cdot\xi(x_{2}),\xi(x_{3})\big)\\
&=&\frac{\xi^{(i_1+1)}(x_{1})\wedge \xi^{(i_2)}(x_2)\wedge \xi^{(i_3-1)}(x_{3})\cdot e^{-\frac{i_2t}{3}}}{\xi^{(i_1+1)}(x_{1})\wedge \xi^{(i_2-1)}(x_2)\wedge \xi^{(i_3)}(x_{3})\cdot e^{-\frac{(i_2-1)t}{3}}}\cdot\\
&&\frac{\xi^{(i_1-1)}(x_{1})\wedge \xi^{(i_2+1)}(x_2)\wedge \xi^{(i_3)}(x_{3})\cdot e^{-\frac{i_2t}{3}}e^{\frac{2t}{3}}}{\xi^{(i_1)}(x_{1})\wedge \xi^{(i_2+1)}(x_2)\wedge \xi^{(i_3-1)}(x_{3})\cdot e^{-\frac{i_2t}{3}}e^{-\frac{t}{3}}}\cdot\\
&&\frac{\xi^{(i_1)}(x_{1})\wedge \xi^{(i_2-1)}(x_2)\wedge \xi^{(i_3+1)}(x_{3})\cdot e^{-\frac{(i_2-1)t}{3}}}{\xi^{(i_1-1)}(x_{1})\wedge \xi^{(i_2)}(x_2)\wedge \xi^{(i_3+1)}(x_{3})\cdot e^{-\frac{i_2t}{3}}}\\
&=&e^t\cdot T^{\mathbf i}(\xi(\mathbf x)).
\end{eqnarray*}

Proof of (2). Again, we may assume without loss of generality that the connected component of $S^1\setminus\{x_1,x_2,x_3\}$ whose closure contains $y_1$, $y_2$, $y_3$ is 
\[\{q\in S^1:x_3\leq q\leq x_1\}.\] 
Then $\xi_t(\mathbf y)=b^{\mathbf i_2}_{\xi(\mathbf x_2)}(t)\cdot\xi(\mathbf y)$. Since the triple ratio is a projective invariant, (2) follows immediately. 

Proof of (3). We may assume without loss of generality that the connected component of $S^1\setminus\{x_1,x_2,x_3\}$ whose closure contains $r_1$, $s_1$, $r_2$ is 
\[\{q\in S^1:x_3\leq q\leq x_1\}.\] 
Then either $x_3\leq s_2\leq x_1$, or $s_2=x_2$. Observe that $C^{\mathbf k}(\xi_t(\mathbf p))$ depends only on $\xi_t^{(1)}(s_2)$ (and not on the rest of the flag $\xi_t(s_2)$). From the definition of $b^{\mathbf i_2}_{\xi(\mathbf x_2)}(t)$ it is clear that $\xi_t^{(1)}(s_2)=b^{\mathbf i_2}_{\xi(\mathbf x_2)}(t)\cdot\xi^{(1)}(s_2)$, $\xi_t(r_1)=b^{\mathbf i_2}_{\xi(\mathbf x_2)}(t)\cdot\xi(r_1)$, $\xi_t(r_2)=b^{\mathbf i_2}_{\xi(\mathbf x_2)}(t)\cdot\xi(r_2)$, and $\xi_t(s_1)=b^{\mathbf i_2}_{\xi(\mathbf x_2)}(t)\cdot\xi(s_1)$. Statement (3) then follows from the projective invariance of the cross ratio.
\end{proof}

\begin{thm}\label{thm:eruption preserves Frenet}
Let $\mathbf{i}:=(i_1,i_2,i_3)$ be a triple of positive integers that sum to $n$, and let $\mathbf x:=(x_1,x_2,x_3)$ be a cyclically ordered triple of points in $S^1$. If $\xi$ is a Frenet curve, then for all $t \in \Rbbb$, the continuous curve $\xi_t=\left(\epsilon^{\mathbf i}_{\mathbf x}\right)_t(\xi)$ is also a Frenet curve. In particular, the $\mathbf i$-elementary shear flow with respect to $\mathbf x$, $\left(\epsilon^{\mathbf i}_{\mathbf x}\right)_t$, is a well-defined flow on $\widetilde{\mathrm{Fre}}(V)$.
\end{thm}

\begin{proof}
First, we prove that $\xi_t$ satisfies property (1) in the definition of Frenet curves (Definition \ref{def:Frenet}). Choose pairwise distinct points $p_1$, $\dots$, $p_k$ in $S^1$, which we can assume to be cyclically ordered with respect to the cyclic order on $S^1$, and choose positive integers $n_1$, $\dots$, $n_k$ that sum to $d\leq n$. 
We have to show that \[\dim\left(\sum_{j=1}^k\xi_t^{(n_j)}(p_j)\right)=d.\] 
By Proposition \ref{prop:Fock-Goncharov} it is sufficient to show that the $k$-tuple of flags $(\xi_t(p_1),\dots,\xi_t(p_k))$ is positive. This is what we are going to prove now. We extend the cyclic order on $\{p_1,\dots,p_k\}$ to a cyclic order on $\{p_1,\dots,p_k\}\cup\{x_1,x_2,x_3\}$ (it is possible that $x_i=p_j$ for some $i,j$). This allows us to consider $\{p_1,\dots,p_k\}\cup\{x_1,x_2,x_3\}$ as the vertices of a planar polygon $M$ inscribed in the circle $S^1$. Choose a triangulation $\Tmc$ of $M$ such that the vertices of $\Tmc$ are the vertices of $M$, and the triangle with vertices $x_1,x_2,x_3$ is a triangle in $\Tmc$.

For each interior edge $e$ of $\Tmc$, let $r_{e,1}$ and $r_{e,2}$ be the endpoints of $e$ respectively, and let $s_{e,1}$ and $s_{e,2}$ be the two vertices of $M$ such that the triangles with vertices $r_{e,1}$, $r_{e,2}$, $s_{e,1}$ and $r_{e,1}$, $r_{e,2}$, $s_{e,2}$ are both triangles of $\Tmc$, and $r_{e,1}<s_{e,1}<r_{e,2}<s_{e,2}<r_{e,1}$ according to the clockwise orientation on $\partial M$ (see Figure~\ref{fig:polygon}). We set $\mathbf F_t(e):=\big(\xi_t(r_{e,1}),\xi_t(s_{e,1}),\xi_t(r_{e,2}),\xi_t(s_{e,2})\big)$. Similarly, for each triangle $T$ of $\Tmc$, let $x_{T,1}$, $x_{T,2}$ and $x_{T,3}$ be the vertices of $T$ such that $x_{T,1}<x_{T,2}<x_{T,3}<x_{T,1}$ (see Figure~\ref{fig:polygon}), and set $\mathbf F_t(T):=\big(\xi_t(x_{T,1}),\xi_t(x_{T,2}),\xi_t(x_{T,3})\big)$. Since $\xi_0=\xi$, by Theorem \ref{prop:positive}, we have that:
\begin{itemize}
\item for every interior edge $e$ of $\Tmc$ and for all pairs of positive integers $\mathbf k$ that sum to $n$, we have $C^{\mathbf k}_{\mathbf F_0(e)}<0$.
\item for every triangle $T$ of the triangulation $\Tmc$ and for all triples of positive integers $\mathbf i$ that sum to $n$, the triple ratio satisfies $T^{\mathbf i}_{\mathbf F_0(T)}>0$.
\end{itemize}
It then follows from Proposition \ref{prop:eruption projective invariants} that the above two statements about the cross ratio and triple ratio above also hold for ${\mathbf F}_t(e)$ and $\mathbf{F}_t(T)$. Therefore Theorem \ref{prop:positive} implies that the $k$-tuple $(\xi_t(p_1),\dots,\xi_t(p_k))$ is positive, so Proposition \ref{prop:Fock-Goncharov} implies
\[\dim\left(\sum_{j=1}^k\xi_t^{(n_j)}(p_j)\right)=d.\] 

Next, we show that $\xi_t$ satisfies property (2) in the definition of Frenet curves (Definition \ref{def:Frenet}).
Let $p_0$ be a point in $S^1$, and consider any sequence $\{(p_{i,1},\dots,p_{i,k})\}_{i=1}^\infty$ of pairwise distinct $k$-tuples in $S^1$, such that $\lim_{i\to\infty}p_{i,j}=p_0$ for all $j=1$, $\dots$, $k$. If $x_{m+1}<p_0<x_{m-1}$ for some $m=1$, $2$, $3$, then $x_{m+1}<p_{i,j}<x_{m-1}$ for sufficiently large $i$ and for all $j=1$, $\dots$, $k$. Since $\xi_t(p)=b^{\mathbf i_m}_{\xi(\mathbf x_m)}(t)\cdot\xi(p)$ for all $x_{m+1}<p<x_{m-1}$, we have
\[\lim_{i\to\infty}\sum_{j=1}^k\xi_t^{(n_j)}(p_{i,j})=b^{\mathbf i_m}_{\xi(\mathbf x_m)}(t)\cdot\lim_{i\to\infty}\sum_{j=1}^k\xi^{(n_j)}(p_{i,j})=b^{\mathbf i_m}_{\xi(\mathbf x_m)}(t)\cdot\xi^{(d)}(p_0)=\xi_t^{(d)}(p_0),\]
where $\mathbf i_m:=(i_m,i_{m+1},i_{m-1})$ and $\mathbf x_m:=(x_m,x_{m+1},x_{m-1})$ for $m=1$, $2$, $3$, and arithmetic in the subscripts are done modulo $3$.

Now, suppose that $p_0=x_m$ for some $m=1$, $2$, $3$. Then for sufficiently large $i$ and for all $j=1$, $\dots$, $k$, we know that $x_{m-1}<p_{i,j}<x_{m+1}$. Observe that $\xi_t(p)=b^{\mathbf i_{m+1}}_{\xi(\mathbf x_{m+1})}(t)\cdot \xi(p)$ for all $x_{m-1}\leq p\leq x_m$, and $\xi_t(p)=b^{\mathbf i_{m-1}}_{\xi(\mathbf x_{m-1})}(t)\cdot \xi(p)$ for all $x_m\leq p\leq x_{m+1}$. Also, recall from Lemma \ref{lem:fixflag1} that $a^{\mathbf i_{m}}_{\xi(\mathbf x_{m})}(t):=b^{\mathbf i_{m+1}}_{\xi(\mathbf x_{m+1})}(-t)b^{\mathbf i_{m-1}}_{\xi(\mathbf x_{m-1})}(t)$ fixes the flag $\xi(x_m)$. Let 
\[A:=\{j:x_{m-1}<p_{i,j}\leq x_m\text{ for sufficiently large }i\}\] 
and let 
\[B:=\{j:x_m<p_{i,j}<x_{m+1}\text{ for sufficiently large }i\}.\] 
By taking subsequences, we may assume that $A\cup B=\{1,\dots,k\}$ is a disjoint union. Then
\begin{eqnarray*}
\lim_{i\to\infty}\sum_{j=1}^k\xi_t^{(n_j)}(p_{i,j})&=&b^{\mathbf i_{m+1}}_{\xi(\mathbf x_{m+1})}(t)\cdot\lim_{i\to\infty}\sum_{j\in A}\xi^{(n_j)}(p_{i,j})+b^{\mathbf i_{m-1}}_{\xi(\mathbf x_{m-1})}(t)\cdot\lim_{i\to\infty}\sum_{j\in B}\xi^{(n_j)}(p_{i,j})\\
&=&b^{\mathbf i_{m+1}}_{\xi(\mathbf x_{m+1})}(t)\cdot\lim_{i\to\infty}\left(\sum_{j\in A}\xi^{(n_j)}(p_{i,j})+a^{\mathbf i_{m}}_{\xi(\mathbf x_{m})}(t)\cdot\sum_{j\in B}\xi^{(n_j)}(p_{i,j})\right).
\end{eqnarray*}

Since $\xi$ is Frenet, 
\[\lim_{i\to\infty}\sum_{j\in A}\xi^{(n_j)}(p_{i,j})\subset\xi^{(d)}(p_0)\]
and
\[\lim_{i\to\infty}a^{\mathbf i_{m}}_{\xi(\mathbf x_{m})}(t)\cdot\sum_{j\in B}\xi^{(n_j)}(p_{i,j})\subset a^{\mathbf i_{m}}_{\xi(\mathbf x_{m})}(t)\cdot\xi^{(d)}(p_0)=\xi^{(d)}(p_0).\]
This implies that $\displaystyle\lim_{i\to\infty}\sum_{j=1}^k\xi_t^{(n_j)}(p_{i,j})\subset\xi^{(d)}(p_0)$. Since $\displaystyle\sum_{j=1}^k\xi_t^{(n_j)}(p_{i,j})$ has dimension $d$, we have 
\[\displaystyle\lim_{i\to\infty}\sum_{j=1}^k\xi_t^{(n_j)}(p_{i,j})=\xi_t^{(d)}(p_0).\]

Therefore $\xi_t$ is Frenet, and \[\left(\epsilon^{\mathbf i}_{\mathbf x}\right)_t:\widetilde{\mathrm{Fre}}(V)\to\widetilde{\mathrm{Fre}}(V)\]
is well-defined. The fact that $\left(\epsilon^{\mathbf i}_{\mathbf x}\right)_t$ is a flow, i.e. that 
$\left(\epsilon^{\mathbf i}_{\mathbf x}\right)_{t'} \circ \left(\epsilon^{\mathbf i}_{\mathbf x}\right)_t = \left(\epsilon^{\mathbf i}_{\mathbf x}\right)_{t+t'}$, is immediate from the definition. 
\end{proof}

\subsection{The elementary shearing flow} \label{sec:elementaryshearing}
To define the elementary shearing flow, consider a pair $\mathbf E:=(E_1,E_2)$ of transverse flags in $\Fmc(V)$, and a pair $\mathbf k:=(k_1,k_2)$ of positive integers that sum to $n$. For any real number $t$, we denote by $d^{\mathbf k}_{\mathbf E}(t)$ the projective transformation in $\PGL(V)$ that is the projectivization of the linear automorphism of $V$ which acts as
\begin{itemize}
\item the identity on $E_1^{(k_1)}$,
\item scaling by $e^{-\frac{t}{2}}$ on $E_2^{(k_2)}$.
\end{itemize}
So with respect to the decomposition $V = E_1^{(k_1)} + E_2^{(k_2)}$, the projective transformation $d^{\mathbf k}_{\mathbf E}(t)$ can be represented by the matrix 
\[
\left(\begin{matrix}
\mathrm{Id}_{k_1} & 0 \\
0 & e^{-\frac{t}{2}} \mathrm{Id}_{k_2}
\end{matrix}
\right)
.\]

\begin{lem}\label{lem:fixflag}
Let $\mathbf k:=(k_1,k_2)$ be a pair of positive integers that sum to $n$, and let $\mathbf E:=(E_1,E_2)$ be a transverse pair of flags. For $m=1$, $2$, set $\mathbf{k}_m := (k_m,k_{m+1})$ and $\mathbf E_m: = (E_m,E_{m+1})$, where the arithmetic in the subscripts are done modulo $2$. Then for all $t\in \Rbbb$, 
\begin{enumerate}
\item $d^{\mathbf k}_{\mathbf E}(t)$ fixes $\mathbf E$, and
\item$d^{\mathbf k_1}_{\mathbf E_1}(t) = d^{\mathbf k_2}_{\mathbf E_2}(-t)$.
\end{enumerate}
\end{lem}

\begin{proof}
Since $\mathbf E$ is a transverse pair of flags, we may choose a basis $\{e_1,\dots,e_n\}$ of $V$ such that $e_j$ is a non-zero vector in $E_1^{(j)}\cap E_2^{(n-1-j)}$ for all $j=1$, $\dots$, $n-1$. Observe from the definition that $d^{\mathbf k}_{\mathbf E}(t)$ fixes the point $[e_j]$ in $\Pbbb(V)$ for all $j$, so it necessarily fixes $\mathbf E$. This proves (1). (2) is obvious from the definition of $d^{\mathbf k}_{\mathbf E}(t)$.
\end{proof}

With this, we can define the elementary shearing flow. 

\begin{definition}\label{def:elementary shearing flow}
Let $\mathbf r:=(r_1,r_2)$ be a pair of distinct points in $S^1$ and let $\mathbf k:=(k_1,k_2)$ be a pair of positive integers that sum to $n$. For $m=1$, $2$, set $\mathbf r_m:=(r_m,r_{m+1})$ and $\mathbf k_m:=(k_m,k_{m+1})$, where the subscripts are done modulo $2$. The $\mathbf k$-\emph{elementary shearing flow with respect to $\mathbf{r}$} is the continuous flow
\[(\psi^{\mathbf k}_{\mathbf r})_t:\widetilde{\mathrm{Fre}}(V)\to\widetilde{\mathrm{Fre}}(V)\]
defined by $\xi \mapsto \xi_t = \left(\psi^{\mathbf k}_{\mathbf r}\right)_t (\xi)$, where 
\begin{equation}\label{eqn:elementaryshear}
\xi_t(p)=\left\{
\begin{array}{ll}
d^{\mathbf k_{2}}_{\xi(\mathbf r_{2})}(t)\cdot \xi(p)&\text{if }r_2\leq p\leq r_1;\\
d^{\mathbf k_1}_{\xi(\mathbf r_1)}(t)\cdot \xi(p)&\text{if }r_1\leq p\leq r_{2}.
\end{array}\right. 
\end{equation}
\end{definition}

\begin{remark}
 In the case when $\dim(V)=3$, Goldman \cite{Goldman_convex, Goldman_bulge} introduced the hyperbolic shear and the bulge deformations associated to $\mathbf r$. The hyperbolic shear associated to $\mathbf r$ is equal to $\left(\psi^{(1,2)}_{\mathbf r}\right)_{t}\circ\left( \psi^{(2,1)}_{\mathbf r}\right)_{t}$, and the bulge associated to $\mathbf r$ is equal to $\left(\psi^{(1,2)}_\mathbf r\right)_{-t} \circ \left(\psi^{(2,1)}_\mathbf r\right)_{t}$. These elementary shearing and bulging flows are also described \cite{Goldman_bulge} and also in \cite[Section 4.1]{Wienhard_Zhang}. 
 \end{remark}
 
By Lemma \ref{lem:fixflag}, for all pairs of positive integers $\mathbf k:=(k_1,k_2)$ that sum to $n$ and for all pairs of distinct points $\mathbf r:=(r_1,r_2)$ in $S^1$, $\xi_t:=\left(\psi^{\mathbf k}_{\mathbf r}\right)_t(\xi)$ is a well-defined continuous curve, obtained by deforming the two subsegments of $\xi(S^1)$ given by $r_1$ and $r_2$ using two different projective transformations. To prove that $(\psi^{\mathbf k}_{\mathbf r})_t$ is well-defined, we now verify that $\xi_t$ is indeed a well-defined Frenet curve for any $t$. The arguments are analogous to the arguments for the eruption flows. 

For $m=1$, $2$, let $\mathbf k_m:=(k_m,k_{m+1})$ and let $\mathbf r_m:=(r_m,r_{m+1})$, where arithmetic in the subscripts are done modulo $2$. For convenience, we consider the curve $d^{\mathbf k_{m+1}}_{\xi(\mathbf r_{m+1})}(-t)\cdot\xi_t$, which is projectively equivalent to $\xi_t$. Setting 
\begin{equation}\label{eqn:c computation}
c^{\mathbf k_m}_{\mathbf E_m}(t):=d^{\mathbf k_{m+1}}_{\mathbf E_{m+1}}(-t)d^{\mathbf k_{m}}_{\mathbf E_{m}}(t)=d^{\mathbf k_{m}}_{\mathbf E_{m}}(2t),
\end{equation}
the curve $d^{\mathbf k_{m+1}}_{\xi(\mathbf r_{m+1})}(-t)\cdot\xi_t:S^1\to\widetilde{\mathrm{Fre}}(V)$ has the following description:
\[d^{\mathbf k_{m+1}}_{\xi(\mathbf r_{m+1})}(-t)\cdot\xi_t(p)=\left\{
\begin{array}{ll}
\xi(p)&\text{if }r_{m+1}\leq p\leq r_m;\\
c^{\mathbf k_m}_{\xi(\mathbf r_m)}(t)\cdot \xi(p)&\text{if }r_m\leq p\leq r_{m+1}.
\end{array}\right.\]
Furthermore, the following lemma gives a useful matrix representation for the projective transformation $c^{\mathbf k_m}_{\mathbf E_m}(t)$ when we choose an appropriate basis for $V$. It follows immediately from Lemma \ref{lem:fixflag}.
\begin{lem}\label{lem:other shearing}
Let $\mathbf k:=(k_1,k_2)$ be a pair of positive integers that sum to $n$, and let $\mathbf E:=(E_1,E_2)$ be a transverse pair of flags. For $m=1$, $2$, let $\mathbf k_m:=(k_m,k_{m+1})$ and $\mathbf E_m:=(E_m,E_{m+1})$, where arithmetic in the subscripts are done modulo $2$. Let $\{e_{m,1},\dots,e_{m,n}\}$ be a basis of $V$ such that for all $l=1$, $\dots$, $n-1$, $E_m^{(l)}=\Span_{\Rbbb}\{e_{m,1},\dots,e_{m,l}\}$. Then $c^{\mathbf k_m}_{\mathbf E_m}(t)$ is represented in this basis by an upper triangular matrix where the first $i_m$ entries down the diagonal are $1$ and the last $i_{m+1}$ entries down the diagonal are $e^{-t}$.
\end{lem}

To prove that $\xi_t$ is a Frenet curve we first observe how the cross ratios and triple ratios evaluated at certain flags along $\xi_t$ change with $t$.

\begin{prop}\label{prop:shearing projective invariants}
Let $\mathbf k:=(k_1,k_2)$ be a pair of positive integers that sum to $n$, and let $\mathbf r:=(r_1,r_2)$ be a pair of distinct points in $S^1$. Let $ \left(\psi^{\mathbf k}_{\mathbf r}\right)_t$ be the $\mathbf{k}$-elementary shear flow with respect to $\mathbf{r}$, and let $\xi_t:=\left(\psi^{\mathbf k}_{\mathbf r}\right)_t(\xi)$. Then we have the following: 
\begin{enumerate}
\item Let $s_1,s_2$ be points in $S^1$ such that $\mathbf{a}:=(r_1,s_1,r_2,s_2)$ is a quadruple of cyclically ordered points in $S^1$. Then for any pair $\mathbf{l}:=(l_1,l_2)$ of positive integers that sum to $n$, we have $C^{\mathbf l}(\xi_t(\mathbf{a}))=e^{t\delta_{{\mathbf l}, {\mathbf k}}}C^{\mathbf k}(\xi(\mathbf{a}))$, where $\delta_{{\mathbf l}, {\mathbf k}} = 1$ if $\mathbf l ={ \mathbf k}$ and $0$ otherwise. 
\item For any quadruple of cyclically ordered points $\mathbf b:=(y_1,z_1,y_2,z_2)$ in $S^1$ such that $y_1,z_1,y_2,z_2$ lie in the closure of a connected component of $S^1\setminus\{r_1,r_2\}$, we have $C^{\mathbf l}(\xi_t(\mathbf{b}))=C^{\mathbf l}(\xi(\mathbf{b}))$ for any pair $\mathbf l:=(l_1,l_2)$ of positive integers that sum to $n$.
\item For any triple of cyclically ordered points $\mathbf x:=(x_1,x_2,x_3)$ in $S^1$ such that $x_1,x_2,x_3$ lie in the closure of a connected component of $S^1\setminus\{r_1,r_2\}$, we have $T^{\mathbf i}(\xi_t(\mathbf x))=T^{\mathbf i}(\xi(\mathbf x))$ for any triple $\mathbf i:=(i_1,i_2,i_3)$ of positive integers that sum to $n$.
\end{enumerate}
\end{prop}

For the proof we use a description of the cross ratio using determinants. Let $K_1, K_2\in\Pbbb(V^*)$ be hyperplanes in $V$ and let $P_1, P_2\in\Pbbb(V)$ be lines in $V$, such that $P_i$ is not contained in $K_j$ for all $i$, $j= 1$, $2$. For any basis $\{k_{i,1},\dots,k_{i,n-1}\}$ for the hyperplane $K_i$, there is a covector $\alpha_i$ with $\ker (\alpha_i) = K_i$ such that
\[\alpha_i(v_j)=\det(k_{i,1},\dots,k_{i,n-1},v_j).\] 
for any vector $v_j$ in $P_j$. We use this to express the cross ratio. 

\begin{proof}
Proof of (1). Let $\{f_1,\dots,f_n\}$ be a basis of $V$ such that $f_i$ is a non-zero vector in $\xi^{(i)}(r_1)\cap\xi^{(n-i+1)}(r_2)$. For $m=1$, $2$, let $e_m$ be a non-zero vector in $\xi^{(1)}(s_m)$, and write 
\[e_m=\sum_{k=1}^n\alpha_{m,k}f_k\]
for some real numbers $\alpha_{m,k}$. Then we can compute that
\begin{eqnarray*}
&&C^{\mathbf l}(\xi_t(\mathbf{a}))\\
&=& \frac{\det(f_1,\dots,\hat{f}_{l_1+1},\dots, f_n,d^{\mathbf k}_{\xi(\mathbf r)}(-t) \cdot e_2)}{\det(f_1,\dots,\hat{f}_{l_1+1},\dots, f_n, d^{\mathbf k}_{\xi(\mathbf r)}(t) \cdot e_1)}\cdot\frac{\det(f_1,\dots,\hat{f}_{l_1},\dots, f_n,d^{\mathbf k}_{\xi(\mathbf r)}(t) \cdot e_1)}{\det(f_1,\dots,\hat{f}_{l_1},\dots, f_n,d^{\mathbf k}_{\xi(\mathbf r)}(-t) \cdot e_2)}\\
&=& \frac{\alpha_{2,l_1+1}\det(f_1,\dots,\hat{f}_{l_1+1},\dots, f_n,d^{\mathbf k}_{\xi(\mathbf r)}(-t) \cdot f_{l_1+1})}{\alpha_{1,l_1+1}\det(f_1,\dots,\hat{f}_{l_1+1},\dots, f_n, d^{\mathbf k}_{\xi(\mathbf r)}(t) \cdot f_{l_1+1})}\cdot\\
&&\hspace{5cm}\frac{\alpha_{1,l_1}\det(f_1,\dots,\hat{f}_{l_1},\dots, f_n,d^{\mathbf k}_{\xi(\mathbf r)}(t) \cdot f_{l_1})}{\alpha_{2,l_1}\det(f_1,\dots,\hat{f}_{l_1},\dots, f_n,d^{\mathbf k}_{\xi(\mathbf r)}(-t) \cdot f_{l_1})}.
\end{eqnarray*}

If $l_1\neq k_1$, then $d^{\mathbf k}_{\xi(\mathbf r)}(t)$ scales $f_{l_1}$ and $f_{l_1+1}$ by the same amount, so
\begin{eqnarray*}
&&C^{\mathbf l}(\xi_t(\mathbf{a}))\\
&=&\frac{\alpha_{2,l_1+1}\det(f_1,\dots,\hat{f}_{l_1+1},\dots, f_n, f_{l_1+1})}{\alpha_{1,l_1+1}\det(f_1,\dots,\hat{f}_{l_1+1},\dots, f_n, f_{l_1+1})}\cdot\frac{\alpha_{1,l_1}\det(f_1,\dots,\hat{f}_{l_1},\dots, f_n,f_{l_1})}{\alpha_{2,l_1}\det(f_1,\dots,\hat{f}_{l_1},\dots, f_n,f_{l_1})}\\
&=&\frac{\det(f_1,\dots,\hat{f}_{l_1+1},\dots, f_n, e_2)}{\det(f_1,\dots,\hat{f}_{l_1+1},\dots, f_n,e_1)}\cdot\frac{\det(f_1,\dots,\hat{f}_{l_1},\dots, f_n,e_1)}{\det(f_1,\dots,\hat{f}_{l_1},\dots, f_n,e_2)}\\
&=&C^{\mathbf l}(\xi_t(\mathbf{a})).
\end{eqnarray*}

On the other hand, 
\begin{eqnarray*}
&&C^{\mathbf k}(\xi_t(\mathbf{a}))\\
&=&\frac{\alpha_{2,k_1+1}\det(f_1,\dots,\hat{f}_{k_1+1},\dots, f_n, e^{\frac{t}{2}}f_{k_1+1})}{\alpha_{1,k_1+1}\det(f_1,\dots,\hat{f}_{k_1+1},\dots, f_n, e^{-\frac{t}{2}}f_{k_1+1})}\cdot\frac{\alpha_{1,k_1}\det(f_1,\dots,\hat{f}_{k_1},\dots, f_n,f_{k_1})}{\alpha_{2,k_1}\det(f_1,\dots,\hat{f}_{k_1},\dots, f_n,f_{k_1})}\\
&=&e^t\cdot \frac{\det(f_1,\dots,\hat{f}_{k_1+1},\dots, f_n, e_2)}{\det(f_1,\dots,\hat{f}_{k_1+1},\dots, f_n,e_1)}\cdot\frac{\det(f_1,\dots,\hat{f}_{k_1},\dots, f_n,e_1)}{\det(f_1,\dots,\hat{f}_{k_1},\dots, f_n,e_2)}\\
&=&e^t\cdot C^{\mathbf k}(\xi_t(\mathbf{a})).
\end{eqnarray*}
This proves (1).

Proof of (2). We may assume without loss of generality that $r_1\leq y_1,z_1,y_2,z_2\leq r_2<r_1$. Then $\xi_t(y_m)=d^{\mathbf k}_{\xi(\mathbf r)}(t)\cdot\xi(y_m)$ and $\xi_t(z_m)=d^{\mathbf k}_{\xi(\mathbf r)}(t)\cdot\xi(z_m)$ for all $m=1$, $2$, and the projective invariance of the cross ratio immediately gives (2). The proof of (3) follows analoguously using the projective invariance of the triple ratio. 
\end{proof}

With Propositon \ref{prop:shearing projective invariants} we can now prove the following theorem using the same arguments (with obvious modifications) as in the proof of Theorem \ref{thm:eruption preserves Frenet}. We omit the proof to avoid repetition.

\begin{thm}\label{thm:shearing preserves Frenet}
Let $\mathbf k:=(k_1,k_2)$ be a pair of positive integers that sum to $n$, and let $\mathbf r:=(r_1,r_2)$ be a pair of distinct points in $S^1$. If $\xi$ is a Frenet curve, then for all $t\in\Rbbb$, the continuous curve $\xi_t=(\psi^\mathbf k_\mathbf r)_t(\xi)$ is also a Frenet curve. In particular, $\left(\psi^{\mathbf k}_{\mathbf r}\right)_t$, the $\mathbf k$-elementary shear flow with respect to $\mathbf r$, is a well-defined flow on $\widetilde{\mathrm{Fre}}(V)$.
\end{thm}

\subsection{The elementary flows descend}
The elementary flows on $\widetilde{\mathrm{Fre}}(V)$ in fact descend to well-defined flows on $\mathrm{Fre}(V)$. 

\begin{prop}
Let $\mathbf i:=(i_1,i_2,i_3)$ be a triple of positive integers that sum to $n$ and let $\mathbf x:=(x_1,x_2,x_3)$ be a cyclically ordered triple of points in $S^1$. Then the $\mathbf i$-elementary eruption flow with respect to $\mathbf{x}$, $\left(\epsilon^{\mathbf i}_{\mathbf x}\right)_t$, descends from $\widetilde{\mathrm{Fre}}(V)$ to a flow on $\mathrm{Fre}(V)$. 

Let $\mathbf k:=(k_1,k_2)$ be a pair of positive integers that sum to $n$ and $\mathbf r:=(r_1,r_2)$ a pair of distinct points in $S^1$. Then the $\mathbf k$-elementary shearing flow with respect to $\mathbf{r}$, 
 $\left(\psi^{\mathbf k}_{\mathbf r}\right)_t$, descends from $\widetilde{\mathrm{Fre}}(V)$ to a flow on $\mathrm{Fre}(V)$. 
\end{prop}

\begin{proof}
We only prove the statement for the elementary eruption flows; the proof for the elementary shearing flows is the same, with the obvious modifications. 

Let $\xi:S^1\to\Fmc(V)$ be any Frenet curve, and for $m=1$, $2$, $3$, let $\mathbf i_m:=(i_m,i_{m+1},i_{m-1})$ and $\mathbf x_m:=(x_m,x_{m+1},x_{m-1})$, where the arithmetic in the subscripts are done modulo $3$. For any projective transformation $g$ in $\PGL(V)$, observe that $gb^{\mathbf i_m}_{\xi(\mathbf x_m)}g^{-1}=b^{\mathbf i_m}_{g\cdot\xi(\mathbf x_m)}$. Hence, for any $p$ in $S^1$ and any $g$ in $\PGL(V)$, 
\begin{eqnarray*}
g\cdot\left(\epsilon^{\mathbf i}_{\mathbf x}\right)_t(\xi)(p)&=&\left\{
\begin{array}{ll}
gb^{\mathbf i_1}_{\xi(\mathbf x_1)}(t)\cdot \xi(p)&\text{if }x_2\leq p\leq x_3;\\
gb^{\mathbf i_2}_{\xi(\mathbf x_2)}(t)\cdot \xi(p)&\text{if }x_3\leq p\leq x_1;\\
gb^{\mathbf i_3}_{\xi(\mathbf x_3)}(t)\cdot \xi(p)&\text{if }x_1\leq p\leq x_2
\end{array}\right.\\
&=&\left\{
\begin{array}{ll}
b^{\mathbf i_1}_{g\cdot\xi(\mathbf x_1)}(t)g\cdot\xi(p)&\text{if }x_2\leq p\leq x_3;\\
b^{\mathbf i_2}_{g\cdot\xi(\mathbf x_2)}(t)g\cdot\xi(p)&\text{if }x_3\leq p\leq x_1;\\
b^{\mathbf i_3}_{g\cdot\xi(\mathbf x_3)}(t)g\cdot\xi(p)&\text{if }x_1\leq p\leq x_2
\end{array}\right.\\
&=&\left(\epsilon^{\mathbf i}_{\mathbf x}\right)_t(g\cdot\xi)(p).
\end{eqnarray*}
\end{proof}

We use the same notation and terminology for the flows on $\widetilde{\mathrm{Fre}}(V)$ and $\mathrm{Fre}(V)$.

\subsection{Properties of the elementary flows}
In this subsection, we establish some basic properties of elementary eruption and shearing flows.

\begin{prop}\label{prop:descend}
The elementary shearing and eruption flows on $ \mathrm{Fre}(V)$ satisfy the following properties. 
\begin{enumerate}
\item Let $\mathbf i:=(i_1,i_2,i_3)$ and $\mathbf j:=(j_1,j_2,j_3)$ be two triples of positive integers that sum to $n$, and let $\mathbf x:=(x_1,x_2,x_3)$ be any cyclically ordered triple of points in $S^1$. Then the flows $\left(\epsilon^{\mathbf i}_{\mathbf x}\right)_t$ and $\left(\epsilon^{\mathbf j}_{\mathbf x}\right)_{t'}$ on $\mathrm{Fre}(V)$ commute:
\[\left(\epsilon^{\mathbf i}_{\mathbf x}\right)_t\circ\left(\epsilon^{\mathbf j}_{\mathbf x}\right)_{t'}=\left(\epsilon^{\mathbf j}_{\mathbf x}\right)_{t'}\circ\left(\epsilon^{\mathbf i}_{\mathbf x}\right)_t\]
for all real numbers $t,t'$. 

\item Let $\mathbf i:=(i_1,i_2,i_3)$ and $\mathbf j:=(j_1,j_2,j_3)$ be two triples of positive integers that sum to $n$, and let $\mathbf x:=(x_1,x_2,x_3)$ be any cyclically ordered triple of points in $S^1$. If $\mathbf y:=(y_1,y_2,y_3)$ is a cyclically ordered triple of points in $S^1$ such that 
 $y_1,y_2,y_3$ lie in the closure of a connected component of $S^1\setminus\{x_1,x_2,x_3\}$, then the flows $\left(\epsilon^{\mathbf i}_{\mathbf x}\right)_t$ and $\left(\epsilon^{\mathbf j}_{\mathbf y}\right)_{t'}$ on $\mathrm{Fre}(V)$ commute:
 \[\left(\epsilon^{\mathbf i}_{\mathbf x}\right)_t\circ\left(\epsilon^{\mathbf j}_{\mathbf y}\right)_{t'}=\left(\epsilon^{\mathbf j}_{\mathbf y}\right)_{t'}\circ\left(\epsilon^{\mathbf i}_{\mathbf x}\right)_t\]
 for all real numbers $t,t'$.

\item For any pairs $\mathbf k:=(k_1,k_2)$ and $\mathbf l:=(l_1,l_2)$ of positive integers that sum to $n$, and pairs $\mathbf r:=(r_1,r_2)$ and $\mathbf s:=(s_1,s_2)$ of distinct points in $S^1$, such that 
$s_1,s_2$ lie in the closure of a connected component of $S^1\setminus\{r_1,r_2\}$, the flows $\left(\psi^{\mathbf k}_{\mathbf r}\right)_t$ and $\left(\psi^{\mathbf l}_{\mathbf s}\right)_{t'}$ commute: 
\[\left(\psi^{\mathbf k}_{\mathbf r}\right)_t\circ\left(\psi^{\mathbf l}_{\mathbf s}\right)_{t'}=\left(\psi^{\mathbf l}_{\mathbf s}\right)_{t'}\circ\left(\psi^{\mathbf k}_{\mathbf r}\right)_t\]
for all real numbers $t,t'$.
\item For any pair $\mathbf r:=(r_1,r_2)$ of distinct points in $S^1$ and cyclically ordered triple $\mathbf x:=(x_1,x_2,x_3)$, such that $x_1,x_2,x_3$ lie in the closure of a connected component of $S^1\setminus\{r_1,r_2\}$, the eruption flows associated to $\mathbf x$ and the shearing flows associated to $\mathbf{r}$ commute. i.e. 
\[\left(\psi^{\mathbf k}_{\mathbf r}\right)_t\circ\left(\epsilon^{\mathbf i}_{\mathbf x}\right)_{t'}=\left(\epsilon^{\mathbf i}_{\mathbf x}\right)_{t'}\circ\left(\psi^{\mathbf k}_{\mathbf r}\right)_t\]
for all pairs $\mathbf k:=(k_1,k_2)$ of positive integers that sum to $n$, all triples $\mathbf i:=(i_1,i_2,i_3)$ of positive integers that sum to $n$, and all real numbers $t,t'$.
\end{enumerate}
\end{prop}

\begin{proof}
For $m=1$, $2$, $3$, set
\[\begin{array}{ll}
\mathbf x_m:=(x_m,x_{m+1},x_{m-1}),&\mathbf y_m:=(y_m,y_{m+1},y_{m-1}),\\
\mathbf i_m:=(i_m,i_{m+1},i_{m-1}),& \mathbf j_m:=(j_m,j_{m+1},j_{m-1}),
\end{array}\] 
where arithmetic in the subscripts are done modulo $3$. Also, for $m=1$, $2$, set
\[\mathbf r_m:=(r_m,r_{m+1}),\,\,\, \mathbf s_m:=(s_m,s_{m+1}),\,\,\, \mathbf k_m:=(k_m,k_{m+1}),\,\,\,\mathbf l_m:=(l_m,l_{m+1})\] 
for $m=1$, $2$, where arithmetic in the subscripts are done modulo $2$.

Proof of (1). Set $\xi_1:=\left(\epsilon^{\mathbf i}_{\mathbf x}\right)_t\circ\left(\epsilon^{\mathbf j}_{\mathbf x}\right)_{t'}(\xi)$, $\xi_2:=\left(\epsilon^{\mathbf j}_{\mathbf x}\right)_{t'}\circ\left(\epsilon^{\mathbf i}_{\mathbf x}\right)_t(\xi)$, $\eta_1:=\left(\epsilon^{\mathbf j}_{\mathbf x}\right)_{t'}(\xi)$, and $\eta_2:=\left(\epsilon^{\mathbf i}_{\mathbf x}\right)_t(\xi)$. 
By Proposition \ref{prop:eruption projective invariants} and Proposition \ref{prop:parameterize triple}, there is a projective transformation $g$ in $\PGL(V)$ such that $g\cdot \xi_1(\mathbf x)=\xi_2(\mathbf x)$. We now show that $g\cdot\xi_1=\xi_2$.

Observe that for fixed $m=1$, $2$, $3$ and for all $x_{m+1}\leq p\leq x_{m-1}$, 
\[\xi_1(p)=b^{\mathbf i_m}_{\eta_1(\mathbf x_m)}(t)b^{\mathbf j_m}_{\xi(\mathbf x_m)}(t')\cdot\xi(p)\,\,\,\,\text{ and }\,\,\,\,\xi_2(p)=b^{\mathbf j_m}_{\eta_2(\mathbf x_m)}(t')b^{\mathbf i_m}_{\xi(\mathbf x_m)}(t)\cdot\xi(p).\]
It follows from the definitions that $b^{\mathbf i_m}_{\xi(\mathbf x_m)}(t)$, $b^{\mathbf j_m}_{\xi(\mathbf x_m)}(t')$, $b^{\mathbf i_m}_{\eta_1(\mathbf x_m)}(t)$ and $b^{\mathbf j_m}_{\eta_2(\mathbf x_m)}(t')$ fix $\xi^{(1)}(x_m)$. In particular, $\xi^{(1)}(x_m)=\xi_1^{(1)}(x_m)=\xi_2^{(1)}(x_m)$.

Moreover, the product 
\[b^{\mathbf j_m}_{\eta_2(\mathbf x_m)}(t')b^{\mathbf i_m}_{\xi(\mathbf x_m)}(t)b^{\mathbf j_m}_{\xi(\mathbf x_m)}(t')^{-1}b^{\mathbf i_m}_{\eta_1(\mathbf x_m)}(t)^{-1}\] 
maps $\xi_1(x_{m-1})$ to $\xi_2(x_{m-1})$ and $\xi_1(x_{m+1})$ to $\xi_2(x_{m+1})$. 
By Remark~\ref{rem:transitive} this completely determines the projective transformation $g$, which is 
\[g = b^{\mathbf j_m}_{\eta_2(\mathbf x_m)}(t')b^{\mathbf i_m}_{\xi(\mathbf x_m)}(t)b^{\mathbf j_m}_{\xi(\mathbf x_m)}(t')^{-1}b^{\mathbf i_m}_{\eta_1(\mathbf x_m)}(t)^{-1} .\] 
Thus, $g\cdot\xi_1(p)=\xi_2(p)$ for all $x_{m+1}\leq p\leq x_{m-1}$. Since this is true for all $m=1$, $2$, $3$, we see that $g\cdot\xi_1(p)=\xi_2(p)$ for all $p$ in $S^1$. This proves (1). 

The proofs of (2), (3) and (4) are similar, so we only give the proof of (2) to avoid repetition.

Proof of (2). We can assume without loss of generality that $x_3\leq y_1<y_2<y_3\leq x_1$. Set $\xi_1:=\left(\epsilon^{\mathbf i}_{\mathbf x}\right)_t\circ\left(\epsilon^{\mathbf j}_{\mathbf y}\right)_{t'}(\xi)$, $\xi_2:=\left(\epsilon^{\mathbf j}_{\mathbf y}\right)_{t'}\circ\left(\epsilon^{\mathbf i}_{\mathbf x}\right)_t(\xi)$, $\eta_1:=\left(\epsilon^{\mathbf j}_{\mathbf y}\right)_{t'}(\xi)$, and $\eta_2:=\left(\epsilon^{\mathbf i}_{\mathbf x}\right)_t(\xi)$. Applying Proposition \ref{prop:eruption projective invariants} and Proposition \ref{prop:Fock-Goncharov parametrization}, we deduce that there is some projective trasformation $g$ in $\PGL(V)$ such that $g\cdot\xi_1(\mathbf x)=\xi_2(\mathbf x)$ and $g\cdot\xi_1(\mathbf y)=\xi_2(\mathbf y)$. We now show that $g\cdot\xi_1=\xi_2$.

For all $x_1\leq p\leq x_2$, observe that
\[\xi_1(p)=b^{\mathbf i_3}_{\eta_1(\mathbf x_3)}(t)b^{\mathbf j_2}_{\xi(\mathbf y_2)}(t')\cdot\xi(p)\,\,\,
\text{ and } \,\,\,\xi_2(p)=b^{\mathbf j_2}_{\eta_2(\mathbf y_2)}(t')b^{\mathbf i_3}_{\xi(\mathbf x_3)}(t)\cdot\xi(p).\]
Similarly, 
\[\xi_1^{(1)}(x_3)=b^{\mathbf i_3}_{\eta_1(\mathbf x_3)}(t)b^{\mathbf j_2}_{\xi(\mathbf y_2)}(t')\cdot\xi^{(1)}(x_3)\,\,\,\text{ and }\,\,\,\xi_2^{(1)}(x_3)=b^{\mathbf j_2}_{\eta_2(\mathbf y_2)}(t')b^{\mathbf i_3}_{\xi(\mathbf x_3)}(t)\cdot\xi^{(1)}(x_3).\] 
Remark~\ref{rem:transitive} implies $g = b^{\mathbf i_3}_{\eta_1(\mathbf x_3)}(t)b^{\mathbf j_2}_{\xi(\mathbf y_2)}(t') \left(b^{\mathbf j_2}_{\eta_2(\mathbf y_2)}(t')b^{\mathbf i_3}_{\xi(\mathbf x_3)}(t)\right)^{-1}$, so $g\cdot\xi_1(p)=\xi_2(p)$ for all $x_1\leq p\leq x_2$. Repeating a similar argument for each of the intervals 
\[\{p\in S^1:x_2\leq p\leq x_3\},\,\,\,\{p\in S^1:y_1\leq p\leq y_2\},\,\,\,\{p\in S^1:y_2\leq p\leq y_3\},\] 
\[\text{and}\,\,\,\{p\in S^1:y_3\leq p\leq x_1\},\,\,\,\{p\in S^1:x_3\leq p\leq y_1\}\] 
shows that $g\cdot\xi_1(p)=\xi_2(p)$ for all $p\in S^1$.
\end{proof}

\section{Ideal triangulations and bridge systems} \label{sec:comb}
In this section we recall the definition of an ideal triangulations on $S$, and define the notion of a compatible bridge system. We explain how to combinatorially describe pairs of vertices in an ideal triangulation. This description is an important tool in Section \ref{sec:Hitchin_param}, Section \ref{sec:Hitchin}, and Section~\ref{sec:main theorem}.

\subsection{Ideal triangulations of $S$} \label{ideal triangulation}
We recall some facts about ideal triangulations on $S$. For more details, see \cite{Thurstonbook, Cassonbook, Pennerbook, Bo01}.

To describe ideal triangulations it is useful to endow $S$ with an auxiliary hyperbolic metric, such that we have a hyperbolic surface, which we denote by $X$. 
 An \emph{ideal triangulation} on $X$ is a family $\Tmc$ of finitely many disjoint, simple, unoriented geodesics on $X$, whose union is a closed set in $X$ that cuts $X$ into finitely many open sets, each of which is isometric to an ideal triangle in $\Hbbb^2$. We call the connected components of $X\setminus\bigcup_{c\in\Tmc}c$ the \emph{ideal triangles} of the ideal triangulation $\Tmc$. 

The geodesics in $\Tmc$ might be closed or not. Every non-closed geodesic $e$ in $\Tmc$ is isolated, i.e. there is an open set in $X$ containing $e$ that does not intersect any other geodesic in $\Tmc$. 
Every closed geodesic $c$ in $\Tmc$ is closed as a subset of $X$, and is not isolated; in fact for every collar neighborhood of $c$, both of the connected components of the complement of $c$ in this neighborhood intersect a non-closed geodesic in $\Tmc$. Therefore instead of closed geodesics in $\Tmc$ we talk about \emph{non-isolated edges} of $\Tmc$, and instead of non-closed geodesics we talk about \emph{isolated edges} of $\Tmc$. The advantage of this terminology is that this also makes sense when we lift the ideal triangulation of $X$ to an ideal triangulation $\widetilde{\Tmc}$ of $\widetilde {X} = \Hbbb^2$.
The ideal triangulation $\widetilde{\Tmc}$ is an infinite family of geodesics, whose complement is an infinite collection of ideal triangles, each of which is bijectively mapped to an ideal triangle of $\Tmc$ by the covering map. 

Even though we used a hyperbolic metric on $S$ to define it, an ideal triangulation is of combinatorial nature and can be described completely independent of a metric in terms of the Gromov boundary $\partial\Gamma$ of the fundamental group $\Gamma$. 
The hyperbolic metric induces a $\Gamma$-equivariant homeomorphism between $\partial\Gamma$ and the visual boundary $\partial \Hbbb^2$ of the hyperbolic plane. 
An unoriented geodesic on $\widetilde{X} =\Hbbb^2$ is determined by its two endpoints in $\partial\Gamma$. Choosing an ordering of this pair of points amounts to choosing an orientation on the corresponding geodesic. Thus we can identify the set of (unoriented) geodesics on the universal cover $\widetilde{S}$ with the set of unordered pairs of distinct points in $\partial \Gamma$. This allows to identify the ideal triangulation $\widetilde{\Tmc}$ with a subset of the space of unordered pairs of distinct points $\partial \Gamma$, and the set of ideal triangles in $\widetilde{\Tmc}$ with a subset of the spaces of pairwise distinct triples in $\partial\Gamma$. Since $\widetilde{\Tmc}$ is the lift of an ideal triangulation $\Tmc$ on $S$, these subsets are $\Gamma$-invariant. 

We denote by $\widetilde{\Qmc}$ the set of isolated edges in $\widetilde{\Tmc}$, by $\widetilde{\Pmc}$ the set of non-isolated edges in $\widetilde{\Tmc}$, and by $\widetilde{\Theta}$ the set of ideal triangles of $\widetilde{\Tmc}$. We denote by $\widetilde{\Vmc}$ the set of vertices of $\widetilde{\Tmc}$, i.e. the set of all points in $\partial\Gamma$ which arise as an endpoint of an edge of $\widetilde{\Tmc}$. All these sets are $\Gamma$-invariant, and there are natural identifications of their orbit spaces, namely of $\Qmc := \widetilde{\Qmc}/\Gamma$ with the set of isolated edges in $\Tmc$, of $\Pmc := \widetilde{\Pmc}/\Gamma$ with the set of non-isolated edges in $\Tmc$, and $\Theta := \widetilde{\Theta}/\Gamma$ with the set of triangles of $\Tmc$. Note that, by the Gauss-Bonnet theorem we have \[|\Theta|=4g-4\,\,,\,\,\, |\Qmc|=6g-6\,\,,\,\,\, 1\leq |\Pmc|\leq 3g-3\\,\,\,\text{ and }\,\,\,|\Tmc|=|\Qmc|+|\Pmc|,\] where $g$ is the genus of $S$.
When $|\Pmc|=3g-3$, $\Pmc$ is a \emph{pants decomposition} of $S$.

We collect the following useful observations: 
\begin{enumerate}
\item Any ideal triangle of $\widetilde{\Tmc}$ is bound by three isolated edges, i.e. if $\{x_1,x_2,x_3\}$ is an ideal triangle in $\widetilde{\Theta}$, then $\{x_i,x_{i+1}\}$ is an isolated edge in $\widetilde{\Tmc}$ for $i=1$, $2$, $3$. 
\item Any isolated edge bounds two triangles, i.e. if $\{r_1,r_2\}$ is an isolated edge in $\widetilde{\Tmc}$, then there are points $z_1$ and $z_2$ in $\partial\Gamma$ such that $\{r_1,r_2,z_1\}$ and $\{r_1,r_2,z_2\}$ are ideal triangles in $\widetilde{\Theta}$. 
\item Isolated edges in $\Tmc$ accumulate onto non-isolated edges, i.e. if $\{r_1,r_2\}$ is an isolated edge in $\widetilde{\Tmc}$, then there exist non-isolated edges $\{s_1,s_2\}$ and $\{y_1,y_2\}$ in $\widetilde{\Tmc}$ with $r_1=s_1$ and $r_2=y_2$. 
\item Every non-isolated edge in $\Tmc$ is a closed geodesic, i.e. if $\{r_1,r_2\}$ is a non-isolated edge in $\widetilde{\Tmc}$ , then there exists a group element $\gamma$ in $\Gamma$, whose attracting and repelling fixed point is $r_1$ and $r_2$ respectively. 
\item For every collar neighborhood $U$ of a non-isolated edge in $c$ in $\Tmc$, both connected components of the complement $U\setminus c$ intersect isolated edges of $\Tmc$. This means that if $\{r_1,r_2\}$ is a non-isolated edge in $\widetilde{\Tmc}$, then there are sequences $\{y_i\}_{i=1}^\infty,\{z_i\}_{i=1}^\infty\subset\widetilde{\Vmc}$ such that
\begin{itemize}
\item $(r_1,y_i,r_2,z_i)$ is cyclically ordered for all $i$,
\item $\{r_1,y_i\}$ is an isolated edge in $\widetilde{\Tmc}$ for all $i$ and $\lim_{i\to\infty}y_i=r_2$, or $\{r_2,y_i\}$ is an isolated edge in $\widetilde{\Tmc}$ for all $i$ and $\lim_{i\to\infty}y_i=r_1$, and
\item $\{r_1,z_i\}$ is an isolated edge in $\widetilde{\Tmc}$ for all $i$ and $\lim_{i\to\infty}z_i=r_2$, or $\{r_2,z_i\}$ is an isolated edge in $\widetilde{\Tmc}$ for all $i$ and $\lim_{i\to\infty}z_i=r_1$.
\end{itemize}
\end{enumerate}

We close this subsection with the statement that any Frenet curve is determined by its behaviour at the vertices of an ideal triangulation. The proof of this lemma is given in Appendix \ref{app:technical}.

\begin{lem}\label{lem:technical2}
Let $\Tmc$ be an ideal triangulation of $S$ and let $\widetilde{\Tmc}$ be the lift to the universal covering. 
Let $\widetilde{\Vmc} \subset S^1 = \partial \Gamma$ denote the set of vertices of $\widetilde{\Tmc}$, and let $(\xi_j)_{j = 1}^\infty$ be a family of Frenet curves in $\widetilde{\mathrm{Fre}}(V)$.
 If $\lim_{j\to\infty}\xi_j(p)=\xi_0(p)$ for all vertices $p$ in $\widetilde{\Vmc}$, then $\lim_{j\to\infty}\xi_j(p)=\xi_0(p)$ for all points $p$ in $\partial\Gamma$.
\end{lem}

\subsection{Bridge system}

 Let $\Tmc$ be an ideal triangulation on $S$. Using $\Tmc$, we now introduce the notion of a \emph{bridge system} on $S$ compatible with $\Tmc$. This is the choice of further topological data on $S$ that is needed to completely determine a marking on $S$. More precisely, if a homeomorphism from $S$ to itself preserves an ideal triangulation $\Tmc$ (viewed as a collection of $\Gamma$-orbits of pairs of points in $\partial\Gamma$), it is not necessarily isotopic to the identity; for example, one can perform a Dehn twist around a non-isolated edge in $\Tmc$. However, if the homeomorphism preserves both the ideal triangulation and a compatible bridge system, then it is isotopic to the identity. There are of course many ways to fix topological data on $S$ to remove the ambiguity of Dehn twists around the non-isolated edges in $\Tmc$; the way we choose to do so is specially suited for constructing flows on the Hitchin component.

\begin{definition}
Let $\widetilde{\Tmc}$ be an ideal triangulation of $\widetilde{S}$ and let $\{r_1,r_2\}$ be a non-isolated edge in $\widetilde{\Tmc}$.
\begin{itemize}
\item A \emph{bridge} across $\{r_1,r_2\}$ is an (unordered) pair $\{T_1,T_2\}$ of ideal triangles in $\widetilde{\Theta}$ such that for all $m=1$, $2$, one of the vertices of $T_m$ is $r_1$ or $r_2$, and the other two vertices of $T_1$ lie in a different connected component of $\partial\Gamma\setminus\{r_1,r_2\}$ from the other two vertices of $T_2$ (see Figure \ref{fig:idealtriangulation}).
\item A \emph{bridge system} compatible with $\widetilde{\Tmc}$, denoted $\widetilde{\Jmc}=\widetilde{\Jmc}_{\widetilde{\Tmc}}$, is a minimal $\Gamma$-invariant collection of bridges such that for every non-isolated edge $\{r_1,r_2\}$ in $\widetilde{\Tmc}$, there is a bridge $\{T_1,T_2\}$ across $\{r_1,r_2\}$ that lies in $\widetilde{\Jmc}$. We also say that $\Jmc=\Jmc_\Tmc:=\widetilde{\Jmc}/\Gamma$ is a \emph{bridge system} compatible with $\Tmc$.
\end{itemize}
\end{definition}

When we choose a hyperbolic metric on $S$ and work on the hyperbolic surface $X$, the bridge system $\Jmc$ can be realized as a collection of ``short" transverse geodesic segments, one for each closed geodesic in $\Pmc$. Each geodesic segment in $\Jmc$ intersects a unique closed geodesic in $\Pmc$ transversely, and each closed geodesic in $\Pmc$ intersects a unique geodesic segment in $\Jmc$ transversely. In our figures we therefore represent bridges by lines that are transverse to the non-isolated edges of the triangulation. The endpoints of the lines lie in the two triangles forming the bridge.

\begin{figure}[ht]
\centering
\includegraphics[scale=0.8]{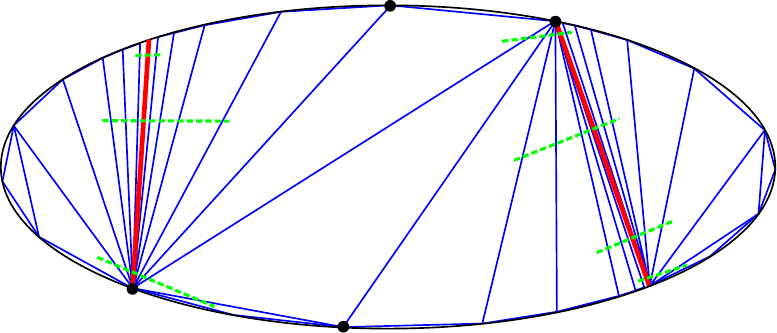}
\small
\put (-153, 130){$s_2$}
\put (-90, 124){$r_2$}
\put (-171, -4){$s_1$}
\put (-253, 10){$r_1$}
\caption{The thin blue lines represent isolated edges of $\widetilde{\Tmc}$, the thick red lines represent non-isolated edges of $\widetilde{\Tmc}$, and the green dotted lines represent bridges in $\widetilde{\Jmc}$.}\label{fig:idealtriangulation}
\end{figure}

\subsection{Combinatorial description of a pair of distinct vertices of $\widetilde{\Tmc}$}\label{sec:combinatorial description}
Let $\Tmc$ be an ideal triangulation on $S$. Next, we give a combinatorial description of an arbitrary pair $(x,y)$ of distinct vertices of $\widetilde{\Tmc}$. This is useful later (see Section \ref{sec:edge and triangle invariants}, Section~\ref{sec:warmup}, and Section~\ref{sec:semielementary}) to keep track of how the flags assigned to $x$ and $y$ by a Frenet curve change relative to each each other when we deform the Frenet curve.

We may choose a hyperbolic metric on $S$ to have at our disposal a hyperbolic surface $X$. As described in the previous subsections, all the constructions we make are of combinatorial nature and do not depend on this metric, but the choice of metric allows us to talk about geodesics as well as their intersection. 

For any pair $(x,y)$ of distinct vertices of $\widetilde{\Tmc}$, first consider the set 
\[\Emc'_{(x,y)}:=\big\{\{r_1,r_2\}\in\widetilde{\Tmc}:r_1<x<r_2<y<r_1\big\}.\]
The set $\Emc'_{(x,y)}$ is the (possibly infinite) collection of edges in $\widetilde{\Tmc}$ that transversely intersect the geodesic in the universal cover of $S$ whose endpoints are $x$ and $y$. If $\{x,y\}$ is an edge of the ideal triangulation $\widetilde{\Tmc}$, then $\Emc'_{(x,y)}$ is empty.
In general, $\Emc'_{(x,y)}$ admits a natural ordering $\leq$ that can be described as follows: orient both components of $\partial\Gamma\setminus\{x,y\}$ from $x$ to $y$, and define $\{r_1,r_2\}\leq\{r_1',r_2'\}$ if $r_1$ and $r_1'$ (hence $r_2$ and $r_2'$) lie in the same connected component of $\partial\Gamma\setminus\{x,y\}$, $r_1$ weakly precedes $r_1'$, and $r_2$ weakly precedes $r_2'$. 
With respect to this ordering, $\Emc'_{(x,y)}$ does not necessarily have a minimal (resp. maximal) element. Thus, we enlarge $\Emc'_{(x,y)}$ to a set $\Emc_{(x,y)}$ that has an ordering which restricts to the natural ordering on $\Emc'_{(x,y)}$, and always has a maximum and a minimum.

Note that $\Emc'_{(x,y)}$ does not have a minimum (resp. maximum) if and only if there is some vertex $p$ (resp. $q$) of $\widetilde{\Tmc}$ such that 
\begin{itemize}
\item the edge $\{x,p\}$ (resp. $\{y,q\}$) is a non-isolated edge of $\Tmc$, 
\item there is a sequence $\{z_i\}_{i=1}^\infty$ of vertices of $\widetilde{\Tmc}$ that converges to $x$ (resp. $y$), and 
\item the edge $\{z_i,p\}$ (resp. $\{z_i,p\}$) is contained in $\Emc'_{(x,y)}$ for all $i$.
\end{itemize} 
We therefore set
\[\Emc_{(x,y)}:=\left\{\begin{array}{ll}\Emc'_{(x,y)}&\text{if }\Emc'_{(x,y)}\text{ has a max and a min},\\
\Emc'_{(x,y)}\cup\big\{\{x,p\}\big\}&\text{if }\Emc'_{(x,y)}\text{ has a max but no min},\\
\Emc'_{(x,y)}\cup\big\{\{y,q\}\big\}&\text{if }\Emc'_{(x,y)}\text{ has a min but no max},\\
\Emc'_{(x,y)}\cup\big\{\{x,p\},\{y,q\}\big\}&\text{if }\Emc'_{(x,y)}\text{ has neither a max nor a min}.
\end{array}\right.\]

There are only finitely many non-isolated edges (possibly none) in $\widetilde{\Tmc}$ that lie in $\Emc_{(x,y)}$.
We now decompose the set $\Emc_{(x,y)}$ with respect to the these non-isolated edges. Suppose first that $\Emc_{(x,y)}$ contains a non-isolated edge. Let $e_1$, $\dots$, $e_k$ be the non-isolated edges in $\Emc_{(x,y)}$, enumerated according to the ordering on $\Emc_{(x,y)}$.
For any $s=1$, $\dots$, $k$ we set 
\[\Emc_{(x,y),s}=\Emc_s:=\{e\in\Emc_{(x,y)}:e\text{ shares a vertex with }e_s\}.\]
The set $\Emc_s$ is infinite, but has a well-defined minimum and maximum (see Figure~\ref{fig:abcd}).

If $e$ is an edge in $\Emc_{(x,y)}$ that shares a common vertex $w$ with some $e_s$, then every edge in $\Emc_{(x,y)}$ that lies between $e$ and $e_s$ also has $w$ as a vertex. Therefore if we set 
\[\Fmc_s^-:=\{e\in\Emc_{(x,y)}:e<e'\text{ for all }e'\in\Emc_s\},\]
\[\Fmc_s^+:=\{e\in\Emc_{(x,y)}:e>e'\text{ for all }e'\in\Emc_s\},\]
then $\Emc_{(x,y)}=\Fmc_s^-\cup\Emc_s\cup\Fmc_s^+$ is a disjoint union. The set $\Fmc_s^\pm$ might be empty.

We now set 
\[\Emc_{(x,y),s,s+1}=\Emc_{s,s+1}:=\left\{\begin{array}{ll}
\Fmc_1^-&\text{if }s=0;\\
\Fmc_s^+\cap\Fmc_{s+1}^-&\text{if }0<s<k;\\
\Fmc_k^+&\text{if }s=k\\
\end{array}\right.\]
(see Figure~\ref{fig:abcd1}). The set $\Emc_{s,s+1}$ is finite (possibly empty) for all $s=0$, $\dots$, $k$, and in particular has a minimum and maximum if it is non-empty. This gives us a decomposition of $\Emc_{(x,y)}$ into the disjoint union
\[\Emc_{(x,y)}=\bigcup_{s=1}^k\Emc_s\cup\bigcup_{s=0}^k\Emc_{s,s+1}\] 
(see Figure~\ref{Exy}).

In the case when $\Emc_{(x,y)}$ does not contain a non-isolated edge, we adopt the convention that $k=0$, i.e. $\Emc_{(x,y)}=\Emc_{(0,1)}$.

\begin{figure}[ht]
\centering
\includegraphics[scale=0.8]{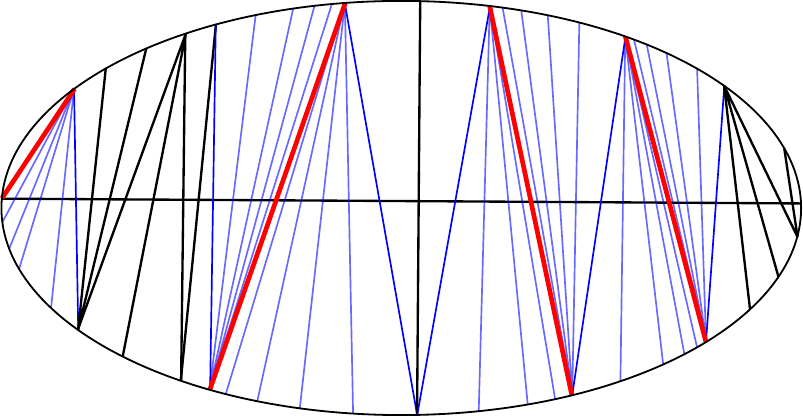}
\small
\put (-316, 82){$x$}
\put (2, 81){$y$}
\caption{The geodesics drawn above that intersect $\{x,y\}$ transversely represent the edges in $\Emc_{(x,y)}$. The thick red lines are non-isolated edges of $\widetilde{\Tmc}$, the thin blue lines geodesics are edges in $\Emc_s$ for some $s$, and the black lines are edges in $\Emc_{s,s+1}$ for some $s$. }\label{Exy}
\end{figure}

One can think of the data $\Emc_s$ as encoding the combinatorics of the geodesic with endpoints $x$ and $y$ in a ``neighborhood" of the non-isolated edge $e_s$, while $\Emc_{s,s+1}$ encodes the combinatorics as it ``moves between" $e_s$ and $e_{s+1}$.

\subsection{The ends of $\Emc_s$ and $\Emc_{s,s+1}$.}\label{sec:ends}
Recall that $(x,y)$ is a pair of distinct vertices of $\widetilde{\Tmc}$. In this section we introduce the backward and the forward end of the sets $\Emc_s=\Emc_{(x,y),s}$ and $\Emc_{s,s+1}=\Emc_{(x,y),s,s+1}$ defined in Section \ref{sec:combinatorial description}. These ends are usually ideal triangles in $\widetilde{\Theta}$, but might collapse to non-isolated edges in $\widetilde{\Tmc}$ in the cases when the end is the backward end (resp. forward end) of $\Emc_1$ (resp. $\Emc_k$) when the minimum (resp. maximum) of $\Emc_{(x,y)}$ is a non-isolated edge. The ends are defined in such a way that, provided $\Emc_{s,s+1}$ is non-empty, the forward end of $\Emc_{s,s+1}$ agrees with the backward end of $\Emc_{s+1}$, and the backward end of $\Emc_{s,s+1}$ agrees with the forward end of $\Emc_s$. If $\Emc_{s,s+1}$ is empty, then the forward end of $\Emc_s$ agrees with the backward end of $\Emc_{s+1}$. 

We use the set $\Emc_{(x,y)}$, the decomposition $\Emc_{(x,y)}=\bigcup_{s=1}^k\Emc_s\cup\bigcup_{s=0}^k\Emc_{s,s+1}$ defined in Section \ref{sec:combinatorial description}, and the forward and backward ends of $\Emc_s$ and $\Emc_{s,s+1}$to control the behaviour of a Frenet curve at $y$ when we are given its behaviour at $x$. 

Recall that $\widetilde{\Vmc}$ denotes the set of vertices of $\widetilde{\Tmc}$. For any $s=1$, $\dots$, $k$, consider $\Emc_s$ as defined in Section~\ref{sec:combinatorial description}. Let the edges $\{a_s,b_s\}$ and $\{c_s,d_s\}$ be the minimum and maximum of $\Emc_s$ respectively, such that $b_s$ and $d_s$ are vertices of $e_s$. If $s=1$, $\dots$, $k-1$, or $s=k$ and $\Emc_{k,k+1}$ is non-empty, let $d_s'$ be the vertex of $\widetilde{\Tmc}$ such that $\{c_s,d_s'\}$ is the successor of $\{c_s,d_s\}$ with respect to the natural ordering. Similarly, if $s=2$, $\dots$, $k$, or $s=1$ and $\Emc_{0,1}$ is non-empty, let $b_s'$ be the vertex of $\widetilde{\Tmc}$ in $\partial\Gamma$ such that $\{a_s,b_s'\}$ is the predecessor of $\{a_s,b_s\}$. On the other hand, if $s=k$ and $\Emc_{k,k+1}$ is empty, let $d_k':=y$ and if $s=1$ and $\Emc_{0,1}$ is empty, let $b_0':=x$ (see Figure \ref{fig:abcd}).

\begin{figure}[ht]
\centering
\includegraphics[scale=0.7]{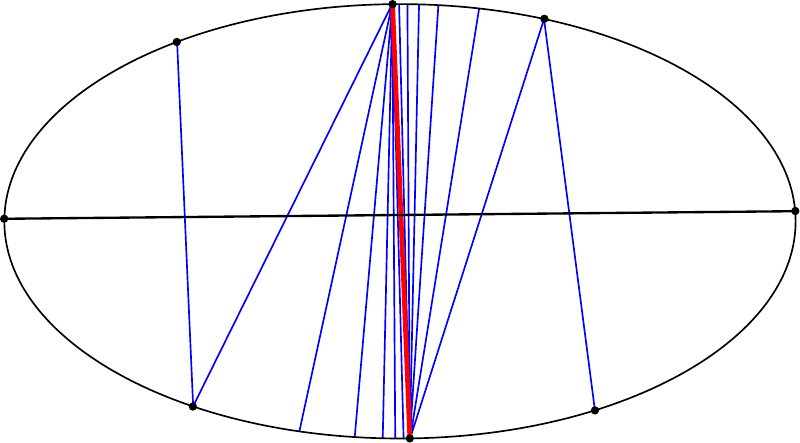}
\small
\put (-140, 151){$b_s$}
\put (-213, 139){$b_s'$}
\put (-207, 5){$a_s$}
\put (-135, -7){$d_s$}
\put (-70, 2){$d_s'$}
\put (-87, 146){$c_s$}
\put (-276, 74){$x$}
\put (1, 77){$y$}
\caption{The lines in $\Emc_s\cup\big\{\{a_s,b_s'\}\big\}\cup\big\{\{c_s,d_s'\}\big\}$, where the thick red line is the non-isolated edge $e_s$.}\label{fig:abcd}
\end{figure}

\begin{definition}\label{def:end} 
The triple $\{a_s,b_s,b_s'\}$ is the backward \emph{end} of $\Emc_s$, the triple $\{c_s,d_s,d_s'\}$ is the forward end of $\Emc_s$. 
\end{definition}

\begin{remark}\label{rem:special_ends}
If the minimum of $\Emc_{(x,y)}$ is a non-isolated edge of $\Tmc$, then $\Emc_{0,1}$ is empty and the two vertices $b_1$ and $b_1'$ of the backward end of $\Emc_1$ agree. Similarly, if the maximum $\Emc_{(x,y)}$ is a non-isolated edge of $\Tmc$, then $\Emc_{k,k+1}$ is empty and the two vertices $d_k$ and $d_k'$ of the forward end of $\Emc_k$ agree. In these two cases, the ends collapse to non-isolated edges in $\widetilde{\Tmc}$. In all other cases, the ends of $\Emc_s$ are ideal triangles in $\widetilde{\Theta}$.
\end{remark}

\begin{figure}[ht]
\centering
\includegraphics[scale=0.8]{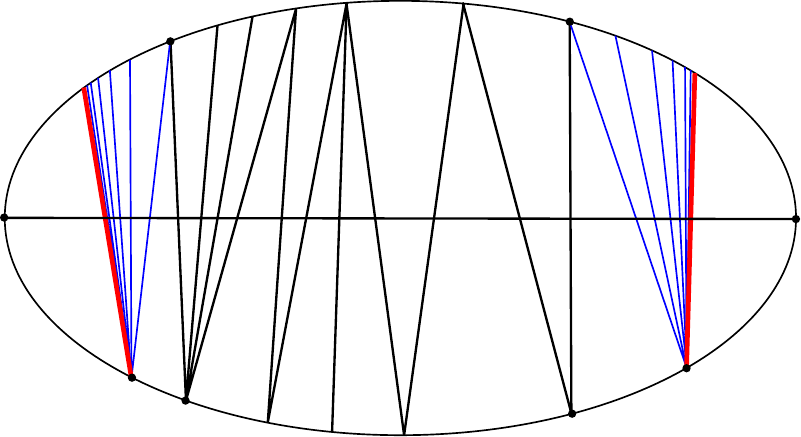}
\small
\put (-276, 18){$b'_{s,s+1}$}
\put (-266, 156){$a_{s,s+1}$}
\put (-254, 7){$b_{s,s+1}$}
\put (-44, 19){$d_{s,s+1}'$}
\put (-89, 1){$d_{s,s+1}$}
\put (-94, 164){$c_{s,s+1}$}
\put (-315, 82){$x$}
\put (1, 82){$y$}
\caption{The thick red lines are $e_s$ and $e_{s+1}$, the thin blue lines lie in $\Emc_s$ and $\Emc_{s+1}$, and the black lines lie in $\Emc_{s,s+1}$.}\label{fig:abcd1}
\end{figure}

Next, suppose that $s=0$, $\dots$, $k$ and that $\Emc_{s,s+1}$ is non-empty. 
To define the backward end of $\Emc_{s,s+1}$ for $s>0$, we consider the minimum $\{a_{s,s+1},b_{s,s+1}\}$ of $\Emc_{s,s+1}$, where $a_{s,s+1}$ is the common vertex of $\{a_{s,s+1},b_{s,s+1}\}$ and its predecessor. Then let $b'_{s,s+1}$ denote the vertex of its predecessor that is not $a_{s,s+1}$. Similarly, for the forward end of $\Emc_{s,s+1}$ for $s<k$, we consider the maximum $\{c_{s,s+1},d_{s,s+1}\}$ of $\Emc_{s,s+1}$, where $c_{s,s+1}$ is the common vertex of $\{c_{s,s+1},d_{s,s+1}\}$ and its successor. Then let $d'_{s,s+1}$ denote the vertex of its successor that is not $c_{s,s+1}$ (see Figure \ref{fig:abcd1}). For $s= 0$ we set $b'_{0,1}:=x$ and for $s=k$, we set $d'_{k,k+1}:=y$. 

\begin{definition}\label{def:end1} 
The triple $\{a_{s,s+1},b_{s,s+1},b_{s,s+1}'\}$ is called the backward \emph{end} of $\Emc_{s,s+1}$, and the triple $\{c_{s,s+1},d_{s,s+1},d_{s,s+1}'\}$ is called the forward end of $\Emc_{s,s+1}$. 
\end{definition}

Observe that for $s=1$, $\dots$, $k-1$, if $\Emc_{s,s+1}$ is non-empty, then the forward end of $\Emc_{s,s+1}$ agrees with the backward end of $\Emc_{s+1}$, and the backward end of $\Emc_{s,s+1}$ agrees with the forward end of $\Emc_s$. On the other hand, if $\Emc_{s,s+1}$ is empty, then the forward end of $\Emc_s$ agrees with the backward end of $\Emc_{s+1}$. All of these ends are ideal triangles in $\widetilde{\Theta}$.
%
\section{Parametrizing Frenet curves} \label{sec:Hitchin_param}
The goal of this section is to describe a new parametrization of $\Hit_V(S)$. This parametrization is a slight, but non-trivial modification of the Bonahon-Dreyer parametrization \cite{BD1}, which we also describe. The parametrization is with respect to a fixed ideal triangulation and a fixed compatible bridge system, and given by invariants associated to the edges and triangles of the triangulation using cross ratios and triple ratios. Along the way we show, thatthese edge and triangle invariants completely determine the projective class a Frenet curve $[\xi]$ in $\mathrm{Fre}(V)$, even if the curve it is not $\rho$-equivariant for some representation $\rho:\Gamma \to \PGL(V)$. 
%

\subsection{Edge and triangle invariants}\label{sec:edge and triangle invariants}
We fix an ideal triangulation $\Tmc$ and a compatible bridge system $\Jmc$. 
Given a Frenet curve $\xi:\partial\Gamma\to\Fmc(V)$, we associate invariants to the edges and ideal triangles of the ideal triangulation $\Tmc$. These invariants were introduced by Fock and Goncharov \cite{FockGoncharov}, and are based on the cross ratios and triple ratios described in Section~\ref{sec:projective}.
They are also used in the Bonahon-Dreyer parametrization \cite{BD1} of the Hitchin component.

\begin{figure}[ht]
\centering
\includegraphics[scale=0.6]{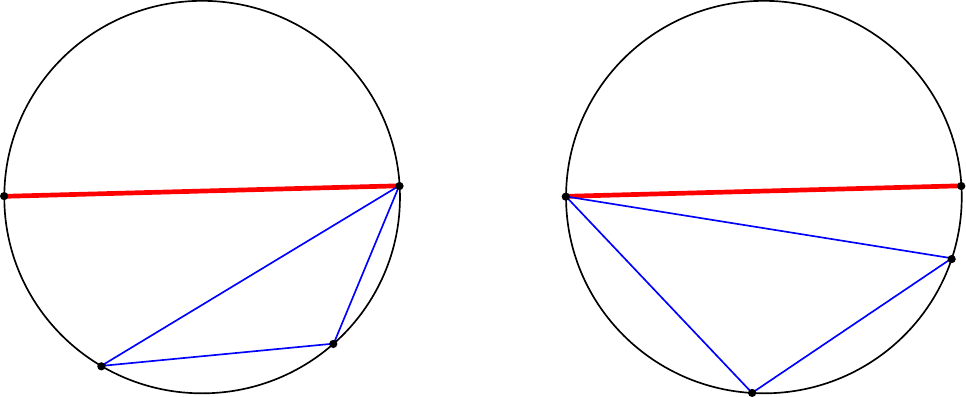}
\small
\put (-292, 58){$q_m$}
\put (-162, 60){$p_m$}
\put (-258, 4){$w_m$}
\put (-183, 10){$z_m$}
\put (-129, 58){$p_m$}
\put (1, 60){$q_m$}
\put (-70, -4){$z_m$}
\put (-3, 39){$w_m$}
\caption{The thick red line is the non-isolated edge $\{r_1,r_2\}$, and the blue triangle is $T_m$.}\label{fig:pq}
\end{figure}

\subsubsection{The edge invariants}\label{sec:edge}
Let $\{r_1,r_2\}$ be an edge of the ideal triangulation $\widetilde{\Tmc}$.
If $\{r_1,r_2\}$ is a non-isolated edge, let $J=\{T_1,T_2\}$ be a bridge in $\widetilde{\Jmc}$ across $\{r_1,r_2\}$. 
For $m=1$, $2$, we set $p_m$ (resp. $q_m$) to be the vertex of the edge $\{r_1,r_2\}$ that is (resp. is not) a vertex of $T_m$, and let $z_m,w_m$ be the other two vertices of $T_m$ such that either $(p_m,z_m,w_m,q_m)$ or $(q_m,w_m,z_m,p_m)$ is cyclically ordered (see Figure \ref{fig:pq}). By switching the roles of $T_1$ and $T_2$ if necessary, we may assume that $T_1$ (resp. $T_2$) lies to the left (resp. right) of $\mathbf r:=(r_1,r_2)$, which we view as an oriented geodesic from $r_1$ to $r_2$. This means that $(r_1,z_1,r_2,z_2)$ is a cyclically ordered quadruple of points in $\partial \Gamma$. 

If $\{r_1,r_2\}$ is an isolated edge, let $z_1$ and $z_2$ be the vertices of $\widetilde{\Tmc}$ such that $\{r_1,r_2,z_m\}$ is an ideal triangle of $\widetilde{\Tmc}$ for $m=1$, $2$, and $(r_1,z_1,r_2,z_2)$ is a cyclically ordered quadruple of points in $\partial\Gamma$ (see Figure \ref{fig:idealtriangulation}). 

\begin{definition}\label{def:edgeinvariant}
Let $\mathbf k:=(k_1,k_2)$ be a pair of positive integers that sum to $n$, let $\mathbf r:=(r_1,r_2)$ be an edge in $\widetilde{\Tmc}$, let $\mathbf z:=(z_1,z_2)$ be as defined above, and set $\mathbf a:=(r_1,z_1,r_2,z_2)$. The $\mathbf k$-\emph{edge invariant} along $\mathbf r$ is the function 
\[\sigma^{\mathbf k}_{\mathbf r}:\mathrm{Fre}(V)\to\Rbbb\] 
defined by $\sigma^{\mathbf k}_{\mathbf r}[\xi]:=\log\left(-C^{\mathbf k}(\xi(\mathbf a))\right)$.
\end{definition}

\begin{remark}
For a non-isolated edge $\{r_1,r_2\}$, the edge invariants along $\mathbf r:=(r_1,r_2)$ depend on the bridge $J$ in $\widetilde{\Jmc}$ across $\{r_1,r_2\}$. We write $\sigma^{\mathbf k}_{\mathbf r,J}[\xi]$ when we want to emphasize this dependence. When $\xi$ is $\rho$-equivariant for some representation $\rho:\Gamma\to\PGL(V)$, then the edge invariants along $\mathbf r$ do not depend on $J$ in $\widetilde{\Jmc}$. 
\end{remark}

By Theorem \ref{prop:positive} and the projective invariance of the cross ratio, the $\mathbf k$-edge invariant $\sigma^{\mathbf k}_{\mathbf r}$ is well-defined for all pairs of integers $\mathbf k:=(k_1,k_2)$ that sum to $n$ and all pairs $\mathbf r:=(r_1,r_2)$ of points in $S^1$ such that $\{r_1,r_2\}$ is an edge in $\widetilde{\Tmc}$.
From the symmetries of the cross ratio,
\[\sigma^{\mathbf k_1}_{\mathbf r_1}=\sigma^{\mathbf k_2}_{\mathbf r_2},\]
if $\mathbf k_m:=(k_m,k_{m+1})$, and $\mathbf r_m:=(r_m,r_{m+1})$ for $m=1$, $2$.

\subsubsection{The triangle invariants}\label{sec:triangle}
Recall that $\widetilde{\Theta}$ is the set of ideal triangles of $\widetilde{\Tmc}$.

\begin{definition}
Let $\mathbf i:=(i_1,i_2,i_3)$ be a triple of positive integers that sum to $n$, and let $\mathbf x:=(x_1,x_2,x_3)$ be a cyclically ordered triple of points in $\partial\Gamma$ such that $\{x_1,x_2,x_3\}$ is an ideal triangle in $\widetilde{\Theta}$. The $\mathbf i$-\emph{triangle invariant} of $\mathbf{x}$ is the function 
\[\tau^{\mathbf i}_{\mathbf x}:\mathrm{Fre}(V)\to\Rbbb, \] 
where $\tau^{\mathbf i}_{\mathbf x}[\xi]:=\log\left(T^{\mathbf i}(\xi(\mathbf x))\right)$.
\end{definition}

By Theorem \ref{prop:positive} and the projective invariance of the triple ratio, the triangle invariant $\tau^{\mathbf i}_{\mathbf x}$ is well-defined for all triples $\mathbf i:=(i_1,i_2,i_3)$ of positive integers that sum to $n$ and for all cyclically ordered triples $\mathbf x:=(x_1,x_2,x_3)$ in $\partial\Gamma$ such that $\{x_1,x_2,x_3\}$ is an ideal triangle in $\widetilde{\Tmc}$. From the symmetries of the triple ratio, we see that 
\[\tau^{\mathbf i_1}_{\mathbf x_1}=\tau^{\mathbf i_2}_{\mathbf x_2}=\tau^{\mathbf i_3}_{\mathbf x_3},\]
where $\mathbf i_m:=(i_m,i_{m+1},i_{m-1})$, and $\mathbf x_m:=(x_m,x_{m+1},x_{m-1})$ 
for $m=1$, $2$, $3$.

The edge and triangle invariants determine the $\PGL(V)$-orbit of a Frenet curve. 

\begin{prop}\label{prop:invariants determine curve}
Let $[\xi_1],[\xi_2]$ be projective classes of Frenet curves in $\mathrm{Fre}(V)$. Then $[\xi_1]=[\xi_2]$ if and only if all of the following hold:
\begin{itemize}
\item For all edges $\{r_1, r_2\}$ of $\widetilde{\Tmc}$ and all pairs of positive integers $\mathbf k:=(k_1,k_2)$ that sum to $n$, the $\mathbf k$-edge invariants of $[\xi_1]$ and $[\xi_2]$ along $\mathbf r:=(r_1,r_2)$ agree:
$\sigma^{\mathbf k}_{\mathbf r}[\xi_1]=\sigma^{\mathbf k}_{\mathbf r}[\xi_2]$.
\item For all ideal triangles $\{x_1,x_2,x_3\}$ in $\widetilde{\Theta}$ such that $\mathbf x:=(x_1,x_2,x_3)$ is cyclically ordered, and for all triples of positive integers $\mathbf i:=(i_1,i_2,i_3)$ that sum to $n$, the $\mathbf i$-triangle invariants of $\mathbf x$ at $[\xi_1]$ and $[\xi_2]$ agree: $\tau^{\mathbf i}_{\mathbf x}[\xi_1]=\tau^{\mathbf i}_{\mathbf x}[\xi_2]$.
\end{itemize}
In other words, the map 
\[\begin{array}{crcl}
\Phi:&\mathrm{Fre}(V)&\to&\Rbbb^{|\widetilde{\Pmc}|\cdot|\widetilde{\Jmc}|\cdot(n-1)}\times\Rbbb^{|\widetilde{\Qmc}|\cdot(n-1)}\times\Rbbb^{|\widetilde{\Theta}|\cdot\frac{(n-1)(n-2)}{2}}\\
 &[\xi]&\mapsto&(\Sigma_1,\Sigma_2,\Sigma_3)
\end{array}\]
is injective, where 
\begin{eqnarray*}
\Sigma_1&:=&\left(\sigma_{\mathbf r,J}^{\mathbf k}[\xi]\right)_{\mathbf k\in\Amc;\,\{r_1,r_2\}\in\widetilde{\Pmc};\,J\in\widetilde{\Jmc}\text{ across }\{r_1,r_2\}},\\
\Sigma_2&:=&\left(\sigma_{\mathbf r}^{\mathbf k}[\xi]\right)_{\mathbf k\in\Amc;\,\{r_1,r_2\}\in\widetilde{\Qmc}},\\
\Sigma_3&:=&\left(\tau^{\mathbf i}_{\mathbf x}[\xi]\right)_{\mathbf i\in\Bmc;\,\{x_1,x_2,x_3\}\in\widetilde{\Theta},\,x_1<x_2<x_3<x_1}.
\end{eqnarray*}
Here, recall that $\widetilde{\Pmc}$ and $\widetilde{\Qmc}$ are respectively the set of non-isolated and isolated edges in $\widetilde{\Tmc}$. Furthermore, $\Amc$ denotes the set of pairs of positive integers that sum to $n$, and $\Bmc$ denotes the set of triples of positive integers that sum to $n$.
\end{prop}

\begin{remark}
Let us stress that in Proposition \ref{prop:invariants determine curve} we do not assume any equivariance properties for the Frenet curves. In fact, the ideal triangulation $\widetilde{\Tmc}$ does not even need to be the lift of an ideal tirangulation of $S$. We will later apply Proposition \ref{prop:invariants determine curve} to non-equivariant Frenet curves in our construction of flows on $\Hit_V(S)$. 

In the case of equivariant Frenet curves, Proposition \ref{prop:invariants determine curve} has been proved by Bonahon and Dreyer \cite[Theorem 2]{BD1}. In this case they also determine the image, see Section~\ref{sec:Bonahon-Dreyer}. The analogue of Proposition \ref{prop:invariants determine curve} holds for equivariant positive maps, but it does not hold for general non-equivariant positive maps. 
\end{remark}

\begin{proof}[Proof of Proposition \ref{prop:invariants determine curve}]
Let $\{x_1,x_2,x_3\}$ be an ideal triangle of $\Tmc$, such that $\mathbf x:=(x_1,x_2,x_3)$ is cyclically ordered. Let $\xi_1$ and $\xi_2$ be the representatives of $[\xi_1]$ and $[\xi]_2$ respectively, normalized such that $\xi_1(x_1)=\xi_2(x_1)$, $\xi_1(x_2)=\xi_2(x_2)$ and $\xi_1^{(1)}(x_3)=\xi_2^{(1)}(x_3)$. Since 
\[T^{\mathbf i}(\xi_1(\mathbf x))=T^{\mathbf i}(\xi_2(\mathbf x))\]
for all positive triples of integers $\mathbf i:=(i_1,i_2,i_3)$ that sum to $n$, Proposition \ref{prop:Fock-Goncharov parametrization} implies that $\xi_1(x_3)=\xi_2(x_3)$. 

By Lemma \ref{lem:technical2}, it is sufficient to prove that $\xi_1(y) = \xi_2(y)$ for any vertex $y$ of $\widetilde{\Tmc}$ that is not $x_1$, $x_2$ or $x_3$. For this we use the set $\Emc_{(x_1,y)}$, which gives a combinatorial description of the pair $(x_1,y)$ as described in Section \ref{sec:combinatorial description}, and its decomposition as 
\[\Emc_{(x_1,y)}=\bigcup_{s=1}^k\Emc_{(x_1,y),s}\cup\bigcup_{s=0}^k\Emc_{(x_1,y),s,s+1}=\bigcup_{s=1}^k\Emc_s\cup\bigcup_{s=0}^k\Emc_{s,s+1}.\]

By relabelling $x_1$, $x_2$ and $x_3$ if necessary, we may assume that $\{x_2,x_3\}$ is the minimum of $\Emc_{(x_1,y)}$. Then the ideal triangle $\{x_1,x_2,x_3\}$ is the backward end of $\Emc_{0,1}$ (see Definition \ref{def:end} and Definition \ref{def:end1} for the definition of ends.) Furthermore, recall that all the other ends of $\Emc_{s,s+1}$ and $\Emc_s$ are also ideal triangles, except for possibly the forward end of $\Emc_k$ (in which case $\Emc_{k,k+1}$ is empty, see Remark \ref{rem:special_ends}). Since we have already established that $\xi_1$ and $\xi_2$ agree on this backward end of $\Emc_{0,1}$, and $y$ is a vertex of the forward end of $\Emc_{k,k+1}$ (or of $\Emc_k$ if $\Emc_{k,k+1}$ is empty), it is sufficient to show that the following hold for $\Emc=\Emc_{s,s+1}$ or $\Emc=\Emc_s$:
\begin{enumerate}
\item[($\dagger$)] If $\xi_1$ and $\xi_2$ agree on the backward end of $\Emc$, then they must agree on its forward end as well. 
\end{enumerate}

Since $\Emc_{s,s+1}$ is finite, Proposition \ref{prop:Fock-Goncharov parametrization} immediately implies ($\dagger$) when $\Emc=\Emc_{s,s+1}$.

To prove ($\dagger$) when $\Emc=\Emc_s$, let $\{b_s, d_s\}=e_s$ be the unique non-isolated edge in $\Emc_s$, let $\{ a_s,b_s,b'_s\}$ the backward end, and let $\{c_s,d_s,d'_s\}$ the forward end (see Figure~\ref{fig:abcd}). Observe that 
\[\Emc_s\cup\big\{\{a_s,b_s'\}\big\}\cup\big\{\{c_s,d_s'\}\big\} = \Emc_{s,-}\cup\Emc_{s,+}, \] 
where 
$\Emc_{s,-}:=\{e\in\Emc_s:e\leq e_s\}\cup\big\{\{a_s,b_s'\}\big\}$ and $\Emc_{s,+}:=\{e\in\Emc_s:e\geq e_s\}\cup\big\{\{c_s,d_s'\}\big\}$. (Observe that in the case when the maximum of $\Emc_s$ is a non-isolated edge, then $\Emc_{s,+}=\{e_s\}$. This can only happen when $s=k$ and $\Emc_{s,s+1}$ is empty, see Remark \ref{rem:special_ends}.) We denote by $\Vmc_{s,\pm}$ the vertices of $\Emc_{s,\pm}$. Proposition \ref{prop:Fock-Goncharov parametrization} and the continuity of $\xi_1$ and $\xi_2$ imply that there is a projective transformation $g_{\pm}$ in $\PGL(V)$ such that $g_{\pm}\cdot \xi_1(q)=\xi_2(q)$ for any vertex $q$ in $\Vmc_{s,\pm}$. 

Since $\xi_1$ and $\xi_2$ agree on the backward end of $\Emc_s$, which is contained in $\Emc_{s,-}$, Remark \ref{rem:transitive} implies that $g_-=\id$. Thus, $\xi_1(b_s)=\xi_2(b_s)$ and $\xi_1(d_s)=\xi_2(d_s)$. This finishes the proof when the maximum of $\Emc_s$ is a non-isolated edge. In the case when the maximum of $\Emc_s$ is an isolated edge, this allows us to deduce that $g_+$ fixes both $\xi_1(b_s)$ and $\xi_1(d_s)$.
 Let $\mathbf e:=(b_s,d_s)$ and let $J$ be any bridge across $e_s$. By assumption, for all pairs of positive integers $\mathbf k:=(k_1,k_2)$ that sum to $n$, $\sigma^{\mathbf k}_{\mathbf e,J}[\xi_1]=\sigma^{\mathbf k}_{\mathbf e,J}[\xi_2]$. Thus, there is some vertex $p$ of $\Emc_s$ that is not $b_s$ or $d_s$, such that $\xi_1^{(1)}(p)=\xi_2^{(1)}(p)$. This implies that $g_+$ fixes $\xi_1(b_s)$, $\xi_1(d_s)$ and $\xi_1^{(1)}(p)$, so Remark \ref{rem:transitive} implies that $g_+=\id$. Therefore $\xi_1$ and $\xi_2$ agree on the forward end of $\Emc_s$. This finishes the proof. 
\end{proof}

\subsection{The symplectic closed edge invariants}\label{sec:new invariants}
In this section we replace the edge invariants associated to the non-isolated edges in $\widetilde{\Tmc}$ by new invariants, which we call the \emph{symplectic closed edge invariants}. The symplectic closed edge invariants behave more naturally under the shearing flows along these non-isolated edges (see Lemma~\ref{lem:new invariant 3}). In the companion paper \cite{SunZhang}, it is shown that the symplectic closed edge invariant can be used to give an easy description of the Goldman symplectic structure on the Hitchin component $\Hit_V(S)$. Note that the symplectic closed edge invariants resemble, but are not the same as the edge invariants in the Bonahon-Dreyer parametrization in \cite{BD2}.

Recall that to define the edge invariants, we used that the cross ratios of a configuration of four cyclically oriented points along a Frenet curve is always negative. The first step to define the symplectic closed edge invariant is to control the sign of the cross ratio of particular configurations of points, that do not necessarily lie along a Frenet curve. 

For this, let $\{r_1,r_2\}$ be a non-isolated edge of $\widetilde{\Tmc}$ and let $J=\{T_1,T_2\}$ be the bridge across $\{r_1,r_2\}$, such that $T_1$ (resp. $T_2)$ lies to the left (resp. right) of the oriented edge $\mathbf r:=(r_1,r_2)$.
As before, for $m=1$, $2$, let $p_m$ (resp. $q_m$) be the vertex of the edge $\{r_1,r_2\}$ that is (resp. is not) a vertex of $T_m$, and let $z_m,w_m$ be the other two vertices of $T_m$ such that either $(p_m,z_m,w_m,q_m)$ or $(q_m,w_m,z_m,p_m)$ is cyclically ordered (see Figure \ref{fig:pq}). 

\begin{prop}\label{prop:new invariant well-defined}
Let $\xi:\partial\Gamma\to\Fmc(V)$ be any Frenet curve (without any equivariance assumptions), and let $u_m$ be the unique unipotent projective transformation in $\PGL(V)$ that fixes the flag $\xi(p_m)$ and sends the flag $\xi(z_m)$ to $\xi(q_m)$. Then, for all pairs $\mathbf k:=(k_1,k_2)$ of positive integers that sum to $n$, the cross ratio satisfies 
\[C^{\mathbf k}(\xi(r_1),u_2\cdot\xi(w_2),\xi(r_2),u_1\cdot\xi(w_1))<0.\]
\end{prop}

\begin{proof}
By Theorem \ref{prop:positive}, the Frenet curve $\xi$ is a positive map. For $m=1$, $2$, choose a basis $B_m=\{f_{m,1},\dots,f_{m,n}\}$ such that $[f_{m,i}]=\xi^{(i)}(p_m)\cap\xi^{(n-i+1)}(z_m)$ for $i=1$, $\dots$, $n-1$. By Definition~\ref{def:positive} we have unipotent projective transformations $v_{1,m}$ and $v_{2,m}$ in $\PGL(V)$ that are totally positive with respect to $B_m$, such that up to transforming everything by a projective transformation, $\xi(w_m)=v_{1,m}\cdot\xi(z_m)$ and $\xi(q_m)=v_{1,m}v_{2,m}\cdot\xi(z_m)$. In particular, $v_{1,m}v_{2,m}$ fixes $\xi(p_m)$ and sends $\xi(z_m)$ to $\xi(q_m)$, so this implies that $u_m=v_{1,m}v_{2,m}$. Hence, $u_m\cdot\xi(w_m)=v_{1,m}v_{2,m}v_{1,m}\cdot\xi(z_m)$, which means that any cyclic permutation of 
\[\left(\xi(p_m), \xi(z_m), \xi(q_m), u_m\cdot\xi(w_m)\right)\,\,\,\text{ or }\,\,\,\left(\xi(q_m), \xi(z_m), \xi(p_m), u_m\cdot\xi(w_m)\right)\] 
is a positive quadruple of flags. Hence, by Proposition \ref{prop:Fock-Goncharov parametrization}, we see that 
\begin{eqnarray*}
&&C^{\mathbf k}(\xi(r_1),u_2\cdot\xi(w_2),\xi(r_2),\xi(z_2)),\\
&&C^{\mathbf k}(\xi(r_1),u_1\cdot\xi(w_1),\xi(r_2),\xi(z_1)),\text{ and }\\
&&C^{\mathbf k}(\xi(r_1),\xi(z_2),\xi(r_2),\xi(z_1))
\end{eqnarray*}
are negative. This implies that
\begin{eqnarray*}
&&C^{\mathbf k}(\xi(r_1),u_2\cdot\xi(w_2),\xi(r_2),u_1\cdot\xi(w_1))\\
&=&\frac{C^{\mathbf k}(\xi(r_1),\xi(z_2),\xi(r_2),\xi(z_1))\cdot C^{\mathbf k}(\xi(r_1),u_2\cdot\xi(w_2),\xi(r_2),\xi(z_2))}{C^{\mathbf k}(\xi(r_1),u_1\cdot\xi(w_1),\xi(r_2),\xi(z_1))}<0.
\end{eqnarray*}
\end{proof}

\begin{remark}
If we defined $u_m$ with $w_m$ in place of $z_m$, then $u_2^{-1}u_1$ is the Bonahon-Dreyer slithering map \cite[Section 5]{BD2} associated to the pair $(T_1,T_2)$. 
\end{remark}

With this, we make the following definition.

\begin{definition}\label{def:symplectic closed-edge invariant}
For any non-isolated edge $\{r_1,r_2\}$ in $\widetilde{\Tmc}$, any bridge $J$ in $\widetilde{\Jmc}$ across $\{r_1,r_2\}$, and any pair of positive integers $\mathbf k:=(k_1,k_2)$ that sum to $n$, the $\mathbf k$-\emph{symplectic closed edge invariant} along $\mathbf r:=(r_1,r_2)$ is the function $\alpha^{\mathbf k}_{\mathbf r,J}:\mathrm{Fre}(V)\to\Rbbb$ defined by
\[\alpha^{\mathbf k}_{\mathbf r,J}[\xi]:=\log\left(-C^{\mathbf k}(\xi(r_1),u_2\cdot\xi(w_2),\xi(r_2),u_1\cdot\xi(w_1))\right),\]
where $u_m$ is the unique unipotent projective transformation in $\PGL(V)$ that fixes $\xi(p_m)$ and sends $\xi(z_m)$ to $\xi(q_m)$. 
\end{definition}

The projective invariance of the cross ratio implies that $\alpha^{\mathbf k}_{\mathbf r,J}[\xi]$ does not depend on the choice of representative $\xi$ in $[\xi]$, so Proposition \ref{prop:new invariant well-defined} implies that $\alpha^{\mathbf k}_{\mathbf r,J}[\xi]$ is indeed well-defined. Note that while the edge invariant $\sigma^\mathbf k_{\mathbf r,J}$ depends only on $\xi(r_1)$, $\xi(r_2)$, $\xi^{(1)}(z_1)$, and $\xi^{(1)}(z_2)$, the symplectic closed edge invariant $\alpha^\mathbf k_{\mathbf r,J}$ depends on $\xi(r_1)$, $\xi(r_2)$, $\xi(z_1)$, $\xi(z_2)$, $\xi^{(1)}(w_1)$, and $\xi^{(1)}(w_2)$.

Setting $\mathbf k_1:=(k_1,k_2)$, $\mathbf k_2:=(k_2,k_1)$, $\mathbf r_1:=(r_1,r_2)$ and $\mathbf r_2:=(r_2,r_1)$, we have $\alpha^{\mathbf k_1}_{\mathbf r_1,J}[\xi]=\alpha^{\mathbf k_2}_{\mathbf r_2,J}[\xi]$. 

We now establish the key property of the symplectic closed edge invariant, which the usual edge invariant does not satisfy. 

\begin{lem}\label{lem:new invariant 1}
Let $\xi$ and $\xi'$ are two Frenet curves normalized such that $\xi(r_1)=\xi'(r_1)$, $\xi(r_2)=\xi'(r_2)$, and $\xi^{(1)}(w_2)=\xi'^{(1)}(w_2)$. Suppose further that $\xi(z_2)=\xi'(z_2)$. Then the symplectic closed edge invariants $\alpha^{\mathbf k}_{\mathbf r,J}[\xi]$ and $\alpha^{\mathbf k}_{\mathbf r,J}[\xi']$ agree for all pairs of positive integers $\mathbf k:=(k_1,k_2)$ that sum to $n$, if and only if there exists a unipotent projective transformation $v$ in $\PGL(V)$ such that $v\cdot\xi(p_1)=\xi(p_1)$, $v\cdot\xi(z_1)=\xi'(z_1)$ and $v\cdot\xi^{(1)}(w_1)=\xi'^{(1)}(w_1)$.
\end{lem}

Lemma \ref{lem:new invariant 1} is an immediate consequence of the following lemma.

\begin{lem}\label{lem:new invariant 0}
Let $H_1,H_2$ be a transverse pair of flags in $\Fmc(V)$. For $m=1$, $2$, $3$, 
\begin{itemize}
\item let $F_m$ be a flag in $\Fmc(V)$ such that $(H_1,F_m,H_2)$ is a generic triple of flags in $\Fmc(V)^{[3]}$, \item let $P_m\in\Pbbb(V)$ be a line in $V$ such that $P_m+H_1^{(k_1)}+F_m^{(k_2-1)}=V$ for all pairs $\mathbf k:=(k_1,k_2)$ of integers that sum to $n$, with $k_1=0$, $\dots$, $n-1$, and 
\item let $u_m$ be the unipotent projective transformation in $\PGL(V)$ that fixes $H_1$ and satisfies $u_m\cdot F_m=H_2$.
\end{itemize} 
Then the equality $C^{\mathbf k}\big(H_1,u_1\cdot P_1,H_2,u_2\cdot P_2\big)=C^{\mathbf k}\big(H_1,u_1\cdot P_1,H_2,u_3\cdot P_3\big)$ holds for all $\mathbf k$ if and only if there exists a unipotent projective transformation $u$ in $\PGL(V)$ that fixes $H_1$ and satisfies $u\cdot F_2=F_3$ and $u\cdot P_2=P_3$. 
\end{lem}

\begin{proof}
First, observe that for $m=1$, $2$, $3$ and any $\mathbf k$, we have
\[V=u_m\cdot V=u_m\cdot P_m+u_m\cdot H_1^{(k_1)}+u_m\cdot F_m^{(k_2-1)}=u_m\cdot P_m+H_1^{(k_1)}+H_2^{(k_2-1)}.\]
Thus, both $C^{\mathbf k}\big(H_1,u_1\cdot P_1,H_2,u_2\cdot P_2\big)$ and $C^{\mathbf k}\big(H_1,u_1\cdot P_1,H_2,u_3\cdot P_3\big)$ are well-defined for all $\mathbf k$.

Now, suppose that the unipotent projective transformation $u$ exists. Then $u_2=u_3u$ because both $u_3u$ and $u_2$ are unipotent projective transformations that fix $H_1$ and send $F_2$ to $H_2$. In particular, we have $u_2\cdot P_2=u_3u\cdot P_2=u_3\cdot P_3$. This implies that $C^{\mathbf k}\big(H_1,u_1\cdot P_1,H_2,u_2\cdot P_2\big)=C^{\mathbf k}\big(H_1,u_1\cdot P_1,H_2,u_3\cdot P_3\big)$ for all pairs of positive integers $\mathbf k$ that sum to $n$.

Conversely, suppose that $C^{\mathbf k}\big(H_1,u_1\cdot P_1,H_2,u_2\cdot P_2\big)=C^{\mathbf k}\big(H_1,u_1\cdot P_1,H_2,u_3\cdot P_3\big)$ for all pairs of positive integers $\mathbf k$ that sum to $n$. A straightforward computation verifies that $u_2\cdot P_2=u_3\cdot P_3$. Setting $u:=u_3^{-1}u_2$, it is clear that $u\cdot H_1=H_1$, $u\cdot F_2=F_3$ and $u\cdot P_2=P_3$. Since $u_2$ and $u_3$ are both unipotent and fix $H_1$, this implies that $u$ is also unipotent.
\end{proof}

Next, we determine how the $\mathbf k$-symplectic closed edge invariant along $\mathbf r$ changes under the $\mathbf k$-elementary shearing flow along $\mathbf r$.

\begin{lem}\label{lem:new invariant 3}
Let $\xi:\partial\Gamma\to\Fmc(V)$ be a Frenet curve, let $\mathbf k:=(k_1,k_2)$ be a pair of positive integers that sum to $n$, and let $\left(\psi^{\mathbf k}_{\mathbf r}\right)_t$ be the $\mathbf k$-elementary shearing flow with respect to $\mathbf r$. We set $\xi_t:=\left(\psi^{\mathbf k}_{\mathbf r}\right)_t(\xi)$. Then 
\[\alpha^{\mathbf k}_{\mathbf r,J}[\xi_t]=\alpha^{\mathbf k}_{\mathbf r,J}[\xi]+t.\]
\end{lem}

\begin{proof}
For $m=1$, $2$, let $u_m$ be the unipotent group element in $\PGL(V)$ that fixes $\xi(p_m)$ and sends $\xi(z_m)$ to $\xi(q_m)$. Similarly, let $u_m(t)$ be the unipotent group element that fixes $\xi_t(p_m)$ and sends $\xi_t(z_m)$ to $\xi_t(q_m)$. Let $d(t):=d^{\mathbf k}_{\xi(\mathbf r)}(t)$, where $d^{\mathbf k}_{\xi(\mathbf r)}(t)$ is the projective transformation used to define the $\mathbf k$-elementary shearing flow along $\mathbf r$ (see Section \ref{sec:elementaryshearing}).
 By definition, 
\[d(t)\cdot\xi(z_1)=\xi_t(z_1),\,\,\, d(t)\cdot\xi(w_1)=\xi_t(w_1),\]
and $d(t)$ fixes both $\xi(r_1)=\xi_t(r_1)$ and $\xi(r_2)=\xi_t(r_2)$. 
Therefore the product $u_1(t)d(t)u_1^{-1}$ fixes $\xi(p_1)$ and 
\begin{eqnarray*}
u_1(t)d(t)u_1^{-1}\cdot\xi(q_1)&=&u_1(t)d(t)\cdot\xi(z_1)\\
&=&u_1(t)\cdot\xi_t(z_1)\\
&=&\xi_t(q_1)\\
&=&\xi(q_1).
\end{eqnarray*}

This implies that in the basis $\{f_1,\dots,f_n\}$ of $V$ such that, for all $i=1$, $\dots$, $n-1$, the vector $f_i$ lies in $\xi^{(i)}(p_1)\cap\xi^{(n-i+1)}(q_1)$, both $u_1(t)d(t)u_1^{-1}$ and $d(t)$ are represented by diagonal matrices and both $u_1(t)$ and $u_1$ are represented by upper-triangular unipotent matrices. Hence, $u_1(t)d(t)u_1^{-1}=d(t)$, so we can conclude that
\begin{eqnarray*}
u_1(t)\cdot\xi_t(w_1)&=&u_1(t)d(t)\cdot\xi(w_1)\\
&=&u_1(t)d(t)u_1^{-1}\cdot(u_1\cdot \xi(w_1))\\
&=&d(t)\cdot(u_1\cdot \xi(w_1)).
\end{eqnarray*}
Similarly, we also have that $u_2(t)\cdot\xi_t(w_2)=d(t)^{-1}\cdot(u_2\cdot\xi( w_2))$. This allows us to conclude that
\begin{eqnarray*}
\alpha^{\mathbf k}_{\mathbf r,J}[\xi_t]&=&\log\Big(-C^{\mathbf k}\big(\xi_t(r_1),u_2(t)\cdot\xi_t(w_2),\xi_t(r_2),u_1(t)\cdot\xi_t(w_1)\big)\Big)\\
&=&\log\Big(-C^{\mathbf k}\big(\xi(r_1),d(t)^{-1}\cdot(u_2\cdot\xi(w_2)),\xi(r_2),d(t)\cdot(u_1\cdot\xi(w_1))\big)\Big)\\
&=&\log\Big(-e^t\cdot C^{\mathbf k}\big(\xi(r_1),u_2\cdot\xi(w_2),\xi(r_2),u_1\cdot\xi(w_1)\big)\Big)\\
&=&\alpha^{\mathbf k}_{\mathbf r,J}[\xi]+t.
\end{eqnarray*}
where the second last equality follows from the same computation as the one performed in the proof of Proposition \ref{prop:shearing projective invariants} (1).
\end{proof}

\subsection{Parametrizing the Hitchin component}\label{sec:Bonahon-Dreyer}
In this section we review the Bonahon-Dreyer parametrization of the Hitchin component $\Hit_V(S)\subset\mathrm{Fre}(V)$, viewed as the locus of $\PGL(V)$-orbits of Frenet curves that are $\rho$-equivariant for some representation $\rho:\Gamma\to\PGL(V)$. We then give a slight reparametrization using the symplectic closed edge invariants.

We first observe, that for a Frenet curve $\xi:\partial\Gamma\to\Fmc(V)$ that is $\rho$-equivariant, the edge invariants along non-isolated edges do not depend on the chosen bridge. 
More precisely, let $\{r_1, r_2\}$ be a non-isolated edge of $\widetilde{\Tmc}$, and let $J_1,J_2$ be two bridges in $\widetilde{\Jmc}$ across $\{r_1,r_2\}$. (Recall that $\Jmc$ is the chosen bridge system that is compatible to the ideal triagulation $\Tmc$.) Then for any pair $\mathbf k:=(k_1,k_2)$ of positive integers that sum to $n$ the $\mathbf k$-edge invariants along $\mathbf r:=(r_1,r_2)$ agree: $\sigma^{\mathbf k}_{\mathbf r,J_1}=\sigma^{\mathbf k}_{\mathbf r,J_2}$. So, we can safely drop the dependence on the bridge from the notation.

Note that the edge invariants and triangle invariants are $\Gamma$-invariant: for any $\gamma\in\Gamma$, we have $\sigma_{\mathbf r}^{\mathbf k}=\sigma_{\gamma\cdot \mathbf r}^{\mathbf k}$ and $\tau^{\mathbf i}_{\mathbf x}=\tau^{\mathbf i}_{\gamma\cdot \mathbf x}$. Thus, the map given in Proposition~\ref{prop:invariants determine curve},
\[\Phi: \mathrm{Fre}(V)\to \Rbbb^{|\widetilde{\Pmc}|\cdot|\widetilde{\Jmc}|\cdot(n-1)}\times\Rbbb^{|\widetilde{\Qmc}|\cdot(n-1)}\times\Rbbb^{|\widetilde{\Theta}|\cdot\frac{(n-1)(n-2)}{2}},\]
where $\widetilde{\Qmc}$ and $\widetilde{\Pmc}$ are respectively the set of isolated and non-isolated edges in $\widetilde{\Tmc}$, restricts to an injective map 
\begin{equation}\label{eqn:para}
\Phi|_{\Hit_V(S)}:\Hit_V(S)\to\Rbbb^{|\Tmc|\cdot(n-1)}\times\Rbbb^{|\Theta|\cdot\frac{(n-1)(n-2)}{2}}.
\end{equation}

Bonahon and Dreyer \cite[Section 4]{BD1} explicitly characterize the image of the map, by a family of equalities, called the \emph{closed leaf equalities}, and a family of inequalities, called \emph{closed leaf inequalities}, which hold on the image. Each such equality and inequality is associated to a non-isolated edge in $\Tmc$.

To specify the closed leaf equalities and inequalities, recall that Labourie \cite[Theorem 1.5]{Labourie2006} showed that for any Hitchin representation $\rho:\Gamma \to \PGL(V)$ the following holds:
for any $\gamma\in\Gamma\setminus\{\id\}$, $\rho(\gamma)$ has a (necessarily unique) lift to $\SL(V)$ that is diagonalizable over $\Rbbb$ with positive, pairwise distinct eigenvalues, which we denote by 
\[\lambda_1(\rho(\gamma))>\dots>\lambda_n(\rho(\gamma)).\] 

To obtain the \emph{closed leaf equalities}, Bonahon and Dreyer observed that the eigenvalue data of $\rho(\gamma_m)$ can be expressed in terms of the triangle and edge invariants. In order to write this down, 
we introduce a notation to label the triangles to the left and the right of the (oriented) geodesic $\mathbf r:=(r_1, r_2)$, see the following Figure. (We will use this notation also in Section~\ref{sec:main theorem}, where we recall it, and advise the reader to then have a look back at Figure \ref{fig:closededge}.) 

\begin{figure}[ht]
\centering
\includegraphics[scale=0.8]{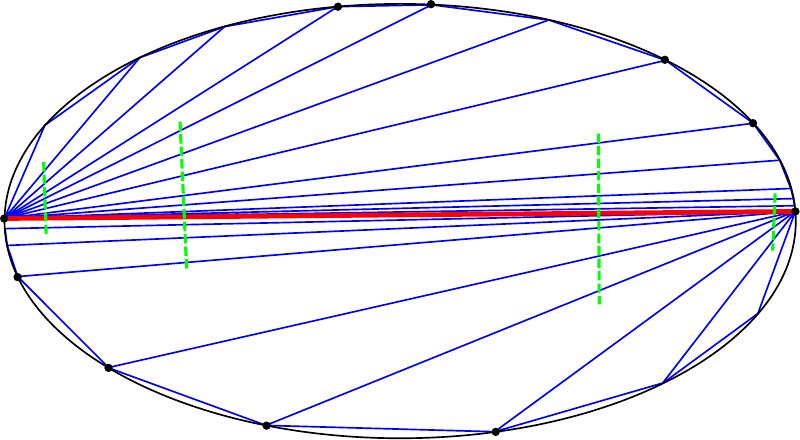}
\small
\put (-339, 85){$p_1=q_2$}
\put (-0, 86){$q_1=p_2$}
\put (-185, 154){$T_1$}
\put (-165, 147){$T_{1,2}$}
\put (-105, 124){$T_{1,H_1+1}=\gamma_1\cdot T_1$}
\put (-260, 49){$T_{2,1}=T_2$}
\put (-150, 14){$\gamma_2^{-1}\cdot T_2$}
\put (-205, 172){$z_1=z_{1,0}$}
\put (-150, 172){$w_1=z_{1,1}$}
\put (-51, 149){$\gamma_1\cdot z_1=z_{1,H_1}$}
\put (-17, 124){$\gamma_1\cdot w_1=z_{1,H_1+1}$}
\put (-220, -4){$\gamma_2^{-1}\cdot w_2$}
\put (-125, -6){$\gamma_2^{-1}\cdot z_2$}
\put (-340, 62){$w_2=z_{2,1}$}
\put (-300, 22){$z_2=z_{2,0}$}
\caption{The thick red line is $\{r_1,r_2\}$ and the thin blue lines are isolated edges in $\widetilde{\Qmc}$, and the green dotted lines are bridges in $\widetilde{\Jmc}$, and the second bridge from the left is $\{T_1,T_2\}$.}\label{fig:closededge}
\end{figure}

\begin{notation}\label{not:closed edge}
Let $J=\{T_1,T_2\}$ be a bridge across the non-isolated edge $\{r_1,r_2\}$, and let $T_1$ and $T_2$ lie to the left and right of $(r_1,r_2)$ respectively. For $m=1$, $2$, let $p_m$ (resp. $q_m$) be the vertex of the edge $\{r_1,r_2\}$ that is (resp. is not) a vertex of $T_m$. Then let $\gamma_m$ be the primitive group element in $\Gamma$ with $p_m$, $q_m$ as its repelling and attracting fixed points respectively. Also, for all integers $h$, let $T_{m,h}$ denote the ideal triangle in $\widetilde{\Theta}$ defined by the following properties (see Figure \ref{fig:closededge} when $p_1=q_2$):
\begin{itemize}
\item $T_{m,1}=T_m$,
\item $T_{m,h}$ has $p_m$ as a vertex for all integers $h$,
\item $T_{m,h}$ and $T_{m,h+1}$ share a common edge, denoted $e_{m,h}$, for all integers $h$,
\item there is a positive integer $H_m$ such that $T_{m,kH_m+h}=\gamma_m^k\cdot T_{m,h}$ for all integers $h,k$.
\end{itemize}
Let $z_{m,h}$ denote the vertex of $e_{m,h}$ that is not $p_m$, set $\mathbf p_{m,h}:=(p_m,z_{m,h})$, and set
\[\mathbf t_{m,h}:=\left\{\begin{array}{ll}
(p_m,z_{m,h-1},z_{m,h})&\text{if }p_m=r_m;\\
(p_m,z_{m,h},z_{m,h-1})&\text{if }p_m=r_{3-m}.
\end{array}
\right.\] 
\end{notation}

\begin{remark}
The triple $\mathbf t_{m,h}$ is cyclically ordered.
\end{remark}

\begin{definition}\label{def:closed edge subset}
Let $J=\{T_1,T_2\}$ be a bridge in $\widetilde{\Jmc}$. For $m=1$, $2$ and any integer $h$, let $T_{m,h}$ be the ideal triangle and let $H_m$ be the integer as defined in Notation \ref{not:closed edge}. The subset $\widetilde{\Theta}(J,T_m):=\{T_{m,1},T_{m,2},\dots,T_{m,H_m}\}$ of $\widetilde{\Theta}$ is a \emph{closed edge subset}.
\end{definition}

For any $i=1$, $\dots$, $n-1$ and $m=1$, $2$, set
\begin{equation}\label{eqn:length}
\ell^i_\rho(\gamma_m):=\log\frac{\lambda_i(\rho(\gamma_m))}{\lambda_{i+1}(\rho(\gamma_m))}.
\end{equation}
Bonahon and Dreyer computed a formula for $\ell^i_\rho(\gamma_m)$ in terms of the triangle invariants of the ideal triangles in the closed edge subset $\widetilde{\Theta}(J,T_m)$, and the edge invariants along the edges of the ideal triangles in $\widetilde{\Theta}(J,T_m)$. Explicitly,
\begin{equation}\label{eqn:length formula}
\ell^{n-i}_\rho(\gamma_m)=\left\{\begin{array}{ll}
\displaystyle\sum_{h=1}^{H_m}\left(\sigma_{\mathbf p_{m,h}}^{(i,n-i)}[\xi]+\sum_{j+k=n-i}\tau_{\mathbf t_{m,h}}^{(i,j,k)}[\xi]\right)&\text{if }p_m=r_{m};\\
\displaystyle-\sum_{h=1}^{H_m}\left(\sigma_{\mathbf p_{m,h}}^{(i,n-i)}[\xi]+\sum_{j+k=n-i}\tau_{\mathbf{t}_{m,h}}^{(i,j,k)}[\xi]\right)&\text{if }p_m=r_{3-m}.
\end{array}.\right.
\end{equation}

As an immediate consequence, the edge invariants and triangle invariants have to satisfy the following inequalities, called the \emph{closed leaf inequalities} associated to the non-isolated edge $\{r_1,r_2\}$ in $\widetilde{\Tmc}$:
\begin{itemize}
\item If $p_m=r_m$, then for all $i=1$, $\dots$, $n-1$,
\[\sum_{h=1}^{H_m}\left(\sigma_{\mathbf p_{m,h}}^{(i,n-i)}[\xi]+\sum_{j+k=n-i}\tau_{\mathbf t_{m,h}}^{(i,j,k)}[\xi]\right)>0.\]
\item If $p_m=r_{3-m}$, then for all $i=1$, $\dots$, $n-1$,
\[\sum_{h=1}^{H_m}\left(\sigma_{\mathbf p_{m,h}}^{(i,n-i)}[\xi]+\sum_{j+k=n-i}\tau_{\mathbf{t}_{m,h}}^{(i,j,k)}[\xi]\right)<0.\]
\end{itemize}

There are $n-1$ such inequalities for each non-isolated edge in $\Tmc$.

Also, since $\gamma_1=\gamma_2$ if $p_1=p_2$ and $\gamma_1=\gamma_2^{-1}$ if $p_1\neq p_2$, the edges invariants and triangle invariants have to satisfy the following equalities, called the \emph{closed leaf equalities} associated to the non-isolated edge $\{r_1,r_2\}$ in $\widetilde{\Pmc}$:

\begin{enumerate}
\item If $p_1\neq p_2$, then for all $i=1$, $\dots$, $n-1$,
\[\sum_{h=1}^{H_1}\left(\sigma_{\mathbf p_{1,h}}^{(i,n-i)}[\xi]+\sum_{j+k=n-i}\tau_{\mathbf{t}_{1,h}}^{(i,j,k)}[\xi]\right)=\sum_{h=1}^{H_2}\left(\sigma_{\mathbf p_{2,h}}^{(n-i,i)}[\xi]+\sum_{j+k=i}\tau_{\mathbf{t}_{2,h}}^{(n-i,j,k)}[\xi]\right).\]
\item If $p_1=p_2$, then for all $i=1$, $\dots$, $n-1$,
\[\sum_{h=1}^{H_1}\left(\sigma_{\mathbf p_{1,h}}^{(i,n-i)}[\xi]+\sum_{j+k=n-i}\tau_{\mathbf{t}_{1,h}}^{(i,j,k)}[\xi]\right)=-\sum_{h=1}^{H_2}\left(\sigma_{\mathbf p_{2,h}}^{(i,n-i)}[\xi]+\sum_{j+k=n-i}\tau_{\mathbf t_{2,h}}^{(i,j,k)}[\xi]\right).\]
\end{enumerate}

There are $n-1$ such identities for each non-isolated edge in $\Tmc$. These are sums of invariants associated to $\widetilde{\Theta}(J,T_1)$ versus those associated to $\widetilde{\Theta}(J,T_2)$. Observe also that the closed leaf equalities and inequalities are associated to the non-isolated edge $\{r_1,r_2\}$. They do not depend on the choice of bridge across $\{r_1,r_2\}$ because of the $\rho$-equivariance of $\xi$. 

\begin{remark}
In Bonahon and Dreyer's \cite{BD1} description of the closed leaf equalities and inequalities, they chose orientations on the edges of $\Tmc$. However, that choice is only for notational convenience, and is not necessary to specify the closed leaf equalities and inequalities.
\end{remark}

\begin{notation}\label{not:polytope notation}
Let $W_{\Tmc}$ be the vector subspace of $\Rbbb^{|\Tmc|\cdot(n-1)}\times\Rbbb^{|\Theta|\cdot\frac{(n-1)(n-2)}{2}}$ cut out by the closed leaf equalities, and let $C_{\Tmc}$ be the convex polytope in $W_{\Tmc}$ cut out by the closed leaf inequalities.
\end{notation}

The following theorem of Bonahon and Dreyer states that the closed leaf equalities and inequalities are the only relations between the edge and triangle invariants. 

\begin{thm}\cite[Theorem 17]{BD1}\label{thm:Bonahon-Dreyer}
The edge and triangle invariants give a real analytic diffeomorphism from $\Hit_V(S)$ to $C_{\Tmc}$. 
\end{thm}

In particular, this classifies the image of the map $\Phi|_{\Hit_V(S)}$ defined by (\ref{eqn:para}), and thus gives a real analytic parametrization of $\Hit_V(S)$ by $C_\Tmc$. 

\begin{remark}
Note that $C_{\Tmc}$ only depends on the ideal triangulation $\Tmc$, and not on a choice of associated bridge system $\Jmc$. However the explicit identification $\Phi|_{\Hit_V(S)}$ of $\Hit_V(S)$ with $C_\Tmc$ depends on the choice of $\Jmc$.
\end{remark}

The goal of the rest of this section is to prove that an analogous theorem holds when we replace the edge invariants along non-isolated edges by the symplectic closed edge invariants described in Section~\ref{sec:new invariants}.

\begin{notation}\label{not:index_set}
For each edge in $\Tmc$, we choose one representative $\{r_1,r_2\}$ in $\widetilde{\Tmc}$ and an orientation $(r_1,r_2)$ on $\{r_1,r_2\}$. Then let $\widehat{\Tmc}$ denote the collection of all such choices, one for each edge in $\Tmc$. Similarly, define $\widehat{\Pmc}$ and $\widehat{\Qmc}$ using only the non-isolated edges and isolated edges respectively. Also, for each ideal triangle in $\Theta$, we choose one representative $\{x_1,x_2,x_3\}$ in $\widetilde{\Theta}$, and an order $\mathbf x:=(x_1,x_2,x_3)$ on $\{x_1,x_2,x_3\}$ that is cyclically ordered. Then let $\widehat{\Theta}$ denote the collection of all such choices, one for each ideal triangle in $\Theta$. 
\end{notation}

\begin{thm}\label{thm:reparametrization}
The map 
\[\begin{array}{crcl}
\Omega=\Omega_{\Tmc,\Jmc}:&\Hit_V(S)&\to&\Rbbb^{|\Qmc|\cdot(n-1)}\times\Rbbb^{|\Pmc|\cdot(n-1)}\times\Rbbb^{|\Theta|\cdot\frac{(n-1)(n-2)}{2}}\\
 &[\xi]&\mapsto&(\Sigma_1,\Sigma_2,\Sigma_3), 
\end{array}\]
with 
\begin{eqnarray*}
\Sigma_1&:=&\left(\sigma_{\mathbf r}^{\mathbf k}[\xi]\right)_{\mathbf k\in\Amc;\,\mathbf r\in\widehat{\Qmc}},\\
\Sigma_2&:=&\left(\alpha_{\mathbf r}^{\mathbf k}[\xi]\right)_{\mathbf k\in\Amc;\,\mathbf r\in\widehat{\Pmc}},\\
\Sigma_3&:=&\left(\tau^{\mathbf i}_{\mathbf x}[\xi]\right)_{\mathbf i\in\Bmc;\,\mathbf x\in\widehat{\Theta}}.
\end{eqnarray*}
is a real-analytic diffeomorphism onto the convex polytope $C_\Tmc$ defined in Notation \ref{not:polytope notation}. Here, $\Amc$ is the set of pairs of positive integers that sum to $n$, and $\Bmc$ is the set of triples of positive integers that sum to $n$.
\end{thm}

\begin{remark}
From the symmetry properties of the triangle invariants, the edge invariants along isolated edges, and the symplectic closed edge invariant, we see that $\Omega$ does not depend on the choice of $\widehat{\Qmc}$, $\widehat{\Pmc}$ and $\widehat{\Theta}$.
\end{remark}

Using Lemma \ref{lem:new invariant 3}, we now prove the equivalence between the Bonahon-Dreyer edge invariants along non-isolated edges and the symplectic closed edge invariants. 

\begin{lem}\label{lem:new invariant 2}
Given any projective classes of Frenet curves $[\xi_0]$ in $\Hit_V(S)\subset \mathrm{Fre}(V)$, let $\Hit_V(S)^{[\xi_0]}$ denote the set of $[\xi]$ in $\Hit_V(S)$ such that all the triangle invariants and edge invariants along ISOLATED edges of $\widetilde{\Tmc}$ agree for $[\xi_0]$ and $[\xi]$. Then the map $A:\Hit_V(S)^{[\xi_0]}\to\Rbbb^{|\Pmc|\cdot(n-1)}$ given by
\[A:[\xi]\mapsto\big(\alpha^{\mathbf k}_{\mathbf r}[\xi]\big)_{\mathbf r\in\widehat{\Pmc}; \mathbf k\in\Amc}, \] 
where $\Amc$ is the set of pairs of positive integers that sum to $n$, is a real-analytic diffeomorphism. 
\end{lem}

Theorem \ref{thm:reparametrization} then follows immediately from Lemma \ref{lem:new invariant 2} and Theorem \ref{thm:Bonahon-Dreyer}.

\begin{proof}[Proof of Lemma \ref{lem:new invariant 2}]
It follows immediately from Theorem \ref{thm:Bonahon-Dreyer} that if $\xi_t$ is a one-parameter family of Frenet curves corresponding to a real-analytic family of Hitchin representations $\rho_t:\Gamma\to\PGL(V)$, then for any point $p$ in $\partial\Gamma$, $\xi_t(p)$ is a real-analytic path in $\Fmc(V)$. This implies the real-analyticity of $A$. 

First, we show that $A$ is surjective. For any pair $\mathbf r:=(r_1,r_2)$ in $\widehat{\Pmc}$ and any pair of positive integers $\mathbf k:=(k_1,k_2)$ that sum to $n$, recall that $\sigma_\mathbf r^\mathbf k$ denotes the Bonahon-Dreyer edge invariants along the non-isolated edge $\{r_1,r_2\}$. Since all the triangle invariants and edge invariants along isolated edges are constant on $\Hit_V(S)^{[\xi_0]}$, it follows from Lemma \ref{lem:new invariant 3} and (1) of Proposition \ref{prop:shearing projective invariants} that $\alpha_\mathbf r^\mathbf k-\sigma_\mathbf r^\mathbf k$ is also a constant on $\Hit_V(S)^{[\xi_0]}$. Since Theorem \ref{thm:Bonahon-Dreyer} implies that the map $\Hit_V(S)^{[\xi_0]}\to\Rbbb^{|\Pmc|\cdot(n-1)}$ given by $[\xi]\mapsto\big(\sigma^{\mathbf k}_{\mathbf r}[\xi]\big)_{\mathbf r\in\widehat{\Pmc}; \mathbf k\in\Amc}$ is surjective, the same is true for $A$.

Next, we show that $A$ is injective. Let $[\xi]\neq[\xi']$ be a pair of projective classes of Frenet curves in $\Hit_V(S)^{[\xi_0]}$. By Theorem \ref{thm:Bonahon-Dreyer}, there is a non-isolated edge $\{r_1,r_2\}$ in $\widetilde{\Tmc}$ and a pair of positive integers $\mathbf k:=(k_1,k_2)$ such that $\sigma_{\mathbf r}^{\mathbf k}[\xi]\neq\sigma_{\mathbf r}^{\mathbf k}[\xi']$, where $\mathbf r:=(r_1,r_2)$. For $m=1$, $2$, let $p_m$ (resp. $q_m$) be the vertex of the edge $\{r_1,r_2\}$ that is (resp. is not) a vertex of $T_m$, and let $z_m,w_m$ be the other two vertices of $T_m$ such that either $(p_m,z_m,w_m,q_m)$ or $(q_m,w_m,z_m,p_m)$ is cyclically ordered (see Figure~\ref{fig:pq}). Since all the triangle invariants and edge invariants along isolated edges for $[\xi]$ and $[\xi']$ agree, Proposition \ref{prop:Fock-Goncharov parametrization} and the continuity of Frenet curves imply that there are representatives $\xi$ and $\xi'$ of $[\xi]$ and $[\xi']$ respectively, and a projective transformation $g$ in $\PGL(V)$ such that 
\begin{itemize}
\item $\xi(r_1)=\xi'(r_1)$, $\xi(r_2)=\xi'(r_2)$, $\xi(z_2)=\xi'(z_2)$, $\xi(w_2)=\xi'(w_2)$, 
\item $g\cdot \xi(r_1)=\xi(r_1)$, $g\cdot \xi(r_2)=\xi(r_2)$, $g\cdot \xi(w_1)=\xi'(w_1)$ and $g\cdot\xi(z_1)=\xi'(z_1)$. 
\end{itemize}

We argue by contradiction and assume that $\alpha^{\mathbf k}_{\mathbf r}[\xi]=\alpha^{\mathbf k}_{\mathbf r}[\xi']$ for all pairs of positive integers $\mathbf k:=(k_1,k_2)$ that sum to $n$. By Lemma \ref{lem:new invariant 1}, we see that there is a unipotent projective transformation $u$ in $\PGL(V)$ such that $u\cdot \xi(p_1)=\xi(p_1)$, $u\cdot \xi(z_1)=\xi'(z_1)$ and $u\cdot\xi^{(1)}(w_1)=\xi'^{(1)}(w_1)$, which implies that $u=g$. In particular, $u$ fixes $\xi(r_1)$ and $\xi(r_2)$, so $q=u=\id$. However, this implies that $\sigma_{\mathbf r}^{\mathbf k}[\xi]=\sigma_{\mathbf r}^{\mathbf k}[\xi']$ for all $\mathbf k$, which is a contradiction and finishes the proof. 
\end{proof}

\section{Deforming Hitchin representations}\label{sec:Hitchin}
\label{sec:deform_Hitchin}
In this section we use the elementary eruption and shearing flows defined in Section \ref{sec:elementary} to construct flows on $\Hit_V(S)\subset\mathrm{Fre}(V)$. The idea is straight forward: we choose an ideal triangulation $\Tmc$ on $S$, and perform elementary eruption flows and shearing flows on the ideal triangles and edges of $\widetilde{\Tmc}$ in a ``$\Gamma$-invariant'' way to obtain a flow in $\Hit_V(S)$. However, the $\Gamma$-orbits of edges in $\widetilde{\Tmc}$ and triangles in $\widetilde{\Theta}$ are infinite, so defining these flows in a $\Gamma$-invariant way involves taking the product of infinitely many elementary eruption and shearing flows. In general, such products do not converge. The main goal of this section is to describe how to resolve these convergence issues.

\subsection{The tangent space to $\Hit_V(S)$}\label{sec:tangent}
Let us fix an ideal triangulation $\Tmc$ on $S$ and a compatible bridge system $\Jmc$. Recall that $\Theta$ denotes the set of triangles of $\Tmc$. We denote by 
\[W=W_{\Tmc}\subset \Rbbb^{|\Tmc|\cdot(n-1)}\times\Rbbb^{|\Theta|\cdot\frac{(n-1)(n-2)}{2}}\]
the linear subspace cut out by the closed leaf equalities (see Notation \ref{not:polytope notation}).
 The image of the diffeomorphism \[\Omega: \Hit_V(S)\to \Rbbb^{|\Tmc|\cdot(n-1)}\times\Rbbb^{|\Theta|\cdot\frac{(n-1)(n-2)}{2}}\] given by Theorem \ref{thm:reparametrization} is an open subset of $W$. Thus, for every $[\xi]$ in $\Hit_V(S)$ we identify $W$ with $T_{[\xi]}\Hit_V(S)$.
 
Our goal now is to assign to every tangent vector $\mu$ in $W=T_{[\xi]}\Hit_V(S)$ a flow $\phi^\mu_t$ on $\Hit_V(S)$ that is defined by performing the elementary shearing and eruption flows on $\mathrm{Fre}(V)$ in a ``$\Gamma$-invariant" way. More precisely, following Notation \ref{not:index_set}, we choose an oriented representative for each edge in $\Tmc=\widetilde{\Tmc}/\Gamma$, and denote the set of such choices by $\widehat{\Tmc}$. Similarly, we choose a cyclically ordered representative for each ideal triangle in $\Theta=\widetilde{\Theta}/\Gamma$, and denote the set of all such choices by $\widehat{\Theta}$. With this notation, a vector $\mu$ in $W$ is a tuple of real numbers 
\[\mu=\left(\left(\mu^{\mathbf k}_{\mathbf r}\right)_{\mathbf k\in\Amc;\mathbf r\in \widehat{\Tmc}},\left(\mu^{\mathbf i}_{\mathbf x}\right)_{\mathbf i\in\Bmc;\mathbf x\in\widehat{\Theta}}\right), \]
where $\Amc$ is the set of pairs of positive integers that sum to $n$, $\Bmc$ is the set of triples of positive integers that sum to $n$.
The map $\phi^\mu_t$ is an infinite product of the $\mathbf{k}$-elementary shearing flows along all $\Gamma$-translates of $\mathbf{r}$ in $\widehat{\Tmc}$, rescaled by $\mu^{\mathbf k}_{\mathbf r}$, and the $\mathbf{i}$-elementary eruption flows with respect to all $\Gamma$-translates of $\mathbf{x}$ in $\widehat{\Theta}$, rescaled by $\mu^{\mathbf i}_{\mathbf x}$. 
The key point is to show that this infinite product gives rise to a well-defined flow on $\Hit_V(S)$. 
In order to do this we decompose the product and write it as limits of special finite products of elementary flows. We call these finite products semi-elementary flows. Their definition relies on an exhaustion of the ideal triangulation $\widetilde{\Tmc}$ and the set of triangles $\widetilde{\Theta}$ by finite subsets. 

\subsection{Semi-elementary flows}\label{sec:closed_edge_subset}

We introduce the class of \emph{semi-elementary} flows on $\mathrm{Fre}(V)$, which are products of elementary flows on a larger and larger collection of ideal triangles and edges. 

We first specify this larger and larger collection. For this, recall that for any bridge $J=\{T_1,T_2\}$ in the bridge system $\widetilde{\Jmc}$ and any $m=1$, $2$, we defined the closed edge subset $\widetilde{\Theta}(J,T_m)\subset\widetilde{\Theta}$ (see Definition \ref{def:closed edge subset}). Denote the collection of all closed edge subsets of $\widetilde{\Theta}$ by $\Dmc$. To specify our larger and larger collection, we define a graph whose vertex set is $\Dmc$, and construct an exhaustion of this graph.

\begin{definition}\label{def:adjacent}\
\begin{itemize}
\item A pair of distinct ideal triangles $T,T'$ in $\widetilde{\Theta}$ are \emph{adjacent} if $T$ and $T'$ share a common edge in $\widetilde{\Tmc}$, or if $\{T,T'\}$ is a bridge in $\widetilde{\Jmc}$.
\item A pair of distinct closed edge subsets $\widetilde{\Theta}(J_1,T_1),\widetilde{\Theta}(J_2,T_2)$ in $\Dmc$ are \emph{adjacent} if there exist triangles $T$ in $\widetilde{\Theta}(J_1,T_1)$ and $T'$ in $\widetilde{\Theta}(J_2,T_2)$ such that $T$ and $T'$ are adjacent (see Figure \ref{fig:adjacent}).
\item A pair of distinct closed edge subsets $\widetilde{\Theta}(J_1,T_1)$ and $\widetilde{\Theta}(J_2,T_2)$ in $\Dmc$ are \emph{opposite} if $J_1=J_2=\{T_1,T_2\}$ (see Figure \ref{fig:opposite}).
\end{itemize}
\end{definition} 

\begin{figure}[ht]
\centering
\includegraphics[scale=0.3]{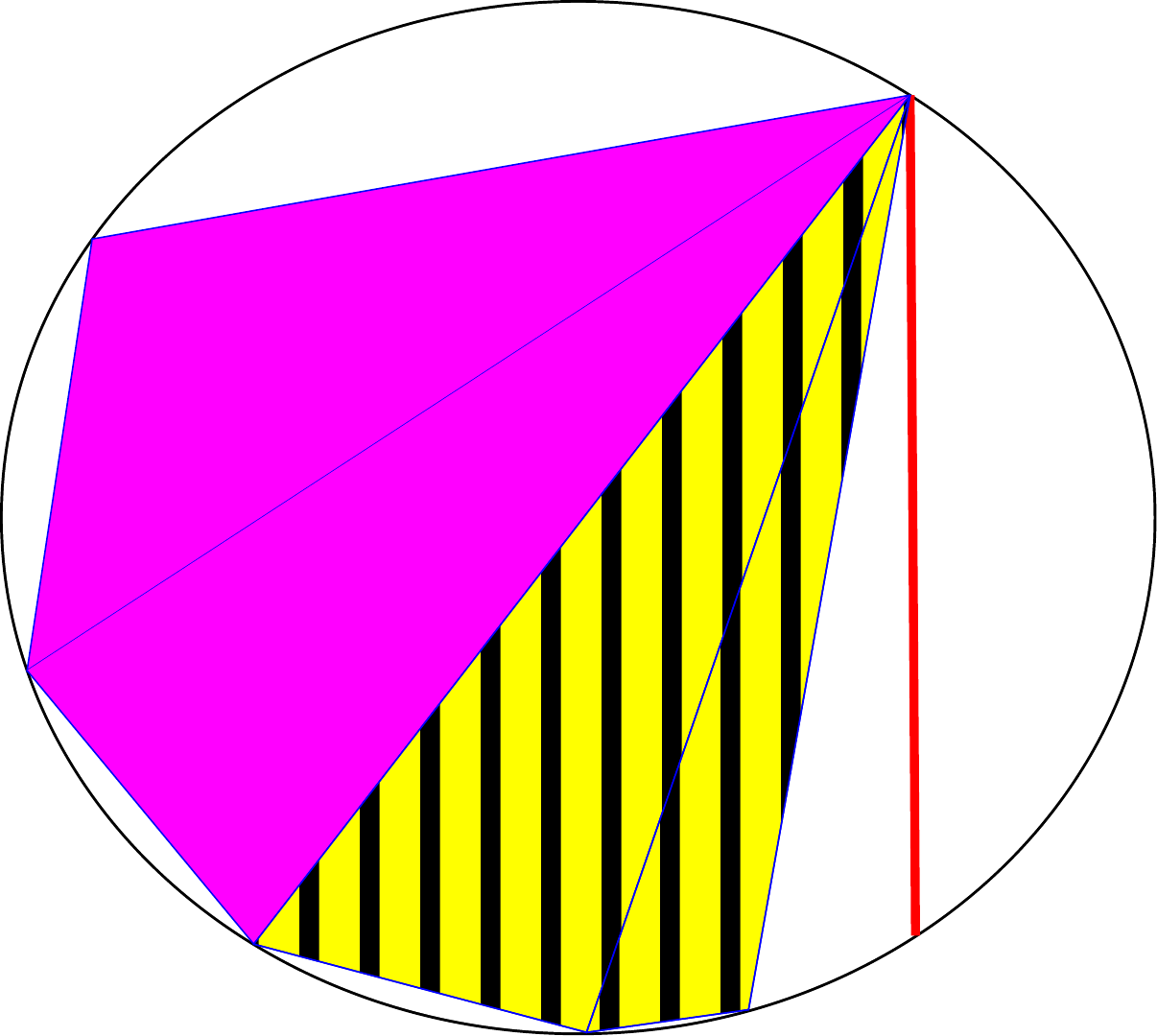}
\small
\caption{The two closed edge subsets given by the yellow striped region and the purple region are adjacent}\label{fig:adjacent}.
\end{figure}

\begin{figure}[ht]
\centering
\includegraphics[scale=0.3]{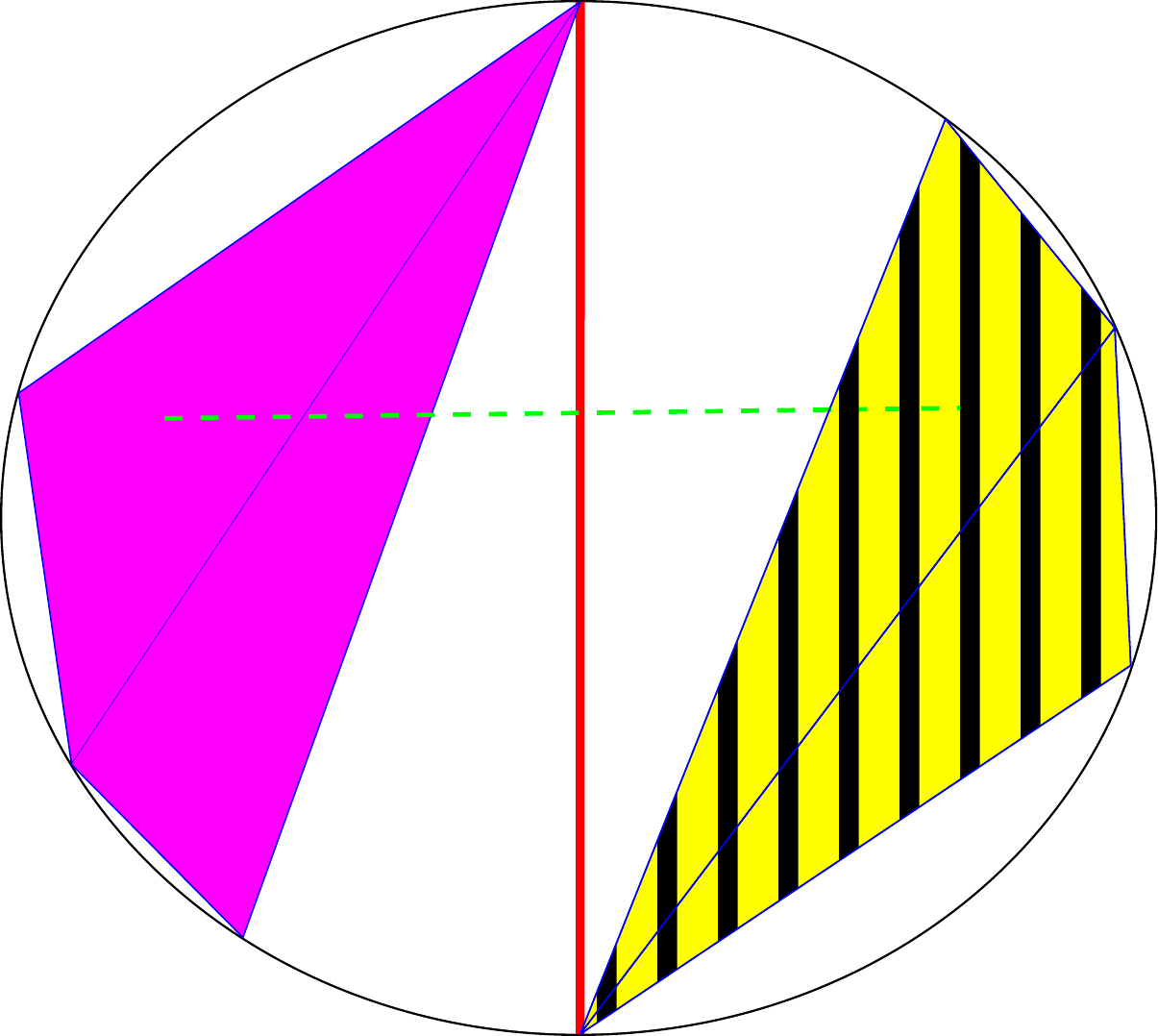}
\small
\caption{The two closed edge subsets given by the yellow striped region and the purple region are opposite.}
\label{fig:opposite}
\end{figure}

Being opposite defines an equivalence relation on $\Dmc$ where each equivalence class contains exactly two points in $\Dmc$. Furthermore, if two points in $\Dmc$ are opposite, then they are adjacent. Adjacency allows us to define the \emph{closed edge graph}. 
\begin{definition}
The \emph{closed edge graph} is the graph whose vertex set is $\Dmc$, and two vertices are joined by an edge if they are adjacent in the sense of Definition \ref{def:adjacent}. 
\end{definition} 

In order to define an exhaustion of the closed edge graph, we choose a closed edge subset $p_0$ in $\Dmc$ as a base point, and let $K_1:=\{p_0\}$. Then for all positive integers $j$, let 
\begin{eqnarray*}
K_j'&:=&K_j\cup\{p\in\Dmc:p\text{ is adjacent to some }p'\in K_j\},\\
K_{j+1}&:=&K_j'\cup\{p\in\Dmc:p\text{ is opposite to some }p'\in K_j'\}.
\end{eqnarray*}
This iteratively defines a nested sequence $(K_1,K_2,\dots)$ of subsets of $\Dmc$ such that 
\[\bigcup_{j=1}^\infty K_j=\Dmc.\] 

Recall that $\widetilde{\Qmc}$ denotes the set of isolated edges in $\widetilde{\Tmc}$. From the exhaustion $(K_1,K_2,\dots)$ of $\Dmc$, we get an exhaustion of $\widetilde{\Qmc}\cup\widetilde{\Theta}$ as follows. For each positive integer $j$, let $N_j$ denote the set of ideal triangles that lie in some closed edge subset in $K_j$, i.e.
\[N_j:=\{T\in\widetilde{\Theta}:T\in p\text{ for some }p\in K_j\}.\] 
Using this, define $M_j$ to be the union of the ideal triangles in $N_j$ and the edges of the ideal triangles in $N_j$, i.e.
\[M_j:=\{o\in\widetilde{\Qmc}\cup\widetilde{\Theta}:o\in N_j\text{ or }o\text{ is an edge of an ideal triangle in }N_j\}.\]
Note that each $ M_j$ is a finite set, and $(M_1,M_2,\dots)$ is a nested sequence whose union is $\widetilde{\Qmc}\cup\widetilde{\Theta}$. We use these sets $M_j$ to define the semi-elementary flows on $\mathrm{Fre}(V)$. 

Given a vector $\mu$ in $W$ and any $o$ in $\widetilde{\Tmc}\cup\widetilde{\Theta}$, we define the flow $\left(\phi^\mu_o\right)_t$ on $\mathrm{Fre}(V)$ as follows. Recall that $\Amc$ denotes the set of pairs of positive integers that sum to $n$, and $\Bmc$ denotes the set of triples of positive integers that sum to $n$.
\begin{itemize}
\item If $o=\{r_1,r_2\}$ is an edge in $\widetilde{\Tmc}$, let $\mathbf r_m:=(r_m,r_{m+1})$ for $m=1$, $2$, where the arithmetic in the subscripts are done modulo $2$. Define
\[\left(\phi^\mu_o\right)_t:=\prod_{\mathbf k\in\Amc}\left(\psi^{\mathbf k}_{\mathbf r_1}\right)_{\mu^{\mathbf k}_{\mathbf r_1}\cdot t}=\prod_{\mathbf k\in\Amc}\left(\psi^{\mathbf k}_{\mathbf r_2}\right)_{\mu^{\mathbf k}_{\mathbf r_2}\cdot t}, \] 
where $\left(\psi^{\mathbf k}_{\mathbf r}\right)_t$ is the $\mathbf k$-elementary shearing flow with respect to $\mathbf r$.

\item If $o=\{x_1,x_2,x_3\}$ is an ideal triangle in $\widetilde{\Theta}$ and $(x_1,x_2,x_3)$ is cyclically ordered, we set $\mathbf x_m:=(x_m,x_{m+1},x_{m-1})$ for $m=1$, $2$, $3$, where the arithmetic in the subscripts are done modulo $3$. Define
\[\left(\phi^\mu_o\right)_t:=\prod_{\mathbf i\in\Bmc}\left(\epsilon^{\mathbf i}_{\mathbf x_1}\right)_{\mu^{\mathbf i}_{\mathbf x_1}\cdot t}=\prod_{\mathbf i\in\Bmc}\left(\epsilon^{\mathbf i}_{\mathbf x_2}\right)_{\mu^{\mathbf i}_{\mathbf x_2}\cdot t}=\prod_{\mathbf i\in\Bmc}\left(\epsilon^{\mathbf i}_{\mathbf x_3}\right)_{\mu^{\mathbf i}_{\mathbf x_3}\cdot t}, \]
where $\left(\epsilon^{\mathbf i}_{\mathbf x}\right)_t$ is the $\mathbf i$-elementary eruption flow with respect to $\mathbf x$. 
\end{itemize}

\begin{definition} 
Let $M_j$ be the subset of $\widetilde{\Qmc}\cup\widetilde{\Theta}$ as defined above, and let $\mu$ be a vector in $W$ (see Notation \ref{not:polytope notation}). The $(\mu,M_j)$-\emph{semi elementary flow} is defined by
\[\left(\phi^\mu_{M_j}\right)_t:=\prod_{o\in M_j}(\phi^\mu_o)_t:\mathrm{Fre}(V)\to\mathrm{Fre}(V).\]
\end{definition}

Since $M_j$ is a finite set for all positive integers $j$, Proposition \ref{prop:descend} ensures that $\left(\phi^\mu_{M_j}\right)_t$ is a commuting product, and is well-defined for all real numbers $t$.

\subsection{$(\Tmc,\Jmc)$-parallel flows}\label{sec:parallel flows}
Using the semi-elementary flows just introduced, we now define the $(\Tmc,\Jmc)$-parallel flow on $\Hit_V(S)$ in the direction of a tangent vector $\mu$ in $W = T_{[\xi]}\Hit_V(S)$, where we identify $\Hit_V(S)$ with the locus of projective classes of Frenet curves in $\mathrm{Fre}(V)$ that are $\rho$-equivariant for some representation $\rho:\Gamma\to\PGL(V)$.

For this we first consider the flow 
\[ \left(\phi^\mu_{\Qmc,\Theta}\right)_t:\Hit_V(S)\to\Hit_V(S),\] which is defined to be 
\[\left(\phi^\mu_{\Qmc,\Theta}\right)_t[\xi]:=\lim_{j\to\infty}\left(\phi^\mu_{M_j}\right)_t[\xi], \] 
where $(M_1,M_2,\dots)$ is the exhaustion of $\widetilde{\Qmc} \cup \widetilde{\Theta}$ defined in Section \ref{sec:closed_edge_subset}. 

We prove in Proposition \ref{prop:general} that $\left(\phi^\mu_{\Qmc,\Theta}\right)_t$ is a well-defined flow and does not depend on the choice of the base point $p_0$ we chose to define the exhaustion $(M_1,M_2,\dots)$.
We emphasize that, as a flow on $\Hit_V(S)$, $\left(\phi^\mu_{\Qmc,\Theta}\right)_t$ is not necessarily defined for all $t$; it is possible to find a vector $\mu$ in $W$ for some group element $\gamma$ in $\Gamma$, the $i$-length $\ell^i_\rho(\gamma)$ of $\gamma$ defined by (\ref{eqn:length}) descreases linearly with $t$ under the flow $\left(\phi^\mu_{\Qmc,\Theta}\right)_t$. Also, note that the flow $\left(\phi^\mu_{\Qmc,\Theta}\right)_t$ does not yet involve the invariants along the non-isolated edges of $\Tmc$. 

We now define the counterpart which involves the invariants along the non-isolated edges of $\Tmc$. Let $c$ be a non-isolated edge of $\Tmc$, and let $\Gamma\cdot\{r_1, r_2\} \subset \widetilde \Tmc$ be the $\Gamma$-orbit determined by $c$. Choose an enumeration $\{o_1,o_2,\dots\}$ of $\Gamma\cdot\{r_1,r_2\}$. For any vector $\mu$ in $W$, consider the flow 
\[\left(\phi^\mu_c\right)_t:\Hit_V(S)\to\Hit_V(S),\] which is defined to be 
\[\left(\phi^\mu_c\right)_t[\xi]:=\prod_{l=1}^\infty\left(\phi^\mu_{o_l}\right)_t[\xi] = \lim_{r \to \infty} \prod_{l=1}^r\left(\phi^\mu_{o_l}\right)_t[\xi].\]

By Proposition \ref{prop:descend}, $\prod_{l=1}^r\left(\phi^\mu_{o_l}\right)_t$ is a commuting product for all positive integers $r$, and is a well-defined flow on $\mathrm{Fre}(V)$ for all real numbers $t$. The fact that $\left(\phi^\mu_c\right)_t$ is well-defined for all real numbers $t$ is a consequence of Proposition \ref{prop:shearing closed edge}.

\begin{remark}
The flows $\left(\phi^\mu_c\right)_t$ are examples of generalized twist flows which have been studied by Goldman in the context of representation varieties of surface groups into reductive Lie groups \cite{Goldman_twist}. 
\end{remark}

Using $\left(\phi^\mu_c\right)_t$ and $\left(\phi^\mu_{\Qmc,\Theta}\right)_t$, we can now give the proper definition of the $(\Tmc, \Jmc)$-parallel flow on $\Hit_V(S)$ associated to any vector $\mu$ in $W$. 

\begin{definition}\label{def:parallelflow}
For any vector $\mu$ in $W$, define the $(\Tmc,\Jmc)$-\emph{parallel flow} associated to $\mu$, $\phi^\mu_t:\Hit_V(S)\to\Hit_V(S)$ to be 
\begin{equation}\label{eqn:TJ}
\phi^\mu_t:=\left(\prod_{c\in\Pmc}\left(\phi^\mu_c\right)_t\right)\circ\left(\phi^\mu_{\Qmc,\Theta}\right)_t.
\end{equation}
\end{definition}

The next theorem relates the $(\Tmc,\Jmc)$-parallel flows to the parametrization $\Omega=\Omega_{\Tmc,\Jmc}$ of $\Hit_V(S)$ defined in Theorem \ref{thm:reparametrization}.

\begin{thm}\label{thm:main theorem}
Let $\Omega:\Hit_V(S)\to C_\Tmc$ be the real-analytic diffeomorphism defined in Theorem \ref{thm:reparametrization}. For any projective class of Frenet curves $[\xi]$ in $\Hit_V(S)$ and any vector $\mu$ in $W$, let 
\[I_{[\xi],\mu}:=\{t\in\Rbbb:\Omega[\xi]+t\cdot\mu\text{ satisfy the closed leaf inequalities}\}.\] 
Then for any $t\in I_{[\xi],\mu}$, 
\[\phi^\mu_t[\xi]=\Omega^{-1}\big(\Omega[\xi]+t\mu\big).\] 
\end{thm}

Theorem \ref{thm:main theorem} implies in particular that (\ref{eqn:TJ}) is a commuting product. A corollary of Theorem \ref{thm:main theorem} is the following.

\begin{cor}\label{cor:tjflowtangent}
Every pair of $(\Tmc,\Jmc)$-parallel flows on $\Hit_V(S)$ commute, and the space of $(\Tmc,\Jmc)$-parallel flows on $\Hit_V(S)$ is naturally in bijection with $T_{[\xi]}\Hit_V(S)$. 
In particular, the pair $(\Tmc,\Jmc)$ determines a trivialization of $T\Hit_V(S)$. 
\end{cor}

In the upcoming companion paper \cite{SunZhang}, it is shown that this trivialization of $T\Hit_V(S)$ is in fact symplectic with respect to the Goldman symplectic form on $\Hit_V(S)$, and every $(\Tmc,\Jmc)$-parallel flow is a Hamiltonian flow.

Theorem \ref{thm:main theorem} follows from Proposition \ref{prop:shearing closed edge} and Proposition \ref{prop:general} below.

\section{Well-definedness of $(\Tmc, \Jmc)$-parallel flows}\label{sec:main theorem}
We show now that the $(\Tmc, \Jmc)$-parallel flows are well-defined. For this we treat the flows $\left(\phi^\mu_c\right)_t$ associated to a non-isolated edge $c$, and the flows $\left(\phi^\mu_{\Qmc,\Theta}\right)_t$ separately. We start with the flows associated to a non-isolated edge, since here the situation is significantly simpler. Even though $\left(\phi^\mu_c\right)_t$ is a composition of infinitely many elementary shearing flows, the edges along which the shearing happen are locally finite. In the case of $\left(\phi^\mu_{\Qmc,\Theta}\right)_t$ however, the neighborhoods of the non-isolated edges are deformed by infinitely many elementary eruption and shearing flows.

\subsection{Well-definedness of $\left(\phi^\mu_c\right)_t$.}\label{sec:warmup}
We consider the linear subspace 
\[W=W_{\Tmc}\subset \Rbbb^{|\Tmc|\cdot(n-1)}\times\Rbbb^{|\Theta|\cdot\frac{(n-1)(n-2)}{2}}\]
that is cut out by the closed leaf equalities (see Notation \ref{not:polytope notation}). At every point $[\xi]$ in $\Hit_V(S)$, we identify the tangent space $T_{[\xi]}\Hit_V(S)$ with $W$, using the parametrization of $\Hit_V(S)$ given in Theorem~\ref{thm:reparametrization}. 

Recall that we chose for every edge in the ideal triangulation $\Tmc=\widetilde{\Tmc}/\Gamma$ an oriented representative, and denoted the set of these representatives by $\widehat{\Tmc}$. Similarly, for any ideal triangle in $\Theta=\widetilde{\Theta}/\Gamma$ we choose a cyclically ordered representative and denoted the set of these representatives by $\widehat{\Theta}$ (see Notation \ref{not:index_set}).
With this a tangent vector $\mu$ in $W$ can be identified with a tuple of real numbers 
\[\mu=\left(\left(\mu^{\mathbf k}_{\mathbf r}\right)_{\mathbf k\in\Amc;\mathbf r\in \widehat{\Tmc}},\left(\mu^{\mathbf i}_{\mathbf x}\right)_{\mathbf i\in\Bmc;\mathbf x\in\widehat{\Theta}}\right), \]
where $\Amc$ is the set of ordered pairs of positive integers that sum to $n$, $\Bmc$ is the set of ordered triples of positive integers that sum to $n$. 

Let $c$ be a non-isolated edge in $\Tmc$, and let $\mathbf r:=(r_1,r_2)$ be the corresponding ordered pair in $\widehat{\Tmc}$. We define 
\[W_c:=\left\{\mu\in W:\begin{array}{l}\mu^{\mathbf i}_{\mathbf x}=0\text{ for all }\mathbf x\in\widehat{\Theta}\text{ and }\mathbf i\in\Bmc\\
\mu^{\mathbf k}_{\mathbf s}=0\text{ for all }\mathbf s\in\widehat{\Tmc}\setminus\{\mathbf r\}\text{ and } \mathbf k\in\Amc
\end{array}\right\}.\]
Note that $W_c\subset W$ is a $(n-1)$-dimensional linear subspace. Furthermore, from the closed leaf inequalities, one verifies that $\Omega[\xi]+t\mu$ lies in the image of $\Omega: \Hit_V(S) \to W$ (see Theorem ~\ref{thm:reparametrization}) for any real number $t$, any vector $\mu$ in $W_c$, and any $[\xi]$ in $\Hit_V(S)$. In other words, $I_{[\xi],\mu}=\Rbbb$ for all vectors $\mu$ in $W_c$. 

Let $\Pi_c:W\to W_c$ be the projection defined by
\begin{itemize}
\item $\Pi_c(\mu)^{\mathbf k}_{\mathbf s}=\left\{\begin{array}{ll}
\mu^{\mathbf k}_{\mathbf s}&\text{if }\mathbf s=\mathbf r;\\
0&\text{otherwise}
\end{array}
\right.$ for all $\mathbf k\in\Amc$,
\item $\Pi_c(\mu)^{\mathbf i}_{\mathbf x}=0$ for all $\mathbf x\in\widehat{\Theta}$ and for all $\mathbf i\in\Bmc$.
\end{itemize}
Choose an enumeration $\Gamma\cdot\{r_1,r_2\}=\{o_1,o_2,\dots\}$, and observe that $(\phi_{o_l}^\mu)_t=\left(\phi_{o_l}^{\Pi_c(\mu)}\right)_t$ for all positive integers $l$. The next proposition implies a special case of Theorem \ref{thm:main theorem} when $\phi^\mu_t=(\phi^\mu_c)_t$, i.e. when $\mu$ lies in $W_c$.

\begin{prop}\label{prop:shearing closed edge}
Let $\xi$ be a representative of $[\xi]$ in $\Hit_V(S)$, let $\mu$ be a vector in $W$, let $t$ be a real number, let $c$ be a non-isolated edge in $\Tmc$, and let 
\[[\xi_0]:=\Omega^{-1}\left(\Omega[\xi]+t\Pi_c(\mu)\right)\in\Hit_V(S).\] 
Pick any triangle $\{x_1,x_2,x_3\}$ in $\widetilde{\Theta}$, and choose representatives $\xi_j$ (resp. $\xi_0$) of $\prod_{l=1}^j\left(\phi^\mu_{o_l}\right)_t[\xi]$ (resp. $[\xi_0]$) such that 
\[\xi_j(x_1)=\xi(x_1),\,\,\,\,\xi_j(x_2)=\xi(x_2)\,\,\,\,\text{ and }\,\,\,\,\xi_j^{(1)}(x_3)=\xi^{(1)}(x_3)\]
for all non-negative integers $j$. Then
\[\lim_{j\to\infty}\xi_j=\xi_0.\]
\end{prop}

\begin{remark}
Proposition \ref{prop:shearing closed edge} implies in particular that $\left(\phi^\mu_c\right)_t$ is well-defined, and does not depend on the enumeration of $\Gamma\cdot\{r_1,r_2\}$. 
\end{remark}

In order to prove Proposition \ref{prop:shearing closed edge} we use the following lemma that gives a convenient expression for how the Frenet curve $\xi_j = \prod_{l=1}^j\left(\phi^\mu_{o_l}\right)_t[\xi]$ changes with $t$. 

\begin{lem}\label{lem:o2}
Let $\xi:\partial\Gamma\to\Fmc(V)$ be Frenet, let $o=\{r_1,r_2\}$ be an edge in $\widetilde{\Tmc}$, and let $\{f_1,\dots,f_n\}$ be a basis of $V$ such that $\xi^{(l)}(r_1)=\Span_\Rbbb(f_1,\dots,f_l)$ for all $l=1$, $\dots$, $n-1$. For $m=1$, $2$, let $\mathbf r_m:=(r_m,r_{m+1})$. There is a representative $\xi_t$ of $\left(\phi^\mu_o\right)_t[\xi]$ such that 
\[\xi_t(x)=\left\{\begin{array}{ll}
\xi(x)&\text{if }r_2\leq x\leq r_1;\\
c_{\xi(\mathbf r_1)}(t)\cdot\xi(x)&\text{if }r_1\leq x\leq r_2
\end{array}
\right.\] 
where
\begin{equation}\label{eqn:short}
c_{\xi(\mathbf r_1)}(t):=\prod_{\mathbf k\in\Amc}c^{\mathbf k}_{\xi(\mathbf r_1)}\left(\mu^{\mathbf k}_{\mathbf r_1}\cdot t\right)\in\PGL(V).
\end{equation}
Here, recall that $c^{\mathbf k}_{\xi(\mathbf r_1)}$ is the projectivization of the linear map that acts as the identity on $\xi^{(k_1)}(r_1)$ and as scaling by $e^{-t}$ on $\xi^{(k_2)}(r_2)$ (see (\ref{eqn:c computation}) in Section~\ref{sec:elementaryshearing}).

In particular, $c_{\xi(\mathbf r_1)}(t)$ is represented by an upper triangular matrix 
\[\left[\begin{array}{ccccc}
\lambda_1&*&\dots&*&*\\
0&\lambda_2&\dots&*&*\\
\vdots&\vdots&\ddots&\vdots&\vdots\\
0&0&\dots&\lambda_{n-1}&*\\
0&0&\dots&0&\lambda_n
\end{array}\right]\]
in the basis $\{f_1,\dots,f_n\}$, where $\frac{\lambda_{k}}{\lambda_{k+1}}=\exp\left(\mu^{(k,n-k)}_{\mathbf r_1}\cdot t\right)$ for all $k=1$, $\dots$, $n-1$. Furthermore, if $\Span_\Rbbb\{f_n,\dots,f_{n-i+1}\}=\xi^{(i)}(r_2)$ for all $i=1$, $\dots$, $n-1$, then the matrix representing $c_{\xi(\mathbf r_1)}(t)$ in the basis $\{f_1,\dots,f_n\}$ is diagonal.
\end{lem}
\begin{proof}
This is a consequence of Lemma \ref{lem:other shearing}, and a straight forward computation. 
\end{proof}

\begin{proof}[Proof of Proposition~\ref{prop:shearing closed edge}]
By Lemma \ref{lem:technical2}, it is sufficient to prove that $\xi_j$ converges to $\xi_0$ on the vertices of $\widetilde{\Tmc}$. We may assume without loss of generality that $\mathbf x:=(x_1,x_2,x_3)$ is a cyclically ordered triple. Proposition \ref{prop:shearing projective invariants}(3), implies that for any triple of positive integers $\mathbf i := (i_1,i_2,i_3)$ that sum to $n$, we have 
\[T^{\mathbf i}(\xi_j(\mathbf x))=T^{\mathbf i}(\xi_0(\mathbf x)).\]
Proposition \ref{prop:Fock-Goncharov parametrization} then implies that $\xi_j(x_3)=\xi_0(x_3)$ for all $j$. 

Now, let $y$ be any vertex of $\widetilde{\Tmc}$ that is not $x_1$ $x_2$ or $x_3$. We use the combinatorial description of pairs of distinct points in $\widetilde{\Tmc}$ developed in Section~\ref{sec:combinatorial description} to prove that $\xi_j(y)$ converges to $\xi_0(y)$. After possibly relabelling the vertices of $\{x_1,x_2,x_3\}$, we may assume without loss of generality that $\{x_2,x_3\}$ is the minimal element of $\Emc_{(x_1,y)}$. We decompose
\[\Emc_{(x_1,y)}=\bigcup_{s=1}^k\Emc_{(x_1,y),s}\cup\bigcup_{s=0}^k\Emc_{(x_1,y),s,s+1}=\bigcup_{s=1}^k\Emc_s\cup\bigcup_{s=0}^k\Emc_{s,s+1}\]
as we described in Section \ref{sec:combinatorial description}. We have already established that for all $j$, $\xi_j$ and $\xi_0$ agree on $\{x_1,x_2,x_3\}$, which is the backward end of $\Emc_{0,1}$. Also, observe that all the ends of $\Emc_s$ and $\Emc_{s,s+1}$ are ideal triangles, except possibly the forward end of $\Emc_{k}$, in which case $\Emc_{k,k+1}$ is empty (see Remark \ref{rem:special_ends}). Thus, by Remark \ref{rem:transitive}, to prove that $\lim_{j\to\infty}\xi_j(y)=\xi_0(y)$, it is sufficient to prove that for $\Emc=\Emc_{s,s+1}$ or $\Emc=\Emc_s$, the following holds:
\begin{enumerate}
\item[($\dagger$)] If $\lim_{j\to\infty}\xi_j$ and $\xi_0$ agree on the backward end $\Delta_-$ of $\Emc$, then they must agree on its forward end $\Delta_+$ as well. 
\end{enumerate}

We first prove ($\dagger$) in the case when $\Emc=\Emc_{s,s+1}$. Since all the edges in $\Emc_{s,s+1}$ are isolated, it is clear that for all positive integers $j$, the sextuples $(\xi_j(\Delta_-),\xi_j(\Delta_+))$ and $(\xi(\Delta_-),\xi(\Delta_+))$ are projectively equivalent. Furthermore, Proposition \ref{prop:Fock-Goncharov parametrization} implies that $(\xi_0(\Delta_-),\xi_0(\Delta_+))$ and $(\xi(\Delta_-),\xi(\Delta_+))$ are projectively equivalent. Since $\lim_{j\to\infty}\xi_j(\Delta_-)=\xi_0(\Delta_-)$, Remark \ref{rem:transitive} immediately implies ($\dagger$) in this case.

Next, we prove ($\dagger$) when $\Emc=\Emc_s$. Recall that $\Gamma\cdot\{r_1,r_2\}$ is the $\Gamma$-orbit in $\widetilde{\Tmc}$ corresponding to the non-isolated edge $c$ in $\Tmc$. We orient the unique non-isolated edge $e_s$ in $\Emc_s$ such that $x_1$ and $y$ lie to the left and right of $e_s$ respectively. If $e_s$ does not lie in $\Gamma\cdot\{r_1,r_2\}$, or $e_s$ is the maximum of $\Emc_s$ (in which case $s=k$ and $\Emc_{k,k+1}$ is empty), then the same argument as above implies ($\dagger$). 

On the other hand, if $e_s=\gamma\cdot\{r_1,r_2\}$ for some group element $\gamma$ in $\Gamma$, we can assume without loss of generality that $\gamma\cdot r_1$ and $\gamma\cdot r_2$ are respectively the backward and forward endpoints of $e_s$. Let $\mathbf r:=(r_1,r_2)$, and let $\{f_1,\dots,f_n\}$ be a basis of $V$ such that $\Span_\Rbbb(f_i)=\xi^{(i)}(\gamma\cdot r_1)\cap\xi^{(n-i+1)}(\gamma\cdot r_2)$. Lemma \ref{lem:o2} implies that for sufficiently large $j$, $(\xi_j(\Delta_-),\xi_j(\Delta_+))$ and $(\xi(\Delta_-),c_{\xi(\gamma\cdot \mathbf r)}(t)\cdot \xi(\Delta_+))$ are projectively equivalent. 
Furthermore, $c_{\xi(\gamma\cdot\mathbf r)}(t)$ is represented in the basis $\{f_1,\dots,f_n\}$ by the diagonal matrix 
\begin{equation}\label{eqn:diagonal}
\left[\begin{array}{ccc}
\lambda_1&\dots&0\\
\vdots&\ddots&\vdots\\
0&\dots&\lambda_n
\end{array}\right]
\end{equation}
where $\frac{\lambda_{k_1}}{\lambda_{k_1+1}}=\exp\left(\mu^{\mathbf k}_{\mathbf r}\cdot t\right)$ for all pairs of positive integers $\mathbf k:=(k_1,k_2)$ that sum to $n$.

Let $J=\{T_1,T_2\}$ be a bridge in $\widetilde{\Jmc}$ across $\{r_1,r_2\}$, such that $T_1$ and $T_2$ lie to the left and right of $(r_1,r_2)$ respectively. Let $\gamma_m$ be the primitive element that fixes $r_1$ and $r_2$, where the repelling fixed point of $\gamma_m$ is the point which is a vertex of the triangle $T_m$. Proposition \ref{prop:Fock-Goncharov parametrization} implies that for $m=1$, $2$, there is a unique projective transformation $g_m$ in $\PGL(V)$ such that $\xi_0(p)=g_m\cdot\xi(p)$ for every vertex $p$ of the ideal triangles in 
\[\bigcup_{d=-\infty}^{\infty}\gamma_m^d\cdot\widetilde{\Theta}(J,T_m), \]
the union of the $\gamma_m$ translates of the closed edge subset $\widetilde{\Theta}(J,T_m)$ (see Definition \ref{def:closed edge subset}).

Since $r_1$ and $r_2$ are the accumulation points of the vertices of these ideal triangles, the continuity of $\xi$ and $\xi_0$ implies that 
\[\xi_0(r_1)=g_m\cdot\xi(r_1)\,\,\,\,\text{ and }\,\,\,\,\xi_0(r_2)=g_m\cdot\xi(r_2)\]
for both $m=1$, $2$. In particular, the projective transformation $g_1^{-1}g_2$ is represented in the basis $\{f_1,\dots,f_n\}$ by a diagonal matrix, and $(\xi_0(\Delta_-),\xi_0(\Delta_+))$ and $(\xi(\Delta_-),g_1^{-1}g_2\cdot \xi(\Delta_+))$ are projectively equivalent.

Since by assumption $\lim_{j\to\infty}\xi_j(\Delta_-)=\xi_0(\Delta_-)$, to finish the proof of ($\dagger$) in this case, it is sufficient to show that $(\xi_j(\Delta_-),\xi_j(\Delta_+))$ and $(\xi_0(\Delta_-),\xi_0(\Delta_+))$ are projectively equivalent for sufficiently large $j$, which is equivalent to showing that $g_1^{-1}g_2=c_{\xi(\gamma\cdot\mathbf r)}(t)$. Let 
\[\eta_1(x):=\left\{\begin{array}{ll}
\xi(x)&\text{if }r_1\leq x\leq r_2;\\
c_{\xi(\gamma\cdot\mathbf r)}(t)\cdot \xi(x)&\text{if }r_2\leq x\leq r_1,
\end{array}
\right.\]
and 
\[\eta_2(x):=\left\{\begin{array}{ll}
\xi(x)&\text{if }r_1\leq x\leq r_2;\\
g_1^{-1}g_2\cdot \xi(x)&\text{if }r_2\leq x\leq r_1.
\end{array}
\right.\]

By Lemma \ref{lem:new invariant 3}
\[\alpha^{\mathbf k}_{\mathbf r,J}[\eta_1]=\alpha^{\mathbf k}_{\mathbf r,J}[\xi]+\mu^{\mathbf k}_{\mathbf r}\cdot t=\alpha^{\mathbf k}_{\mathbf r,J}[\xi_0]=\alpha^{\mathbf k}_{\mathbf r,J}[\eta_2]\]
for all pairs of positive integers $\mathbf k:=(k_1,k_2)$ that sum to $n$. Denote the vertices of $T_2$ by $p_2,z_2,w_2$ such that $p_2$ lies in $\{r_1,r_2\}$ and either $(p_2,z_2,w_2,q_2)$ or $(q_2,w_2,z_2,p_2)$ are cyclically ordered, where $q_2$ is the point in $\{r_1,r_2\}$ that is not $p_2$. Lemma \ref{lem:new invariant 1} then implies that there is some unipotent projective transformation $u$ in $\PGL(V)$ that fixes $\xi(p_2)$, sends $g_1^{-1}g_2\cdot\xi(z_2)$ to $c_{\xi(\gamma\cdot\mathbf r)}(t)\cdot\xi(z_2)$, and sends $g_1^{-1}g_2\cdot\xi^{(1)}(w_2)$ to $c_{\xi(\gamma\cdot\mathbf r)}(t)\cdot\xi^{(1)}(w_2)$. Since $g_1^{-1}g_2$ and $c_{\xi(\gamma\cdot\mathbf r)}(t)$ both fix $\xi(p_2)$, this allows us, using Remark \ref{rem:transitive}, to deduce that $ug_1^{-1}g_2=c_{\xi(\gamma\cdot\mathbf r)}(t)$. In the basis $\{f_1,\dots,f_n\}$, $g_1^{-1}g_2$ and $c_{\xi(\gamma\cdot\mathbf r)}(t)$ are both diagonal, and $u$ is a unipotent matrix. Hence, $u=\id$ and $g_1^{-1}g_2=c_{\xi(\gamma\cdot\mathbf r)}(t)$. This ends the proof. 
\end{proof}


\subsection{Behavior near the non-isolated edges.}\label{sec:closed edge}
We would like to apply a similar argument as in the proof Proposition \ref{prop:shearing closed edge} to prove that $\left(\phi^\mu_{\Qmc,\Theta}\right)_t$ is well-defined. However, the argument is more delicate now, since neighbourhoods around non-isolated edges are deformed by infinitely many elementary eruption and shearing flows. We thus have to more precisely describe the limit of the semi-elementary flows near the non-isolated edges. 


\subsubsection{Main technical statement} \label{technical}
Our next goal is to state the main technical theorem we need to prove that $\left(\phi^\mu_{\Qmc,\Theta}\right)_t$ is well-defined. We first need the following lemma, which gives an explicit description of a representative in the projective class of Frenet curves $\left(\phi^\mu_o\right)_t[\xi]$ when $o$ is an ideal triangle in $\widetilde{\Theta}$. This is the counterpart of Lemma~\ref{lem:o2}, which gives such an explicit description when $o$ is an edge in $\widetilde{\Tmc}$. 

\begin{lem}\label{lem:o}
Let $\xi:\partial\Gamma\to\Fmc(V)$ be Frenet, and let $o$ be an ideal triangle in $\widetilde{\Theta}$ with vertices $x_1,x_2,x_3$, enumerated such that $\mathbf x:=(x_1,x_2,x_3)$ is cyclically ordered. For $m=1$, $2$, $3$, let $\mathbf x_m:=(x_m,x_{m+1},x_{m-1})$. If $\{f_1,\dots,f_n\}$ is a basis of $V$ such that $\xi^{(l)}(x_1)=\Span_\Rbbb(f_1,\dots,f_l)$ for all $l=1$, $\dots$, $n-1$, then there is a representative $\xi_t$ of $\left(\phi^\mu_o\right)_t[\xi]$ such that 
\[\xi_t(x)=\left\{\begin{array}{ll}
a'_{\xi(\mathbf x_3)}(t)\cdot\xi(x)&\text{if }x_2\leq x\leq x_3;\\
\xi(x)&\text{if }x_3\leq x\leq x_1;\\
a_{\xi(\mathbf x_1)}(t)\cdot\xi(x)&\text{if }x_1\leq x\leq x_2,
\end{array}
\right.\]
where
\[a_{\xi(\mathbf x_1)}(t):=\prod_{\mathbf i\in\Bmc}a^{\mathbf i}_{\xi(\mathbf x_1)}\left(\mu^{\mathbf i}_{\mathbf x_1}\cdot t\right)\in\PGL(V)\]
and
\[a'_{\xi(\mathbf x_3)}(t):=\prod_{\mathbf i\in\Bmc}a^{\mathbf i}_{\xi(\mathbf x_3)}\left(-\mu^{\mathbf i}_{\mathbf x_3}\cdot t\right)\in\PGL(V).\] 
Here $a^{\mathbf i}_{\xi(\mathbf x)}(t)$ is the projectivization of the linear map that acts as scaling by $e^{\frac{2t}{3}}$ on $\xi^{(i_1)}(x_1)$ and scaling by $e^{-\frac{t}{3}}$ on $\xi^{(i_2)}(x_2)+\xi^{(i_3)}(x_3)$ (see (\ref{eqn:a computation}) in Section~\ref{sec:elementaryeruption}).

In particular, $a_{\xi(\mathbf x_1)}(t)$ is represented by an upper-triangular matrix 
\[\left[\begin{array}{ccccc}
\lambda_1&*&\dots&*&*\\
0&\lambda_2&\dots&*&*\\
\vdots&\vdots&\ddots&\vdots&\vdots\\
0&0&\dots&\lambda_{n-1}&*\\
0&0&\dots&0&\lambda_n
\end{array}\right]\]
in the basis $\{f_1,\dots,f_n\}$, where 
\[\frac{\lambda_k}{\lambda_{k+1}}=\exp\left(\sum_{i_2+i_3=n-k}\mu^{(k,i_2,i_3)}_{\mathbf x_1}\cdot t\right)\] 
for all $k=1$, $\dots$, $n-1$. Furthermore, $a_{\xi(\mathbf x_1)}(t)$ fixes $\xi^{(1)}(x_3)$.
\end{lem}

\begin{proof}
This is a consequence of Lemma \ref{lem:other eruption}, and a straight forward computation. 
\end{proof}

Recall that for any non-isolated edge $\{r_1, r_2\}$ in $\widetilde{\Tmc}$, any bridge $J=\{T_1, T_2\}$ across $\{r_1,r_2\}$, and any $m=1$, $2$, we defined the closed edge subset $\widetilde{\Theta}(J,T_m)\subset\widetilde{\Theta}$ (see Definition \ref{def:closed edge subset}). Assume without loss of generality that $T_1$ and $T_2$ lie to the left and right of the oriented edge $(r_1,r_2)$ respectively, and let $p_m$ be the vertex of $T_m$ that is either $r_1$ or $r_2$. Recall that in Notation \ref{not:closed edge}, we defined a collection of ideal triangles $\{\dots,T_{m,-1},T_{m,0},T_{m,1},T_{m,2},\dots\}$ such that all these triangles share $p_m$ as a common vertex, $T_{m,1}=T_m$, and $T_{m,h}$ shares a common edge $e_{m,h}$ with $T_{m,h+1}$. Also, we defined a positive integer $H_m$ such that $\widetilde{\Theta}(J,T_m)=\{T_{m,1},T_{m,2},\dots,T_{m,H_m}\}$ (see Figure \ref{fig:closededge}). Let $z_{m,h}$ denote the vertex of $e_{m,h}$ that is not $p_m$, set $\mathbf p_{m,h}:=(p_m,z_{m,h})$, and set
\[\mathbf t_{m,h}:=\left\{\begin{array}{ll}
(p_m,z_{m,h-1},z_{m,h})&\text{if }p_m=r_m;\\
(p_m,z_{m,h},z_{m,h-1})&\text{if }p_m=r_{3-m}.
\end{array}
\right.\]

Let $[\xi]$ be a projective class of Frenet curves in $\Hit_V(S)$, let $\mu$ be a vector in $W$, and let $t$ be a real number. For each triple in $\mathbf t_{m,h}$, define
\begin{equation}\label{eqn:triangle transform}
a_{\xi(\mathbf t_{m,h})}(t):=\left\{\begin{array}{rl}
\displaystyle\prod_{\mathbf i\in\Bmc}a^{\mathbf i}_{\xi(\mathbf t_{m,h})}\left(\mu^{\mathbf i}_{\mathbf t_{m,h}}\cdot t\right)&\text{if }p_m=r_m;\\
\displaystyle\prod_{\mathbf i\in\Bmc}a^{\mathbf i}_{\xi(\mathbf t_{m,h})}\left(-\mu^{\mathbf i}_{\mathbf t_{m,h}}\cdot t\right)&\text{if }p_m=r_{3-m},\\
\end{array}\right.
\end{equation}
and for each pair $\mathbf p_{m,h}$, define 
\begin{equation}\label{eqn:edge transform}
c_{\xi(\mathbf p_{m,h})}(t):=\left\{\begin{array}{rl}
\displaystyle\prod_{\mathbf i\in\Amc}c^{\mathbf k}_{\xi(\mathbf p_{m,h})}\left(\mu^{\mathbf k}_{\mathbf p_{m,h}}\cdot t\right)&\text{if }p_m=r_m;\\
\displaystyle\prod_{\mathbf i\in\Amc}c^{\mathbf k}_{\xi(\mathbf p_{m,h})}\left(-\mu^{\mathbf k}_{\mathbf p_{m,h}}\cdot t\right)&\text{if }p_m=r_{3-m}.
\end{array}\right.
\end{equation}
Here, recall that $\mu^\mathbf i_{\mathbf t_{m,h}}$ and $\mu^\mathbf k_{\mathbf p_{m,h}}$ are the entries of the vector $\mu$ (see Section \ref{sec:tangent}).

Using this, we set 
\begin{equation}\label{eqn:closed edge subset transform}
a_m(t)=a_{\xi,m}(t):=a_{\xi(\mathbf t_{m,1})}(t)c_{\xi(\mathbf p_{m,1})}(t)\dots a_{\xi(\mathbf t_{m,H_m})}(t)c_{\xi(\mathbf p_{m,H_m})}(t).
\end{equation}
It follows from Lemma \ref{lem:o2} and Lemma \ref{lem:o} that each $a_{\xi(\mathbf t_{m,h})}(t)$ and each $c_{\xi(\mathbf p_{m,h})}(t)$ fixes $\xi(p_m)$, so $a_m(t)$ also fixes $\xi(p_m)$. In other words, if $\{f_1,\dots,f_n\}$ is a basis of $V$ such that $f_i$ lies in $\xi^{(i)}(r_1)\cap\xi^{(n-i+1)}(r_2)$, then $a_m(t)$ is upper triangular if $p_m=r_1$ and lower triangular if $p_m=r_2$. Let $u_m(t)=u_{\xi,m}(t)$ be the unipotent projective transformation that fixes $\xi(p_m)$ such that 
\[a_m(t)=u_m(t)h_m(t)\] 
for some projective transformation $h_m(t)=h_{\xi,m}(t)$ that fixes both $\xi(r_1)$ and $\xi(r_2)$.

We consider the exhaustion $(M_1,M_2,\dots)$ of $\widetilde{\Qmc}\cup\widetilde{\Theta}$ used to define $\left(\phi^\mu_{\Qmc,\Theta}\right)_t$ (see Section \ref{sec:closed_edge_subset}).The following is the main technical theorem that we prove in Section \ref{sec:closed edge}. Informally, it tells us how $\left(\phi^\mu_{\Qmc,\Theta}\right)_t$, evaluated at the vertices of a non-isolated edge and at the vertices of the two ideal triangles associated to a bridge across this non-isolated edge, changes with $t$.

Let $q_m$ be the vertex of the non-isolated edge $\{r_1,r_2\}$ that is not $p_m$, and let $\gamma_m$ be the primitive element in $\Gamma$ that has $p_m$ and $q_m$ as its repellor and attractor respectively (see Notation \ref{not:closed edge}). 

\begin{thm}\label{thm:unipotent closed edge}
Let $\{r_1,r_2\}$ be a non-isolated edge in $\widetilde{\Pmc}$ and let $J=\{T_1,T_2\}$ be a bridge in $\widetilde{\Jmc}$ across $\{r_1,r_2\}$. Suppose that $[\xi]$ is a projective class of Frenet curves in $\Hit_V(S)$ and let $\xi$ be a representative of $[\xi]$. Let $\xi_j$ be the representative of $\left(\phi^\mu_{M_j}\right)_t[\xi]$ normalized such that $\xi_j(p_1)=\xi(p_1)$, $\xi_j(z_{1,0})=\xi(z_{1,0})$ and $\xi_j^{(1)}(z_{1,1})=\xi^{(1)}(z_{1,1})$. 
\begin{enumerate}
\item The infinite product 
\begin{equation}\label{eqn:unipotent}
u_{m,\infty}(t):=\lim_{d\to\infty}\big(u_m(t)h_m(t)\rho(\gamma_m)\big)^d\cdot\big(h_m(t)\rho(\gamma_m)\big)^{-d}
\end{equation}
 is a well-defined unipotent projective transformation in $\PGL(V)$ that fixes $\xi(p_m)$. Furthermore,
\begin{eqnarray*}
\lim_{j\to\infty}\xi_j(x_1)&=&u_{1,\infty}(t)\cdot\xi(r_1),\\
\lim_{j\to\infty}\xi_j(x_2)&=&u_{1,\infty}(t)\cdot\xi(r_2),\\
\lim_{j\to\infty}\xi_j(z_{2,0})&=&u_{1,\infty}(t)u_{2,\infty}(t)^{-1}\cdot\xi(z_{2,0}),\\
\lim_{j\to\infty}\xi_j^{(1)}(z_{2,1})&=&u_{1,\infty}(t)u_{2,\infty}(t)^{-1}\cdot\xi^{(1)}(z_{2,1}).
\end{eqnarray*}
\item $\displaystyle\lim_{d\to\infty}\lim_{j\to\infty}\xi_j(\gamma_m^d\cdot z_{m,0})=\lim_{j\to\infty}\xi_j(q_m)$.
\end{enumerate}
\end{thm}

\subsubsection{The proof of Theorem \ref{thm:unipotent closed edge}} 
Recall that $\widetilde{\Theta}(J,T_m)$ denote a closed edge subset (see Definition~\ref{def:closed edge subset}). We denote by $\widetilde{\Tmc}(J,T_m)$ the set of edges of the triangles contained in the closed edge subset $\widetilde{\Theta}(J,T_m)$, $\widetilde{\Tmc}(J,T_m)=\{e_{m,1},\dots,e_{m,H_m}\}$.

It is a direct consequence of the definition of the sets $M_j$ in the exhaustion of $\widetilde{\Theta} \cup \widetilde{Q}$ that there is an integer $D'$ such that for sufficiently large $j$ 
\begin{itemize}
\item $\gamma_m^d\cdot\left(\widetilde{\Theta}(J,T_m)\cup\widetilde{\Tmc}(J,T_m)\right)\subset M_j$ for all $0\leq d\leq j+D'$, and
\item $\gamma_m^d\cdot\left(\widetilde{\Theta}(J,T_m)\cup\widetilde{\Tmc}(J,T_m)\right)\cap M_j=\emptyset$ for all $d> j+D'$.
\end{itemize}
For any sufficiently large $j$, set $D:=j+D'$, and let $E:=2(D+1)(H_1+H_2)$. 
 We enumerate the set 
\[\bigcup_{d=0}^D\bigcup_{m=1}^2\gamma_m^d\cdot\left(\widetilde{\Theta}(J,T_m)\cup\widetilde{\Tmc}(J,T_m)\right)=\{o_1,o_2,o_3,\dots o_E\}\subset\widetilde{\Qmc}\cup\widetilde{\Theta},\] 
in the following order 
\[\begin{array}{l}
T_{1,1},\,\,e_{1,1},\,\,\dots\,\,,\,\,T_{1,H_1},\,\,e_{1,H_1},\\
T_{2,1},\,\,e_{2,1},\,\,\dots\,\,,\,\,T_{2,H_2},\,\,e_{2,H_2},\\
T_{1,H_1+1},\,\,e_{1,H_1+1},\,\,\dots\,\,,\,\,T_{1,2H_1},\,\,e_{1,2H_1},\\
T_{2,H_2+1},\,\,e_{2,H_2+1},\,\,\dots\,\,,\,\,T_{2,2H_2},\,\,e_{2,2H_2},\\
\hspace{3cm}\vdots\\
T_{1,DH_1+1},\,\,e_{1,DH_1+1},\,\,\dots\,\,,\,\,T_{1,(D+1)H_1},\,\,e_{1,(D+1)H_1},\\
T_{2,DH_2+1},\,\,e_{2,DH_2+1},\,\,\dots\,\,,\,\,T_{2,(D+1)H_2},\,\,e_{2,(D+1)H_2}.
\end{array}\]

\begin{remark}\label{rem:simplify}
Fix a sufficiently large positive integer $j$. Observe that the projective classes of Frenet curves
\[\left(\phi^\mu_{M_j}\right)_t[\xi]\,\,\,\text{ and }\,\,\,\prod_{h=1}^{E}\left(\phi^\mu_{o_h}\right)_t[\xi]\] 
contain representatives that when restricted to the vertices of the ideal triangles in 
\[\bigcup_{d=-\infty}^{\infty}\gamma_m^d\cdot\widetilde{\Theta}(J,T_m),\]
are projectively equivalent. To prove Theorem \ref{thm:unipotent closed edge}, we can thus work with $\prod_{h=1}^{E}\left(\phi^\mu_{o_h}\right)_t$ in place of $\left(\phi^\mu_{M_j}\right)_t$. 

\begin{notation}\label{not:xi}
For all positive integers $h$, let $\xi_{m,h}$ be the representative of 
\[\prod_{r=1}^h\left(\phi^\mu_{o_r}\right)_t[\xi]\] 
such that $\xi_{m,h}(p_m)=\xi(p_m)$, $\xi_{m,h}(z_{m,0})=\xi(z_{m,0})$, and $\xi_{m,h}^{(1)}(z_{m,1})=\xi^{(1)}(z_{m,1})$.
\end{notation}

The first step is to prove the following lemma, which is an analog of Theorem \ref{thm:unipotent closed edge} for the flows $\prod_{h=1}^{2(H_1+H_2)}\left(\phi^\mu_{o_h}\right)_t$, i.e. it tells us how $\xi_t:=\prod_{h=1}^{2(H_1+H_2)}\left(\phi^\mu_{o_h}\right)_t(\xi)$, evaluated at $r_1$ and $r_2$ and at the vertices of the two ideal triangles associated to a bridge across $\{r_1,r_2\}$, changes with $t$.
\end{remark}

\begin{lem}\label{lem:computation1}
Let $J$ be a bridge in $\widetilde{\Jmc}$ across the non-isolated edge $\{r_1,r_2\}$ of $\widetilde{\Tmc}$. Let $\xi:\partial\Gamma\to\Fmc(V)$ be any Frenet curve, and let $\{f_1,\dots,f_n\}$ be a basis of $V$ such that $f_i$ lies in $\xi^{(i)}(r_1)\cap\xi^{(n-i+1)}(r_2)$ for all $i=1$, $\dots$, $n$. Then for all real numbers $t$, the projective transformation $a_m(t)$ in $\PGL(V)$ defined by (\ref{eqn:closed edge subset transform}) satisfies the following properties. 
\begin{enumerate}
\item $a_m(t)$ fixes the flag $\xi(p_m)$.
\item Let $\lambda_{m,1},\dots,\lambda_{m,n}$ be the diagonal entries down the diagonal of the matrix representing $a_m(t)$ in the basis $\{f_1,\dots,f_n\}$. Then for all $k=1$, $\dots$, $n-1$,
\begin{equation}\label{eqn:closed_leaf_par1}
\frac{\lambda_{1,k}}{\lambda_{1,k+1}}=\left\{\begin{array}{ll}
\displaystyle\exp\left(\sum_{h=1}^{H_1}\left(\mu^{(k,n-k)}_{\mathbf p_{1,h}}+\sum_{i_2+i_3=n-k}\mu^{(k,i_2,i_3)}_{\mathbf t_{1,h}}\right) t\right)&\text{if }p_1=r_1;\\
\displaystyle\exp\left(-\sum_{h=1}^{H_1}\left(\mu^{(n-k,k)}_{\mathbf p_{1,h}}+\sum_{i_2+i_3=k}\mu^{(n-k,i_2,i_3)}_{\mathbf t_{1,h}}\right) t\right)&\text{if }p_1=r_2,\\
\end{array}\right.
\end{equation}
and 
\begin{equation}\label{eqn:closed_leaf_par2}
\frac{\lambda_{2,k}}{\lambda_{2,k+1}}=\left\{\begin{array}{ll}
\displaystyle\exp\left(-\sum_{h=1}^{H_2}\left(\mu^{(k,n-k)}_{\mathbf p_{2,h}}+\sum_{i_2+i_3=n-k}\mu^{(k,i_2,i_3)}_{\mathbf t_{2,h}}\right) t\right)&\text{if }p_2=r_1;\\
\displaystyle\exp\left(\sum_{h=1}^{H_2}\left(\mu^{(n-k,k)}_{\mathbf p_{2,h}}+\sum_{i_2+i_3=k}\mu^{(n-k,i_2,i_3)}_{\mathbf t_{2,h}}\right) t\right)&\text{if }p_2=r_2.
\end{array}\right.
\end{equation}
In particular, the closed leaf equalities imply that all the diagonal entries of $a_1(t)$ and $a_2(t)$ agree.
\item $\xi_{m,2(H_1+H_2)}(z)=a_m(t)\cdot\xi(z)$ for all 
\[z\in\bigcup_{d=1}^\infty\bigcup_{m=1}^2\gamma_m^d\cdot\{z_{m,1},\dots,z_{m,H_m}\}\cup\{r_1,r_2\}.\] 
\item $\xi_{m,2(H_1+H_2)}(z_{3-m,0})=a_m(t)a_{3-m}(t)^{-1}\cdot\xi(z_{3-m,0})$ 
\item $\xi_{m,2(H_1+H_2)}^{(1)}(z_{3-m,1})=a_m(t)a_{3-m}(t)^{-1}\cdot\xi^{(1)}(z_{3-m,1})$
\end{enumerate}
\end{lem}

\begin{proof} 
(1) and (2) are immediate consequences of Lemma \ref{lem:o2} and Lemma~\ref{lem:o} and a straightforward computation. We prove (3), (4), (5) in the case when $m=1$; the case when $m=2$ is identical. 

By Lemma \ref{lem:o}, we see that $\xi_{1,1}(p_1)=\xi(p_1)$, $\xi_{1,1}(z_{1,0})=\xi(z_{1,0})$, $\xi_{1,1}^{(1)}(z_{1,1})=\xi^{(1)}(z_{1,1})$, and $\xi_{1,1}(z)=a_{\xi(\mathbf t_{1,1})}(t)\cdot\xi(z)$,
for all 
\[z\in\bigcup_{d=0}^\infty\gamma_1^d\cdot\{z_{1,1},\dots,z_{1,H_1}\}\cup\bigcup_{d=-\infty}^\infty\gamma_2^d\cdot\{z_{2,1},\dots,z_{2,H_2}\}\cup\{r_1,r_2\},\]
where $a_{\xi(\mathbf t_{1,1})}(t)$ is the projective transformation in $\PGL(V)$ defined by (\ref{eqn:triangle transform}). 

By definition, $\xi_{1,2}(p_1)=\xi_{1,1}(p_1)=\xi(p_1)$, $\xi_{1,2}(z_{1,0})=\xi_{1,1}(z_{1,0})=\xi(z_{1,0})$, $\xi_{1,2}^{(1)}(z_{1,1})=\xi_{1,1}^{(1)}(z_{1,1})=\xi^{(1)}(z_{1,1})$, and $\xi_{1,2}(z_{1,1})=\xi_{1,1}(z_{1,1})$. Then by Lemma \ref{lem:o2}, $\xi_{1,2}(z)=c_{\xi_{1,1}(\mathbf p_{1,1})}(t)\cdot\xi_{1,1}(z)$ for all 
\[z\in\{z_{1,2},\dots,z_{1,H_1}\}\cup\bigcup_{d=1}^\infty\gamma_1^d\cdot\{z_{1,1},\dots,z_{1,H_1}\}\cup\bigcup_{d=-\infty}^\infty\gamma_2^d\cdot\{z_{2,1},\dots,z_{2,H_2}\}\cup\{r_1,r_2\},\]
where $c_{\xi_{1,1}(\mathbf p_{1,1})}(t)$ is the projective transformation in $\PGL(V)$ defined by (\ref{eqn:edge transform}). 
The following is a key observation for this proof:
\[\xi_{1,2}(z)=c_{\xi_{1,1}(\mathbf p_{1,1})}(t)a_{\xi(\mathbf t_{1,1})}(t)\cdot\xi(z)=a_{\xi(\mathbf t_{1,1})}(t)c_{\xi(\mathbf p_{1,1})}(t)\cdot\xi(z),\]
for all 
\[z\in\{z_{1,2},\dots,z_{1,H_1}\}\cup\bigcup_{d=1}^\infty\gamma_1^d\cdot\{z_{1,1},\dots,z_{1,H_1}\}\cup\bigcup_{d=-\infty}^\infty\gamma_2^d\cdot\{z_{2,1},\dots,z_{2,H_2}\}\cup\{r_1,r_2\}.\]

Iterating the above procedure $H_1$ times proves that $\xi_{1,2H_1}$ satisfies the following:
\begin{itemize}
\item $\xi_{1,2H_1}(p_1)=\xi(p_1)$, $\xi_{1,2H_1}(z_{1,0})=\xi(z_{1,0})$, and $\xi_{1,2H_1}^{(1)}(z_{1,1})=\xi^{(1)}(z_{1,1})$,
\item $\xi_{1,2H_1}(z)=a_1(t)\cdot\xi(z)$ for all 
\[z\in\bigcup_{d=1}^\infty\gamma_1^d\cdot\{z_{1,1},\dots,z_{1,H_1}\}\cup\bigcup_{d=-\infty}^\infty\gamma_2^d\cdot\{z_{2,1},\dots,z_{2,H_2}\}\cup\{r_1,r_2\},\]
\item $\xi_{1,h}(z)=\xi_{1,2H_1}(z)$ for all $z\in\{p_1,z_{1,0},z_{1,1},\dots,z_{1,H_1}=\gamma_1\cdot z_{1,0}\}$ and $h\geq 2H_1$, and
\item $\xi_{1,h}^{(1)}(\gamma_1\cdot z_{1,1})=\xi_{1,2H_1}^{(1)}(\gamma_1\cdot z_{1,1})$ for all $h\geq 2H_1$.
\end{itemize}

From the above properties of $\xi_{1,2H_1}$, we see that $a_1(t)^{-1}$ sends $\xi_{2,2H_1}(p_2)=\xi(p_2)$, $\xi_{2,2H_1}(z_{2,0})=\xi(z_{2,0})$, and $\xi_{2,2H_1}^{(1)}(z_{2,1})=\xi^{(1)}(z_{2,1})$ to $\xi_{1,2H_1}(p_2)$, $\xi_{1,2H_1}(z_{2,0})$, and $\xi_{1,2H_1}^{(1)}(z_{2,1})$ respectively. Since $\xi_{1,2H_1}$ and $\xi_{2,2H_2}$ differ by a projective transformation, Remark \ref{rem:transitive} implies that $\xi_{2,2H_1}=a_1(t)^{-1}\cdot\xi_{1,2H_1}$. We can thus repeat the above argument for the case when $m=2$ (with $\xi_{2,2H_1}$ in place of $\xi$) to prove that $\xi_{2,2(H_1+H_2)}$ satisfies the following:
\begin{itemize}
\item $\xi_{2,2(H_1+H_2)}(p_2)=\xi(p_2)$, $\xi_{2,2(H_1+H_2)}(z_{2,0})=\xi(z_{2,0})$, and 
$\xi_{2,2(H_1+H_2)}^{(1)}(z_{2,1})=\xi^{(1)}(z_{2,1})$,
\item $\xi_{2,2(H_1+H_2)}(z)=a_2(t)\cdot\xi_{2,2H_1}(z)$ for all 
\[z\in\bigcup_{d=1}^\infty\gamma_2^d\cdot\{z_{2,1},\dots,z_{2,H_2}\}\cup\bigcup_{d=-\infty}^\infty\gamma_1^d\cdot\{z_{1,1},\dots,z_{1,H_1}\}\cup\{r_1,r_2\},\]
\item $\xi_{2,h}(z)=\xi_{2,2(H_1+H_2)}(z)$ for all $z\in\{p_2,z_{2,0},z_{2,1},\dots,z_{2,H_2}=\gamma_2\cdot z_{2,0}\}$ and $h\geq 2(H_1+H_2)$, and
\item $\xi_{2,h}^{(1)}(\gamma_2\cdot z_{2,1})=\xi_{2,2(H_1+H_2)}^{(1)}(\gamma_2\cdot z_{2,1})$ for all $h\geq 2(H_1+H_2)$.
\end{itemize}

Again from the above properties of $\xi_{2,2(H_1+H_2)}$, observe that $a_1(t)a_2(t)^{-1}$ sends $\xi_{2,2(H_1+H_2)}(p_1)$, $\xi_{2,2(H_1+H_2)}(z_{1,0})$, and $\xi_{2,2(H_1+H_2)}^{(1)}(z_{1,1})$ to $\xi(p_1)=\xi_{1,2(H_1+H_2)}(p_1)$, $\xi(z_{1,0})=\xi_{1,2(H_1+H_2)}(z_{1,0})$, and $\xi^{(1)}(z_{1,1})=\xi_{1,2(H_1+H_2)}^{(1)}(z_{1,1})$ respectively. Since $\xi_{2,2(H_1+H_2)}$ and $\xi_{1,2(H_1+H_2)}$ differ by a projective transformation, Remark \ref{rem:transitive} implies that $a_1(t)a_2(t)^{-1}\cdot\xi_{2,2(H_1+H_2)}=\xi_{1,2(H_1+H_2)}$. As a consequence, we see that $\xi_{1,2(H_1+H_2)}(z)=a_1(t)a_2(t)^{-1}\cdot\xi_{2,2(H_1+H_2)}(z)=\xi_{1,2H_1}(z)=a_1(t)\cdot\xi(z)$ for all
\[z\in\bigcup_{d=1}^\infty\bigcup_{m=1}^2\gamma_1^d\cdot\{z_{m,1},\dots,z_{m,2H_1}\}\cup\{x_1,x_2\},\] 
so (3) holds. Similarly,
\[\xi_{1,2(H_1+H_2)}(z_{2,0})=a_1(t)a_2(t)^{-1}\cdot\xi_{2,2(H_1+H_2)}(z_{2,0})=a_1(t)a_2(t)^{-1}\cdot\xi(z_{2,0}),\]
so (4) holds. The same computation proves (5).
\end{proof}

When the Frenet curve $\xi:\partial\Gamma\to\Fmc(V)$ is $\rho$-equivariant for some representation $\rho:\Gamma\to\PGL(V)$, we can prove the following lemma, which is an analog of Lemma \ref{lem:computation1} for the flows $\prod_{h=1}^{E}\left(\phi^\mu_{o_h}\right)_t$. Here, recall that $E:=2(D+1)(H_1+H_2)$.

\begin{lem}\label{lem:computation2}
Let $J$ be a bridge in $\widetilde{\Jmc}$ across the non-isolated edge $\{r_1,r_2\}$ of $\widetilde{\Tmc}$. Let $\xi:\partial\Gamma\to\Fmc(V)$ be a $\rho$-equivariant Frenet curve for some representation $\rho:\Gamma\to\PGL(V)$ and let $a_m(t)$ be the projective transformation in $\PGL(V)$ defined by (\ref{eqn:closed edge subset transform}). For all non-negative integers $d$, set
\[a'_{m,d}(t):=\rho(\gamma_m)^da_m(t)\rho(\gamma_m)^{-d}\,\,\,\,\text{ and }\,\,\,\,a_{m,d}(t):=a'_{m,0}(t)a'_{m,1}(t)\dots a'_{m,d}(t).\] 
Then the following statements hold:
\begin{enumerate}
\item $a_{m,d}(t)$ fixes the flag $\xi(p_m)$.
\item $\xi_{m,E}(\gamma_m^d\cdot z_{m,0})=a_{m,d-1}(t)\cdot\xi(\gamma_m^d\cdot z_{m,0})$ and $\xi_{m,E}^{(1)}(\gamma_m^d\cdot z_{m,1})=a_{m,d-1}(t)\cdot\xi^{(1)}(\gamma_m^d\cdot z_{m,1})$ for all $d=0$, $\dots$, $D$, where $a_{m,-1}(t):=\id$.
\item $\xi_{m,E}(z)=a_{m,D}(t)\cdot\xi(z)$ for all 
\[z\in\bigcup_{d=D+1}^\infty\bigcup_{m=1}^2\gamma_m^d\cdot\{z_{m,1},\dots,z_{m,H_m}\}\cup\{r_1,r_2\}.\]
\item $\xi_{m,E}(z_{3-m,0})=a_{m,D}(t)a_{3-m,D}(t)^{-1}\cdot\xi(z_{3-m,0})$.
\item $\xi_{m,E}^{(1)}(z_{3-m,1})=a_{m,D}(t)a_{3-m,D}(t)^{-1}\cdot\xi^{(1)}(z_{3-m,1})$.
\end{enumerate}
\end{lem}

\begin{proof}
In the proof of Lemma \ref{lem:computation1}, we stopped the iterative procedure after $D(H_1+H_2)$ iterations. To prove this lemma, apply the iterative procedure $(D+1)(H_1+H_2)$ times, and use the observation that since $\xi$ is $\rho$-equivariant, we have 
\[\rho(\gamma_m)^{d}a_m(t)\rho(\gamma_m)^{-d}=a_{\xi(\mathbf t_{m,dH_m+1})}(t)c_{\xi(\mathbf p_{m,dH_m+1})}(t)\dots a_{\xi(\mathbf t_{m,(d+1)H_m})}(t)c_{\xi(\mathbf p_{m,(d+1)H_m})}(t)\]
for all $m=1$, $2$ and $d=0$, $\dots$, $D$. 
\end{proof}

Observe that $a_{m,d}(t)$ defined in Lemma \ref{lem:computation2} can be rewritten as 
\[a_{m,d}(t)=(a_m(t)\rho(\gamma_m))^{d}\rho(\gamma_m)^{-d}.\] 
We need to understand the limit as $d$ grows to infinity of the unipotent part of $a_{m,d}(t)$. This motivates the following lemma.

\begin{lem}\label{lem:matrix}
Let $X$ be an $n\times n$ diagonal matrix whose diagonal entries are positive and strictly decreasing down the diagonal, and let $U$ be an $n\times n$ unipotent upper triangular matrix, i.e.
\[X=\left(\begin{matrix}
\lambda_1 &0 &\dots&0&0 \\ 
0 & \lambda_2 &\dots &0&0 \\ 
\vdots&\vdots&\ddots&\vdots&\vdots\\
0&0&\dots& \lambda_{n-1}&0\\
0&0&\dots& 0&\lambda_n
\end{matrix}\right),\,\,\,\,\,\,\,\, 
U=\left(\begin{matrix}
1 & u_{1,2} & \cdots &u_{1,n-1}& u_{1,n} \\ 
0 & 1 & \dots &u_{2,n-1}&u_{2,n}\\
\vdots & \vdots &\ddots&\vdots & \vdots\\ 
 0& 0 &\dots& 1 & u_{n-1,n} \\ 
 0& 0 & \dots&0 & 1
\end{matrix} \right)\] 
where $0<\lambda_1<\cdots <\lambda_n$.
Then the sequence 
\[\{V_d := (UX)^dX^{-d} \}_{d=1}^\infty\] 
converges to a unipotent upper-triangular matrix whose $(i,j)$-entry for all $j>i$ is 
\[\sum_{k=0}^{j-i} \sum_{i=t_0<t_1< \cdots <t_{k}=j} \left(\prod_{s=1}^{k-1}\frac{u_{t_s,t_{s+1}}\left(\frac{\lambda_{t_{s}}}{\lambda_{t_k}}\right)}{1-\frac{\lambda_{t_{s}}}{\lambda_{t_k}}}\right) \cdot \frac{u_{t_0,t_1}}{1-\frac{\lambda_{t_0}}{\lambda_{t_k}}}.\]
\end{lem}
\begin{proof}
First, by induction on $d$, one can prove that $V_d$ is a unipotent upper triangular matrix whose $(i,j)$-entry $V_{i,j,d}$ is given by
\[\begin{array}{l}
\displaystyle V_{i,j,d} = \sum_{k=1}^{j-i} \sum_{i=t_0< \cdots <t_{k}=j} \left( u_{t_0,t_1} u_{t_1,t_2} \cdots u_{t_{k-1},t_k}\cdot\right.\\
\left.\hspace{4cm}\displaystyle \sum_{k\leq l_1+\cdots + l_k \leq d} \left(\frac{\lambda_{t_0}}{\lambda_{t_k}}\right)^{l_1-1} \left(\frac{\lambda_{t_1}}{\lambda_{t_k}}\right)^{l_2} \cdots \left(\frac{\lambda_{t_{k-1}}}{\lambda_{t_k}}\right)^{l_k}\right).
 \end{array}\]
for all $j>i$. Here, the second summation is over positive integers $t_1,\cdots, t_{k-1}$ such that $i=t_0< t_1<\cdots < t_{k-1}<t_k=j$, and the last summation is over positive integers $l_1,\cdots, l_k$ such that $k \leq l_1+\cdots + l_k \leq d$.

Since $0<\lambda_1<\dots<\lambda_n$, it is easy to see that if $i=t_0<t_1<\dots<t_k=j$, then
\[\lim_{d\to\infty}\left(\sum_{k\leq l_1+\cdots + l_k \leq d} \left(\frac{\lambda_{t_0}}{\lambda_{t_k}}\right)^{l_1-1} \left(\frac{\lambda_{t_1}}{\lambda_{t_k}}\right)^{l_2} \cdots \left(\frac{\lambda_{t_{k-1}}}{\lambda_{t_k}}\right)^{l_k}\right)=\prod_{s=1}^{k-1}\frac{\frac{\lambda_{t_s}}{\lambda_{t_k}}}{1-\frac{\lambda_{t_s}}{\lambda_{t_k}}}\cdot \frac{1}{1-\frac{\lambda_{t_0}}{\lambda_{t_k}}}.\] 
Thus, 
\begin{equation}\label{eqn:Vijd}
\lim_{d \rightarrow\infty} V_{i,j,d} = \sum_{k=0}^{j-i} \sum_{i=t_0<t_1< \cdots <t_{k}=j} \left(\prod_{s=1}^{k-1}\frac{u_{t_s,t_{s+1}}\left(\frac{\lambda_{t_{s}}}{\lambda_{t_k}}\right)}{1-\frac{\lambda_{t_{s}}}{\lambda_{t_k}}}\right) \cdot \frac{u_{t_0,t_1}}{1-\frac{\lambda_{t_0}}{\lambda_{t_k}}}
\end{equation}
for all $j>i$. Since $V_d$ is upper triangular and unipotent for all $d$, the same is true for $\lim_{d\to\infty}V_d$.
\end{proof}

By making an appropriate change of basis, we see that Lemma \ref{lem:matrix} also holds if $U$ is a unipotent lower triangular matrix and $X$ is a diagonal matrix whose entries are decreasing down the diagonal. 

\begin{proof}[Proof of Theorem \ref{thm:unipotent closed edge}]
Let $\{f_1,\dots,f_n\}$ be the basis of $V$ such that $f_i$ lies in $\xi(r_1)^{(i)}\cap\xi(r_2)^{(n-i+1)}$ for all $i=1$, $\dots$, $n$.

Proof of (1). For $m=1$, $2$, let $a_m(t)$ be the projective transformation in $\PGL(V)$ defined by (\ref{eqn:closed edge subset transform}). Then we can write 
\[a_m(t)=u_m(t)h_m(t)\]
where $u_m(t)$ is a unipotent projective transformation that fixes $\xi(p_m)=\xi_{m,2H_m}(p_m)$ and $h_m(t)$ is a projective transformation that fixes $\xi(r_1)$ and $\xi(r_2)$. By Lemma \ref{lem:computation1}(2), $h_1(t)=h_2(t)$ is explicitly represented in the basis $\{f_1,\dots,f_n\}$ by the unique determinant $1$, diagonal matrix whose diagonal entries $\lambda_1$, $\dots$, $\lambda_n$ down the diagonal satisfy, and in fact are determined by, the equation (\ref{eqn:closed_leaf_par1}) and (\ref{eqn:closed_leaf_par2}) for all $k=1$, $\dots$, $n-1$.

Observe that $h_m(t)$ commutes with $\rho(\gamma_m)$ since they have the same attracting and repelling fixed flags. Also, since $t\in I_{[\xi],\mu}$, we see from the description of the closed leaf inequalities in Section \ref{sec:Bonahon-Dreyer} that the matrix representative of $\rho(\gamma_m)h_m(t)$ in the basis $\{f_1,\dots,f_n\}$ has diagonal entries that are increasing down the diagonal if $p_m=r_1$ and decreasing down the diagonal if $p_m=r_2$. At the same time, the matrix representing $u_m(t)$ in the basis $\{f_1,\dots,f_n\}$ is upper triangular when $p_m=r_1$ and lower triangular when $p_m=r_2$. Thus, by Lemma \ref{lem:matrix}, the limit
\[u_{m,\infty}(t):=\lim_{d\to\infty}\big(u_m(t)h_m(t)\rho(\gamma_m)\big)^d\big(h_m(t)\rho(\gamma_m)\big)^{-d}\]
exists, and is a unipotent projective transformation that fixes $\xi(p_m)$.

We previously observed that for the purposes of Theorem \ref{thm:unipotent closed edge}, we may use $\prod_{h=1}^{E}\left(\phi^\mu_{o_h}\right)_t$ in place of $\left(\phi^\mu_{M_j}\right)_t$ (see Remark \ref{rem:simplify}). Also, recall that we denote by $\xi_{m,h}$ the representative of 
\[\prod_{r=1}^h\left(\phi^\mu_{o_r}\right)_t[\xi]\] 
such that $\xi_{m,h}(p_m)=\xi(p_m)$, $\xi_{m,h}(z_{m,0})=\xi(z_{m,0})$, and $\xi_{m,h}^{(1)}(z_{m,1})=\xi^{(1)}(z_{m,1})$ (see Notation \ref{not:xi}). It is thus sufficient to show that
\begin{eqnarray*}
\lim_{D\to\infty}\xi_{1,E}(q_1)&=&u_{1,\infty}(t)\cdot\xi(q_1),\\
\lim_{D\to\infty}\xi_{1,E}(z_{2,0})&=&u_{1,\infty}(t)u_{2,\infty}(t)^{-1}\cdot\xi(z_{2,0}),\\
\lim_{D\to\infty}\xi_{1,E}^{(1)}(z_{2,1})&=&u_{1,\infty}(t)u_{2,\infty}(t)^{-1}\cdot\xi^{(1)}(z_{2,1}).
\end{eqnarray*}

By Lemma \ref{lem:computation2}(3), 
\begin{eqnarray}\label{eqn:unipotent compute}
\lim_{D\to\infty}\xi_{1,E}(q_1)&=&\lim_{D\to\infty}a_{1,D}(t)\cdot\xi(q_1)\nonumber\\
&=&\lim_{D\to\infty}\big(a_1(t)\rho(\gamma_1)\big)^{D}\rho(\gamma_1)^{-D}\cdot\xi(q_1)\nonumber\\
&=&\lim_{D\to\infty}\big(u_1(t)h_1(t)\rho(\gamma_1)\big)^{D}\big(h_1(t)\rho(\gamma_1)\big)^{-D}\cdot\xi(q_1)\\
&=&u_{1,\infty}(t)\cdot\xi(q_1).\nonumber
\end{eqnarray}
Since $h_1(t)=h_2(t)$ commute with $\rho(\gamma_1)$ and $\rho(\gamma_2)$, Lemma \ref{lem:computation2}(4) implies
\begin{eqnarray*}
&&\lim_{D\to\infty}\xi_{1,E}(z_{2,0})\\
&=&\lim_{D\to\infty}a_{1,D}(t)a_{2,D}(t)^{-1}\cdot\xi(z_{2,0})\\
&=&\lim_{D\to\infty}\big(u_1(t)h_1(t)\rho(\gamma_1)\big)^{D+1}\rho(\gamma_1)^{-D-1}\rho(\gamma_2)^{D+1}\big(u_2(t)h_2(t)\rho(\gamma_2)\big)^{-D-1}\cdot\xi(z_{2,0})\\
&=&u_{1,\infty}(t)u_{2,\infty}(t)^{-1}\cdot\xi(z_{2,0}).
\end{eqnarray*}
Similarly, using Lemma \ref{lem:computation2}(5) in place of Lemma \ref{lem:computation2}(4),
\[\lim_{D\to\infty}\xi_{1,E}^{(1)}(z_{2,1})=u_{1,\infty}(t)u_{2,\infty}(t)^{-1}\cdot\xi^{(1)}(z_{2,1}).\]

(2) Using Lemma \ref{lem:computation2}(3), the same computation as the equation (\ref{eqn:unipotent compute}) proves that
\[\lim_{D\to\infty} \xi_{m,E}(q_m)=u_{m,\infty}(t)\cdot\xi(q_m).\]
At the same time, Lemma \ref{lem:computation2}(2) tells us that for sufficiently large integers $D$ and any integer $d\geq 1$, 
\begin{eqnarray*}
&&\xi_{m,E}(\gamma_m^d\cdot z_{m,0})\\
&=&a_{m,d-1}(t)\cdot\xi(\gamma_m^d\cdot z_{m,0})\\
&=&\big(u_m(t)h_m(t)\rho(\gamma_m)\big)^d\rho(\gamma_m)^{-d}\cdot\xi(\gamma_m^d\cdot z_{m,0})\\
&=&\big(u_m(t)h_m(t)\rho(\gamma_m)\big)^d\big(h_m(t)\rho(\gamma_m)\big)^{-d}\big(h_m(t)\rho(\gamma_m)\big)^d\cdot\xi(z_{m,0}).
\end{eqnarray*}
Since the matrix representative of $\rho(\gamma_m)h_m(t)$ in the basis $\{f_1,\dots,f_n\}$ has diagonal entries that are increasing down the diagonal if $p_m=x_1$ and decreasing down the diagonal if $p_m=x_2$, we see that $\lim_{d\to\infty}\big(h_m(t)\rho(\gamma_m)\big)^d\cdot\xi(z_{m,0})=\xi(q_m)$. Hence,
\[\lim_{d\to\infty}\lim_{D\to\infty}\xi_{m,E}(\gamma_m^d\cdot z_{m,0})=u_{m,\infty}(t)\xi(q_m).\]
This proves (2).
\end{proof}

\subsection{Well-definedness of $\left(\phi^\mu_{\Qmc,\Theta}\right)_t$.}\label{sec:semielementary}
Using Theorem \ref{thm:unipotent closed edge}, we are now ready to prove that $\left(\phi^\mu_{\Qmc,\Theta}\right)_t$ defined in Section \ref{sec:parallel flows} is well-defined. 

We consider the subspace $W_{\Qmc,\Theta}$ of $W = T_\xi\Hit_V(S)$ given by 
\[W_{\Qmc,\Theta}:=\{\mu\in W:\mu^{\mathbf k}_{\mathbf r}=0\text{ for all }\mathbf r\in\widehat{\Pmc},\,\mathbf k\in\Amc\},\] 
where $\Amc$ is the set of pairs of positive integers that sum to $n$. Also, let $\Pi:W\to W_{\Qmc,\Theta}$ be the projection such that 
\begin{itemize}
\item $\Pi(\mu)^{\mathbf k}_{\mathbf r}=\mu^{\mathbf k}_{\mathbf r}$ for all $\mathbf r\in\widehat{\Qmc}$ and $\mathbf k\in\Amc$, 
\item $\Pi(\mu)^{\mathbf k}_{\mathbf r}=0$ for all $\mathbf r\in\widehat{\Pmc}$ and $\mathbf k\in\Amc$, 
\item $\Pi(\mu)^{\mathbf i}_{\mathbf x}=\mu^{\mathbf i}_{\mathbf x}$ for all $\mathbf x\in\widehat{\Theta}$ and $\mathbf i\in\Bmc$.
\end{itemize}

Since the closed leaf equalities only involve the invariants associated to $\Qmc\cup\Theta$, $\Pi$ is indeed well-defined. Observe that, if defined, $\left(\phi^\mu_{\Qmc,\Theta}\right)_t=\left(\phi^{\Pi(\mu)}_{\Qmc,\Theta}\right)_t$ because $I_{[\xi],\mu}=I_{[\xi],\Pi(\mu)}$, and the sequence $(M_1,M_2,\dots)$ defined in Section \ref{sec:closed_edge_subset} is an exhaustion of $\widetilde{\Qmc}\cup\widetilde{\Theta}$. The next proposition is the analog of Proposition \ref {prop:shearing closed edge}, and relates the flow $\left(\phi^\mu_{\Qmc,\Theta}\right)_t$ with the parametrization $\Omega$ of $\Hit_V(S)$ given in Theorem \ref{thm:reparametrization}. In particular, it implies that $\left(\phi^\mu_{\Qmc,\Theta}\right)_t$ is well-defined, and does not depend on the choice of base point $p_0$ in $\Dmc$, even though the semi-elementary flows $\left(\phi^\mu_{M_j}\right)_t$ obviously do. 

\begin{prop}\label{prop:general}
Let $\xi$ be a representative of $[\xi]$ in $\Hit_V(S)$, let $\mu$ be a vector in $W$, let $t$ be a real number in $I_{[\xi],\mu}$, and let 
\[[\xi_0]:=\Omega^{-1}\left(\Omega[\xi]+t\Pi(\mu)\right)\in\Hit_V(S).\] 
Pick any ideal triangle $\{x_1,x_2,x_3\}$ in $\widetilde{\Theta}$, and choose representatives $\xi_j$ (resp. $\xi_0$) of $\left(\phi^\mu_{M_j}\right)_t[\xi]$ (resp. $[\xi_0]$) such that 
\[\xi_j(x_1)=\xi(x_1),\,\,\,\,\xi_j(x_2)=\xi(x_2)\,\,\,\,\text{ and }\,\,\,\,\xi_j^{(1)}(x_3)=\xi^{(1)}(x_3)\]
for all non-negative integers $j$. Then
\[\lim_{j\to\infty}\xi_j=\xi_0.\]
\end{prop}

\begin{proof}
By Proposition \ref{lem:technical2}, it is sufficient to prove that $\xi_j$ converges to $\xi_0$ on the vertices of $\widetilde{\Tmc}$. Since
\[\lim_{j\to\infty}T^{\mathbf i}(\xi_j(x_1),\xi_j(x_2),\xi_j(x_3))=T^{\mathbf i}(\xi_0(x_1),\xi_0(x_2),\xi_0(x_3))\] 
for all ordered triples of positive integers $\mathbf i$ that sum to $n$, Proposition \ref{prop:Fock-Goncharov parametrization}, implies that $\lim_{j\to\infty}\xi_j(x_3)=\xi_0(x_3)$ for all $j$.

Let $y$ be any vertex of $\widetilde{\Tmc}$ that is not $x_1$, $x_2$ or $x_3$. We again use the combinatorial description of pairs of distinct points in $\widetilde{\Tmc}$ developed in Section~\ref{sec:combinatorial description} to prove that $\xi_j(y)$ converges to $\xi_0(y)$. After possibly relabelling the vertices of $\{x_1,x_2,x_3\}$, we may assume without loss of generality that $\{x_2,x_3\}$ is the minimal element of $\Emc_{(x_1,y)}$. Recall the decomposition 
\[\Emc_{(x_1,y)}=\bigcup_{s=1}^k\Emc_{(x_1,y),s}\cup\bigcup_{s=0}^k\Emc_{(x_1,y),s,s+1}=\bigcup_{s=1}^k\Emc_s\cup\bigcup_{s=0}^k\Emc_{s,s+1}.\]
We just observed that, $\lim_{j\to\infty}\xi_j$ and $\xi_0$ agree on $\{x_1,x_2,x_3\}$, which is the backward end of $\Emc_{0,1}$. Also, observe that all the ends of $\Emc_s$ and $\Emc_{s,s+1}$ are ideal triangles, except possibly the forward end of $\Emc_{k}$, in which case $\Emc_{k,k+1}$ is empty. Thus, by Remark \ref{rem:transitive}, to prove that $\lim_{j\to\infty}\xi_j(y)=\xi_0(y)$, it is sufficient to prove that for $\Emc=\Emc_{s,s+1}$ or $\Emc=\Emc_s$, the following holds:
\begin{enumerate}
\item[($\dagger$)] If $\lim_{j\to\infty}\xi_j$ and $\xi_0$ agree on the backward end $\Delta_-$ of $\Emc$, then they must agree on its forward end $\Delta_+$ as well. 
\end{enumerate}

First, we prove ($\dagger$) when $\Emc=\Emc_{s,s+1}$. Recall that the number of vertices of the edges in $\Emc_{s,s+1}$ is finite. We can thus apply Proposition \ref{prop:eruption projective invariants}, Proposition \ref{prop:shearing projective invariants} and Proposition \ref{prop:Fock-Goncharov parametrization} to these vertices to deduce that the sequence of sextuples $\big\{(\xi_j(\Delta_-),\xi_j(\Delta_+)\big\}_{j=1}^\infty$ converges to $(\xi_0(\Delta_-),\xi_0(\Delta^+))$ up to projective transformations. This immediately implies ($\dagger$).

Next, we prove ($\dagger$) when $\Emc=\Emc_s$. We orient the unique non-isolated edge $e_s$ in $\Emc_s$ such that $x_1$ and $y$ lie to the left and right of $e_s$ respectively. Let $r_1$ and $r_2$ be respectively the backward and forward endpoints of $e_s$ equipped with its orientation, and let $\{f_1,\dots,f_n\}$ be a basis of $V$ such that $f_i$ lies in $\xi^{(i)}(r_1)\cap\xi^{(n-i+1)}(r_2)$. Let $J=\{T_1,T_2\}$ be a bridge across $e_s$, $T_1$ and $T_2$ lie to the left and right of $\mathbf r:=(r_1,r_2)$ respectively. The same argument that we used above proves ($\dagger$) when $e_s$ is the maximum of $\Emc_s$, in which case $s=k$ and $\Emc_{k,k+1}$ is empty, and also that 
\begin{itemize}
\item $\lim_{j\to\infty}\xi_j(\Delta_-)=\xi_0(\Delta_-)$ if and only if $\lim_{j\to\infty}\xi_j(T_1)=\xi_0(T_1)$,
\item $\lim_{j\to\infty}\xi_j(\Delta_+)=\xi_0(\Delta_+)$ if and only if $\lim_{j\to\infty}\xi_j(T_2)=\xi_0(T_2)$.
\end{itemize}
Thus, we may assume that $e_s$ is not the maximum of $\Emc_s$, and it is sufficient to prove that if $\lim_{j\to\infty}\xi_j(T_1)=\xi_0(T_1)$, then $\lim_{j\to\infty}\xi_j(T_2)=\xi_0(T_2)$. To prove this, we recall the following notation.

For $m=1$, $2$, let $p_m$ (resp. $q_m$) be the vertex of the edge $e_s$ that is (resp. is not) a vertex of $T_m$, and let $\gamma_m$ be the primitive group element in $\Gamma$ with $p_m$ and $q_m$ as its repelling and attracting fixed points respectively. Recall that the collection of ideal triangles $\{\dots,T_{m,-1},T_{m,0},T_{m,1},T_{m,2},\dots\}$, specified in Notation \ref{not:closed edge} all share $p_m$ as a common vertex, $T_{m,1}=T_m$, and $T_{m,h}$ shares a common edge $e_{m,h}$ with $T_{m,h+1}$. Also, $H_m$ is the positive integer such that the closed edge subset $\widetilde{\Theta}(J,T_m)$ (Definition \ref{def:closed edge subset}) is given by $\widetilde{\Theta}(J,T_m)=\{T_{m,1},T_{m,2},\dots,T_{m,H_m}\}$ (see Figure \ref{fig:closededge}). Let $z_{m,h}$ denote the vertex of $e_{m,h}$ that is not $p_m$, set $\mathbf p_{m,h}:=(p_m,z_{m,h})$, and set
\[\mathbf t_{m,h}:=\left\{\begin{array}{ll}
(p_m,z_{m,h-1},z_{m,h})&\text{if }p_m=r_m;\\
(p_m,z_{m,h},z_{m,h-1})&\text{if }p_m=r_{3-m}.
\end{array}
\right.\] 

Let $\hat{\xi}_j$ be the representative of $\left(\phi^\mu_{M_j}\right)_t[\xi]$ such that $\hat{\xi}_j(p_1)=\xi(p_1)$, $\hat{\xi}_j(z_{1,0})=\xi( z_{1,0})$, and $\hat{\xi}_j^{(1)}(z_{1,1})=\xi^{(1)}(z_{1,1})$. By Theorem \ref{thm:unipotent closed edge}(1),
\begin{eqnarray*}
\lim_{j\to\infty}\hat{\xi}_j(q_1)&=&u_{1,\infty}(t)\cdot\xi(q_1),\\
\lim_{j\to\infty}\hat{\xi}_j(z_{2,0})&=&u_{1,\infty}(t)u_{2,\infty}(t)^{-1}\cdot\xi( z_{2,0}),\\
\lim_{j\to\infty}\hat{\xi}_j^{(1)}(z_{2,1})&=&u_{1,\infty}(t)u_{2,\infty}(t)^{-1}\cdot\xi^{(1)}(z_{2,1}).
\end{eqnarray*}
for some unipotent elements $u_{1,\infty}(t)$ and $u_{2,\infty}(t)$ in $\PGL(V)$ that fix $\xi(p_1)$ and $\xi(p_2)$ respectively. In particular, Lemma \ref{lem:new invariant 1} implies that for all pairs of positive integers $\mathbf k:=(k_1,k_2)$ that sum to $n$, $\lim_{j\to\infty}\alpha^{\mathbf k}_{\mathbf r}[\xi_j]=\alpha^{\mathbf k}_{\mathbf r}[\xi]=\alpha^{\mathbf k}_{\mathbf r}[\xi_0]$.

By Proposition \ref{prop:Fock-Goncharov parametrization}, Proposition \ref{prop:eruption projective invariants} and Proposition \ref{prop:shearing projective invariants}, we see that for $m=1$, $2$, there is a projective transformation $h_m$ in $\PGL(V)$ such that $h_m\cdot\lim_{j\to\infty}\xi_j(p_{m})=\xi_0(p_{m})$ and
\[h_m\cdot \lim_{j\to\infty}\xi_j(\gamma_m^d\cdot z_{m,0})=\xi_0(\gamma_m^d\cdot z_{m,0})\] 
 and all integers $d$. By Theorem \ref{thm:unipotent closed edge}(2), we have
\begin{eqnarray*}
h_m\cdot\lim_{j\to\infty}\xi_j(q_m)&=&\lim_{d\to\infty}\lim_{j\to\infty}\xi_j(\gamma_m^d\cdot z_{m,0})\\
&=&\lim_{d\to\infty}\xi_0(\gamma_m^d\cdot z_{m,0})\\
&=&\xi_0(q_m).
\end{eqnarray*}
In particular, $h_m\cdot\lim_{j\to\infty}\xi_j(r_m)=\xi_0(r_m)$ for both $m=1$, $2$. 

The assumption that $\lim_{j\to\infty}\xi_j(T_1)=\xi_0(T_1)$, together with Remark \ref{rem:transitive}, implies that $h_1=\id$, i.e. $\lim_{j\to\infty}\xi_j(r_m)=\xi_0(r_m)$ for $m=1$, $2$ and $\lim_{j\to\infty}\xi_j(\gamma_1^d\cdot z_{1,0})=\xi_0(\gamma_1^d\cdot z_{1,0})$ for all non-negative integers $d$. In particular, $h_2$ fixes both $\lim_{j\to\infty}\xi_j(p_2)=\xi_0(p_2)$ and $\lim_{j\to\infty}\xi_j(q_2)=\xi_0(q_2)$. 

Now, since $\lim_{j\to\infty}\alpha^{\mathbf k}_{\mathbf r}[\xi_j]=\alpha^{\mathbf k}_{\mathbf r}[\xi_0]$ for all pairs of positive integers $\mathbf k:=(k_1,k_2)$ that sum to $n$, Lemma \ref{lem:new invariant 0} implies that there is a unipotent projective transformation $u$ that fixes $\xi_0(p_2)$, sends $\lim_{j\to\infty}\xi_j(z_{2,0})$ to $\xi_0(z_{2,0})$, and sends $\lim_{j\to\infty}\xi_j^{(1)}(z_{2,1})$ to $\xi_0^{(1)}(z_{2,1})$. It follows from Remark \ref{rem:transitive} that $u=h_2$, so $h_2=u=\id$. Thus, $\lim_{j\to\infty}\xi_j(T_2)=\xi_0(T_2)$.
\end{proof} 

\subsection{Closed-form formulas for the $(\Tmc,\Jmc)$-parallel flows}\label{sec:closedform}
Fix an ideal triangulation $\Tmc$ of $S$ and a compatible bridge system $\Jmc$. Let $\Theta$ denote the set of ideal triangles of $\Tmc$. Also, fix a vector $\mu$ in $W_\Tmc$ (see Notation \ref{not:polytope notation}). Recall that $\phi_t^\mu$ denotes the $(\Tmc,\Jmc)$-parallel flow associated to the $\mu$. Our proof of Theorem \ref{thm:main theorem} allows us to describe $\phi_t^\mu$ explicitly, in the sense of the following pair of propositions.

Recall that $\Amc$ is the set of ordered pairs of positive integers that sum to $n$, $\Bmc$ is the set of ordered triples of positive integers that sum to $n$. Also, recall that for all triples of positive integers that sum to $n$ and all cyclically ordered triples $\mathbf x:=(x_1,x_2,x_3)$ in $\partial\Gamma$, $a^{\mathbf i}_{\xi(\mathbf x)}$ is the projective transformation defined by (\ref{eqn:a computation}) in Section \ref{sec:elementaryeruption}. Similarly, for all pairs of positive integers that sum to $n$ and all distinct pairs $\mathbf r:=(r_1,r_2)$ in $\partial\Gamma$, $c^{\mathbf k}_{\xi(\mathbf r)}$ is the projective transformation defined by (\ref{eqn:c computation}) in Section \ref{sec:elementaryshearing} respectively. 

\begin{prop}\label{prop:secondpaper1}
Let $\mathbf x:=(x_1,x_2,x_3)$ and $\mathbf y:=(y_1,y_2,y_3)$ be cyclically ordered triples of points in $\partial\Gamma$ such that $\{x_1,x_2,x_3\}$ and $\{y_1,y_2,y_3\}$ are ideal triangles in $\widetilde{\Theta}$. Let $\xi$ be a Frenet curve whose projective class lies in $\Hit_V(S)$, let $t\in I_{[\xi],\mu}$, and let $\xi_t$ be the representative of $\phi^\mu_t[\xi]$ such that $\left(\xi(x_1),\xi(x_2),\xi^{(1)}(x_3) \right)=\left(\xi_t(x_1),\xi_t(x_2),\xi_t^{(1)}(x_3) \right)$.
\begin{enumerate}
\item If $x_1=y_2$ and $x_2=y_1$, let $\mathbf p:=(y_1,y_2)$. Then
\[\left(\xi_t(y_1),\xi_t(y_2),\xi_t^{(1)}(y_3)\right)=\prod_{\mathbf k\in\Amc}c^{\mathbf k}_{\xi(\mathbf p)}\left(-\mu^{\mathbf k}_{\mathbf p}\cdot t\right)\cdot\left(\xi(y_1),\xi(y_2),\xi^{(1)}(y_3)\right).\]
\item If $x_2=y_1$, $x_3=y_2$ and $x_1=y_3$, then 
\[\left(\xi_t(y_1),\xi_t(y_2),\xi_t^{(1)}(y_3)\right)=\prod_{\mathbf i\in\Bmc}a^{\mathbf i}_{\xi(\mathbf y)}\left(\mu^{\mathbf i}_{\mathbf y}\cdot t\right)\cdot\left(\xi(y_1),\xi(y_2),\xi^{(1)}(y_3)\right)\]
where the product is taken in decreasing order of $i_2$. This makes sense because $a^{\mathbf i}_{\xi(\mathbf y)}$ and $a^{\mathbf j}_{\xi(\mathbf y)}$ commute if $i_2=j_2$, $\mathbf i:=(i_1,i_2,i_3)$ and $\mathbf j:=(j_1,j_2,j_3)$.
\end{enumerate}
Furthermore, if $\{f_1,\dots,f_n\}$ is a basis for $V$ such that $f_i$ lies in $\xi^{(i)}(x_1)\cap\xi^{(n-i+1)}(x_2)$ for all $i=1$, $\dots$, $n-1$, then
\[\prod_{\mathbf k\in\Amc}c^{\mathbf k}_{\xi(\mathbf p)}\left(\mu^{\mathbf k}_{\mathbf p}\cdot t\right)\,\,\,\text{ and }\,\,\,\prod_{\mathbf i\in\Bmc}a^{\mathbf i}_{\xi(\mathbf x)}\left(\mu^{\mathbf i}_{\mathbf x}\cdot t\right)\]
can be written as matrices whose entries have explicit algebraic formulas in terms of the coordinates of the vector $e^{t\mu}$.
\end{prop}

\begin{proof}
This follows from a computation using Lemma \ref{lem:other eruption} and Lemma \ref{lem:other shearing}.
\end{proof}

The second proposition, Proposition \ref{prop:secondpaper2}, is the analogous statement for bridges. 

\begin{prop}\label{prop:secondpaper2}
Let $\{r_1,r_2\}$ be an edge in $\widetilde{\Tmc}$ and let $J=\{T_1,T_2\}$ be any bridge across $\{r_1,r_2\}$ such that $T_1$ and $T_2$ lie to the left and right of $\mathbf r:=(r_1,r_2)$ respectively. For $m=1$, $2$, let $p_m$ (resp. $q_m$) be the vertex of the edge $\{r_1,r_2\}$ that is (resp. is not) a vertex of $T_m$, and let $z_m$ and $w_m$ be the other two vertices of $T_m$ such that either $(p_m,z_m,w_m,q_m)$ or $(q_m,w_m,z_m,p_m)$ is cyclically ordered (see Figure \ref{fig:pq}). Let $\xi$ be a Frenet curve whose projective class lies in $\Hit_V(S)$, let $t\in I_{[\xi],\mu}$, and let $\xi_t$ be the representative of $\phi^\mu_t[\xi]$ such that $\left(\xi_t(p_1),\xi_t(z_1),\xi_t^{(1)}(w_1)\right)=\left(\xi(p_1),\xi(z_1),\xi^{(1)}(w_1)\right)$. Then
\[\left(\xi_t(p_2),\xi_t(z_2),\xi_t^{(1)}(w_2)\right)=u_{1,\infty}(t)\cdot c_{\xi(\mathbf r)}(t)^{-1}\cdot u_{2,\infty}(t)^{-1}\cdot\left(\xi(p_2),\xi(z_2),\xi^{(1)}(w_2)\right),\]
where $u_{m,\infty}(t)$ is the unipotent projective transformation in $\PGL(V)$ defined by (\ref{eqn:unipotent}) in Theorem \ref{thm:unipotent closed edge} and $c_{\xi(\mathbf r)}(t)$ is the projective transformation defined by (\ref{eqn:short}) in Lemma \ref{lem:o2}. Furthermore, if $\{f_1,\dots,f_n\}$ is a basis for $V$ such that $f_i$ lies in $\xi^{(i)}(r_1)\cap\xi^{(n-i+1)}(r_2)$ for all $i=1$, $\dots$, $n-1$, then $u_{1,\infty}(t)\cdot c_{\xi(\mathbf r_2)}(t)\cdot u_{2,\infty}(t)^{-1}$ can be written as matrices whose entries have explicit algebraic formulas in terms of the coordinates of the vector $e^{t\mu}$.
\end{prop}

\begin{proof}[Proof of Proposition \ref{prop:secondpaper2}]
Let $\overline{\xi}_t$ be the representative of $\left(\phi^\mu_{\Qmc,\Theta}\right)_t[\xi]$ such that $\overline{\xi}_t(p_1)=\xi(p_1)$, $\overline{\xi}_t(z_1)=\xi(z_1)$ and $\overline{\xi}_t^{(1)}(w_1)=\xi^{(1)}(w_1)$. Then observe that 
\begin{align*}
\xi_t(p_2)&=c_{\overline{\xi}_t(\mathbf r)}(t)^{-1}\cdot \overline{\xi}_t(p_2),\\ 
\xi_t(z_2)&=c_{\overline{\xi}_t(\mathbf r)}(t)^{-1}\cdot \overline{\xi}_t(z_2),\text{ and }\\
\xi_t^{(1)}(w_2)&=c_{\overline{\xi}_t(\mathbf r)}(t)^{-1}\cdot \overline{\xi}_t^{(1)}(w_2).
\end{align*}
By Theorem \ref{thm:unipotent closed edge}(1), $\overline{\xi}_t(r_1)=u_{1,\infty}(t)\cdot \xi(r_1)$ and $\overline{\xi}_t(r_2)=u_{1,\infty}(t)\cdot \xi(r_2)$, so 
\[c_{\overline{\xi}_t(\mathbf r)}(t)^{-1}=u_{1,\infty}(t)c_{\xi(\mathbf r)}(t)^{-1}u_{1,\infty}(t)^{-1}.\] 
At the same time, Theorem \ref{thm:unipotent closed edge} implies that
\begin{align*}
\overline{\xi}_t(p_2)&=u_{1,\infty}(t)\cdot \xi(p_2)=u_{1,\infty}(t)u_{2,\infty}(t)^{-1}\cdot \xi(p_2),\\ 
\overline{\xi}_t(w_2)&=u_{1,\infty}(t)u_{2,\infty}(t)^{-1}\cdot\xi(w_2),\text{ and }\\
\overline{\xi}_t^{(1)}(z_2)&=u_{1,\infty}(t)u_{2,\infty}(t)^{-1}\cdot\xi^{(1)}(z_2).
\end{align*}
This proves that
\[\left(\xi_t(p_2),\xi_t(z_2),\xi_t^{(1)}(w_2)\right)=u_{1,\infty}(t)\cdot c_{\xi(\mathbf r)}(t)^{-1}\cdot u_{2,\infty}(t)^{-1}\cdot\left(\xi(p_2),\xi(z_2),\xi^{(1)}(w_2)\right).\]

From the proof of Theorem \ref{thm:unipotent closed edge}, we see that (\ref{eqn:Vijd}) in the proof of Lemma \ref{lem:matrix} gives an explicit algebraic formula for the entries of $u_{m,\infty}(t)$, and hence of 
\begin{equation}\label{eqn:explicit3}
u_{1,\infty}(t)\cdot c_{\xi(\mathbf r)}(t)^{-1}\cdot u_{2,\infty}(t)^{-1},
\end{equation}
in terms of the coordinates of $\Omega([\xi])$ and $e^{t\mu}$. 
\end{proof}

\section{Pants decompositions, flows and Darboux coordinates}\label{sec:Goldman}
In this section we consider a particular ideal triangulation and a particular compatible bridge system, that are subordinate to a pants decomposition of $S$. Using this fixed data, we specify a family of \emph{special} $(\Tmc,\Jmc)$-parallel flows on $\Hit_V(S)$. This family consists of $(2g-2)(n^2-1)=\dim(\Hit_V(S))$ flows (that necessarily pairwise commute) whose tangent fields form a global frame of the tangent bundle $T\Hit_V(S)$. Using this, we construct a particular coordinate system on $\Hit_V(S)$ whose coordinate functions are explicitly given in terms of the parametrization $\Omega=\Omega_{\Tmc,\Jmc}$ of the Hitchin component. In the case when $n=2$, this coordinate system agrees with the Fenchel-Nielsen coordinates on Teichm\"uller space (up to scaling). 

In the upcoming companion paper \cite{SunZhang}, it is shown that at every point in $\Hit_V(S)$, the tangent fields to the special $(\Tmc,\Jmc)$-parallel flows give a Darboux basis of the tangent space to $\Hit_V(S)$ with respect to the Goldman symplectic structure. It follows that the coordinate system we define is a global Darboux coordinate system for $\Hit_V(S)$, and every coordinate function is the Hamiltonian function of a special $(\Tmc,\Jmc)$-parallel flow, see Section \ref{sec:Hamiltonian} for more details. 

There are four different types of special $(\Tmc,\Jmc)$-parallel flows: the families of $n-1$ {\em twist flows} and of $n-1$ {\em length flows} that are associated to simple closed curves in the pants decomposition, and the families of $\frac{(n-1)(n-2)}{2}$ {\em eruption flows} and of $\frac{(n-1)(n-2)}{2}$ {\em hexagon flows} that are associated to the pairs of pants. These are explained in more detail in Section \ref{sec:special}.

\subsection{Triangulation subordinate to a pants decomposition}\label{sec:special1}
We specify the ideal triangulation $\Tmc$ and the compatible bridge system $\Jmc$ subordinate to a pair of pants decompositions of $S$. 
Let us fix an auxiliary hyperbolic structure on $S$. Let $\Pmc$ be a family of $3g-3$ pairwise non-intersecting simple closed geodesics on $S$. They define a decomposition of $S$ into $2g-2$ pairs of pants. We denote by $\Pbbb$ the set of these $2g-2$ pairs of pants. For each pair of pants $P$ in $\Pbbb$, we choose peripheral group elements $\alpha_{P,1},\alpha_{P,2},\alpha_{P,3}$ in $\pi_1(P)$ such that $\alpha_{P,1}\alpha_{P,2}\alpha_{P,3}=\id$, and $P$ lies to the right of its boundary components, oriented according to $\alpha_{P,1}$, $\alpha_{P,2}$ and $\alpha_{P,3}$. 

By choosing base points, the inclusion of $P$ into $S$ induces an inclusion of $\pi_1(P)$ into $\Gamma$, so we can view $\alpha_{P,1}$, $\alpha_{P,2}$ and $\alpha_{P,3}$ as group elements in $\Gamma$. 
We can then define a $\Gamma$-invariant ideal triangulation of the universal cover of $S$ by 
\[\widetilde{\Tmc}:=\bigcup_{P\in\Pbbb}\bigcup_{m=1}^3\Gamma\cdot\big\{\{\alpha_{P,m}^-,\alpha_{P,m}^+\},\{\alpha_{P,m}^-,\alpha_{P,m+1}^-\}\big\},\]
where $\gamma^-,\gamma^+$ denote the repelling and attracting fixed points of the group element $\gamma$ in $\Gamma$ respectively, and arithmetic in the subscripts are done modulo $3$. Then $\Tmc:=\widetilde{\Tmc}/\Gamma$ is an ideal triangulation. The set of non-isolated (equivalently closed) edges of $\Tmc$ is exactly $\Pmc$. The set of ideal triangles $\widetilde{\Theta}$ of the triangulation $\widetilde{\Tmc}$ is given by 
\begin{equation}\label{eqn:pants triangles}
\widetilde{\Theta}=\bigcup_{P\in\Pbbb}\Gamma\cdot\big\{\{\alpha_{P,1}^-, \alpha_{P,2}^-, \alpha_{P,3}^-\},\;\{\alpha_{P,1}^-, \alpha_{P,3}^-, \alpha_{P,3}\cdot\alpha_{P,2}^-\}\big\}.
\end{equation}

We fix a bridge system $\Jmc$ compatible with $\Tmc$, such that both endpoints of all the bridges in $\widetilde{\Jmc}$ lie in the ideal triangles in $\widetilde{\Theta}$ of the form $\gamma\cdot\{\alpha_{P,1}^-, \alpha_{P,2}^-, \alpha_{P,3}^-\}$ for some $P$ in $\Pbbb$ and some $\gamma$ in $\Gamma$.

For the rest of this article, we assume that $(\Tmc,\Jmc)$ is of this type. Note that the ideal triangulation and the bridge system do not depend on the auxiliary hyperbolic metric on $S$.

\begin{figure}[ht]
\centering
\includegraphics[scale=0.5]{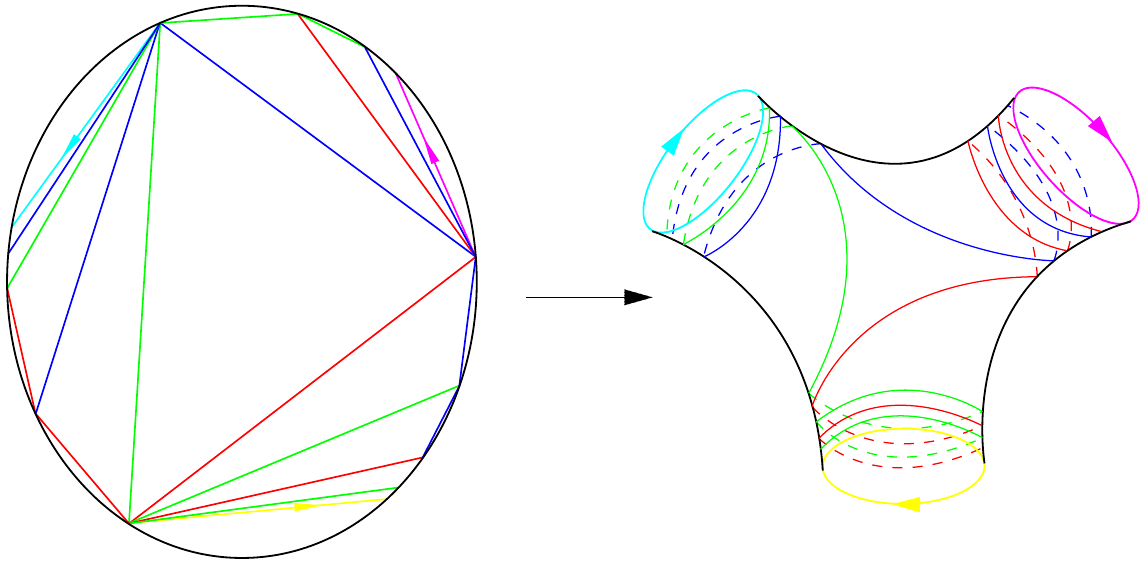}
\small
\caption{This figure shows the ideal triangulation of a single pair of pants and its lift to the universal covering.}\label{fig:triangulation}
\end{figure}

\subsection{The special flows}\label{sec:special}
Next, we describe the special flows associated to $(\Tmc,\Jmc)$.

Recall that we identify the tangent space to $\Hit_V(S)$ with 
\[W=W_\Tmc\subset\Rbbb^{(9g-9)(n-1)}\times\Rbbb^{(2g-2)(n-1)(n-2)},\] 
the linear subspace cut out by the $(3g-3)(n-1)$ closed leaf equalities (see Section \ref{sec:tangent}). We denote an arbitrary vector $\mu$ in $W$ by
\[\begin{array}{l}
\mu=\left(\left(\mu^{\mathbf k}_{\mathbf r}\right)_{\mathbf k\in\Amc;\mathbf z\in\widehat{\Tmc}},\left(\mu^{\mathbf i}_{\mathbf x}\right)_{\mathbf i\in\Bmc;\mathbf x\in\widehat{\Theta}}\right),
\end{array}\]
where $\Amc$ is the set of ordered pairs of positive integers that sum to $n$, $\Bmc$ is the set of ordered triples of positive integers that sum to $n$, $\widehat{\Tmc}$ is a set of oriented representatives for each $\Gamma$-orbit in $\widetilde{\Tmc}$, and $\widehat{\Theta}$ is the set of cyclically ordered representatives for each $\Gamma$-orbit in $\widetilde{\Theta}$ (see Notation \ref{not:index_set}). Recall that 
\[\mu^{\mathbf k_1}_{\gamma\cdot \mathbf r_1}=\mu^{\mathbf k_1}_{\mathbf r_1}=\mu^{\mathbf k_2}_{\mathbf r_2},\] 
for any group element $\gamma$ in $\Gamma$, and any $\mathbf k_m:=(k_m,k_{m+1})$ and $\mathbf r_m:=(r_m,r_{m+1})$ for $m=1$, $2$. Similarly, 
\[\mu^{\mathbf i_1}_{\gamma\cdot \mathbf x_1}=\mu^{\mathbf i_1}_{\mathbf x_1}=\mu^{\mathbf i_2}_{\mathbf x_2}=\mu^{\mathbf i_3}_{\mathbf x_3}\]
for any group element $\gamma$ in $\Gamma$, and $\mathbf i_m:=(i_m,i_{m+1},i_{m-1})$ and $\mathbf x_m:=(x_m,x_{m+1},x_{m-1})$ for all $m=1$, $2$, $3$.

In the previous sections, we split the parameters of the parametrization $\Omega_{\Tmc,\Jmc}$ into two types, parameters associated to edges in $\Tmc$, and parameters associated to ideal triangles in $\Theta$. 
Here, where we construct flows associated to the non-isolated edges (simple closed curves) in $\Pmc$ on the one hand, and to the pair of pants in $\Pbbb$ on the other hand, it is more convenient to split the parameters in a different way. Namely into those that are associated to non-isolated edges, and those that are associated to ideal triangles in $\Theta$. For this note that any isolated edge of $\Tmc$ is the side of an ideal triangle in $\Theta$. Therefore a convenient way is to consider $\overline{\Bmc}$, the set of ordered triples $\mathbf j:=(j_1,j_2,j_3)$ of non-negative integers that sum to $n$, such that $0\leq j_1,j_2, j_3\leq n-1$.

We can then denote an arbitrary vector $\mu$ in $W$ by
\[\begin{array}{l}
\mu=\left(\left(\mu^{\mathbf k}_{\mathbf r}\right)_{\mathbf k\in\Amc;\mathbf z\in\widehat{\Pmc}},\left(\mu^{\mathbf j}_{\mathbf x}\right)_{\mathbf j\in\overline{\Bmc};\mathbf x\in\widehat{\Theta}}\right).
\end{array}\]
So, for example if $\mathbf{j} = (0, j_2, j_3)$, then $\mu^{\mathbf j}_{\mathbf x} = \mu^{(j_2,j_3)}_{(x_2,x_3)}$ is the parameter associated to the isolated edge represented by $(x_2,x_3)$.

We now first introduce the eruption and the hexagon flow associated to a pair of pants $P$ in $\Pbbb$. 

For every pair of pants $P$ in $\mathbb{P}$, we need a way to label the two ideal triangles in which it is cut. For this, set $x_m:=\alpha_{P,m}^-$ for $m=1$, $2$, $3$, and let $y_1:=\alpha_{P,1}^-$, $y_2:=\alpha_{P_3}^-$, $y_3:=\alpha_{P,3}\cdot\alpha_{P,2}^-$ in $\partial\Gamma$, where $\alpha_{P,1}$, $\alpha_{P,2}$ and $\alpha_{P,3}$ are the group elements in $\pi_1(P)$ corresponding to the peripheral curves in $P$ as described in Section \ref{sec:special1}. Observe that
\begin{itemize}
\item $x_1=y_1<x_2<x_3=y_2<y_3<x_1=y_1$,
\item $[x_1,x_2,x_3]$ and $[y_1,y_2,y_3]$ are the two ideal triangles that lie in $P$,
\item if $\{T_1,T_2\}$ is a bridge in $\widetilde{\Jmc}$ across a lift of a boundary curve of $P$, and $[T_1]$ is the triangle in $P$, then $[T_1]=[x_1,x_2,x_3]$.
\end{itemize}

Note that the conjugacy classes $[\alpha_{P,1}]$, $[\alpha_{P,2}]$, $[\alpha_{P,3}]$ are naturally in bijection with the three boundary components $c_1,c_2,c_3$ of $P$, equipped with the orientation so that $P$ lies to the right of each boundary component. Therefore, specifying an order on $\{c_1,c_2,c_3\}$ induces an order on $\{x_1,x_2,x_3\}$.

\begin{definition}
Let $c_1,c_2,c_3$ be the non-isolated edges in $\Pmc$ that are the three boundary components of $P$. An order on $\{c_1,c_2,c_3\}$ is said to be \emph{cyclic} if the induced order on the vertices of the ideal triangle $\{x_1,x_2,x_3\}$ is cyclically ordered.
\end{definition}

The only cyclic orders on $\{c_1,c_2,c_3\}$ are $(c_m,c_{m+1},c_{m-1})$ for $m=1$, $2$, $3$.

We set $\mathbf{P}:=(P,c_1,c_2,c_3)$, where $P$ is a pair of pants in $\Pbbb$, and $(c_1,c_2,c_3)$ is a cyclic order on the boundary components of $P$. We then have a cyclic order 
$\mathbf x:=(x_1,x_2,x_3)$ on the triangle $[x_1,x_2,x_3]$. We will refer to any such choice of $\mathbf x$ as a triple \emph{associated} to $\mathbf P$.

\begin{figure}[ht]
\centering
\includegraphics[scale=0.5]{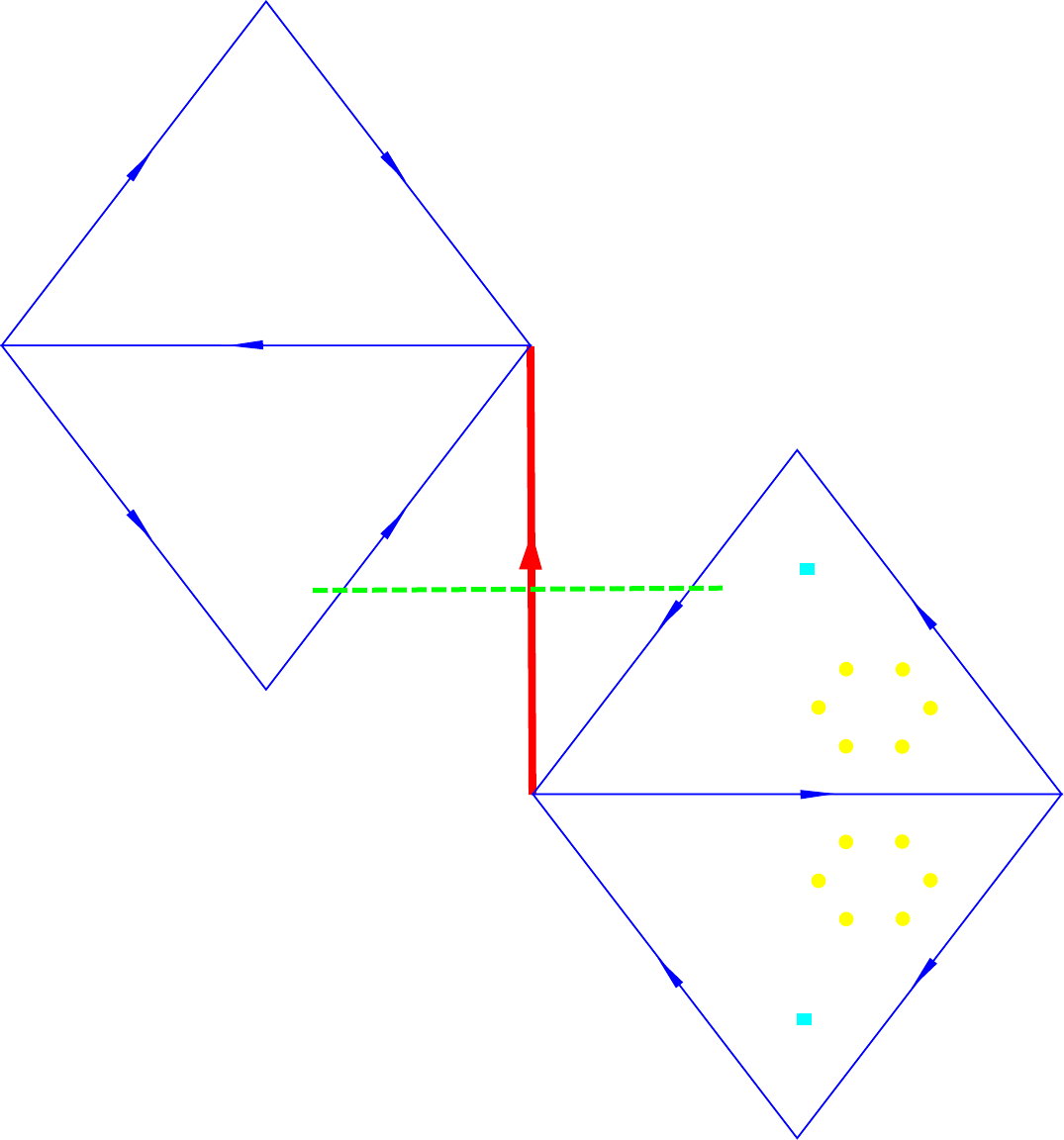}
\footnotesize
\put (-158, 84){$x_1=y_1$}
\put (-70, 171){$x_2$}
\put (1, 83){$x_3=y_2$}
\put (-69, -5){$y_3$}
\put (-66, 146){$\frac{1}{2}$}
\put (-70, 20){$-\frac{1}{2}$}
\put (-55, 119){$1$}
\put (-46, 119){$-1$}
\put (-55, 88){$1$}
\put (-46, 88){$-1$}
\put (-31, 103){$1$}
\put (-72, 103){$-1$}
\put (-55, 76){$1$}
\put (-46, 76){$-1$}
\put (-55, 46){$1$}
\put (-46, 46){$-1$}
\put (-31, 61){$1$}
\put (-72, 61){$-1$}
\caption{A non-isolated edge of $\widetilde{\Pmc}$ is drawn in red (thick), a bridge in $\widetilde{\Jmc}$ across the non-isolated edge is drawn in green (dotted), and isolated edges are draw in blue (thin). Each colored dot in an ideal triangle represents an ordered triple of positive integers that sum to $n$. The picture gives a diagramatic representation for the eruption flows (turquoise squares) and hexagon flows (yellow discs). The numbers above each of the colored dots are the corresponding coordinates of the vector $\mu$ in $W$.}\label{fig:specialadmissiblelabellingsEH}
\end{figure}

\begin{notation}\label{not:overline}
For any triple of positive integers $\mathbf i:=(i_1,i_2,i_3)$ that sum to $n$, let $\overline{\mathbf i}:=(i_1,i_3,i_2)$.
\end{notation}

\begin{definition}
Let $\mathbf P$, $\mathbf x$, $\mathbf y$ be as above, and let $\mathbf i:=(i_1,i_2,i_3)$ be a triple of positive integers that sum to $n$. The $\mathbf i$-{\em eruption flow} associated to $\mathbf P$ is the $(\Tmc, \Jmc)$ parallel flow $\phi^\mu_t$ as defined in Definition \ref{def:parallelflow}, where $\mu$ is the vector in $W$ that satisfies
\begin{itemize}
\item for all pairs $\mathbf r:=(r_1,r_2)$ such that $\{r_1,r_2\}$ is a non-isolated edge in $\widetilde{\Tmc}$, and all pairs of positive integers $\mathbf l:=(l_1,l_2)$ that sum to $n$, 
\[\mu^{\mathbf l}_{\mathbf r}=0,\]
\item for all cyclically ordered triples $\mathbf t:=(t_1,t_2,t_3)$ such that $\{t_1,t_2,t_3\}$ is an ideal triangle in $\widetilde{\Theta}$, and all triples of integers $\mathbf j:=(j_1,j_2,j_3)$ that sum to $n$ and $0\leq j_1,j_2, j_3\leq n-1$,
\[\mu^{\mathbf j}_{\mathbf t}=\begin{cases}\displaystyle\frac{1}{2}&\text{if }\mathbf t=\mathbf x\text{ and }\mathbf j=\mathbf i;
\cr-\displaystyle\frac{1}{2}&\text{if }\mathbf t=\mathbf y\text{ and }\mathbf j=\overline{\mathbf i}; \cr0 &\text{otherwise.}\end{cases}.\]
\end{itemize}
(see Figure \ref{fig:specialadmissiblelabellingsEH}). Let $\mathcal{E}_{\mathbf P}^{\mathbf i}$ denote the tangent vector field of this flow. Collectively, we refer to the $\mathbf i$-eruption flows associated to $\mathbf P$ as the \emph{eruption flows associated to }the pair of pants $P$.
\end{definition}

It is straightforward to verify that $\mathcal{E}_{\mathbf P}^{\mathbf i}$ does not depend on the choices made.

More informally, the $\mathbf i$-eruption flow associated to $\mathbf P$ is the flow with the property that in time $t$, it increases $\tau^\mathbf i_\mathbf x$ by $\frac{1}{2}t$, decreases $\tau^{\overline{\mathbf i}}_\mathbf y$ by $\frac{1}{2}t$, and keeps all other coordinate functions of $\Omega$ constant. Observe that if we set $\mathbf i_m:=(i_m,i_{m+1},i_{m-1})$ and $\mathbf P_m:=(P,c_m,c_{m+1},c_{m-1})$ for $m=1$, $2$, $3$, then
$\mathcal{E}_{\mathbf P_1}^{\mathbf i_1}=\mathcal{E}_{\mathbf P_2}^{\mathbf i_2}=\mathcal{E}_{\mathbf P_3}^{\mathbf i_3}$. Thus, there are $(n-1)(n-2)(g-1)$ eruption flows, $\frac{(n-1)(n-2)}{2}$ for each pair of pants in $\Pbbb$. 

To define the hexagon flows, we use the following notation.

\begin{notation}\label{not:shift}
Let $\mathbf i:=(i_1,i_2,i_3)$ be a triple of positive integers that sum to $n$. For any triple of integers $(a_1,a_2,a_3)$ that sum to $0$ such that $-i_m\leq a_m\leq n-i_m-1$ for $m=1$, $2$, $3$, denote $\mathbf i(a_1,a_2,a_3):=(i_1+a_1,i_2+a_2,i_3+a_3)$.
\end{notation}

\begin{definition}
Let $\mathbf P$, $\mathbf x$, $\mathbf y$ be as above, and let $\mathbf i:=(i_1,i_2,i_3)$ be a triple of positive integers that sum to $n$. The $\mathbf i$-{\em hexagon flow} associated to $\mathbf P$ is the $(\Tmc, \Jmc)$-parallel flow $\phi^\mu_t$, where $\mu$ is the vector in $W$ that satisfies 
\begin{itemize}
\item for all pairs $\mathbf r:=(r_1,r_2)$ such that $\{r_1,r_2\}$ is a non-isolated edge in $\widetilde{\Tmc}$, and all pairs of positive integers $\mathbf l:=(l_1,l_2)$ that sum to $n$, 
\[\mu^{\mathbf l}_{\mathbf r}=0,\]
\item for all cyclically ordered triples $\mathbf t:=(t_1,t_2,t_3)$ such that $\{t_1,t_2,t_3\}$ is an ideal triangle in $\widetilde{\Theta}$, and all triples of integers $\mathbf j:=(j_1,j_2,j_3)$ that sum to $n$ and $0\leq j_1,j_2, j_3\leq n-1$,
\[\mu^{\mathbf j}_{\mathbf t}=\left\{\begin{array}{ll}
1&\begin{array}{l}\text{if }\mathbf t=\mathbf x \text{ and } \mathbf j=\mathbf i(0,1,-1),\; \mathbf i(-1,0,1)\text{ or }\mathbf i(1,-1,0);\end{array}\vspace{.2cm}\\
1&\begin{array}{l}\text{if }\mathbf t=\mathbf y \text{ and } \mathbf j=\overline{\mathbf i}(0,-1,1),\; \overline{\mathbf i}(1,0,-1)\text{ or }\overline{\mathbf i}(-1,1,0);\end{array}\vspace{.2cm}\\
-1&\begin{array}{l}\text{if }\mathbf t=\mathbf x\text{ and }\mathbf j=\mathbf i(0,-1,1),\; \mathbf i(1,0,-1)\text{ or }\mathbf i(-1,1,0);\end{array}\vspace{.2cm}\\
-1&\begin{array}{l}\text{if }\mathbf t=\mathbf y\text{ and }\mathbf j=\overline{\mathbf i}(0,1,-1),\; \overline{\mathbf i}(-1,0,1)\text{ or }\overline{\mathbf i}(1,-1,0);\end{array}\vspace{.2cm}\\
0&\hspace{.1cm}\text{otherwise.}
\end{array}\right.\]
\end{itemize}
(see Figure \ref{fig:specialadmissiblelabellingsEH}). Let $\mathcal{H}_{\mathbf P}^{\mathbf i}$ denote the tangent vector field of this flow. Collectively, we refer to the $\mathbf i$-hexagon flows associated to $\mathbf P$ as the \emph{hexagon flows associated to }the pair of pants $P$.
\end{definition}

Just like the eruption flows, observe that $\mathcal{H}_{\mathbf P}^{\mathbf i}$ does not depend on any of the choices made. Also, if we set $\mathbf i_m:=(i_m,i_{m+1},i_{m-1})$ and $\mathbf P_m:=(P,c_m,c_{m+1},c_{m-1})$ for $m=1$, $2$, $3$, then $\mathcal{H}_{\mathbf P_1}^{\mathbf i_1}=\mathcal{H}_{\mathbf P_2}^{\mathbf i_2}=\mathcal{H}_{\mathbf P_3}^{\mathbf i_3}$. Thus, there are $(g-1)(n-1)(n-2)$ hexagon flows, $\frac{(n-1)(n-2)}{2}$ for each pair of pants in $\Pbbb$. 

The following lemma is an easy consequence of the definition of the eruption and hexagon flows.

\begin{lem}\label{lem:pants_length}
Let $P$ be a pair of pants in $\Pbbb$. The eruption flows and the hexagon flows associated to $P$ do not change the representation restricted to $S\backslash P$ (up to conjugation). In particular they preserve the eigenvalues of the holonomy along all pants curves in $\Pmc$.
\end{lem}

\begin{remark}
 When $n=3$, there is only one eruption flow and one hexagon flow for each pair of pants $P$ in $\Pbbb$. In this case, any representation in the Hitchin component is the holonomy of a convex real projective structure on $S$. The eruption flow has a geometric realization which can be seen as changing certain gluing parameters of convex projective triangles. The hexagon flow can be realized as a bulging flow with respect to an edge of the triangulation that lies inside of $P$. We refer the reader to \cite{Wienhard_Zhang} for a geometric description of these flows. 
\end{remark}

Next, we define the twist flows and the length flows, which are associated to the pants curve in $\Pmc$. 

For any curve $c$ in $\mathcal{P}$, let $\mathbf c$ be $c$ equipped with an orientation, and let $P_1$ and $P_2$ be the pairs of pants that share $c$ as a common boundary component (possibly $P_1=P_2$), such that $P_1$ and $P_2$ lie to the right and left of $\mathbf c$ respectively. Choose an oriented, non-isolated edge $\mathbf d:=(x_{1,1},x_{1,2})$ in $\widetilde{\Tmc}$ that is a lift of $\mathbf c$. We may assume without loss of generality that the group element $\alpha_{P_m}$ in $\pi_1(P_m)$ has $x_{1,m}$ and $x_{1,m+1}$ as its repelling and attracting fixed point respectively, and that $\{T_1,T_2\}$ is a bridge in $\widetilde{\Jmc}$, where $T_m:=\{\alpha_{P_m}^-,\beta_{P_m}^-,\gamma_{P_m}^-\}$. Then let $x_{2,m}:=\beta_{P_m}^-$, $x_{3,m}:=\gamma_{P_m}^-$ and $y_{1,m}:=\alpha_{P_m}^-$, $y_{2,m}:=\gamma_{P_m}^-$, $y_{3,m}:=\gamma_{P_m}\cdot\beta_{P_m}^-$ in $\partial\Gamma$. Observe that 
\begin{itemize}
\item $x_{1,1}<x_{2,2}<x_{3,2}<x_{1,2}<x_{2,1}<x_{3,1}<x_{1,1}$,
\item $[x_{1,m},x_{2,m},x_{3,m}]$ and $[y_{1,m},y_{2,m},y_{3,m}]$ are the ideal triangles in $P_m$.
\end{itemize}
We will refer to any such choice of $\mathbf d$ as a pair \emph{associated} to $\mathbf c$.

\begin{definition}
Let $\mathbf c$ be an oriented simple closed curve in $\Pmc$ and $\mathbf d$ a pair associated to $\mathbf c$. Let $\mathbf k:=(k_1,k_2)$ be a pair of positive integers that sum to $n$. The $\mathbf k$-{\em twist flow} associated to $\mathbf c$ is the $(\Tmc, \Jmc)$-parallel flow $\phi^\mu_t$, where $\mu$ is the vector in $W$ that satisfies 
\begin{itemize}
\item for all pairs $\mathbf r:=(r_1,r_2)$ such that $\{r_1,r_2\}$ is a non-isolated edge in $\widetilde{\Tmc}$, and all pairs of positive integers $\mathbf l:=(l_1,l_2)$ that sum to $n$, 
\[\mu^{\mathbf l}_{\mathbf r}=\begin{cases}\displaystyle-\frac{1}{2}&\mathbf l=\mathbf k\text{ and }\mathbf r=\mathbf d;\cr0 &\text{otherwise}\end{cases}\] 
\item for all cyclically ordered triples $\mathbf t:=(t_1,t_2,t_3)$ such that $\{t_1,t_2,t_3\}$ is an ideal triangle in $\widetilde{\Theta}$, and all triples of integers that sum to $n$ and $0\leq j_1,j_2, j_3\leq n-1$, 
\[\mu^{\mathbf j}_{\mathbf t}=0,\] 
(see Figure \ref{fig:specialadmissiblelabellingsSY}).
\end{itemize} 
Let $\mathcal{S}_{\mathbf c}^\mathbf k$ denote the tangent vector field of this flow. We collectively refer to the $\mathbf k$-twist flows associated to $\mathbf c$ as the \emph{twist flows associated to} the simple closed curve $c$ in $\Pmc$.
\end{definition}

\begin{figure}[ht]
\centering
\includegraphics[scale=0.5]{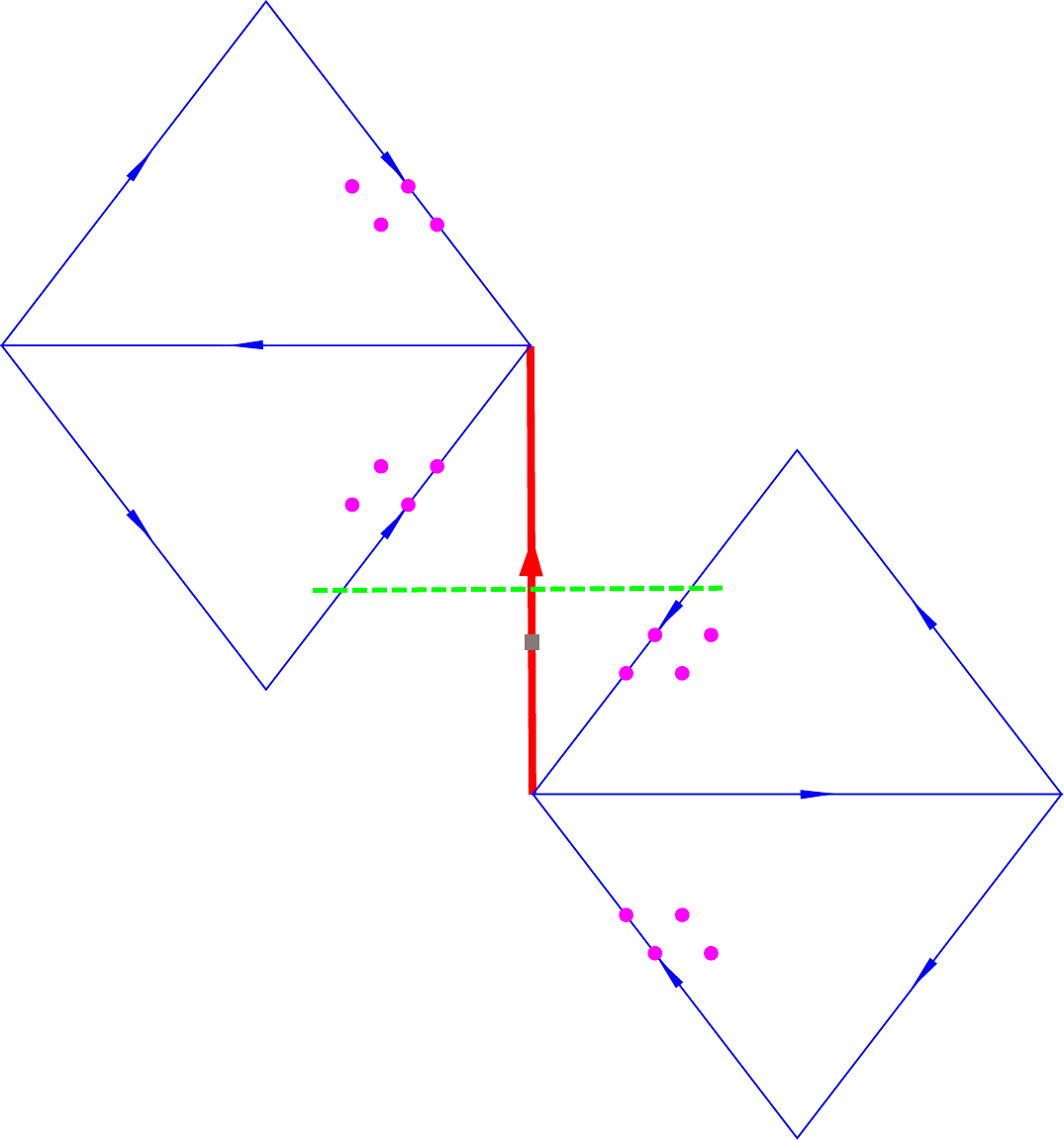}
\footnotesize
\put (-170, 84){$x_{1,1}=y_{1,1}$}
\put (-70, 171){$x_{2,1}$}
\put (1, 83){$x_{3,1}=y_{2,1}$}
\put (-69, -5){$y_{3,1}$}
\put (-129, 192){$x_{1,2}=y_{1,2}$}
\put (-199, 105){$x_{2,2}$}
\put (-300, 192){$x_{3,2}=y_{2,2}$}
\put (-199, 281){$y_{3,2}$}
\put (-145, 121){$-\frac{1}{2}$}
\put (-105, 122){$1$}
\put (-84, 122){$-1$}
\put (-119, 112){$-1$}
\put (-92, 112){$1$}
\put (-105, 44){$1$}
\put (-84, 44){$-1$}
\put (-119, 53){$-1$}
\put (-92, 53){$1$}
\put (-172, 163){$1$}
\put (-151, 163){$-1$}
\put (-185, 153){$-1$}
\put (-158, 153){$1$}
\put (-172, 221){$1$}
\put (-151, 221){$-1$}
\put (-185, 231){$-1$}
\put (-158, 231){$1$}
\caption{A non-isolated edge of $\widetilde{\Pmc}$ is drawn in red (thick), a bridge in $\widetilde{\Jmc}$ across the non-isolated edge is drawn in green (dotted), and isolated edges are draw in blue (thin). Each colored dot in an ideal triangle represents an ordered triple of positive integers that sum to $n$, and each dot along an edge represents an ordered pair of integers that sum to $n$. The picture gives a diagramatic representation for the lozenge flows (purple discs) and twist flows (grey squares). The numbers above each of the colored dots are the corresponding coordinates of the vector $\mu$ in $W$. 
}\label{fig:specialadmissiblelabellingsSY}
\end{figure}

Observe that $\mathcal{S}_{\mathbf c}^{\mathbf k}$ does not depend on any of the choices made. 

If we let $\mathbf c_1$ and $\mathbf c_2$ be the two orientations of $c$, and set $\mathbf k_m:=(k_m,k_{m+1})$ for $m=1$, $2$, then $\mathcal{S}_{\mathbf c_1}^{\mathbf k_1}=\mathcal{S}_{\mathbf c_2}^{\mathbf k_2}$. As such, in total, there are $(3g-3)(n-1)$ twist flows, $n-1$ for each simple closed curve in $\Pmc$. 

These twist flows are generalized twist flows in the sense of Goldman \cite[Section 1]{Goldman_twist}. It is straightforward to verify the following lemma.

\begin{lem}\label{lem:twist_length}
Let $c$ be a simple closed curve in the pair of pants decomposition $\Pmc$. The twist flows associated to $c$ do not change (up to conjugation) the representation restricted to $S\backslash c$. They hence preserve the eigenvalues of the holonomy along all pants curves in $\Pmc$. 
\end{lem} 

To define the length flows, note that for any pair of positive integers $\mathbf k:=(k_1,k_2)$ that sum to $n$, we denote $\mathbf k^1:=(0,k_1,k_2)$, $\mathbf k^2:=(k_1,0,k_2)$, and $\mathbf k^3:=(k_1,k_2,0)$.

\begin{definition}
Let $\mathbf c$ be an oriented simple closed curve in $\Pmc$ and $\mathbf d$ a pair associated to $\mathbf c$. Let $\mathbf k:=(k_1,k_2)$ be a pair of positive integers that sum to $n$. For $m=1$, $2$, set $\mathbf x(m):=(x_{1,m},x_{2,m},x_{3,m})$, and $\mathbf y(m):=(y_{1,m},y_{2,m},y_{3,m})$, and let $\mathbf k_m:=(k_m,k_{m+1})$. 

\begin{enumerate}
\item The $\mathbf k$-{\em lozenge flow} associated to $\mathbf c$ is the $(\Tmc, \Jmc)$-parallel flow $\phi^\mu_t$, where $\mu$ is the vector in $W$ that satisfies 
\begin{itemize}
\item for all pairs $\mathbf r:=(r_1,r_2)$ such that $\{r_1,r_2\}$ is a non-isolated edge in $\widetilde{\Tmc}$, and all pairs of positive integers $\mathbf l:=(l_1,l_2)$ that sum to $n$, 
\[\mu^{\mathbf l}_{\mathbf r}=0\] 
\item for all cyclically ordered triples $\mathbf t:=(t_1,t_2,t_3)$ such that $\{t_1,t_2,t_3\}$ is an ideal triangle in $\widetilde{\Theta}$, and all triples of integers $\mathbf j:=(j_1,j_2,j_3)$ that sum to $n$ and $0\leq j_1,j_2, j_3\leq n-1$, 
\[\mu^{\mathbf j}_{\mathbf t}=\left\{\begin{array}{ll}
1&\text{if }\mathbf t=\mathbf x(1)\text{ and }\mathbf j=\mathbf k_1^3\text{ or }\mathbf k_1^3(0,-1,1);\\
1&\text{if }\mathbf t=\mathbf y(1)\text{ and }\mathbf j=\mathbf k_1^2\text{ or }\mathbf k_1^2(0,1,-1),\\
1&\text{if }\mathbf t=\mathbf x(2)\text{ and }\mathbf j=\mathbf k_2^3\text{ or }\mathbf k_2^3(0,-1,1);\\
1&\text{if }\mathbf t=\mathbf y(2)\text{ and }\mathbf j=\mathbf k_2^2\text{ or }\mathbf k_2^2(0,1,-1);\\
-1&\text{if }\mathbf t=\mathbf x(1)\text{ and }\mathbf j=\mathbf k_1^3(1,-1,0)\text{ or }\mathbf k_1^3(-1,0,1);\\
-1&\text{if }\mathbf t=\mathbf y(1)\text{ and }\mathbf j=\mathbf k_1^2(1,0,-1)\text{ or }\mathbf k_1^2(-1,1,0);\\
-1&\text{if }\mathbf t=\mathbf x(2)\text{ and }\mathbf j=\mathbf k_2^3(1,-1,0)\text{ or }\mathbf k_2^3(-1,0,1);\\
-1&\text{if }\mathbf t=\mathbf y(2)\text{ and }\mathbf j=\mathbf k_2^2(1,0,-1)\text{ or }\mathbf k_2^2(-1,1,0);\\
0&\text{otherwise;}
\end{array}\right.\]
where any positive triple of integers $\mathbf i:=(i_1,i_2,i_3)$ that sum to $n$, $\mathbf i(a_1,a_2,a_3):=(i_1+a_1,i_2+a_2,i_3+a_3)$, see Figure \ref{fig:specialadmissiblelabellingsSY}. Let $\mathcal{Z}_{\mathbf c}^{\mathbf k}$ denote the tangent vector field of this flow. 
\end{itemize}


\item The $\mathbf k$-{\em length flow} associated to $\mathbf c$ is the flow whose tangent vector field $\mathcal{Y}_{\mathbf c}^\mathbf k$ is given by
\[\mathcal{Y}_{\mathbf c}^\mathbf k:=\mathcal{Z}_{\mathbf c}^\mathbf k+ \mathcal{E}_{\mathbf x(1)}^{\mathbf k_1^3(0,-1,1)}-\mathcal{E}_{\mathbf x(1)}^{\mathbf k_1^3(-1,0,1)}+ \mathcal{E}_{\mathbf x(2)}^{\mathbf k_2^3(0,-1,1)}-\mathcal{E}_{\mathbf x(2)}^{\mathbf k_2^3(-1,0,1)}.\]
\end{enumerate}
We collectively refer to the $\mathbf k$-length flows associated to $\mathbf c$ as the \emph{length flows associated to} the simple closed curve $c\in\Pmc$.
\end{definition}

The lozenge flow is simply an intermediate flow used to define the length flow. Just like the twist flows, $\mathcal{Y}_{\mathbf c}^{\mathbf k}$ and $\mathcal{Z}_{\mathbf c}^{\mathbf k}$ do not depend on the choices made. Also, if we let $\mathbf c_1$ and $\mathbf c_2$ be the two orientations of $c$, and set $\mathbf k_m:=(k_m,k_{m+1})$ for $m= 1$, $2$, then we have $\mathcal{Y}_{\mathbf c_1}^{\mathbf k_1}=\mathcal{Y}_{\mathbf c_2}^{\mathbf k_2}$ and $\mathcal{Z}_{\mathbf c_1}^{\mathbf k_1}=\mathcal{Z}_{\mathbf c_2}^{\mathbf k_2}$. Hence, there are $(3g-3)(n-1)$ length flows, $n-1$ for each simple closed curve in $\Pmc$. 

The following lemma is immediate.

\begin{lem}
The length flows associated to $c$ do not change the representation restricted to $S\setminus(P_1\cup P_2)$.
\end{lem}

\begin{definition}
A $(\Tmc,\Jmc)$-parallel flow is said to be \emph{special} if it is an eruption, hexagon, twist or length flow. The tangent vector fields of these flows are respectively called the \emph{eruption, hexagon, twist and length fields}, and are collectively referred to the \emph{special} $(\Tmc,\Jmc)$-parallel vector fields.
\end{definition}

\subsection{A new coordinate system}\label{sec:Hamiltonian} 
For every special $(\Tmc,\Jmc)$-parallel vector field $\Xmc$, let $\Xmc^*$ be the special $(\Tmc,\Jmc)$-parallel vector field given by
\[\Xmc^*:=\left\{\begin{array}{ll}
\Smc^\mathbf k_\mathbf c&\text{if }\Xmc=\Ymc^\mathbf k_\mathbf c;\\
-\Ymc^\mathbf k_\mathbf c&\text{if }\Xmc=\Smc^\mathbf k_\mathbf c;\\
\Emc^\mathbf i_\mathbf P&\text{if }\Xmc=\Hmc^\mathbf i_\mathbf P;\\
-\Hmc^\mathbf i_\mathbf P&\text{if }\Xmc=\Emc^\mathbf i_\mathbf P.
\end{array}\right.\]

\begin{remark}\label{rem:symplectic}
This notation is chosen because in the upcoming companion paper \cite{SunZhang}, it is shown that the pairing of any pair of special $(\Tmc,\Jmc)$-parallel vector fields $\Xmc_1$, $\Xmc_2$ under the Goldman symplectic form $\omega$ on $\Hit_V(S)$ is given by
\[\omega(\Xmc_1,\Xmc_2)=\left\{\begin{array}{ll}
1&\text{if }\Xmc_1=\Xmc_2^*;\\
0&\text{otherwise.}
\end{array}\right.\]
\end{remark}

In order to define the coordinate system we show the following 
\begin{prop} 
For any special $(\Tmc,\Jmc)$-parallel vector field $\Xmc$, there exists a (necessarily unique) real-analytic function 
\[H(\Xmc):\Hit_V(S)\to\Rbbb\] 
whose derivative in the direction of $\Xmc^*$ is $1$, and whose derivative in the direction of all other special $(\Tmc,\Jmc)$-parallel flows is $0$. Furthermore, $H(\Xmc)$ can be written in the coordinate functions of the parametrization $\Omega$ given in Theorem \ref{thm:reparametrization}. In particular, the $(\Tmc,\Jmc)$-parallel vector fields define a global frame of the tangent bundle $T\Hit_V(S)$.
\end{prop}
This is proven in Theorem \ref{thm:Darboux} and Theorem \ref{thm:Darboux2} below. Since the special $(\Tmc,\Jmc)$-parallel flows pairwise commute and their tangent vector fields form a basis of the tangent space to $\Hit_V(S)$ at every point, we have the following corollary.

\begin{cor}\label{cor:coordinates}
The collection of functions
\[\{H(\Xmc):\Xmc\text{ is a special }(\Tmc,\Jmc)\text{-parallel vector field}\}\]
defines a global coordinate system on $\Hit_V(S)$. 
\end{cor}

\begin{remark}
Using the results in \cite{SunZhang} discussed in Remark \ref{rem:symplectic}, the global coordinate system in Corollary \ref{cor:coordinates} is a Darboux coordinate system for the Goldman symplectic form. As a consequence, we have the following corollary.
\end{remark}

\begin{cor}
The following submanifolds of $\Hit_V(S)$ are Lagrangian submanifolds:
\begin{enumerate}
\item The submanifold spanned by the $n-1$ twist flows associated to the $3g-3$ simple closed curves in $\Pmc$ and the $\frac{(n-1)(n-2)}{2}$ eruption flows associated to the $2g-2$ pairs of pants in $\Pbbb$.
\item The submanifold spanned by the $n-1$ twist flows associated to the $3g-3$ simple closed curves in $\Pmc$ and the $\frac{(n-1)(n-2)}{2}$ hexagon flows associated to the $2g-2$ pairs of pants in $\Pbbb$.
\item The submanifold spanned by the $n-1$ length flows associated to the $3g-3$ simple closed curves in $\Pmc$ and the $\frac{(n-1)(n-2)}{2}$ eruption flows associated to the $2g-2$ pairs of pants in $\Pbbb$.
\item The submanifold spanned by the $n-1$ length flows associated to the $3g-3$ simple closed curves in $\Pmc$ and the $\frac{(n-1)(n-2)}{2}$ hexagon flows associated to the $2g-2$ pairs of pants in $\Pbbb$.
\end{enumerate}
\end{cor}
The Lagrangian submanifolds in (1) and (2) are of particular interest since the length functions of all simple closed curves in $\Pmc$ are constant on them. 

The functions $H(\Xmc)$ are explicitely given in the following two theorems. 

\begin{thm}
\label{thm:Darboux}
Let $\mathbf c$ be an oriented, non-isolated edge in $\Tmc$, let $\mathbf d:=(d_1,d_2)$ be a lift of $\mathbf c$, and let $\mathbf k:=(k_1,k_2)$ be a pair of positive integers that sum to $n$.
\begin{enumerate}
\item Let $\mathcal{Y}_{\mathbf c}^{\mathbf k}$ be the tangent vector field of the $\mathbf k$-length flow with respect to $\mathbf c$. Then 
\[H(\mathcal{Y}_{\mathbf c}^{\mathbf k})=-2\alpha_{\mathbf d}^{\mathbf k},
\] 
where $\alpha_{\mathbf d}^{\mathbf k}:\Hit_V(S)\to\Rbbb$ is the symplectic closed edge invariant defined in Definition \ref{def:symplectic closed-edge invariant}. 
\item
Let $\mathcal{S}_{\mathbf c}^\mathbf k$ be the tangent vector field of the $\mathbf k$-twist flow with respect to $\mathbf c$. 
 Let $\gamma$ be the primitive group element in $\Gamma$ with $d_{1}$ and $d_{2}$ as its repelling and attracting fixed points respectively. For any pair of positive integers $\mathbf j:=(j_1,j_2)$ that sum to $n$, let $\ell^\mathbf j_{\mathbf c}:\Hit_V(S)\to\Rbbb^+$ be the function defined by $\ell^{\mathbf j}_{\mathbf c}[\rho]:=\ell^{j_1}_\rho(\gamma)$ (see (\ref{eqn:length})). Then
\[H(\mathcal{S}_{\mathbf c}^\mathbf k)=\sum_{j_1=1}^{k_1}\frac{k_2j_1}{2n} \cdot \ell_{\mathbf c}^{\mathbf j}+
\sum^{k_2-1}_{j_2=1}\frac{ k_1j_2}{2n} \cdot \ell_{\mathbf c}^{\mathbf j}
\]
\end{enumerate}
\end{thm}

In the proof of Theorem \ref{thm:Darboux}, we need the following notation.
\begin{notation}\label{not:shift2}
For any pair of positive integers $\mathbf k:=(k_1,k_2)$ that sum to $n$, and for any integer $a$ satisfying $-k_1<a<k_2$, set $\mathbf k(a,-a):=(k_1+a,k_2-a)$.
\end{notation}

\begin{proof}
To prove (1), we set $I^\mathbf k_{\mathbf c}:=-2\alpha_{\mathbf d}^{\mathbf k}$. It is sufficient to show that if $\Xmc$ is a special $(\Tmc,\Jmc)$-parallel flow, then
\[\mathcal{X}\big(I^\mathbf k_\mathbf c\big)=
\begin{cases}
1&\text{if }\mathcal{X}=\Smc^\mathbf k_{\mathbf c};\cr
0 & \text{otherwise.}
\end{cases}\]
This follows immediately from Theorem \ref{thm:main theorem}. 

To prove (2), we set 
\[L^\mathbf k_{\mathbf c}:= \sum_{j_1=1}^{k_1}\frac{k_2j_1}{2n} \cdot \ell_{\mathbf c}^{\mathbf j}+\sum^{k_2-1}_{j_2=1}\frac{ k_1j_2}{2n} \cdot \ell_{\mathbf c}^{\mathbf j}\] 
 and show that if $\Xmc$ is a special $(\Tmc,\Jmc)$-parallel flow, then
\[\mathcal{X}\big(L^\mathbf k_\mathbf c\big)=
\begin{cases}
-1&\text{if }\mathcal{X}=\Ymc^\mathbf k_{\mathbf c};\cr
0 & \text{otherwise.}
\end{cases}\]
We previously observed in Lemma \ref{lem:pants_length} and Lemma \ref{lem:twist_length} that the derivative of $\ell_{\mathbf c}^{\mathbf j}$, and hence that of $L^\mathbf k_\mathbf c$, in the direction of any eruption, hexagon or twist field is zero. Also, one can compute using Theorem \ref{thm:main theorem} that for any length field $\mathcal{Y}_{\mathbf b}^\mathbf l$,
\[\mathcal{Y}_{\mathbf b}^\mathbf l(\ell^\mathbf j_{\mathbf c} )=\mathcal{Z}_{\mathbf b}^\mathbf l(\ell^\mathbf j_{\mathbf c} ) = \begin{cases}2&\text{if }\mathbf l=\mathbf j(-1,1)\text{ or }\mathbf j(1,-1),\text{ and }\mathbf c=\mathbf b;\cr-4&\text{if }\mathbf l=\mathbf j \text{ and }\mathbf c=\mathbf b;\cr 0&\text{otherwise.}\end{cases}\] 
It follows that
\[\mathcal{Y}_{\mathbf b}^\mathbf l(L^\mathbf k_{\mathbf c})=\mathcal{Z}_{\mathbf b}^\mathbf l(L^\mathbf k_{\mathbf c}) = \begin{cases}-1&\text{if }\mathbf l=\mathbf k\text{ and }\mathbf c=\mathbf b;\cr 0&\text{otherwise.}\end{cases}\] 
\end{proof}

\begin{notation} Recall that for any pair of positive integers $\mathbf k:=(k_1,k_2)$ that sum to $n$, we denoted $\mathbf k^1:=(0,k_1,k_2)$, $\mathbf k^2:=(k_1,0,k_2)$, and $\mathbf k^3:=(k_1,k_2,0)$. Similarily, for any cyclically ordered triple of points $\mathbf x:=(x_1,x_2,x_3)$ in $\partial\Gamma$, we denote $\mathbf x_1:=(x_2,x_3)$, $\mathbf x_2=(x_1,x_3)$, and $\mathbf x_3:=(x_1,x_2)$. 
With this notation, the edge invariants along isolated edges $\sigma_{\mathbf{r}}^{\mathbf{k}}$, that were defined in Section~\ref{sec:edge} can be written as degenerate triangle invariants: 
\[\tau_{\mathbf x}^{\mathbf k_1}:= \sigma_{\mathbf x_1}^{\mathbf k},\,\,\, \tau_{\mathbf x}^{\mathbf k_2}:= \sigma_{\mathbf x_2}^{\mathbf k}\,\,\,\text{ and }\,\,\, \tau_{\mathbf x}^{\mathbf k_3}:= \sigma_{\mathbf x_3}^{\mathbf k}.\]
\end{notation}

\begin{notation}\label{not:3}
For any triple of positive integers $\mathbf i:=(i_1,i_2,i_3)$ that sum to $n$, denote $\mathbf i':=(i_1,i_2+i_3)$.
\end{notation}

\begin{thm}\label{thm:Darboux2}
Let $\mathbf{P}:=(P,c_1,c_2,c_3)$ where $P$ is a pair of pants in $\Pbbb$, and $(c_1,c_2,c_3)$ is a cyclic order on the boundary components of $P$. Let $\mathbf x:=(x_1,x_2,x_3)$ be a triple of points in $\partial\Gamma$ associated to $\mathbf P$, and let $\mathbf y:=(y_1,y_2,y_3)$ be the triple of points in $\partial\Gamma$, such that $y_1=x_1$, $y_2=x_3$, and $[y_1,y_2,y_3]$ is the ideal triangle in $\Theta$ whose union with $[x_1,x_2,x_3]$ is $P$. Let $\mathbf i:=(i_1,i_2,i_3)$ be a triple of positive integers that sum to $n$, and for $m=1$, $2$, $3$, set $\mathbf i_m:=(i_m,i_{m+1},i_{m-1})$. 
\begin{enumerate}
\item For $m=1$, $2$, $3$, let $\mathbf c_m$ be the orientation on $c_m$ such that $P$ lies to the right of $\mathbf c_m$. For all positive integers $k$, we set 
\[\delta_{k}:=\left\{\begin{array}{ll} 
1&\text{ if }k=1,\\
0&\text{ otherwise}. 
\end{array}\right.\]
Then
\begin{eqnarray*}
H(\mathcal{H}_{\mathbf P}^{\mathbf i})&=&\tau_{\mathbf x}^{\mathbf i}- \tau_{\mathbf y}^{\overline{\mathbf i}}+\delta_{i_3}\big(H(\mathcal{S}_{\mathbf c_1}^{\mathbf i_1'})-H(\mathcal{S}_{\mathbf c_1}^{\mathbf i_1'(1,-1)})\big) \\
&&+ \delta_{i_1}\big(H(\mathcal{S}_{\mathbf c_2}^{\mathbf i_2'})-H(\mathcal{S}_{\mathbf c_2}^{\mathbf i_2'(1,-1)})\big)+\delta_{i_2}\big(H(\mathcal{S}_{\mathbf c_3}^{\mathbf i_3'})-H(\mathcal{S}_{\mathbf c_3}^{\mathbf i_3'(1,-1)})\big)
\end{eqnarray*}
Here, recall that for any pair of positive integers $\mathbf k:=(k_1,k_2)$ that sum to $n$, $\mathbf k(a,-a):=(k_1+a,k_1-a)$.
\item Let $\overline{\Bmc}$ denote the set of ordered triples of integers $\mathbf j = ( j_1, j_2, j_3)$, $0 \leq j_1, j_2, j_3\leq n-1$ that sum to $n$ and define
\begin{eqnarray*}
\overline{\Bmc}_1&:=&\{\mathbf j\in\overline{\Bmc}:j_1\geq i_1\text{ and }j_2\leq i_2\},\\
\overline{\Bmc}_2&:=&\{\mathbf j\in\overline{\Bmc}:j_2\geq i_2\text{ and }j_3\leq i_3\},\\
\overline{\Bmc}_3&:=&\{\mathbf j\in\overline{\Bmc}:j_3\geq i_3\text{ and }j_1\leq i_1\}
\end{eqnarray*}
(see Figure \ref{fig:Tn}). Further define for all triples $\mathbf j:=(j_1,j_2,j_3)$ in $\overline{\Bmc}$ 
\begin{equation*}
c_{\mathbf i}^{\mathbf j}:= \begin{cases}\displaystyle\frac{i_1j_3+i_1j_2+i_3j_2}{2n}&\text{if }\mathbf j\in \overline{\Bmc}_1;\vspace{0.1cm}
\cr \displaystyle\frac{i_2j_1+i_2j_3+i_1j_3}{2n}&\text{if }\mathbf j \in \overline{\Bmc}_2;\vspace{0.1cm}
\cr\displaystyle\frac{i_3j_2+i_3j_1+i_2j_1}{2n}&\text{if }\mathbf j \in \overline{\Bmc}_3.\end{cases}
\end{equation*}
Then
\[H(\mathcal{E}_{\mathbf P}^{\mathbf i})=\sum_{\mathbf j \in \overline{\Bmc}} c_{\mathbf i}^{\mathbf j} \cdot \big(\tau_{\mathbf x}^{\mathbf j} +\tau_{\mathbf y}^{\overline{\mathbf j}}\big).
\]
Here, recall that if $\mathbf j:=(j_1,j_2,j_3)$ is a triple of positive integers that sum to $n$, then $\overline{\mathbf j}:=(j_1,j_3,j_2)$.

\end{enumerate}
\begin{figure}[ht]
\centering
\includegraphics[scale=0.5]{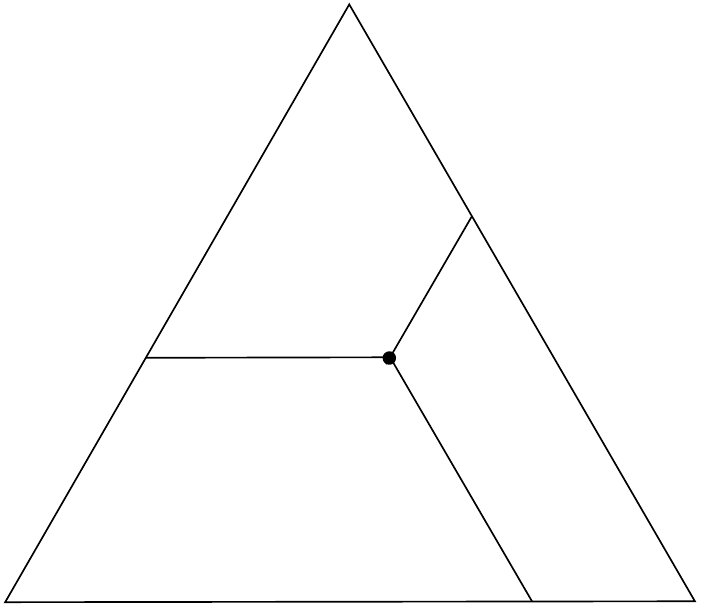}
\small
\put (-177, 0){$x_1$}
\put (0, 0){$x_3$}
\put (-87, 149){$x_2$}
\put (-80, 64){$\mathbf i$}
\put (-110, 25){$\overline{\Bmc}_1$}
\put (-90, 95){$\overline{\Bmc}_2$}
\put (-45, 40){$\overline{\Bmc}_3$}
\caption{Hamiltonian function for the eruption flow $\mathcal{E}_{\mathbf x}^{\mathbf i}$}.\label{fig:Tn}
\end{figure}
\end{thm}

\begin{proof}
For (1) we set 
\begin{eqnarray*}G^{\mathbf i}_{\mathbf P}&:=&\tau_{\mathbf x}^{\mathbf i}- \tau_{\mathbf y}^{\overline{\mathbf i}}+\delta_{i_3}\big(H(\mathcal{S}_{\mathbf c_1}^{\mathbf i_1'})-H(\mathcal{S}_{\mathbf c_1}^{\mathbf i_1'(1,-1)})\big) \\
&&+ \delta_{i_1}\big(H(\mathcal{S}_{\mathbf c_2}^{\mathbf i_2'})-H(\mathcal{S}_{\mathbf c_2}^{\mathbf i_2'(1,-1)})\big)+\delta_{i_2}\big(H(\mathcal{S}_{\mathbf c_3}^{\mathbf i_3'})-H(\mathcal{S}_{\mathbf c_3}^{\mathbf i_3'(1,-1)})\big).
\end{eqnarray*}

 It is sufficient to prove that for any special $(\Tmc,\Jmc)$-vector field $\Xmc$, 
\[\mathcal{X}\big(G^\mathbf i_\mathbf P\big)=
\begin{cases}
1&\text{if }\mathcal{X}=\Emc^{\mathbf i}_{\mathbf P};\cr
0 & \text{otherwise.}
\end{cases}\]

By Theorem \ref{thm:main theorem}, the derivative of $G^{\mathbf i}_{\mathbf P}$ in the direction of any twist field is zero. Also, by Theorem \ref{thm:main theorem} and the symmetry of the hexagon fields, we see that the derivative of $\tau_{\mathbf x}^{\mathbf i}- \tau_{\mathbf y}^{\overline{\mathbf i}}$ in the direction of any hexagon field is zero. Furthermore, by Theorem~\ref{thm:Darboux}(2), we know that the derivative of $F^{\mathbf i}_{\mathbf P}:= G^{\mathbf i}_{\mathbf P}-\tau_{\mathbf x}^{\mathbf i}+ \tau_{\mathbf y}^{\overline{\mathbf i}}$ in the direction of any hexagon field is also zero. Hence, the derivative of $G^{\mathbf i}_{\mathbf P}$ in the direction of any hexagon field is zero.

For any eruption field $\Emc^\mathbf j_\mathbf P$, let $\mathbf t$ be the triple of points in $\partial\Gamma$ associated to $\mathbf P$. Using Theorem \ref{thm:main theorem}, one can compute that 
\begin{equation}\label{eqn:eruption length}
\Emc^{\mathbf j}_{\mathbf P}(\tau_{\mathbf x}^{\mathbf i}-\tau_{\mathbf y}^{\overline{\mathbf i}})=\left\{\begin{array}{ll}
1&\text{if }\mathbf j=\mathbf i\text{ and }\mathbf t=\mathbf x\\
0&\text{otherwise.}
\end{array}\right.
\end{equation}
Furthermore, we know by (2) of Theorem \ref{thm:Darboux} that the derivative of $F^{\mathbf i}_{\mathbf P}$ in the direction of any eruption vector field is zero. Thus,
\[\Emc^{\mathbf j}_{\mathbf P}(G^{\mathbf i}_{\mathbf P})=\left\{\begin{array}{ll}
1&\text{if }\mathbf j=\mathbf i\text{ and }\mathbf t=\mathbf x\\
0&\text{otherwise.}
\end{array}\right.\]

To finish the proof of (1), we need to show that $\mathcal{Y}_{\mathbf c}^\mathbf k(G^{\mathbf i}_{\mathbf P})=0$ for any length field $\mathcal{Y}_{\mathbf c}^\mathbf k$. Let $\mathcal{X}_{\mathbf c}^\mathbf k:=\mathcal{Y}_{\mathbf c}^\mathbf k-\mathcal{Z}_{\mathbf c}^\mathbf k$ and write
\[\Ymc^\mathbf k_{\mathbf c}(G^{\mathbf i}_{\mathbf P})=\Ymc^\mathbf k_{\mathbf c}(F^{\mathbf i}_{\mathbf P})+\Xmc^\mathbf k_{\mathbf c}(\tau_{\mathbf x}^{\mathbf i}- \tau_{\mathbf y}^{\overline{\mathbf i}})+\Zmc^\mathbf k_{\mathbf c}(\tau_{\mathbf x}^{\mathbf i}- \tau_{\mathbf y}^{\overline{\mathbf i}}).\]
Again by Theorem \ref{thm:main theorem} and the symmetry of the lozenge fields, $\mathcal{Z}_{\mathbf c}^\mathbf k(\tau_{\mathbf x}^{\mathbf i}- \tau_{\mathbf y}^{\overline{\mathbf i}})=0$. Also, if $\mathbf x:=(x_1,x_2,x_3)$ is a triple of points in $\partial\Gamma$ associated to $\mathbf P$ and $\mathbf d:=(d_1,d_2)$ is a triple of points in $\partial\Gamma$ associated to $\mathbf c$, then by (2) of Theorem \ref{thm:Darboux}, $\mathcal{Y}_{\mathbf c}^\mathbf k(F^{\mathbf i}_{\mathbf P})=a_1+a_2+a_3$, where
\[a_m=\left\{\begin{array}{ll}
1&\text{if }i_{m-1}=1, \mathbf k=\mathbf i_m'(1,-1) \text{ and }d_1=x_m;\\
-1&\text{if }i_{m-1}=1, \mathbf k=\mathbf i_m' \text{ and }d_1=x_m;\\
0&\text{otherwise,}
\end{array}\right.\]
and $\mathbf c=:(c_1,c_2)$. Since $\Xmc^\mathbf k_\mathbf c$ is a sum of eruption fields and
\[a_m=\left\{\begin{array}{ll}
1&\text{if }i_{m-1}=1, \mathbf k^3(-1,0,1)=\mathbf i_m \text{ and }d_1=x_m;\\
-1&\text{if }i_{m-1}=1, \mathbf k^3(0,-1,1)=\mathbf i_m \text{ and }d_1=x_m;\\
0&\text{otherwise,}
\end{array}\right.\] 
we can compute using (\ref{eqn:eruption length}) that $\mathcal{X}_{\mathbf c}^\mathbf k(\tau_{\mathbf x}^{\mathbf i}- \tau_{\mathbf y}^{\overline{\mathbf i}})=-a_1-a_2-a_3$. Thus, $\Ymc^\mathbf k_{\mathbf c}(G^{\mathbf i}_{\mathbf P})=\Ymc^\mathbf k_{\mathbf c}(F^{\mathbf i}_{\mathbf P})+\Xmc^\mathbf k_{\mathbf c}(\tau_{\mathbf x}^{\mathbf i}- \tau_{\mathbf y}^{\overline{\mathbf i}})=0$.


For (2) we set 
\[K^{\mathbf i}_{\mathbf P}:=\sum_{\mathbf j \in \overline{\Bmc}} c_{\mathbf i}^{\mathbf j} \cdot \big(\tau_{\mathbf x}^{\mathbf j} +\tau_{\mathbf y}^{\overline{\mathbf j}}\big).\] 
We need to show that for any special $(\Tmc,\Jmc)$-parallel vector field $\Xmc$,
\[\mathcal{X}\big(K^{\mathbf i}_{\mathbf P}\big)=
\begin{cases}
-1&\text{if }\mathcal{X}=\Hmc^{\mathbf i}_{\mathbf P};\cr
0 & \text{otherwise.}
\end{cases}\]

By Theorem \ref{thm:Darboux}(2), the derivative of $K^{\mathbf i}_{\mathbf P}$ in the direction of any twist field is zero. Theorem \ref{thm:Darboux}(2) and the anti-symmetry of the eruption fields, imply that the derivative of $K^{\mathbf i}_{\mathbf P}$ in the direction of any eruption vector field is zero. 

We now prove that $\Ymc^\mathbf k_{\mathbf c}(K^{\mathbf i}_{\mathbf P})=0$ for any length vector field $\Ymc^\mathbf k_{\mathbf c}$. Let $\mathbf x:=(x_1,x_2,x_3)$ be a triple of points in $\partial\Gamma$ associated to $\mathbf P$ and let $\mathbf d:=(d_1,d_2)$ be a pair of points in $\partial\Gamma$ associated to $\mathbf c$. This is an obvious consequence of Theorem \ref{thm:main theorem} when neither $d_1$ nor $d_2$ is a $\Gamma$-translate of $x_1,x_2$ or $x_3$. Thus, we can assume that either $d_1$ or $d_2$ is $x_1,x_2$ or $x_3$. Further assume without loss of generality that $c_1=x_1$; the other cases are similar. 

For any $\mathbf i:=(i_1,i_2,i_3)\in\overline{\Bmc}$, denote $\mathbf i'':=(i_1+i_2,0,i_3)$ and $\mathbf i''':=(i_1+i_3,i_2,0)$. Observe that if $\mathbf j:=(j_1,j_2,j_3)$ is a triple in $\overline{\Bmc}_1 \cup\overline{\Bmc}_2$, then 
\begin{equation}
\label{equation:cLozenge}
c_{\mathbf i}^{\mathbf j}= c_{\mathbf i}^{\mathbf j''}+ c_{\mathbf i}^{\mathbf j'''}
\end{equation}
(see Notation \ref{not:3}). Thus,
\[-c_{\mathbf i}^{\mathbf k^3(1,-1,0)}+ c_{\mathbf i}^{\mathbf k^3(0,-1,1)} = c_{\mathbf i}^{(n-1,0,1)} = c_{\mathbf i}^{\mathbf k^3(-1,0,1)}-c_{\mathbf i}^{\mathbf k^3}\]
(see Notation \ref{not:shift}). Theorem \ref{thm:main theorem} then implies that $\Zmc^\mathbf k_{\mathbf c}(K^{\mathbf i}_{\mathbf P})=0$. Since we already know that the derivative of $K^{\mathbf i}_{\mathbf P}$ is zero in the direction of any eruption field, this proves that $\Ymc^\mathbf k_{\mathbf c}(K^{\mathbf i}_{\mathbf P})=0$.

Next, we consider $\Hmc^{\mathbf j}_{\mathbf P'}(K^{\mathbf i}_{\mathbf P})$ for any hexagon field $\Hmc^{\mathbf j}_{\mathbf P'}$. Let $\mathbf y=:(y_1,y_2,y_3)$ be the triple of points in $\partial\Gamma$ associated to $\mathbf P'$. It follows from Theorem \ref{thm:main theorem} that if $[y_1,y_2,y_3]\neq[x_1,x_2,x_3]$ as ideal triangles in $\Theta$, then $\Hmc^{\mathbf j}_{\mathbf P'}(K^{\mathbf i}_{\mathbf P})=0$. Thus, we may assume without loss of generality that $\mathbf y=\mathbf x$. If $\mathbf j\neq \mathbf i$, then one of $\overline{\Bmc}_1 \cup \overline{\Bmc}_2$, $\overline{\Bmc}_2 \cup \overline{\Bmc}_3$ or $\overline{\Bmc}_3 \cup \overline{\Bmc}_1$ contains 
\[\mathbf j(0,1,-1),\,\,\,\mathbf j(-1,1,0),\,\,\,\mathbf j(-1,0,1),\,\,\,\mathbf j(0,-1,1),\,\,\,\mathbf j(1,-1,0)\,\,\,\text{ and }\,\,\,\mathbf j(1,0,-1).\]
Without loss of generality, suppose that they lie in $\overline{\Bmc}_1 \cup \overline{\Bmc}_2$. Then by (\ref{equation:cLozenge}),
\begin{eqnarray*}
\Hmc^{\mathbf j}_{\mathbf P}(K^{\mathbf i}_{\mathbf P})&=&c_{\mathbf i}^{\mathbf j(0,1,-1)} - c_{\mathbf i}^{\mathbf j(-1,1,0)} + c_{\mathbf i}^{\mathbf j(-1,0,1)} - c_{\mathbf i}^{\mathbf j(0,-1,1)} + c_{\mathbf i}^{\mathbf j(1,-1,0)} - c_{\mathbf i}^{\mathbf j(1,0,-1)}
\\&=&c_{\mathbf i}^{\mathbf j''(1,0,-1)}+c_{\mathbf i}^{\mathbf j'''(-1,1,0)} - c_{\mathbf i}^{\mathbf j''}-c_{\mathbf i}^{\mathbf j'''(-1,1,0)} + c_{\mathbf i}^{\mathbf j''(-1,0,1)}+c_{\mathbf i}^{\mathbf j'''} 
\\&&- c_{\mathbf i}^{\mathbf j''(-1,0,1)}-c_{\mathbf i}^{\mathbf j'''(1,-1,0)} + c_{\mathbf i}^{\mathbf j''}+ c_{\mathbf i}^{\mathbf j'''(1,-1,0)} - c_{\mathbf i}^{\mathbf j''(1,0,-1)}-c_{\mathbf i}^{\mathbf j'''}
\\&=&0.
\end{eqnarray*}
Finally, if $\mathbf k=\mathbf i$, then
\begin{eqnarray*}
\Hmc^{\mathbf i}_{\mathbf P}(K^{\mathbf i}_{\mathbf P})&=&2\left(c_{\mathbf i}^{\mathbf i(0,1,-1)} - c_{\mathbf i}^{\mathbf i(-1,1,0)} + c_{\mathbf i}^{\mathbf i(-1,0,1)} - c_{\mathbf i}^{\mathbf i(0,-1,1)} + c_{\mathbf i}^{\mathbf i(1,-1,0)} - c_{\mathbf i}^{\mathbf i(1,0,-1)}\right)\\
&=&2\left(\frac{i_2i_1+(i_1+i_2)(i_3-1)}{2n}-\frac{i_2(i_1-1)+(i_1+i_2)i_3}{2n}\right.\\
&&+\frac{i_3i_2+(i_2+i_3)(i_1-1)}{2n}-\frac{i_3(i_2-1)+(i_2+i_3)i_1}{2n}\\
&&\left.+\frac{i_1i_3+(i_1+i_3)(i_2-1)}{2n}-\frac{i_1(i_3-1)+(i_1+i_3)i_2}{2n}\right)\\
&=&-1.
\end{eqnarray*}
\end{proof}

\appendix
\section{Proof of Proposition \ref{prop:Fock-Goncharov}}\label{app:Fock-Goncharov}
In this appendix, we give a proof of Proposition \ref{prop:Fock-Goncharov}, which we restate here for the convenience of the reader. 

\begin{prop*}
Let $(F_1,\dots,F_k)$ be a positive $k$-tuple of flags in $\Fmc(V)$. Then for any positive integers $n_1$, $\dots$, $n_k$ that sum to $d\leq n$, we have that 
\[\dim\left(\sum_{j=1}^kF_j^{(n_j)}\right)=d.\]
\end{prop*}

We start with the following two notions.

\begin{definition}\
\begin{enumerate}
\item A map $\zeta:S^1\to\Pbbb(V)$ is \emph{convex} if $\zeta(x_1)+\dots+\zeta(x_k)$ is a direct sum for all $k\leq n:=\dim(V)$ and for all pairwise distinct $x_1$, \dots, $x_k$ in $S^1$.
\item Let $F$ be a flag in $\Fmc(V)$ and let $\zeta:S^1\to\Pbbb(V)$ be a convex curve. We say $\zeta$ \emph{osculates} $F$ if there is some point $x$ in $S^1$ such that $\zeta(x)=F^{(1)}$, and for all $l=1$, \dots, $d-1$, 
\[\lim_{i\to\infty}\zeta(x_{1,i})+..+\zeta(x_{l,i})=F^{(l)}\]
for all pairwise distinct points $l$-tuples of points $(x_{1,i},...,x_{l,i})$ in $S^1$ such that $\lim_{i\to\infty}(x_{1,i},\dots,x_{l,i})=(x,\dots,x)$.
\end{enumerate}
\end{definition}

For the proof, we use the following facts. The first is due to Fock and Goncharov \cite[Theorem 1.3 and Section 9.11]{FockGoncharov}. 

\begin{thm}[\cite{FockGoncharov}]
Let $(F_1,\dots,F_k)$ be a positive triple of flags in $\Fmc(V)$. Then there is a continuous convex map $\zeta:S^1\to\Fmc(V)$ that osculates $F_l$ for all $l=1$, \dots, $k$.
\end{thm}

The second is an immediate consequence of Proposition \ref{prop:parameterize triple}.

\begin{prop}
The space $\Fmc(V)^3_+$ of positive triples of flags in $V$ is a connected component of the space of $\Fmc(V)^{[3]}$ of generic triples of flags in $V$.
\end{prop}

Using these, we prove the following key lemma.

\begin{lem}\label{lem:induction}
Let $k\geq 4$ be an integer, let $(F_1,\dots,F_k)$ be a positive $k$-tuple of flags in $\Fmc(V)$ and let $m=1$, \dots, $n-1$. Let $H$ be the flag defined by
\[H^{(i)}:=\left\{\begin{array}{ll}
F_2^{(i)}&\text{if } i\leq m;\\
F_2^{(m)}+F_3^{(i-m)}&\text{if }i>m.
\end{array}\right.\]
Then $(F_1,H,F_4,\dots,F_k)$ is a positive $(k-1)$-tuple of flags.
\end{lem}

\begin{proof}
Note that $(F_1,F_2,F_4,\dots,F_k)$ is a positive $k$-tuple of flags and $H^{(1)}=F_2^{(1)}$. Thus, by Proposition \ref{prop:Fock-Goncharov parametrization}, it is sufficient to prove that $(F_1,H,F_4)$ is a positive triple of flags.

Let $\zeta$ be a continuous convex map that osculates the flags $F_1$, \dots, $F_4$. For all $l=1$, $\dots$, $4$, let $x_l$ be the point in $S^1$ such that $\zeta(x_l)=F_l^{(1)}$. It is a consequence of Proposition \ref{prop:Fock-Goncharov parametrization} that $\Fmc(V)^4_+$ is open in the space of pairwise transverse quadruple of flags. Thus, there are pairwise disjoint open intervals $I_1$, \dots, $I_4\subset S^1$ such that $I_l$ contains $x_l$, and satisfy the following property: For all cyclically ordered $(n-1)$-tuple of points $\mathbf y_l:=(y_{l,1},\dots,y_{l,n-1})$ in $I_l$, let $F_{\mathbf y_l}$ be the flag in $\Fmc(V)$ given by $F_{\mathbf y_l}^{(i)}:=\sum^i_{j=1}\zeta(y_{l,j})$. Then $(F_{\mathbf y_1},\dots,F_{\mathbf y_4})$ is a positive quadruple of flags. 

Now, for each $i=1$, \dots, $n-1$, let $f_i:[0,1]\to S^1$ be continuous paths such that
\begin{itemize}
\item $f_i(0)=y_{2,i}$
\item $f_i(1)=\left\{\begin{array}{ll}
y_{2,i}&\text{if } i\leq m\\
y_{3,i-m}&\text{if }i>m
\end{array}\right.$
\item $f_i(0)\leq f_i(t)\leq f_i(1)\leq f_i(0)$ for all $t$ in $[0,1]$.
\item $\big(f_1(t),\dots,f_{n-1}(t)\big)$ is a cyclically ordered $(n-1)$-tuple of points in $\partial\Gamma$ for all $t$ in $[0,1]$.
\end{itemize}
Such paths $f_i$ exist because $(\mathbf y_2,\mathbf y_3)$ is a cyclically ordered $(2n-2)$-tuple of points.

Since $\zeta$ is convex, we may define $G(t)=G(\mathbf y_2,\mathbf y_3,t)$ to be the flag in $\Fmc(V)$ given by $G(t)^{(j)}:=\sum^j_{i=1}\zeta(f_i(t))$. Since $G(0)=F_{\mathbf y_2}$, we see that $(F_{\mathbf y_1},G(0),F_{\mathbf y_4})$ is a positive triple of flags. The convexity of $\zeta$ also implies that $(F_{\mathbf y_1},G(t),F_{\mathbf y_4})$ is a generic for all $t$. Since $\Fmc(V)^3_+$ is a connected component of $\Fmc(V)^{[3]}$ and $t\mapsto G(t)$ is continuous, it follows that $(F_{\mathbf y_1},G(1),F_{\mathbf y_4})$ is also a positive triple of flags.

Observe that 
\begin{itemize}
\item for $l=1$, $4$, the limit of $F_{\mathbf y_l}$ as $\mathbf y_l$ converges to $(x_l,\dots,x_l)$ is $F_l$,
\item the limit of $G(\mathbf y_2,\mathbf y_3,1)$ as $(\mathbf y_2,\mathbf y_3)$ converges to $(x_2,\dots,x_2,x_3,\dots,x_3)$ is $H$.
\end{itemize}
Since $\Fmc(V)^3_+$ is closed in the space of pairwise transverse triple of flags, this implies that $(F_1,H,F_4)$ is a positive triple of flags.
\end{proof}

\begin{proof}[Proof of Proposition \ref{prop:Fock-Goncharov}]
We prove this by induction on $k$. For the base case $k=3$, this is an immediate consequence of Proposition \ref{prop:parameterize triple}.

Next, we prove the inductive step. Suppose that $k\geq 4$. Let $H$ be the flag in $\Fmc(V)$ defined by
\[H^{(i)}:=\left\{\begin{array}{ll}
F_2^{(i)}&\text{if } i\leq n_2;\\
F_2^{(n_2)}+F_3^{(i-n_2)}&\text{if }i> n_2.
\end{array}\right.\]
By Lemma \ref{lem:induction}, we see that $(F_1,H,F_4,\dots,F_k)$ is a positive $(k-1)$-tuple of flags. We further observe that 
\[\sum_{j=1}^kF_j^{(n_j)}=F_1^{(n_1)}+H^{(n_2+n_3)}+\sum_{j=4}^kF_j^{(n_j)}.\] 
The statement now follows by applying the inductive hypothesis to the right hand side.
\end{proof}

\section{Proof of Lemma \ref{lem:technical2}}\label{app:technical}

In this appendix, we prove Lemma \ref{lem:technical2}, which we restate here for the reader's convenience.

\begin{lem*}
Let $\Vmc$ denote the set of vertices of $\widetilde{\Tmc}$, and for non-negative integers $j$, let $\xi_j$ be a Frenet curve in $\widetilde{\mathrm{Fre}}(V)$. If $\lim_{j\to\infty}\xi_j(p)=\xi_0(p)$ for all vertices $p$ in $\widetilde{\Vmc}$, then $\lim_{j\to\infty}\xi_j(p)=\xi_0(p)$ for all points $p$ in $\partial\Gamma$.
\end{lem*}
 
Let $x,y,z$ be the vertices of a triangle in $\widetilde{\Theta}$, let $w$ be any point in $\partial\Gamma\setminus\Vmc$, and assume without loss of generality that $y<z<x<w<y$ in this cyclic order in $S^1$. By Proposition \ref{prop:Fock-Goncharov parametrization}, to conclude that $\lim_{j\to\infty}\xi_j(w)=\xi_0(w)$, it is sufficient to show that 
\begin{enumerate}
\item For any pair $\mathbf i$ of positive integers that sum to $n$,
\[C^{\mathbf i}(\xi_0(x),\xi_0(z),\xi_0(y),\xi_0(w))=\lim_{j\to\infty}C^{\mathbf i}(\xi_j(x),\xi_j(z),\xi_j(y),\xi_j(w)).\]
\item For any triple of positive integers $\mathbf i$ that sum to $n$, 
\[ T^{\mathbf i}(\xi_0(x),\xi_0(w),\xi_0(y))=\lim_{j\to\infty}T^{\mathbf i}(\xi_j(x),\xi_j(w),\xi_j(y)).\]
\end{enumerate}

Since $\Vmc$ is dense in $\partial\Gamma$, there are sequences $\{a_k\}_{k=1}^\infty$ and $\{b_k\}_{k=1}^\infty$ in $\Vmc$ such that $\lim_{k\to\infty}a_k=\lim_{k\to\infty}b_k=w$ and $y<z<x<a_k<w<b_k<y$ for all $k$. 

\begin{proof}[Proof of (1)]

Since $\xi_0$ is continuous, we have
\begin{eqnarray*}
\lim_{k\to\infty}C^{\mathbf i}(\xi_0(x),\xi_0(z),\xi_0(y),\xi_0(a_k))&=&\lim_{k\to\infty}C^{\mathbf i}(\xi_0(x),\xi_0(z),\xi_0(y),\xi_0(b_k))\\
&=&C^\mathbf {i}(\xi_0(x),\xi_0(z),\xi_0(y),\xi_0(w))
\end{eqnarray*}
for all pairs of positive integers $\mathbf i$ that sum to $n$. Furthermore, it is also well-known (see for example Proposition 2.12 of \cite{Zhang_internal}) that 
\[C^\mathbf i(\xi_j(x),\xi_j(a_k),\xi_j(y),\xi_j(w)),\,\,\,\,C^\mathbf i(\xi_j(x),\xi_j(w),\xi_j(y),\xi_j(b_k))>1\]
for all non-negative integers $j$ and all positive integers $k$. In particular, 
\[\begin{array}{l}
C^{\mathbf i}(\xi_j(x),\xi_j(z),\xi_j(y),\xi_j(a_k))>C^{\mathbf i}(\xi_j(x),\xi_j(z),\xi_j(y),\xi_j(w))\\
\hspace{7cm}>C^{\mathbf i}(\xi_j(x),\xi_j(z),\xi_j(y),\xi_j(b_k)).
\end{array}\]

Since $\lim_{j\to\infty}\xi_j(p)=\xi_0(p)$ for any vertex $p$ in $\Vmc$, we see that 
\begin{eqnarray*}
C^{\mathbf i}(\xi_0(x),\xi_0(z),\xi_0(y),\xi_0(w))&=&\lim_{k\to\infty}C^{\mathbf i}(\xi_0(x),\xi_0(z),\xi_0(y),\xi_0(a_k))\\
&=&\lim_{k\to\infty}\lim_{j\to\infty}C^{\mathbf i}(\xi_j(x),\xi_j(z),\xi_j(y),\xi_j(a_k))\\
&\geq&\lim_{j\to\infty}C^{\mathbf i}(\xi_j(x),\xi_j(z),\xi_j(y),\xi_j(w))\\
&\geq&\lim_{k\to\infty}\lim_{j\to\infty}C^{\mathbf i}(\xi_j(x),\xi_j(z),\xi_j(y),\xi_j(b_k))\\
&=&\lim_{k\to\infty}C^{\mathbf i}(\xi_0(x),\xi_0(z),\xi_0(y),\xi_0(b_k))\\
&=&C^{\mathbf i}(\xi_0(x),\xi_0(z),\xi_0(y),\xi_0(w)).
\end{eqnarray*}
This proves (1).
\end{proof}

To prove (2), we need to use the following lemma. Let $\mathbf i:=(i_1,i_2,i_3)$ be a triple of positive integers that sum to $n$. For any generic quadruple of flags $F_1,F_2,F_3,F_4$ in $\Fmc(V)$, define 
\[\begin{array}{l}T^{(i_1,i_2,F_4,i_3)}(F_1,F_2,F_3):=\\
\displaystyle\frac{F_1^{(i_1+1)}\wedge F_2^{(i_2)}\wedge F_3^{(i_3-1)}\cdot F_1^{(i_1-1)}\wedge F_2^{(i_2)}\wedge F_4^{(1)}\wedge F_3^{(i_3)}\cdot F_1^{(i_1)}\wedge F_2^{(i_2-1)}\wedge F_3^{(i_3+1)}}{F_1^{(i_1+1)}\wedge F_2^{(i_2-1)}\wedge F_3^{(i_3)}\cdot F_1^{(i_1)}\wedge F_2^{(i_2)}\wedge F_4^{(1)}\wedge F_3^{(i_3-1)}\cdot F_1^{(i_1-1)}\wedge F_2^{(i_2)}\wedge F_3^{(i_3+1)}}
\end{array}\]

\begin{lem}\label{lem:monotonicity triple ratio}
Let $\xi:S^1\to\Fmc(V)$ be a Frenet curve and let $x_1<x_4<x_2<x_5<x_3<x_1$ lie in $S^1$ in this cyclic order. For each $m=1$, $\dots$, $5$, let $F_m:=\xi(x_m)$, then
\[T^{(i_1,i_2,F_4,i_3)}(F_1,F_2,F_3)>T^{\mathbf i}(F_1,F_2,F_3)>T^{(i_1,i_2,F_5,i_3)}(F_1,F_2,F_3).\]
(Recall that we assume $\dim(V)\geq 3$.)
\end{lem}

\begin{proof}
Let $K:=F_1^{(i_1-1)}+F_2^{(i_2-1)}+F_3^{(i_3-1)}$. For $m=1$, $2$, $3$, let $L_{m}\subset V$ be a line such that $F_m^{(i_m-1)}+L_{m}=F_m^{(i_m)}$, and let $P_{m}\subset V$ be a plane such that $F_m^{(i_m-1)}+P_{m}=F_m^{(i_m+1)}$. For any point $x$ in $S^1$, let
\[L_x:=\left\{\begin{array}{ll}
\xi^{(1)}(x)&\text{if }x\neq x_1,\,x_2,\,x_3\\
L_{m}&\text{if }x=x_m; m=1,\,2,\,3
\end{array}\right.,\,\,\,P_x:=\left\{\begin{array}{ll}
\xi^{(2)}(x)&\text{if }x\neq x_1,\,x_2,\,x_3\\
P_{m}&\text{if }x=x_m; m=1,\,2,\,3
\end{array}\right.\]
and let $H:=L_{x_1}+L_{x_2}+L_{x_3}(=L_1+L_2+L_3)$. Then define $\xi':S^1\to\Fmc(H)$ by
\[\xi'^{(1)}(x):=(K+L_x)\cap H, \,\,\,\,\,\xi'^{(2)}(x):=(K+P_x)\cap H.\]
One can verify that $\xi'$ does not depend on the choices of $L_{m}$ and $P_{m}$, and is Frenet. 

Furthermore, from the definition of the triple ratio, we see that 
\begin{eqnarray*}
T^{(i_1,i_2,F_4,i_3)}(F_1,F_2,F_3)&=&T^{(1,1,\xi'(x_4),1)}(\xi'(x_1),\xi'(x_2),\xi'(x_3)),\\
T^{\mathbf i}(F_1,F_2,F_3)&=&T^{(1,1,1)}(\xi'(x_1),\xi'(x_2),\xi'(x_3))\\
T^{(i_1,i_2,F_5,i_3)}(F_1,F_2,F_3)&=&T^{(1,1,\xi'(x_5),1)}(\xi'(x_1),\xi'(x_2),\xi'(x_3)).
\end{eqnarray*}
Thus, it is sufficient to prove this lemma in the case when $\dim(V)=3$. That is a straightforward computation (see Proposition 2.3.4 of \cite{Zhang_thesis}).
\end{proof}

\begin{proof}[Proof of (2)]
The Frenet property of $\xi_0$ implies that
\begin{eqnarray*}
\lim_{k\to\infty}T^{(i_1,i_2,\xi_0(a_k),i_3)}(\xi_0(x),\xi_0(w),\xi_0(y))
&=&T^{\mathbf i}(\xi_0(x),\xi_0(w),\xi_0(y))\\
 &=&\lim_{k\to\infty}T^{(i_1,i_2,\xi_0(b_k),i_3)}(\xi_0(x),\xi_0(w),\xi_0(y))
\end{eqnarray*}
for all ordered triples of positive integers $\mathbf i:=(i_1,i_2,i_3)$ that sum to $n$. Also, by Lemma \ref{lem:monotonicity triple ratio}, we have
\[\begin{array}{l}
T^{(i_1,i_2,\xi_j(a_k),i_3)}(\xi_j(x),\xi_j(w),\xi_j(y))>T^{\mathbf i}(\xi_j(x),\xi_j(w),\xi_j(y))\\
\hspace{7cm}>T^{(i_1,i_2,\xi_j(b_k),i_3)}(\xi_j(x),\xi_j(w),\xi_j(y))
\end{array}\]
for all non-negative integers $j$ and all positive integers $k$. 

Since $\lim_{j\to\infty}\xi_j(p)=\xi_0(p)$ for all vertices $p$ in $\Vmc$, this implies that
\begin{eqnarray*}
T^{\mathbf i}(\xi_0(x),\xi_0(w),\xi_0(y))&=&\lim_{k\to\infty}T^{(i_1,i_2,\xi_0(a_k),i_3)}(\xi_0(x),\xi_0(w),\xi_0(y))\\
&=&\lim_{k\to\infty}\lim_{j\to\infty}T^{(i_1,i_2,\xi_j(a_k),i_3)}(\xi_j(x),\xi_j(w),\xi_j(y))\\
&\geq&\lim_{j\to\infty}T^{\mathbf i}(\xi_j(x),\xi_j(w),\xi_j(y))\\
&\geq&\lim_{k\to\infty}\lim_{j\to\infty}T^{(i_1,i_2,\xi_j(b_k),i_3)}(\xi_j(x),\xi_j(w),\xi_j(y))\\
&=&\lim_{k\to\infty}T^{(i_1,i_2,\xi_0(b_k),i_3)}(\xi_0(x),\xi_0(w),\xi_0(y))\\
&=&T^{\mathbf i}(\xi_0(x),\xi_0(w),\xi_0(y)).
\end{eqnarray*}
This proves (2).
\end{proof}

\bibliographystyle{amsalpha}
\bibliography{ref}

\end{document}